\documentclass[reqno,12pt]{amsart}

\usepackage{color} 
\usepackage[pdftex]{graphicx}
\usepackage[pdftex]{hyperref}
\hypersetup{
    unicode=false,          
    pdftoolbar=true,        
    pdfmenubar=true,        
    pdffitwindow=false,     
    pdfstartview={FitH},    
    pdftitle={MCP Article},      
    pdfauthor={Michael Holst},   
    pdfsubject={Mathematics},    
    pdfcreator={Michael Holst},  
    pdfproducer={Michael Holst}, 
    pdfkeywords={PDE, analysis, mathematical physics}, 
    pdfnewwindow=true,      
    colorlinks=true,        
    linkcolor=red,          
    citecolor=blue,         
    filecolor=magenta,      
    urlcolor=cyan           
}
\usepackage{times}
\usepackage{amsfonts}
\usepackage{amsmath}
\usepackage{amsthm}
\usepackage{amssymb}
\usepackage{amsbsy}
\usepackage{amscd}
\usepackage{enumerate}
\usepackage{mathtools}
\usepackage{pgf}
\usepackage{tikz}
\usepackage{tikz-cd}
\usepackage{mdframed}
\usepackage{pbox}
\usepackage{array}
\usepackage{float}
\newtheorem{theorem}{Theorem}[section]
\newtheorem{corollary}[theorem]{Corollary}
\newtheorem{lemma}[theorem]{Lemma}

\newtheorem{definition}[theorem]{Definition}
\newtheorem{remark}[theorem]{Remark}

\numberwithin{equation}{section}

\newcounter{mnote}
\let\oldmarginpar\marginpar
  \renewcommand\marginpar[1]{\-\oldmarginpar[\raggedleft\footnotesize #1]%
  {\raggedright\footnotesize #1}}

\newcommand{\grad}{\nabla}
\newcommand{\reals}{\mathbb{R}}

\newcommand{\lab}{\label}

\newcommand{\floor}[1]{\lfloor{#1}\rfloor}

\definecolor{mve}{rgb}{0.7,0.35,0.15}
\definecolor{brght}{rgb}{0.825,0.2625,0.15}
\definecolor{yello}{rgb}{1,0.925,0.65}
\definecolor{bluu}{rgb}{0.65, 0.95, 1}
\definecolor{bluu2}{rgb}{0.2, 0.5, 0.8}

\newcommand{\brpt}[1]{{{\textcolor{brght}{#1}}}} 

\DeclareMathOperator*{\esssup}{ess\,sup}

\newenvironment{itemizeX}
{\begin{list}{\labelitemi}
 {\setlength{\leftmargin}{1.5em}
  \setlength{\topsep}{0.5em}
  \setlength{\itemsep}{0.5em}
  \setlength{\labelwidth}{50.0em}}}
 {\end{list}}

\newenvironment{enumerateX}
{\begin{list}{(\arabic{enumi})}
 {\usecounter{enumi}
  \setlength{\leftmargin}{2.5em}
  \setlength{\topsep}{0.5em}
  \setlength{\itemsep}{0.5em}
  \setlength{\labelwidth}{50.0em}}}
 {\end{list}}

\newenvironment{enumerateXALI}
{\begin{list}{(\arabic{enumi})}
 {\usecounter{enumi}
  \setlength{\leftmargin}{1em}
  \setlength{\topsep}{0.5em}
  \setlength{\itemsep}{0.5em}
  \setlength{\labelwidth}{50.0em}}}
 {\end{list}}
 \newenvironment{itemizeXALI}
{\begin{list}{\labelitemi}
 {\setlength{\leftmargin}{0.4em}
  \setlength{\topsep}{0.5em}
  \setlength{\itemsep}{0.5em}
  \setlength{\labelwidth}{50.0em}}}
 {\end{list}}
\newenvironment{itemizeXXALI}
{\begin{list}{$\circ$}
 {\setlength{\leftmargin}{0.5em}
  \setlength{\topsep}{0.5em}
  \setlength{\itemsep}{0.5em}
  \setlength{\labelwidth}{50.0em}}}
 {\end{list}}
\begin{document}

\title[Sobolev Spaces on Compact Manifolds]
      {Sobolev-Slobodeckij Spaces on Compact Manifolds, Revisited}

\author[A. Behzadan]{A. Behzadan}
\email{abehzada@ucsd.edu}

\author[M. Holst]{M. Holst}
\email{mholst@ucsd.edu}

\address{Department of Mathematics\\
         University of California San Diego\\
         La Jolla CA 92093}

\thanks{AB was supported by NSF Award~1262982.}
\thanks{MH was supported in part by
        NSF Awards~1262982, 1318480, and 1620366.}

\date{}
\keywords{Sobolev spaces, Compact manifolds, Tensor bundles,
          Differential operators}

\begin{abstract}
In this article we present a coherent rigorous overview of the main properties of Sobolev-Slobodeckij spaces of sections of vector bundles on compact manifolds; results of this type are scattered through the literature and can be difficult to find.
A special emphasis has been put on spaces with noninteger smoothness order, and a special attention has been paid to the peculiar fact that for a general nonsmooth domain $\Omega$ in $\reals^n$, $0<t<1$, and $1<p<\infty$, it is not necessarily true that $W^{1,p}(\Omega)\hookrightarrow W^{t,p}(\Omega)$.
This has dire consequences in the multiplication properties of Sobolev-Slobodeckij spaces and subsequently in the study of Sobolev spaces on manifolds. To the authors' knowledge, some of the proofs, especially those that are pertinent to the properties of Sobolev-Slobodeckij spaces of sections of general vector bundles, cannot be found in the literature in the generality appearing here.
\end{abstract}

\maketitle
{\footnotesize \tableofcontents }


\vspace*{-0.75cm}
\section{Introduction}
 \label{sec:intro}
 Suppose $s\in\reals$ and $p\in (1,\infty)$. With each nonempty open set $\Omega$ in $\reals^n$ we can associate a complete normed function space denoted by $W^{s,p}(\Omega)$ called the Sobolev-Slobodeckij space with smoothness degree $s$ and integrability degree $p$. Similarly, given a compact smooth manifold $M$ and a vector bundle $E$ over $M$, there are several ways to define the normed spaces $W^{s,p}(M)$ and more generally $W^{s,p}(E)$. The main goal of this manuscript is to review these various definitions and rigorously study the key properties of these spaces. Some of the properties that we are interested in are as follows:
\begin{itemizeX}
\item Density of smooth functions
\item Completeness, separability, reflexivity
\item Embedding properties
\item Behavior under differentiation
\item Being closed under multiplication by smooth functions
\begin{equation*}
u\in W^{s,p},\quad \textrm{$\varphi$ is smooth}
\stackrel{\brpt{?}}{\Longrightarrow}\varphi u\in W^{s,p}
\end{equation*}
\item Invariance under change of coordinates
\begin{equation*}
u\in W^{s,p},\quad \textrm{$T$ is a
diffeomorphism}\stackrel{\brpt{?}}{\Longrightarrow} u\circ T\in
W^{s,p}
\end{equation*}
\item Invariance under composition by a smooth function
\begin{equation*}
u\in W^{s,p},\quad \textrm{$F$ is
smooth}\stackrel{\brpt{?}}{\Longrightarrow} F(u)\in W^{s,p}
\end{equation*}
\end{itemizeX}
As we shall see, there are several ways to define $W^{s,p}(E)$. In particular, $\|u\|_{W^{s,p}(E)}$ can be defined using the components of the local representations of $u$ with respect to a fixed augmented total trivialization atlas $\Lambda$, or it can be defined using the notion of connection in $E$. Here are some of the questions that we have studied in this paper regarding this issue:
\begin{itemizeX}
\item Are the different characterizations that exist in the literature equivalent? If not, what is the relationship between the various characterizations of Sobolev-Slobodeckij spaces on $M$?
\item In particular, does the corresponding space depend on the chosen atlas (more precisely the chosen augmented total trivialization atlas) used in the definition?
\item Suppose $f\in W^{s,p}(M)$. Does this imply that the local representation of $f$ with respect to each chart $(U_\alpha,\varphi_\alpha)$ is in $W^{s,p}(\varphi_\alpha(U_\alpha))$? If $g$ is a metric on $M$ and $g\in W^{s,p}$, can we conclude that $g_{ij}\in W^{s,p}(\varphi_\alpha(U_\alpha))$?
\item Suppose that $P:C^\infty(M)\rightarrow C^{\infty}(M)$ is a linear differential operator. Is it possible to gain information about the mapping properties of $P$ by studying the mapping properties of its local representations with respects to charts in a given atlas? For example, suppose that the local representations of $P$ with respect to each chart $(U_\alpha,\varphi_\alpha)$ in an atlas is continuous from $W^{s,p}(\varphi_\alpha(U_\alpha))$ to $W^{\tilde{s},\tilde{p}}(\varphi_\alpha(U_\alpha))$. Is it possible to extend $P$ to a continuous linear map from $W^{s,p}(M)$ to $W^{\tilde{s},\tilde{p}}(M)$?
\end{itemizeX}
To further motivate the questions that are studied in this paper and the study of the key properties mentioned above, let us consider a concrete example. For any two sets $A$ and $B$, let $\textrm{Func}(A,B)$ denote the collection of all functions from $A$ to $B$. Consider the differential operator
\begin{equation*}
\textrm{div}_g: C^{\infty}(TM)\rightarrow \textrm{Func}(M,\reals), \qquad \textrm{div}\,X=(\textrm{tr}\circ \textrm{sharp}_g\circ \grad\circ \textrm{flat}_g) X
\end{equation*}
on a compact Riemannian manifold $(M,g)$ with $g\in W^{s,p}$. Let $\{(U_\alpha,\varphi_\alpha)\}$ be a smooth atlas for $M$. It can be shown that for each $\alpha$
\begin{equation*}
(\textrm{div}_g X)\circ
\varphi_\alpha^{-1}=\sum_{j=1}^n\frac{1}{\sqrt{\textrm{det}\,g_\alpha}}\frac{\partial}{\partial
x^j}\big[(\sqrt{\textrm{det}\,g_\alpha}) (X^j\circ
\varphi_\alpha^{-1}) \big]
\end{equation*}
where $g_\alpha(x)$ is the matrix whose $(i,j)$-entry is $(g_{ij}\circ \varphi_\alpha^{-1})(x)$.
As it will be discussed in detail in Section 10, we call $Q^\alpha:
C^{\infty}(\varphi_\alpha(U_\alpha),\reals^n)\rightarrow
\textrm{Func}(\varphi_\alpha(U_\alpha),\reals)$ defined by
\begin{equation*}
Q^\alpha(Y)=\sum_{j=1}^n\underbrace{\frac{1}{\sqrt{\textrm{det}\,g_\alpha}}\frac{\partial}{\partial
x^j}\big[(\sqrt{\textrm{det}\,g_\alpha}) (Y^j) \big]}_{\brpt{Q^\alpha_j(Y^j)}}
\end{equation*}
the \textit{local representation} of $\textrm{div}_g$  with respect to the local chart $(U_\alpha,\varphi_\alpha)$. Let's say we can prove that for each $\alpha$ and $j$, $Q^\alpha_j$ maps $W^{e,q}_{0}(\varphi_\alpha(U_\alpha))$ to $W^{e-1,q}(\varphi_\alpha(U_\alpha))$. Can we conclude that $\textrm{div}_g$ maps $W^{e,q}(TM)$ to $W^{e-1,q}(M)$? And how can we find exponents $e$ and $q$ such that
\begin{equation*}
Q^\alpha_j: W^{e,q}_{0}(\varphi_\alpha(U_\alpha)) \rightarrow W^{e-1,q}(\varphi_\alpha(U_\alpha))
\end{equation*}
is a well-defined continuous map? We will see how the properties we mentioned above play a key role in answering these questions.\\

Since $W^{0,p}(\Omega)=L^p(\Omega)$, Sobolev-Slobodeckij spaces can be viewed as a generalization of classical Lebesgue spaces. Of course, unlike Lebesgue spaces, some of the key properties of $W^{s,p}(\Omega)$ (for $s\neq 0$) depend on the geometry of the boundary of $\Omega$. Indeed, to thoroughly study the properties of $W^{s,p}(\Omega)$ one should consider the following cases independently:
 \begin{enumerate}
\item $\Omega=\reals^n$
\item $\Omega$ is an arbitrary open subset of $\reals^n$ $\begin{cases}2a) \textrm{bounded}\\2b) \textrm{unbounded}\end{cases}$
\item $\Omega$ is an open subset of $\reals^n$ with
smooth
boundary $\begin{cases}3a) \textrm{bounded}\\3b) \textrm{unbounded}\end{cases}$
\end{enumerate}
Let us mention here four facts to highlight the dependence on domain and some atypical behaviors of certain fractional Sobolev spaces. Let $s\in (0,\infty)$ and $p\in (1,\infty)$.
\begin{itemizeX}
\item  \textbf{Fact 1: }
\begin{align*}
\forall\,j \qquad\frac{\partial}{\partial x^j} : W^{s,p}(\reals^n)\rightarrow W^{s-1,p}(\reals^n)
\end{align*}
is a well-defined bounded linear operator.
\item \textbf{Fact 2: } If we further assume that $s\neq \frac{1}{p}$ and $\Omega$ has smooth boundary
then
\begin{equation*}
\forall\,j \qquad\frac{\partial}{\partial x^j} : W^{s,p}(\Omega)\rightarrow W^{s-1,p}(\Omega)
\end{equation*}
is a well-defined bounded linear operator.
\item  \textbf{Fact 3: } If $\tilde{s}\leq s$, then
\begin{align*}
W^{s,p}(\reals^n)\hookrightarrow W^{\tilde{s},p}(\reals^n)\,.
\end{align*}
${}$
\item  \textbf{Fact 4: } If $\Omega$ does NOT have Lipschitz boundary, then it is NOT necessarily true that
\begin{align*}
W^{1,p}(\Omega)\hookrightarrow W^{\tilde{s},p}(\Omega)
\end{align*}
for $0<\tilde{s}<1$.
\end{itemizeX}
Let $M$ be an $n$-dimensional compact smooth manifold and let $\{(U_\alpha,\varphi_\alpha)\}$ be a smooth atlas for $M$. As we will see, the properties of Sobolev-Slobodeckij spaces of sections of vector bundles on $M$ are closely related to the properties of spaces of locally Sobolev-Slobodeckij functions on domains in $\reals^n$. Primarily we will be interested in the properties of $W^{s,p}(\varphi_{\alpha}(U_\alpha))$ and  $W^{s,p}_{loc}(\varphi_{\alpha}(U_\alpha))$. Also when we want to patch things together consistently and move from "local" to "global", we will need to consider spaces $W^{s,p}(\varphi_{\alpha}(U_\alpha\cap U_\beta))$ and $W^{s,p}(\varphi_{\beta}(U_\alpha\cap U_\beta))$. However, as we pointed out earlier, some of the properties of $W^{s,p}(\Omega)$ depend heavily on the geometry of the boundary of $\Omega$. Considering that the intersection of two Lipschitz domains is not necessarily a Lipschitz domain, we need to consider the following question:
\begin{itemizeX}
\item Is it possible to find an atlas such that the image of each coordinate domain in the atlas (and the image of the intersection of any two coordinate domains in the atlas) under the corresponding coordinate map is either the entire $\reals^n$ or a nonempty bounded set with smooth boundary? And if we define the Sobolev spaces using such an atlas, will the results be independent of the chosen atlas?
\end{itemizeX}
This manuscript is an attempt to collect some results concerning these questions and certain other fundamental questions similar to the ones stated above, and we pay special attention to spaces with noninteger smoothness order and to general sections of vector bundles.
There are a number of standard sources for properties of integer order
Sobolev spaces of functions and related elliptic operators on domains
in $\mathbb{R}^n$ (cf.~\cite{32, Evans2010, 18}),
real order Sobolev spaces of functions (\cite{Gris85,Trie83,37,12,33}),
Sobolev spaces of functions on manifolds (\cite{Trie92, Hebey96, Aubin1998,Hebey2000}),
and Sobolev spaces of sections of vector bundles on manifolds
(\cite{Palais65, Eichhorn2007}).
However, most of these works focus on spaces of functions rather than
general sections, and in many cases the focus is on integer order spaces.
This paper should be viewed as a part of our efforts to build a more complete foundation for the study and use of Sobolev-Slobodeckij spaces on manifolds through a sequence of related manuscripts \cite{holstbehzadan2015b, holstbehzadan2017c, holstbehzadan2018c, holstbehzadan2018d}.


{\bf\em Outline of Paper.}
In Section 2 we summarize some of
the basic notations and conventions used throughout the paper. In Section 3 we will review a number of basic
constructions in linear algebra that are essential in the study
of function spaces of generalized sections of vector bundles. In
Section 4 we will recall some useful tools from analysis and
topology. In particular, a concise overview of some of the main
properties of topological vector spaces is presented in this
 section. Section 5 deals with reviewing some results we need from
differential geometry. The main purpose of this section is to
set the notations, definitions, and conventions straight. This
section also includes some less well known facts about topics
such as higher order covariant derivatives in vector bundles. In
Section 6 we collect the results that we need from the theory of
generalized functions on Euclidean spaces and vector bundles.
Section 7 is concerned with various definitions
and properties of Sobolev spaces that are needed for developing a
coherent theory of such spaces on the vector bundles. In
Section 8 and Section 9 we introduce Lebesgue spaces and
Sobolev-Slobodeckij spaces of sections of vector bundles and we present a
rigorous account of their various properties. Finally in Section 10 we study the continuity of certain differential operators between Sobolev spaces of sections of vector bundles. Although the purpose of sections 3 through 7 is to give a quick overview of the prerequisites that are needed to understand the proofs of the results in later sections  and set the notations straight, as it was pointed out earlier, several theorems and proofs that appear in these sections cannot be found elsewhere in the generality that are stated here.


\section{Notation and Conventions}
\label{sec:notation}

Throughout this paper, $\reals$ denotes the set of real numbers, $\mathbb{N}$ denotes the set of positive integers, and $\mathbb{N}_0$ denotes the set of nonnegative integers. For any nonnegative real number $s$, the integer part of $s$ is denoted by $\floor{s}$. The letter $n$ is a positive integer and stands for the dimension of the space.\\

 $\Omega$ is a nonempty open set in $\reals^n$. The collection of all compact subsets of $\Omega$ will be denoted by $\mathcal{K}(\Omega)$. Lipschitz domain in $\reals^n$ refers to a nonempty bounded open set in $\reals^n$ with Lipschitz continuous boundary.\\

Each element of $\mathbb{N}_0^n$ is called a multi-index. For a multi-index $\alpha=(\alpha_1,\cdots,\alpha_n)\in \mathbb{N}_0^n$, we let
\begin{itemize}
\item $|\alpha|:=\alpha_1+\cdots+\alpha_n$
\item $\alpha!:=\alpha_1!\cdots\alpha_n!$
\end{itemize}
If $\alpha,\beta\in \mathbb{N}_0^n$, we say $\beta\leq \alpha$ provided that $\beta_i\leq \alpha_i$ for all $1\leq i\leq n$. If $\beta\leq \alpha$, we let
\begin{equation*}
{\alpha\choose \beta}:=\frac{\alpha!}{\beta!(\alpha-\beta)!}={\alpha_1\choose \beta_1}\cdots {\alpha_1\choose \beta_1}\,.
\end{equation*}
Suppose that $\alpha\in \mathbb{N}_0^n$. For sufficiently smooth functions $u:\Omega\rightarrow \reals$ (or for any distribution $u$) we define the $\alpha$th order partial derivative of $u$ as follows:
\begin{equation*}
\partial^\alpha u:=\frac{\partial^{|\alpha|}u}{\partial x_1^{\alpha_1}\cdots \partial x_n^{\alpha_n}}\,.
\end{equation*}

We use the notation $A\preceq B$ to mean $A\leq cB$, where $c$ is a positive constant that does not depend on the non-fixed parameters appearing in $A$ and $B$. We write $A\simeq B$ if $A\preceq B$ and $B\preceq A$.\\

For any nonempty set $X$ and $r\in \mathbb{N}$, $X^{\times r}$ stands for $\underbrace{X\times \cdots \times X}_{\textrm{$r$ times}}$.

For any two nonempty sets $X$ and $Y$, $\textrm{Func} (X,Y)$
 denotes the collection of all functions from $X$ to $Y$.

We write $L(X,Y)$ for the space of all \emph{continuous} linear
maps from the normed space $X$ to the normed space $Y$. $L(X,\reals)$ is called the (topological) dual of $X$ and is denoted by $X^*$.  We use the notation $X\hookrightarrow Y$ to mean $X \subseteq Y$ and the
inclusion map is continuous.

$\textrm{GL}(n,\reals)$ is the set of all $n\times n$ invertible
matrices with real entries. Note that $\textrm{GL}(n,\reals)$ can
be identified with an open subset of $\reals^{n^2}$ and so it can
be viewed as a smooth manifold (more precisely,
$\textrm{GL}(n,\reals)$ is a Lie group).

Throughout this manuscript, all manifolds are assumed to be
smooth, Hausdorff, and second-countable.

Let $M$ be an $n$-dimensional compact smooth manifold. The
tangent space of the manifold $M$ at point $p\in M$ is denoted by
$T_pM$, and the cotangent space by $T^{*}_pM$. If
$(U,\varphi=(x^i))$ is a local coordinate chart and $p\in U$, we
denote the corresponding coordinate basis for $T_pM$ by
$\partial_i|_p$ while $\frac{\partial}{\partial x^i}|_{x}$
denotes the basis for the tangent space to $\reals^n$ at $x=\varphi(p)\in
\reals^n$; that is
\begin{equation*}
\varphi_{*}\partial_i=\frac{\partial}{\partial x^i}
\end{equation*}
Note that for any smooth function $f:M\rightarrow \reals$ we have
\begin{equation*}
(\partial_i f)\circ \varphi^{-1}=\frac{\partial}{\partial
x^i}(f\circ \varphi^{-1})
\end{equation*}
The vector space of all $k$-covariant, $l$-contravariant tensors
on $T_pM$ is denoted by $T^k_l(T_pM)$. So each element of
$T^k_l(T_pM)$ is a multilinear map of the form
\begin{equation*}
F:\underbrace{T^{*}_pM\times \cdots \times T^{*}_pM}_{\textrm{$l$
copies }}\times \underbrace{T_pM\times \cdots \times
T_pM}_{\textrm{$k$ copies}}\rightarrow \reals
\end{equation*}

We are primarily interested in the vector bundle of $k\choose l$-tensors
on $M$ whose total space is
\begin{equation*}
T^k_l(M)=\bigsqcup_{p\in M} T^k_l(T_pM)
\end{equation*}
A section of this bundle is called a $k\choose l$-tensor field.
We set $T^k M:=T^k_0(M)$. $TM$ denotes the tangent bundle of $M$
and $T^{*}M$ is the cotangent bundle of $M$. We set $\tau^k_l(M)=C^\infty(M,T^k_l(M))$ and $\chi(M)=C^\infty(M,TM)$.

A symmetric positive definite section of $T^2M$ is called a
 Riemannian metric on $M$. If $M$ is equipped with a Riemannian
metric $g$, the combination $(M,g)$ will be referred to as a
 Riemannian manifold. If there is no possibility of confusion, we
may write $\langle X,Y\rangle$ instead of $g(X,Y)$. The norm induced by $g$ on each tangent space will be denoted by $\|.\|_g$. We say that
$g$ is smooth (or the Riemannian manifold is smooth) if $g\in
C^{\infty}(M,T^2M)$. $d$ denotes the exterior derivative and $\textrm{grad}: C^{\infty}(M)\rightarrow C^\infty (M,TM)$ denotes the gradient operator which is defined by $g(\textrm{grad}\,f,X)=d\,f(X)$ for all $f\in C^{\infty}(M)$ and $X\in C^\infty(M,TM)$.

Given a metric $g$ on
$M$, one can define the musical isomorphisms as follows:
\begin{align*}
\textrm{flat}_g: T_pM&\rightarrow T_p^*M\\
          X&\mapsto X^\flat:=g(X,\,\cdot\,)\,,\\
\textrm{sharp}_g: T^*_pM&\rightarrow T_pM\\
        \psi&\mapsto \psi^\sharp:=\textrm{flat}_g^{-1}(\psi)\,.
\end{align*} Using
$\textrm{sharp}_g$ we can define the $0\choose 2$-tensor field
$g^{-1}$ (which is called the \textbf{inverse metric tensor}) as follows
\begin{equation*}
g^{-1}(\psi_1,\psi_2):=g(\textrm{sharp}_g(\psi_1),\textrm{sharp}_g(\psi_2))\,.
\end{equation*}
 Let $\{E_i\}$ be a local
frame on an open subset $U\subset M$ and $\{\eta^i\}$ be the
corresponding dual coframe. So we can write $X=X^iE_i$ and
$\psi=\psi_i\eta^i$. It is standard practice to denote the
$i^{{\rm th}}$ component of $\textrm{flat}_g X$ by $X_i$ and the
$i^{{\rm th}}$ component of $\textrm{sharp}_g(\psi)$ by $\psi^i$:
\begin{equation*}
\textrm{flat}_g X=X_i \eta^i\,,\quad \textrm{sharp}_g \psi=\psi^i
E_i\,.
\end{equation*}
It is easy to show that
\begin{equation*}
X_i=g_{ij}X^j\,,\quad \psi^i=g^{ij}\psi_j\,,
\end{equation*}
where $g_{ij}=g(E_i,E_j)$ and $g^{ij}=g^{-1}(\eta^i,\eta^j)$. It
is said that $\textrm{flat}_g X$ is obtained from $X$ by lowering
an index and $\textrm{sharp}_g \psi$ is obtained from $\psi$ by
raising an index.

\section{Review of Some Results From Linear Algebra}
   \label{app:spacesuseful}

   In this section we summarize a collection of definitions and results from linear algebra that play an important role in our study of function spaces and differential operators on manifolds.

   There are several ways to construct new vector spaces from old ones: subspaces, products, direct sums, quotients, etc. The ones that are particularly important for the study of Sobolev spaces of sections of vector bundles
   are the vector space of linear maps between two given vector spaces, the tensor product of vector spaces, and the vector space of all densities on a given vector space which we briefly review here in order to set the notations straight.

\begin{itemizeXALI}
\item Let $V$ and $W$ be two vector spaces. The collection of all linear maps from $V$ to $W$ is a new vector space which we denote by $\textrm{Hom}(V,W)$. In particular, $\textrm{Hom}(V,\reals)$ is the (algebraic) dual of $V$. If $V$ and $W$ are finite-dimensional, then $\textrm{Hom}(V,W)$ is a vector space whose dimension is equal to
the product of dimensions of $V$ and $W$. Indeed, if we choose a
basis for $V$ and a basis for $W$, then $\textrm{Hom}(V,W)$ is
isomorphic with the space of matrices with $\textrm{dim}\, W$
rows and $\textrm{dim}\, V$ columns.
\item  Let
$U$ and $V$ be two vector spaces. Roughly speaking, the
tensor product of $U$ and $V$ (denoted by $U\otimes V$) is the
unique vector space (up to isomorphism of vector spaces) such that for any vector space $W$, $\textrm{Hom}(U\otimes V, W)$ is isomorphic to the collection of bilinear maps from $U\times V$ to $W$.
Informally, $U\otimes V$ consists of finite linear combinations of
symbols $u\otimes v$, where $u\in U$ and $v\in V$. It is assumed
that these symbols satisfy the following identities:
\begin{align*}
&(u_1+u_2)\otimes v-u_1\otimes v-u_2\otimes v=0\\
& u\otimes (v_1+v_2)-u\otimes v_1-u\otimes v_2=0\\
& \alpha (u\otimes v)-(\alpha u)\otimes v=0\\
& \alpha(u\otimes v)-u\otimes (\alpha v)=0
\end{align*}
for all $u, u_1, u_2\in U$, $v, v_1,v_2\in V$ and $\alpha \in
\reals$.These identities simply say that the map
\begin{equation*}
\otimes:U\times V\rightarrow U\otimes V,\quad (u,v)\mapsto
u\otimes v
\end{equation*}
is a bilinear map. The image of this map spans $U\otimes V$.
\begin{definition}\lab{winter1}
Let $U$ and $V$ be two vector spaces. Tensor product is a vector
space $U\otimes V$ together with a bilinear map $\otimes:U\times
V\rightarrow U\otimes V,\,\, (u,v)\mapsto u\otimes v$ such that
given any vector space $W$ and any \textbf{bilinear map} $b:U\times V\rightarrow W$, there
is a unique \textbf{linear map} $\bar{b}: U\otimes V\rightarrow W$
with $\bar{b}(u\otimes v)=b(u,v)$. That is the following diagram
commutes:
\begin{center}
\begin{tikzcd}
U\otimes V\arrow[dotted]{rd}{\bar{b}}\\
U\times V \arrow{r}{b} \arrow{u}{\otimes} & W
\end{tikzcd}
\end{center}
\end{definition}
For us, the most useful property of the tensor product of finite dimensional vector spaces is the following property:
\begin{equation*}
\textrm{Hom}(V,W)\cong V^*\otimes W
\end{equation*}
Indeed, the following map is an isomorphism of vector spaces:
\begin{equation*}
F:V^*\otimes W\rightarrow \textrm{Hom}(V,W),\quad
\underbrace{F(v^*\otimes w)}_{\textrm{an element of
$\textrm{Hom}(V,W)$} }(v)=\underbrace{[v^*(v)]}_{\textrm{a real
number}}w
\end{equation*}
It is useful to obtain an expression for the inverse of $F$ too. That is, given $T\in \textrm{Hom}(V,W)$, we want to find an expression for the corresponding element of $V^*\otimes W$. To
this end, let $\{e_i\}_{1\leq i \leq n}$ be a basis for $V$ and
$\{e^i\}_{1\leq i \leq n}$ denote the corresponding dual basis. Let $\{s_a\}_{1\leq
a\leq r}$ be a basis for $W$. Then $\{e^i\otimes s_b\}$ is a basis
for $V^*\otimes W$. Suppose $\sum_{i,a}R^a_i e^i\otimes s_a$ is
the element of $V^*\otimes W$ that corresponds to $T$. We have
\begin{align*}
F(\sum_{i,a}R^a_i e^i\otimes s_a)=T&\Longrightarrow \forall\, u\in
V\quad \sum_{i,a}R^a_iF[e^i\otimes s_a](u)=T(u)\\
&\Longrightarrow
\forall\, u\in V\quad \sum_{i,a}R^a_ie^i(u)s_a=T(u)
\end{align*}
In particular, for all $1\leq j\leq n$
\begin{equation*}
T(e_j)=\sum_{i,a}R^a_i
\underbrace{e^i(e_j)}_{\delta^i_j}s_a=\sum_a R^a_j s_a
\end{equation*}
That is, $R^a_i$ is the entry in the $a^{th}$ row and $i^{th}$
column of the matrix of the linear transformation $T$.
\item Let $V$ be an $n$-dimensional vector space. A density on $V$ is a function $\mu:\underbrace{ V\times \cdots \times V}_{\textrm{$n$ copies}}\rightarrow \reals$ with the property that
\begin{equation*}
\mu(Tv_1,\cdots,Tv_n)=|\textrm{det} T|\mu(v_1,\cdots,v_n)
\end{equation*}
for all $T\in \textrm{Hom}(V,V)$.
\end{itemizeXALI}
We denote the collection of all densities on $V$ by $\mathcal{D}(V)$. It can be shown that $\mathcal{D}(V)$ is a
 one dimensional vector space under the obvious vector space operations. Indeed, if  $(e_1,\cdots,e_n)$ is a basis for $V$, then each element $\mu\in \mathcal{D}(V)$ is uniquely determined by its value at $(e_1,\cdots,e_n)$ because for any $(v_1,\cdots,v_n)\in V^{\times n}$, we have $\mu(v_1,\cdots,v_n)=|\textrm{det} T|\mu(e_1,\cdots,e_n)$ where $T:V\rightarrow V$ is the linear transformation defined by $T(e_i)=v_i$ for all $1\leq i\leq n$. Thus
 \begin{equation*}
 F: \mathcal{D}(V)\rightarrow \reals,\qquad F(\mu)=\mu(e_1,\cdots,e_n)
 \end{equation*}
 will be an isomorphism of vector spaces.

 Moreover, if $\omega\in \Lambda^n (V)$ where $\Lambda^n (V)$ is the collection of all alternating covariant $n$-tensors, then $|\omega|$ belongs to $\mathcal{D}(V)$. Thus if $\omega$ is any nonzero element of $\Lambda^n (V) $, then $\{|\omega|\}$ will be a basis for $\mathcal{D}(V)$ (\cite{Lee3}, Page 428).


\section{Review of Some Results From Analysis and Topology}
\subsection{Euclidean Space}

Let $\Omega$ be a nonempty open set in $\reals^n$ and $m\in
\mathbb{N}_0$. Here is a list of several useful function spaces on
$\Omega$:
\begin{align*}
& C(\Omega)=\{f:\Omega\rightarrow \reals: \textrm{$f$ is
continuous}\}\\
& C^m(\Omega)=\{f: \Omega\rightarrow \reals:
\textrm{$\forall\,|\alpha|\leq m\quad \partial^\alpha f \in
C(\Omega)$}\}\qquad (C^0(\Omega)=C(\Omega))\\
& BC(\Omega)=\{f:\Omega\rightarrow \reals: \textrm{$f$ is
continuous and bounded on $\Omega$}\}\\
& BC^m(\Omega)=\{f\in C^m(\Omega): \textrm{$\forall\,|\alpha|\leq
m\quad \partial^\alpha f$ is bounded on $\Omega$}\}\\
& BC(\bar{\Omega})=\{f:\Omega\rightarrow \reals: \textrm{$f\in
BC(\Omega)$ and $f$ is uniformly continuous on
$\Omega$}\}\\
& BC^m(\bar{\Omega})=\{f:\Omega\rightarrow \reals: \textrm{$f\in
BC^m(\Omega), \forall\,|\alpha|\leq m \quad \partial^\alpha f$ is
uniformly continuous on $\Omega$ }\}\\
&C^\infty(\Omega)=\bigcap_{m\in \mathbb{N}_0}C^m(\Omega),\quad
BC^\infty(\Omega)=\bigcap_{m\in \mathbb{N}_0}BC^m(\Omega), \quad
BC^\infty(\bar{\Omega})=\bigcap_{m\in
\mathbb{N}_0}BC^m(\bar{\Omega})
\end{align*}
\begin{remark}\cite{32}\lab{winter4}
If $g:\Omega\rightarrow \reals$ is in $BC(\bar{\Omega})$, then it
possesses a unique, bounded, continuous extension to the closure
$\bar{\Omega}$ of $\Omega$.
\end{remark}
\noindent \textbf{Notation :} Let $\Omega$ be a nonempty open set in $\reals^n$. The collection of all compact sets in $\Omega$ is denoted by $\mathcal{K}(\Omega)$. If $f:\Omega \rightarrow \reals$ is a function, the support of $f$ is denoted by $\textrm{supp}\,f$. Notice that, in some references $\textrm{supp}\,f$ is defined as the closure of $\{x\in \Omega: f(x)\neq 0\}$ in $\Omega$, while in certain other references it is defined as the closure of $\{x\in \Omega: f(x)\neq 0\}$ in $\reals^n$. Of course, if we are concerned with functions whose support is inside an element of $\mathcal{K}(\Omega)$, then the two definitions agree. For the sake of definiteness, in this manuscript we always use the former interpretation of support. Also support of a distribution will be discussed in Section $6$.
\begin{remark}\lab{winter5}
If $\mathcal{F}(\Omega)$ is any function space on $\Omega$ and
$K\in \mathcal{K}(\Omega)$, then $\mathcal{F}_K(\Omega)$ denotes
the collection of elements in $\mathcal{F}(\Omega)$ whose support
is inside $K$. Also
\begin{equation*}
\mathcal{F}_c(\Omega)=\mathcal{F}_{comp}(\Omega)=\bigcup_{K\in
\mathcal{K}(\Omega)}\mathcal{F}_K(\Omega)
\end{equation*}
\end{remark}
Let $0<\lambda\leq 1$. A function $F: \Omega\subseteq \reals^n
\rightarrow \reals^k$ is called \textit{$\lambda$-Holder continuous} if
there exists a constant $L$ such that
\begin{equation*}
|F(x)-F(y)|\leq L|x-y|^\lambda\quad \forall\,x,y\in \Omega
\end{equation*}
Clearly a $\lambda$-Holder continuous function on $\Omega$ is
uniformly continuous on $\Omega$. $1$-Holder continuous functions
are also called \textit{Lipschitz continuous} functions or simply Lipschitz
functions. We define
\begin{align*}
BC^{m,\lambda}(\Omega)&=\{f:\Omega\rightarrow \reals:
\textrm{$\forall\,|\alpha|\leq m\quad
\partial^\alpha f$ is $\lambda$-Holder continuous and bounded}\}\\
&=\{f\in BC^m(\Omega): \textrm{$\forall\,|\alpha|\leq m\quad
\partial^\alpha f$ is $\lambda$-Holder continuous}\}\\
&=\{f\in BC^m(\bar{\Omega}): \textrm{$\forall\,|\alpha|\leq m\quad
\partial^\alpha f$ is $\lambda$-Holder continuous}\}
\end{align*}
and
\begin{equation*}
BC^{\infty,\lambda}(\Omega):=\bigcap_{m\in \mathbb{N}_0}BC^{m,\lambda}(\Omega)
\end{equation*}
\begin{remark}\lab{fallremcomplip1}
Let $F:\Omega\subseteq \reals^n\rightarrow \reals^k$
($F=(F^1,\cdots,F^k)$). Then
\begin{align*}
\textrm{$F$ is Lipschitz}\Longleftrightarrow \forall\,1\leq i\leq
k\quad \textrm{$F^i$ is Lipschitz}
\end{align*}
Indeed, for each $i$
\begin{equation*}
|F^i(x)-F^i(y)|\leq
\sqrt{\sum_{j=1}^k|F^j(x)-F^j(y)|^2}=|F(x)-F(y)|\leq L|x-y|
\end{equation*}
which shows that if $F$ is Lipschitz so will be its components.
Also if for each $i$, there exists $L_i$ such that
\begin{equation*}
|F^i(x)-F^i(y)|\leq L_i|x-y|
\end{equation*}
Then
\begin{equation*}
\sum_{j=1}^k|F^j(x)-F^j(y)|^2\leq n L^2|x-y|^2
\end{equation*}
where $L=\max{\{L_1,\cdots,L_k\}}$. This proves that if each
component of $F$ is Lipschitz so is $F$ itself.
\end{remark}

\begin{theorem}\cite{9}
Let $\Omega$ be a nonempty open set in $\reals^n$ and let $K\in \mathcal{K}(\Omega)$. There is a function $\psi\in C_c^\infty(\Omega)$ taking values in $[0,1]$ such that $\psi=1$ on a neighborhood of $K$.
\end{theorem}

\begin{theorem}[Exhaustion by Compact Sets]\cite{9}\lab{winter2}
Let $\Omega$ be a nonempty open subset of $\reals^n$. There
exists a sequence of compact subsets $(K_j)_{j\in\mathbb{N}}$
such that $\cup_{j\in \mathbb{N}} \mathring{K}_j=\Omega$ and
\begin{equation*}
K_1\subseteq \mathring{K}_2\subseteq K_2\subseteq \cdots\subseteq
\mathring{K}_j\subseteq K_j\subseteq \cdots
\end{equation*}
Moreover, as a direct consequence, if $K$ is any compact subset
of the open set $\Omega$, then there exists an open set $V$ such
that $K\subseteq V\subseteq \bar{V}\subseteq \Omega$.
\end{theorem}

\begin{theorem}\cite{9}\lab{winter3}
Let $\Omega$ be a nonempty open subset of $\reals^n$. Let
$\{K_j\}_{j\in\mathbb{N}}$ be an exhaustion of $\Omega$ by compact
sets. Define
\begin{equation*}
V_0=\mathring{K}_4,\qquad \forall\, j\in \mathbb{N}\quad
V_j=\mathring{K}_{j+4}\setminus K_j
\end{equation*}
Then
\begin{enumerateXALI}
\item Each $V_j$ is an open bounded set and $\Omega=\cup_j V_j$.
\item The cover $\{V_j\}_{j\in\mathbb{N}_0}$ is \textbf{locally finite} in
$\Omega$, that is, each compact subset of $\Omega$ has nonempty
intersection with only a finite number of the $V_j$'s.
\item There is a family of functions $\psi_j\in
C_c^{\infty}(\Omega)$ taking values in $[0,1]$ such that
$\textrm{supp}\,\psi_j\subseteq V_j$ and
\begin{equation*}
\sum_{j\in \mathbb{N}_0}\psi_j(x)=1\qquad \textrm{for all $x\in
\Omega$}
\end{equation*}
\end{enumerateXALI}
\end{theorem}

\begin{theorem}[\cite{Folland07}, Page 74]\lab{winter8}
Suppose $\Omega$ is an open set in $\reals^n$ and $G:\Omega\rightarrow G(\Omega)\subseteq \reals^n$ is a $C^1$-diffeomorphism (i.e. $G$ and $G^{-1}$ are both $C^1$ maps). If $f$ is a Lebesgue measurable function on $G(\Omega)$, then $f\circ G$ is Lebesgue measurable on $\Omega$. If $f\geq 0$ or $f\in L^1(G(\Omega))$, then
\begin{equation*}
\int_{G(\Omega)}f(x)dx=\int_\Omega f\circ G(x)|\textrm{det}G'(x)|dx
\end{equation*}
\end{theorem}
\begin{theorem}[\cite{Folland07}, Page 79]\lab{winter9}
If $f$ is a nonnegative measurable function on $\reals^n$ such that $f(x)=g(|x|)$ for some function $g$ on $(0,\infty)$, then
\begin{equation*}
\int f(x)dx=\sigma(S^{n-1})\int_0^\infty g(r)r^{n-1}dr
\end{equation*}
where $\sigma(S^{n-1})$ is the surface area of $(n-1)$-sphere.
\end{theorem}
\begin{theorem}(\cite{Apostol74}, Section 12.11)\lab{thmfallconvexapostol1} Suppose $U$ is an open set in
$\reals^n$ and $f: U\rightarrow \reals$ is differentiable. Let
$x$ and $y$ be two points in $U$ and suppose the line segment
joining $x$ and $y$ is contained in $U$. Then there exists a
point $z$ on the line joining $x$ to $y$ such that
\begin{equation*}
f(y)-f(x)=\grad f(z).(y-x)
\end{equation*}
As a consequence, if $U$ is \textbf{convex} and all first order
partial derivatives of $f$ are bounded, then $f$ is Lipschitz on
$U$.
\end{theorem}
\noindent\textbf{Warning:} Suppose $f\in BC^{\infty}(U)$. By the above
item, if $U$ is convex, then $f$ is Lipschitz. However, if $U$ is
not convex, then $f$ is not necessarily Lipschitz. For example,
let $U=\cup_{n=0}^\infty (n,n+1)$ and define
\begin{equation*}
f:U\rightarrow \reals,\qquad f(x)=(-1)^n,\, \forall\,x\in (n,n+1)
\end{equation*}
Clearly all derivatives of $U$ are equal to zero, so $f\in
BC^\infty (U)$. But $f$ is not uniformly continuous and thus it
is not Lipschitz. Indeed, for any $1>\delta>0$, we can let
$x=2-\delta/4$ and $y=2+\delta/4$. Clearly $|x-y|<\delta$,
however, $|f(x)-f(y)|=2$.\\\\
\noindent Of course if $f\in C_c^1(U)$, then $f$ can be extended by
zero to a function in $C_c^1(\reals^n)$. Since $\reals^n$ is
convex, we may conclude that the extension by zero of $f$ is
Lipschitz which implies that $f:U\rightarrow \reals$ is
Lipschitz. As a consequence, $C_c^1(U)\subseteq BC^{0,1}(U)$ and $C_c^\infty(U)\subseteq BC^{\infty,1}(U)$. Also Theorem \ref{sobislip31} and the following theorem provide useful information regarding this isuue.
\begin{theorem}\lab{thmfalllipbound1}
Let $U\subseteq \reals^n$ and $V\subseteq \reals^k$ be two
nonempty open sets and let $T:U\rightarrow V$
$(T=(T^1,\cdots,T^k))$ be a $C^1$ map (that is for each $1\leq
i\leq k$, $T^i\in C^1(U)$). Suppose $B\subseteq U$ is a bounded
set such that $B\subseteq \bar{B}\subseteq U$. Then $T:
B\rightarrow V$ is Lipschitz.
\end{theorem}
\begin{proof}
By Remark \ref{fallremcomplip1} it is enough to show that each
$T^i$ is Lipschitz on $B$. Fix a function $\varphi\in
C_c^\infty(\reals^n)$ such that $\varphi=1$ on $\bar{B}$ and
$\varphi=0$ on $\reals^n\setminus U$. Then $\varphi T^i$ can be
viewed as an element of $C_c^1(\reals^n)$. Therefore it is
Lipschitz ($\reals^n$ is convex) and there exists a constant $L$,
which may depend on $\varphi$, $B$ and $T^i$, such that
\begin{equation*}
|\varphi T^i(x)-\varphi T^i(y)|\leq  L|x-y|\quad \forall\,x,y\in
\reals^n
\end{equation*}
Since $\varphi=1$ on $\bar{B}$ it follows that
\begin{equation*}
|T^i(x)-T^i(y)|\leq L|x-y| \quad \forall\,x,y\in B
\end{equation*}
\end{proof}
\subsection{Normed Spaces}

\begin{theorem}\lab{thmwinterdualab}
Let $X$ and $Y$ be normed spaces. Let $A$ be a dense subspace of $X$ and $B$ be a dense subspace of $Y$. Then
\begin{itemizeX}
\item  $A\times B$ is dense in $X\times Y$;
\item if $T:A\times B\rightarrow \reals$
is a continuous bilinear map, then $T$ has a unique extension to a
continuous bilinear operator $\tilde{T}:X\times Y\rightarrow
\reals$.
\end{itemizeX}
\end{theorem}
\begin{theorem}\cite{32}\lab{thmfallnormedreflexive1}
Let $X$ be a normed space and let $M$ be a closed vector subspace
of $X$. Then
\begin{enumerateXALI}
\item If $X$ is reflexive, then $X$ is a Banach space.
\item $X$ is reflexive if and only if $X^*$ is reflexive.
\item If $X^*$ is separable, then $X$ is separable.
\item If $X$ is reflexive and separable, then so is $X^*$.
\item If $X$ is a reflexive Banach space, then so is $M$.
\item If $X$ is a separable Banach space, then so is $M$.
\end{enumerateXALI}
Moreover, if $X_1,\cdots,X_r$ are reflexive Banach spaces, then
$X_1\times \cdots\times X_r$ equipped with the norm
\begin{equation*}
\|(x_1,\cdots,x_r)\|=\|x_1\|_{X_1}+\cdots+\|x_r\|_{X_r}
\end{equation*}
is also a reflexive Banach space.
\end{theorem}

\subsection{Topological Vector Spaces}
There are different, generally nonequivalent, ways to define
topological vector spaces. The conventions in this section mainly follow Rudin's functional analysis \cite{Rudi73}. Statements in this section are either taken from Rudin's functional analysis, Grubb's distributions and operators \cite{9}, excellent presentation of Reus \cite{Reus1}, and Treves' topological vector spaces \cite{Treves1} or are direct consequences of statements in the aforementioned references. Therefore we will not give the proofs.
\begin{definition}\lab{winter10}
A topological vector space is a vector space $X$ together with a
topology $\tau$ with the following properties:
\begin{enumerate}[i)]
\item For all $x\in X$, the singleton $\{x\}$ is a closed set.
\item The maps
\begin{align*}
& (x,y)\mapsto x+y \qquad (\textrm{from $X\times X$ into $X$})\\
& (\lambda,x)\mapsto \lambda x\qquad (\textrm{from $\reals\times X
$ into $X$})
\end{align*}
are continuous where $X\times X$ and $\reals \times X$ are
equipped with the product topology.
\end{enumerate}
\end{definition}
\begin{definition}\lab{winter11}
Suppose $(X,\tau)$ is a topological vector space and $Y\subseteq
X$.
\begin{itemizeXALI}
\item $Y$ is said to be \textbf{convex} if for all $y_1, y_2\in Y$ and $t\in
(0,1)$ it is true that $ty_1+(1-t)y_2\in Y$.
\item $Y$ is said to be \textbf{balanced} if for all $y\in Y$ and $|\lambda|\leq
1$ it holds that $\lambda y\in Y$. In particular, any balanced
set contains the origin.
\item We say $Y$ is \textbf{bounded} if for any  neighborhood $U$ of the
origin (i.e. any open set containing the origin), there exits
$t>0$ such that $Y\subseteq tU$.
\end{itemizeXALI}
\end{definition}
\begin{theorem}[Important Properties of Topological Vector Spaces]\lab{thmfall1}
\leavevmode
\begin{itemizeXALI}
\item Every topological vector space is Hausdorff.
\item If $(X,\tau)$ is a topological vector space, then
\begin{enumerate}
\item for all $a\in X$:  $E\in \tau \Longleftrightarrow a+E\in
\tau$ (that is $\tau$ is \textbf{translation invariant})
\item for all $\lambda\in \reals\setminus \{0\}$: $E\in \tau \Longleftrightarrow \lambda E\in
\tau$ (that is $\tau$ is \textbf{scale invariant})
\item if $A\subseteq X$ is convex and $x\in X$, then so is $A+x$
\item if $\{A_i\}_{i\in I}$ is a family of convex subsets of $X$,
then $\cap_{i\in I}A_i$ is convex.
\end{enumerate}
\end{itemizeXALI}
\end{theorem}
\noindent\textbf{Note:} Some authors do not include condition (i) in the definition of topological vector spaces. In that case, a topological vector space will not necessarily be Hausdorff.
\begin{definition}\lab{winter12}
Let $(X,\tau)$ be a topological space.
\begin{itemizeXALI}
\item A collection $\mathcal{B}\subseteq \tau$ is said to be a
\textbf{\emph{basis}} for $\tau$, if every element of $\tau$ is a union of
elements in $\mathcal{B}$.
\item Let $p\in X$. If $\gamma\subseteq \tau$ is such that each
element of $\gamma$ contains $p$ and every neighborhood of $p$
(i.e. every open set containing $p$) contains at least one
element of $\gamma$, then we say $\gamma$ is a \textbf{\emph{local base
at $p$}}. If $X$ is a vector space, then the local base $\gamma$
is said to be convex if each element of $\gamma$ is a
convex set.
\item $(X,\tau)$ is called \textbf{first-countable} if
each point has a countable local base.
\item $(X,\tau)$ is called \textbf{second-countable} if there is a
countable basis for $\tau$.
\end{itemizeXALI}
\end{definition}
\begin{theorem}\lab{winter13}
Let $(X,\tau)$ be a topological space and suppose for all $x\in
X$, $\gamma_x$ is a local base at $x$. Then
$\mathcal{B}=\cup_{x\in X}\gamma_x$ is a basis for $\tau$.
\end{theorem}
\begin{theorem}\lab{thmfall2}
Let $X$ be a vector space and suppose $\tau$ is a translation
invariant topology on $X$. Then for all $x_1,x_2\in X$ we have
{\fontsize{10}{10}{\begin{equation*}
\textrm{the collection $\gamma_{x_1}$ is a local base at
$x_1$}\Longleftrightarrow \textrm{the collection
$\{A+(x_2-x_1)\}_{A\in \gamma_{x_1}}$ is a local base at $x_2$}
\end{equation*}}}
\end{theorem}
\begin{remark}\lab{winter14}
Let $X$ be a vector space and suppose $\tau$ is a translation
invariant topology on $X$. As a direct consequence of the
previous theorems the topology $\tau$ is uniquely determined by
giving a local base $\gamma_{x_0}$ at some point $x_0\in X$.
\end{remark}
\begin{definition}\lab{winter15}
Let $(X,\tau)$ be a topological vector space. $X$ is said to be
\textbf{metrizable} if there exists a metric $d: X\times X\rightarrow
[0,\infty)$ whose induced topology is $\tau$. In this case we say
that the metric $d$ is compatible with the topology $\tau$.
\end{definition}
\begin{theorem}\lab{winter16}
Let $(X,\tau)$ be a topological vector space. Then
\begin{itemizeXALI}
\item $X$ is metrizable $\,\Longleftrightarrow\,$  there exists a
metric $d$ on $X$ such that for all $x\in X$,
$\{B(x,\frac{1}{n})\}_{n\in \mathbb{N}}$ is a local base at $x$.
\item A metric $d$ on $X$ is compatible with $\tau$ $\,\Longleftrightarrow\,$ for all $x\in X$,
$\{B(x,\frac{1}{n})\}_{n\in \mathbb{N}}$ is a local base at $x$.
\end{itemizeXALI}
($B(x,\frac{1}{n})$ is the open ball of radius $\frac{1}{n}$ centered at $x$)
\end{theorem}
\begin{definition}\lab{winter17}
Let $X$ be a vector space and $d$ be a metric on $X$. $d$ is said
to be translation invariant provided that
\begin{equation*}
\forall\,x,y,a\in X\qquad d(x+a,y+a)=d(x,y)
\end{equation*}
\end{definition}
\begin{remark}\lab{winter18}
Let $(X,\tau)$ be a topological vector space and suppose $d$ is a
translation invariant metric on $X$. Then the following
statements are equivalent
\begin{enumerate}
\item for all $x\in X$, $\{B(x,\frac{1}{n})\}_{n\in \mathbb{N}}$ is a local base at $x$
\item there exists $x_0\in X$ such that $\{B(x_0,\frac{1}{n})\}_{n\in \mathbb{N}}$ is a local base at $x_0$
\end{enumerate}
Therefore $d$ is compatible with $\tau$ if and only if
$\{B(0,\frac{1}{n})\}_{n\in \mathbb{N}}$ is a local base at the
origin.
\end{remark}
\begin{theorem}\lab{winter19}
Let $(X,\tau)$ be a topological vector space. Then $(X,\tau)$ is
metrizable if and only if it has a countable local base at the
origin. Moreover, if $(X,\tau)$ is metrizable, then one can find a
translation invariant metric that is compatible with $\tau$.
\end{theorem}
\begin{definition}\lab{winter20}
Let $(X,\tau)$ be a topological vector space and let $\{x_n\}$ be
a sequence in $X$.
\begin{itemizeXALI}
\item We say that $\{x_n\}$ converges to a point $x\in X$ provided
that
\begin{equation*}
\forall\, U\in \tau,\,x\in U\quad \exists N\quad \forall\,n\geq
N\quad x_n\in U
\end{equation*}
\item We say that $\{x_n\}$ is a Cauchy sequence provided that
\begin{equation*}
\forall\, U\in \tau,\,0\in U\quad \exists N\quad \forall\,m,n\geq
N\quad x_n-x_m\in U
\end{equation*}
\end{itemizeXALI}
\end{definition}
\begin{theorem}\lab{winter21}
Let $(X,\tau)$ be a topological vector space, $\{x_n\}$ be a
sequence in $X$, and $x,y\in X$. Also suppose $\gamma$ is a local
base at the origin. The following statements are equivalent:
\begin{enumerate}
\item $x_n\rightarrow x$
\item $(x_n-x)\rightarrow 0$
\item $x_n+y\rightarrow x+y$
\item $\forall\, V\in \gamma\quad \exists\, N\quad \forall\,n\geq N\quad x_n-x\in V$
\end{enumerate}
Moreover $\{x_n\}$ is a Cauchy sequence if and only if
\begin{equation*}
\forall\, V\in \gamma\quad \exists\, N\quad \forall\,n,m\geq
N\quad x_n-x_m\in V
\end{equation*}
\end{theorem}
\begin{remark}\lab{winter22}
In contrast with properties like continuity of a function and
convergence of a sequence which depend only on the topology of the
space, the property of being a Cauchy sequence is not a
topological property. Indeed, it is easy to construct examples of
two metrics $d_1$ and $d_2$ on a vector space $X$ that induce the
same topology (i.e. the metrics are equivalent) but have different
collection of Cauchy sequences. However, it can be shown that if
$d_1$ and $d_2$ are two translation invariant metrics that induce
the same topology on $X$, then the Cauchy sequences of $(X,d_1)$
will be exactly the same as the Cauchy sequences of $(X,d_2)$.
\end{remark}
\begin{theorem}\lab{winter23}
Let $(X,\tau)$ be a metrizable topological vector space and $d$
be a translation invariant metric on $X$ that is compatible with
$\tau$. Let $\{x_n\}$ be a sequence in $X$. The following
statements are equivalent:
\begin{enumerate}
\item $\{x_n\}$ is a Cauchy sequence in the topological vector
space $(X,\tau)$.
\item $\{x_n\}$ is a Cauchy sequence in the metric space $(X,d)$.
\end{enumerate}
\end{theorem}
\begin{definition}\lab{winter24}
Let $(X,\tau)$ be a topological vector space. We say $(X,\tau)$
is \textbf{\emph{locally convex}} if it has a convex local base at the
origin.
\end{definition}
Note that, as a consequence of theorems (\ref{thmfall1}) and
(\ref{thmfall2}), the following statements are equivalent:
\begin{enumerate}
\item $(X,\tau)$ is a locally convex topological vector space.
\item There exists $p\in X$ with a convex local base at $p$.
\item For every $p\in X$ there exists a convex local base at $p$.
\end{enumerate}
\begin{definition}\lab{winter25}
Let $(X,\tau)$ be a metrizable locally convex topological vector
space. Let $d$ be a translation invariant metric on $X$ that is
compatible with $\tau$. We say that $X$ is \textbf{complete} if and only if
the metric space $(X,d)$ is a complete metric space. A complete
metrizable locally convex topological vector space is called a
\textbf{Frechet space}.
\end{definition}
\begin{remark}\lab{winter26}
Our previous remark about Cauchy sequences shows that the above
definition of completeness is independent of the chosen
translation invariant metric $d$. Indeed one can show that the
locally convex topological vector space $(X,\tau)$ is complete in
the above sense if and only if every Cauchy net in $(X,\tau)$ is
convergent.
\end{remark}
\begin{definition}\lab{winter27}
A \textbf{seminorm} on a vector space $X$ is a real-valued function
$p:X\rightarrow \reals$ such that
\begin{enumerate}[i.]
\item $\forall\,x,y\in X\qquad p(x+y)\leq p(x)+p(y)$
\item $\forall\,x\in X\,\,\forall\,\alpha\in \reals\qquad p(\alpha
x)=|\alpha| p(x)$
\end{enumerate}
If $\mathcal{P}$ is a family of seminorms on $X$, then we say
$\mathcal{P}$ is \textbf{separating} provided that for all $x\neq 0$ there
exists at least one $p\in\mathcal{P}$ such that $p(x)\neq 0$
(that is if $p(x)=0$ for all $p\in\mathcal{P}$, then $x=0$).
\end{definition}

\begin{remark}
It follows from conditions (i) and (ii) that if $p:X\rightarrow \reals$ is a seminorm, then $p(x)\geq 0$ for all $x\in X$.
\end{remark}

\begin{theorem}\lab{winter28}
Suppose $\mathcal{P}$ is a separating family of seminorms on a
vector space $X$. For all $p\in\mathcal{P}$ and $n\in \mathbb{N}$
 let
 \begin{equation*}
 V(p,n):=\{x\in X: p(x)<\frac{1}{n}\}
 \end{equation*}
Also let $\gamma$ be the collection of all finite intersections
of $V(p,n)$'s. That is,
\begin{equation*}
A\in \gamma \Longleftrightarrow \exists k\in \mathbb{N},\,\exists
p_1,\cdots,p_k\in \mathcal{P},\,\exists n_1,\cdots,n_k\in
\mathbb{N}\,\,\textrm{such that}\,\,A=\cap_{i=1}^k V(p_i,n_i)
\end{equation*}
Then each element of $\gamma$ is a convex balanced subset of $X$.
Moreover, there exists a unique topology $\tau$ on $X$ that
satisfies both of the following properties:
\begin{enumerate}
\item $\tau$ is translation invariant (that is, if $U\in \tau$ and $a\in X$, then
$a+U\in\tau$).
\item $\gamma$ is a local base at the origin for $\tau$.
\end{enumerate}
This unique topology is called the \textbf{natural topology} induced by
the family of seminorms $\mathcal{P}$. Furthermore, if $X$ is
equipped with the natural topology $\tau$, then
\begin{enumerate}[i)]
\item $(X,\tau)$ is a locally convex topological vector space.
\item every $p\in\mathcal{P}$ is a continuous function from $X$ to
$\reals$.
\end{enumerate}
\end{theorem}
\begin{theorem}\lab{winter29}
Suppose $\mathcal{P}$ is a separating family of seminorms on a
vector space $X$. Let $\tau$ be the natural topology induced by
$\mathcal{P}$. Then
\begin{enumerate}
\item $\tau$ is the smallest topology on $X$ that is translation
invariant and with respect to which every $p\in \mathcal{P}$ is
continuous.
\item $\tau$ is the smallest topology on $X$ with respect to which
 addition is continuous and every $p\in \mathcal{P}$ is
 continuous.
\end{enumerate}
\end{theorem}
\begin{theorem}\lab{thmapptvconvergence1}
Let $X$ and $Y$ be two vector spaces and suppose $\mathcal{P}$
and $\mathcal{Q}$ are two separating families of seminorms on $X$
and $Y$, respectively. Equip $X$ and $Y$ with the corresponding
natural topologies. Then
\begin{enumerateXALI}
\item A sequence $x_n$ converges to $x$ in $X$ if and only if for
all $p\in \mathcal{P}$, $p(x_n-x)\rightarrow 0$.
\item A linear operator $T:X\rightarrow Y$ is continuous if and
only if
{\fontsize{10}{10}{\begin{equation*}
\forall\,q\in \mathcal{Q}\quad \exists\,c>0,\,k\in
\mathbb{N},\,p_1,\cdots,p_k\in \mathcal{P}\quad \textrm{such
that}\quad \forall\,x\in X\quad |q\circ T(x)|\leq c\max_{1\leq
i\leq k}p_i(x)
\end{equation*}}}
\item A linear operator $T:X\rightarrow \reals$ is continuous if and
only if
\begin{equation*}
\exists\,c>0,\,k\in \mathbb{N},\,p_1,\cdots,p_k\in
\mathcal{P}\quad \textrm{such that}\quad \forall\,x\in X\quad
|T(x)|\leq c\max_{1\leq i\leq k}p_i(x)
\end{equation*}
\end{enumerateXALI}
\end{theorem}
\begin{theorem}\lab{thmapptvconvergence2}
Let $X$ be a Frechet space and let $Y$ be a topological vector
space. When $T$ is a linear map of $X$ into $Y$, the following
two properties are equivalent
\begin{enumerate}
\item $T$ is continuous.
\item $x_n\rightarrow 0$ in $X$ $\Longrightarrow$ $Tx_n\rightarrow
0$ in $Y$.
\end{enumerate}
\end{theorem}
\begin{theorem}\lab{winter30}
Let $\mathcal{P}=\{p_k\}_{k\in \mathbb{N}}$ be a
\textbf{countable} separating family of seminorms on a vector
space $X$. Let $\tau$ be the corresponding natural topology. Then
 the locally convex topological vector space $(X,\tau)$ is metrizable and the
 following translation invariant metric on $X$ is compatible with
 $\tau$:
 \begin{equation*}
 d(x,y)=\sum_{k=1}^\infty\frac{1}{2^k}\frac{p_k(x-y)}{1+p_k(x-y)}
 \end{equation*}
\end{theorem}
Let $(X,\tau)$ be a locally convex topological vector space. Consider the
topological dual of $X$,
\begin{equation*}
X^*:=\{f: X\rightarrow \reals: \textrm{$f$ is linear and
continuous}\}
\end{equation*}
There are several ways to topologize $X^*$: the weak$^*$
topology, the topology of convex compact convergence, the
topology of compact convergence, and the strong topology (see
\cite{Treves1}, Chapter 19). Here we describe the weak$^*$
topology and the strong topology on $X^*$.
\begin{definition}\lab{winter31}
Let $(X,\tau)$ be a locally convex topological vector space.
\begin{itemizeXALI}
\item The \textbf{weak$^*$ topology} on $X^*$ is the natural topology
induced by the separating family of seminorms $\{p_x\}_{x\in X}$ where
\begin{equation*}
\forall\,x\in X\qquad p_x: X^*\rightarrow \reals,\quad
p_x(f):=|f(x)|
\end{equation*}
A sequence $\{f_m\}$ converges to $f$ in $X^*$ with respect to
the weak$^*$ topology if and only if $f_m(x)\rightarrow f(x)$ in
$\reals$ for all $x\in X$.
\item The \textbf{strong topology} on $X^*$ is the natural topology induced
by the separating family of seminorms $\{p_B\}_{B\subseteq X
\textrm{bounded}}$ where for any bounded subset $B$ of $X$
\begin{equation*}
p_B: X^*\rightarrow \reals\qquad p_B(f):=\sup \{|f(x)|: x\in B\}
\end{equation*}
(it can be shown that for any bounded subset $B$ of $X$ and $f\in
X^*$, $f(B)$ is a bounded subset of $\reals$)
\end{itemizeXALI}
\end{definition}
\begin{remark}\lab{winter32}
\leavevmode
\begin{enumerateXALI}
\item If $X$ is a normed space, then the topology induced by the
norm
\begin{equation*}
\forall\,f\in X^*\qquad \|f\|_{op}=\sup_{\|x\|_X=1}|f(x)|
\end{equation*}
on $X^*$ is the same as the strong topology on $X^*$
(\cite{Treves1}, Page 198).
\item In this manuscript we always consider the topological dual of a locally convex
topological vector space with the strong topology. Of course, it
is worth mentioning that for many of the spaces that we will
consider (including $X=\mathcal{E}(\Omega)$ or $X=D(\Omega)$
where $\Omega$ is an open subset of $\reals^n$) a sequence in
$X^*$ converges with respect to the weak$^*$ topology if and only
if it converges with respect to the strong topology (for more
details on this see the definition and properties of \textbf{Montel
spaces} in section 34.4, page 356 of \cite{Treves1}).
\end{enumerateXALI}
\end{remark}
The following theorem, which is easy to prove, will later be used in the proof of completeness of Sobolev spaces of sections of vector bundles.
 \begin{theorem}[\cite{Reus1}, Page 160]\lab{thmwinter1101}
If $X$ and $Y$ are topological vector spaces and $I:X\rightarrow
 Y$ and $P:Y\rightarrow X$ are continuous linear maps such that $P\circ
 I=id_X$, then $I: X\rightarrow I(X)\subseteq Y$ is a linear
 topological isomorphism and $I(X)$ is closed in $Y$.
 \end{theorem}
Now we briefly review the relationship between the dual of a product of topological vector spaces and the product of the dual spaces. This will play an important role in our discussion of local representations of distributions in vector bundles in later sections.\\\\
Let $X_1,\cdots,X_r$ be topological vector spaces. Recall that the product topology on $X_1\times \cdots\times X_r$ is the smallest topology such
 that the projection maps
 \begin{equation*}
 \pi_k: X_1\times \cdots\times X_r\rightarrow
X_k,\qquad \pi_k(x_1,\cdots,x_r)=x_k
 \end{equation*}
are continuous for all $1\leq k\leq r$. It can be shown that if each $X_k$ is a locally convex topological vector space whose topology is induced by a family of seminorms $\mathcal{P}_k$, then $X_1\times \cdots\times X_r$ equipped with the product topology is a locally convex topological vector space whose topology is induced by the following family of seminorms
\begin{equation*}
\{p_1\circ\pi_1+\cdots+p_r\circ \pi_r: p_k\in\mathcal{P}_k\,\,\forall\, 1\leq k\leq r\}
\end{equation*}
\begin{theorem}[\cite{Reus1}, Page 164]\lab{winter33}
Let $X_1,\cdots,X_r$ be locally convex topological vector spaces. Equip
$X_1\times \cdots\times X_r$ and $X^*_1\times \cdots\times X^*_r$
 with the product topology.
The mapping $\tilde{L}: X_1^*\times \cdots\times X_r^*\rightarrow
(X_1\times \cdots\times X_r)^*$ defined by
\begin{equation*}
\tilde{L}(u_1,\cdots,u_r)=u_1\circ \pi_1+\cdots+u_r\circ \pi_r
\end{equation*}
is a linear topological isomorphism. Its inverse is
\begin{equation*}
L(v)=(v\circ i_1,\cdots,v\circ i_r)
\end{equation*}
where for all $1\leq k\leq r$, $i_k: X_k\rightarrow X_1\times
\cdots\times X_r$ is defined by
\begin{equation*}
i_k(z)=(0,\cdots,0,\underbrace{z}_{\textrm{$k^{th}$
position}},0,\cdots,0)
\end{equation*}
\end{theorem}
The notion of adjoint operator, which frequently appears in the future sections, is introduced in the following theorem.
\begin{theorem}[\cite{Reus1}, Page 163] \lab{thmfallinjectiveadjoint1} Let $X$ and $Y$ be locally
convex topological vector spaces and suppose $T:X\rightarrow Y$
is a continuous linear map. Then
\begin{enumerateXALI}
\item the map
\begin{equation*}
T^*: Y^*\rightarrow X^*\qquad \langle T^*y,x \rangle_{X^*\times
X}=\langle y,Tx \rangle_{Y^*\times Y}
\end{equation*}
is well-defined, linear, and continuous. ($T^*$ is called the
\textbf{adjoint} of $T$.)
\item If $T(X)$ is dense in $Y$, then $T^*:Y^*\rightarrow X^*$ is
injective.
\end{enumerateXALI}
\end{theorem}
\begin{remark}\lab{winter34}
In the subsequent sections we will focus heavily on certain
function spaces on domains $\Omega$ in the Euclidean space. For approximation purposes, it is always
 desirable to have $D(\Omega)(=C_c^{\infty}(\Omega))$ as a dense subspace of our function spaces. However, there is another, may be more profound, reason for being interested in having $D(\Omega)$ as a dense subspace. It is
important to note that we would like to use the term ``function
spaces" for topological vector spaces that can be continuously
embedded in $D'(\Omega)$ (see Section 6 for the definition of $D'(\Omega)$) so that concepts such as differentiation
will be meaningful for the elements of our function spaces. Given
a function space $A(\Omega)$ it is usually helpful to consider its
dual too. In order to be able to view the dual of $A(\Omega)$ as
a function space we need to ensure that $[A(\Omega)]^*$ can be
viewed as a subspace of $D'(\Omega)$. To this end, according to
the above theorem, it is enough to ensure that the identity map from $D(\Omega)$ to $A(\Omega)$ is continuous with
dense image in $A(\Omega)$.
\end{remark}
Let us consider more closely two special cases of Theorem
\ref{thmfallinjectiveadjoint1}.
\begin{enumerateXALI}
\item Suppose $Y$ is a normed space and $H$ is a dense subspace of
$Y$. Clearly the identity map $i: H\rightarrow Y$ is continuous
with dense image. Therefore $i^*: Y^*\rightarrow H^*$ ($F\mapsto
F|_H $) is continuous and injective. Furthermore, by the
Hahn-Banach theorem for all $\varphi\in H^*$ there exists $F\in
Y^*$ such that $F|_H=\varphi$ and $\|F\|_{Y^*}=\|\varphi\|_{H^*}$. So the above map
is indeed bijective and $Y^*$ and $H^*$ are isometrically
isomorphic. As an important example, let $\Omega$ be a nonempty open set in $\reals^n$, $s\geq 0$, and $1<p<\infty$. Consider the space $W^{s,p}_0(\Omega)$ (see Section 7 for the definition of $W^{s,p}_0(\Omega)$). $C_c^\infty(\Omega)$ is a dense subspace of $W^{s,p}_0(\Omega)$. Therefore $W^{-s,p'}(\Omega):= [W^{s,p}_0(\Omega)]^*$ is isometrically
isomorphic to $[(C_c^\infty(\Omega), \|.\|_{s,p})]^*$. In particular, if $F\in W^{-s,p'}(\Omega)$, then
\begin{equation*}
\|F\|_{W^{-s,p'}(\Omega)}=\sup_{0\not\equiv \psi\in C_c^\infty(\Omega)}\frac{|F(\psi)|}{\|\psi\|_{s,p}}
\end{equation*}
\item Suppose $(Y,\|.\|_Y)$ is a normed space, $(X,\tau)$ is a locally convex
topological vector space,  $X\subseteq Y$, and the identity map
$i: (X,\tau)\rightarrow (Y,\|.\|_Y)$ is continuous with dense
image. So $i^*: Y^*\rightarrow X^*$ ($F\mapsto F|_X$) is
continuous and injective and can be used to identify $Y^*$ with a
subspace of $X^*$.
\begin{itemize}
\item \textbf{Question:} Exactly what elements of $X^*$ are in the
image of $i^*$? That is, which elements of $X^*$ ``belong to"
$Y^*$?
\item \textbf{Answer:} $\varphi\in X^*$ belongs to the image of $i^*$
if and only if $\varphi: (X,\|.\|_Y)\rightarrow \reals$ is
continuous, that is, $\varphi\in X^*$ belongs to the image of $i^*$
if and only if {\fontsize{10}{10}{$\sup_{x\in X\setminus
\{0\}}\frac{|\varphi(x)|}{\|x\|_Y}<\infty$}}.
\end{itemize}
So an element $\varphi\in X^*$ can be considered as an element of
$Y^*$ if and only if
\begin{equation*}
\sup_{x\in X\setminus
\{0\}}\frac{|\varphi(x)|}{\|x\|_Y}<\infty\,.
\end{equation*}
Furthermore if we denote the unique corresponding element in $Y^*$ by
$\tilde{\varphi}$ (normally we identify $\varphi$ and
$\tilde{\varphi}$ and we use the same notation for both) then since
$X$ is dense in $Y$
\begin{equation*}
\|\tilde{\varphi}\|_{Y^*}=\sup_{y\in Y\setminus
\{0\}}\frac{|\tilde{\varphi}(y)|}{\|y\|_Y}=\sup_{x\in
X\setminus \{0\}}\frac{|\varphi(x)|}{\|x\|_Y}<\infty
\end{equation*}
\begin{remark}\lab{winter35}
To sum up, given an element $\varphi\in X^*$ in order to show
that $\varphi$ can be considered as an element of $Y^*$ we just
need to show that $\sup_{x\in X\setminus
\{0\}}\frac{|\varphi(x)|}{\|x\|_Y}<\infty$ and in that case, norm
of $\varphi$ as an element of $Y^*$ is $\sup_{x\in X\setminus
\{0\}}\frac{|\varphi(x)|}{\|x\|_Y}$. However, it is important to
notice that if $F:Y\rightarrow \reals$ is a linear map, $X$ is a dense subspace of $Y$, and $F|_X: (X,\|.\|_Y)\rightarrow \reals$
is bounded, that does NOT imply that $F\in Y^*$. It just shows
that there exists $G\in Y^*$ such that $G|_X=F|_X$.
\end{remark}
\end{enumerateXALI}
We conclude this section by a quick review of the inductive limit topology.
\begin{definition}\lab{winter36}
Let $X$ be a vector space and let $\{X_\alpha\}_{\alpha\in I}$ be a family of vector subspaces of $X$ with the property that
\begin{itemize}
\item for each $\alpha \in I$, $X_\alpha$ is equipped with a topology that makes it a locally convex topological vector space, and
\item $\bigcup_{\alpha\in I} X_\alpha=X$.
\end{itemize}
The \textbf{inductive limit topology} on $X$ with respect to the family $\{X_\alpha\}_{\alpha\in I}$ is defined to be the largest topology with
respect to which
\begin{enumerate}
\item $X$ is a locally convex topological vector space, and
\item all the inclusions $X_\alpha\subseteq
X$ are continuous.
\end{enumerate}
\end{definition}
\begin{theorem}\lab{winter37}
Let $X$ be a vector space equipped with the inductive limit topology with respect to $\{X_\alpha\}$ as described above. If $Y$ is a locally convex vector space, then a linear map $T:X\rightarrow Y$ is continuous if and only if $T|_{X_\alpha}: X_\alpha\rightarrow Y$ is continuous for all $\alpha\in I$.
\end{theorem}
\begin{theorem}\lab{winter38}
Let $X$ be a vector space and let $\{X_j\}_{j\in \mathbb{N}_0}$ be a nested family of vector subspaces of $X$:
\begin{equation*}
X_0\subsetneq X_1\subsetneq \cdots \subsetneq X_j\subsetneq \cdots
\end{equation*}
Suppose each $X_j$ is equipped with a topology that makes it a locally convex topological vector space. Equip $X$ with the inductive limit topology with respect to $\{X_j\}$.
Then the following topologies on $X^{\times r}$ are equivalent (=they are the same)
\begin{enumerate}
\item The product topology
\item The inductive limit topology with respect to the family $\{X_j^{\times r}\}$. (For each $j$, $X_j^{\times r}$ is equipped with the product topology)
\end{enumerate}
As a consequence, if $Y$ is a locally convex vector space, then a
linear map $T:X^{\times r}\rightarrow Y$ is continuous if and
only if $T|_{X_j^{\times r}}: X_j^{\times r}\rightarrow Y$ is
continuous for all $j\in \mathbb{N}_0$.
\end{theorem}


\section{Review of Some Results From Differential Geometry}
The main purpose of this section is to set the notations and
terminology straight. To this end we cite the definitions of
several basic terms and a number of basic properties that we will
frequently use. The main reference for the majority of the
definitions is the invaluable book by John M. Lee (\cite{Lee3}).

\subsection{Smooth Manifolds}
 Suppose $M$ is a topological space. We say that $M$ is a
topological manifold of dimension $n$ if it is Hausdorff,
second-countable, and locally Euclidean in the sense that each
point of $M$ has a neighborhood that is homeomorphic to an open
subset of $\reals^n$. It is easy to see that the following
statements are equivalent (\cite{Lee3}, Page 3):
\begin{enumerateXALI}
\item Each point of $M$
 has a neighborhood that is homeomorphic to an open subset of
 $\reals^n$.
\item Each point of $M$
 has a neighborhood that is homeomorphic to an open ball in
 $\reals^n$.
\item Each point of $M$
 has a neighborhood that is homeomorphic to
 $\reals^n$.
\end{enumerateXALI}
By a \textbf{coordinate chart} (or just \textbf{chart}) on $M$ we
mean a pair $(U,\varphi)$, where $U$ is an open subset of $M$ and
$\varphi: U\rightarrow \hat{U}$ is a homeomorphism from $U$ to an
open subset $\hat{U}=\varphi(U)\subseteq \reals^n$. $U$ is called
a \textbf{coordinate domain} or a \textbf{coordinate neighborhood}
of each of its points and $\varphi$ is called a
\textbf{coordinate map}. An \textbf{atlas for $M$} is a
collection of charts whose domains cover $M$. Two charts
$(U,\varphi)$ and $(V,\psi)$ are said to be \textbf{smoothly
compatible} if either $U\cap V=\emptyset$ or the transition map
$\psi\circ \varphi^{-1}$ is a $C^\infty$-diffeomorphism. An atlas
$\mathcal{A}$ is called a \textbf{smooth atlas} if any two charts
in $\mathcal{A}$ are smoothly compatible with each other. A
smooth atlas $\mathcal{A}$ on $M$ is \textbf{maximal} if
it is not properly contained in any larger smooth atlas.  A
\textbf{smooth structure} on $M$ is a maximal smooth atlas. A
\textbf{smooth manifold} is a pair $(M,\mathcal{A})$, where $M$ is a
topological manifold and $\mathcal{A}$ is a smooth structure on
$M$. Any chart $(U,\varphi)$ contained in the given maximal
smooth atlas is called a \textbf{smooth chart}. If $M$ and $N$ are two smooth manifolds, a map $F:M\rightarrow N$ is said to be a smooth $(C^\infty)$ map if for every $p\in M$, there exist smooth charts $(U,\varphi)$ containing $p$ and $(V,\psi)$ containing $F(p)$ such that $F(U)\subseteq V$ and $\psi\circ F\circ \varphi^{-1}\in C^\infty(\varphi(U))$. It can be shown that if $F$ is smooth, then its restriction to every open subset of $M$ is smooth. Also if every $p\in M$ has a neighborhood $U$ such that $F|_U$ is smooth, then $F$ is smooth.
\begin{remark}\lab{winter39}
\leavevmode
\begin{itemizeX}
\item Sometimes we use the shorthand notation $M^n$ to indicate that $M$ is $n$-dimensional.
\item Clearly if $(U,\varphi)$ is a smooth chart and $V$ is an open
subset of $U$, then $(V,\psi)$ where $\psi=\varphi|_V$ is also a
smooth chart (i.e. it belongs to the same maximal atlas).
\item Every smooth atlas $\mathcal{A}$ for $M$ is contained in a unique maximal smooth
atlas, called the \textbf{smooth structure determined by
$\mathcal{A}$}.
\item If $M$ is a compact smooth manifold, then there exists a smooth atlas with finitely many elements that determines the smooth
structure of $M$ (this is immediate from the definition of
compactness).
\end{itemizeX}
\end{remark}
\begin{definition}{\noindent}\lab{winter40}
\begin{itemizeX}
\leavevmode
\item We say that a smooth atlas for a smooth manifold $M$ is a
\textbf{geometrically Lipschitz (GL)} smooth atlas if the image of
each coordinate domain in the atlas under the corresponding
coordinate map is a nonempty bounded open set with Lipschitz
boundary.
\item We say that a smooth atlas for a smooth manifold $M^n$ is a
\textbf{generalized geometrically Lipschitz (GGL)} smooth atlas if the image of
each coordinate domain in the atlas under the corresponding
coordinate map is the entire $\reals^n$ or a nonempty bounded open set with Lipschitz
boundary.
\item We say that a smooth atlas for a smooth manifold $M^n$ is a
\textbf{nice} smooth atlas if the image of each coordinate domain
in the atlas under the corresponding coordinate map is a ball in
$\reals^n$.
\item We say that a smooth atlas for a smooth manifold $M^n$ is a
\textbf{super nice} smooth atlas if the image of each coordinate domain
in the atlas under the corresponding coordinate map is the entire
$\reals^n$.
\item We say that two smooth atlases $\{(U_\alpha,\varphi_\alpha)\}_{\alpha\in I}$ and $\{(\tilde{U}_\beta,\tilde{\varphi}_\beta)\}_{\beta\in J}$
for a smooth manifold $M^n$ are \textbf{geometrically Lipschitz
compatible (GLC)} smooth atlases provided that each atlas is GGL and moreover for all $\alpha\in I$ and $\beta\in J$ with
$U_\alpha\cap\tilde{U}_\beta \neq \emptyset$,
$\varphi_\alpha(U_\alpha\cap\tilde{U}_\beta)$ and
$\tilde{\varphi}_\beta(U_\alpha\cap\tilde{U}_\beta)$ are nonempty
bounded open sets with Lipschitz boundary or the entire $\reals^n$.
\end{itemizeX}
\end{definition}
Clearly every super nice smooth atlas is also a GGL smooth atlas; every nice smooth atlas is also a GL smooth atlas, and every GL smooth atlas is also a GGL smooth atlas. Also
note that two arbitrary GL smooth atlases are not necessarily GLC
smooth atlases because the intersection of two Lipschitz domains is
not necessarily Lipschitz (see e.g. \cite{Bastos2014}, pages
115-117).\\

Given a smooth atlas $\{(U_\alpha,\varphi_\alpha)\}$ for a
compact smooth manifold $M$, it is not necessarily possible to construct
a new atlas
 $\{(U_\alpha,\tilde{\varphi}_\alpha)\}$ such that this new atlas is
 nice; for instance if $U_\alpha$ is not connected
 we cannot find $\tilde{\varphi}_\alpha$ such that
 $\tilde{\varphi}_\alpha(U_\alpha)=\reals^n$ (or any ball in
 $\reals^n$). However, as the following lemma states it is always
 possible to find a refinement that is nice.
 \begin{lemma}\lab{lemfallniceatlas1}
Suppose $\{(U_\alpha,\varphi_\alpha)\}_{1\leq \alpha\leq N}$ is a
smooth atlas for a compact smooth manifold $M$. Then there exists
a finite open cover $\{V_\beta \}_{1\leq \beta\leq
 L}$ of $M$ such that
\begin{equation*}
\forall\,\beta\qquad \exists 1\leq \alpha(\beta)\leq
N\,\,\textrm{s.t.}\,\,\quad V_\beta\subseteq
U_{\alpha(\beta)},\quad \varphi_{\alpha(\beta)}(V_\beta)\,
\textrm{is a ball in $\reals^n$}
\end{equation*}
Therefore $\{(V_\beta,\varphi_{\alpha(\beta)}|_{V_\beta})\}_{1\leq
\beta\leq L}$ is a nice smooth atlas.
\end{lemma}
\begin{proof}
For each $1\leq \alpha\leq N$ and $p\in U_\alpha$, there exists
$r_{\alpha p}>0$ such that $B_{r_{\alpha
p}}(\varphi_\alpha(p))\subseteq \varphi_\alpha(U_\alpha)$. Let
$V_{\alpha p}:=\varphi_\alpha^{-1}(B_{r_{\alpha
p}}(\varphi_\alpha(p)))$. $\bigcup_{1\leq \alpha\leq
N}\bigcup_{p\in U_\alpha}V_{\alpha p}$ is an open cover of $M$ and
so it has a finite subcover $\{V_{\alpha_1p_1},\cdots,
V_{\alpha_Lp_L}\}$. Let $V_\beta=V_{\alpha_\beta p_\beta}$.
Clearly, $V_\beta\subseteq U_{\alpha_\beta}$ and
$\varphi_{\alpha_\beta}(V_\beta)$ is a ball in $\reals^n$.
\end{proof}
\begin{remark}\lab{remwinter41b}
Every open ball in $\reals^n$ is $C^\infty$-diffeomorphic to $\reals^n$. Also compositions of diffeomorphisms is a diffeomorphism. Therefore existence of a finite nice smooth atlas on a compact smooth manifold (which is guaranteed by the above lemma) implies the existence of a finite super nice smooth atlas.
\end{remark}
\begin{lemma}\lab{winter41}
Let $M$ be a compact smooth manifold. Let $\{U_\alpha\}_{1\leq
\alpha\leq N}$ be an open cover of $M$. Suppose $C$ is a closed
set in $M$ (so $C$ is compact) which is contained in $U_\beta$
for some $1\leq \beta\leq N$. Then there exists an open cover
$\{A_\alpha\}_{1\leq \alpha\leq N}$ of $M$ such that $C\subseteq
A_\beta\subseteq \bar{A}_\beta\subseteq U_\beta$ and
$A_\alpha\subseteq \bar{A}_\alpha\subseteq U_\alpha$ for all
$\alpha\neq \beta$.
\end{lemma}
\begin{proof}
Without loss of generality we may assume that $\beta=1$. For each
$1\leq \alpha\leq N$ and $p\in U_\alpha$, there exists $r_{\alpha
p}>0$ such that $B_{2r_{\alpha p}}(\varphi_\alpha(p))\subseteq
\varphi_\alpha(U_\alpha)$. Let $V_{\alpha
p}:=\varphi_\alpha^{-1}(B_{r_{\alpha p}}(\varphi_\alpha(p)))$.
Clearly $p\in V_{\alpha p}\subseteq \bar{V}_{\alpha p}\subseteq
U_\alpha$. Since $M$ is compact, the open cover $\bigcup_{1\leq
\alpha\leq N}\bigcup_{p\in U_\alpha} V_{\alpha p}$ of $M$ has a
finite subcover $\mathcal{A}$. For each $1\leq \alpha\leq N$ let
$E_\alpha=\{p\in U_\alpha: V_{\alpha p}\in \mathcal{A}\}$ and
\begin{equation*}
I_1=\{\alpha: E_\alpha\neq \emptyset\}
\end{equation*}
If $\alpha\in I_1$, we let $W_\alpha=\bigcup_{p\in
E_\alpha}V_{\alpha p}$. For $\alpha \not\in I_1$ choose one point
$p\in U_\alpha$ and let $W_\alpha=V_{\alpha p}$.\\
 $C$ is compact
 so $\varphi_1(C)$ is a compact set inside the open set
$\varphi_1(U_1)$. Therefore there exists an open set $B$ such that
\begin{equation*}
\varphi_1(C)\subseteq B\subseteq \bar{B}\subseteq \varphi_{1}(U_1)
\end{equation*}
Let $W=\varphi_1^{-1}(B)$. Clearly $C\subseteq W\subseteq
\bar{W}\subseteq U_\alpha$. Now Let
\begin{align*}
& A_1=W\bigcup W_1\\
& A_\alpha= W_\alpha\quad \forall \alpha>1
\end{align*}
Clearly $A_1$ contains $W$ which contains $C$. Also union of
$A_\alpha$'s contains $\bigcup_{\alpha=1}^N\bigcup_{p\in E_\alpha}
V_{\alpha p}$ which is equal to $M$. Closure of a union of sets
is a subset of the union of closures of those sets. Therefore for
each $\alpha$, $\bar{A}_\alpha\subseteq U_\alpha$.
\end{proof}
\begin{theorem}[Exhaustion by Compact Sets for
Manifolds]\lab{winter42} Let $M$ be a smooth manifold. There
exists a sequence of compact subsets $(K_j)_{j\in\mathbb{N}}$
such that $\cup_{j\in \mathbb{N}} \mathring{K}_j=M$,
$\mathring{K}_{j+1}\setminus K_j\neq \emptyset$ for all $j$ and
\begin{equation*}
K_1\subseteq \mathring{K}_2\subseteq K_2\subseteq \cdots\subseteq
\mathring{K}_j\subseteq K_j\subseteq \cdots
\end{equation*}
\end{theorem}
\begin{definition}\lab{winter43}
A $C^{\infty}$ partition of unity on a smooth manifold is a collection
of nonnegative $C^{\infty}$ functions $\{\psi_\alpha:M\rightarrow
\reals\}_{\alpha\in A}$ such that
\begin{enumerate}[(i)]
\item the collection of supports, $\{\textrm{supp}\,\psi_\alpha\}_{\alpha\in
A}$ is locally finite in the sense that every point in $M$ has a neighborhood that intersects only finitely many
 of the sets in $\{\textrm{supp}\,\psi_\alpha\}_{\alpha\in
A}$.
\item $\sum \psi_\alpha=1$.
\end{enumerate}
Given an open cover $\{U_\alpha\}_{\alpha\in A}$ of $M$, we say
that a partition of unity $\{\psi_\alpha\}_{\alpha\in A}$ is
subordinate to the open cover $\{U_\alpha\}$ if
$\textrm{supp}\,\psi_\alpha\subseteq U_\alpha$ for every
$\alpha\in A$.
\end{definition}
\begin{theorem}(\cite{loring2011}, Page
146)\lab{thmapp5} Let $M$ be a \textbf{compact} smooth manifold and
$\{U_{\alpha}\}_{\alpha\in A}$ an open cover of $M$. There exists
a $C^\infty$ partition of unity $\{\psi_\alpha\}_{\alpha\in A}$
subordinate to $\{U_{\alpha}\}_{\alpha\in A}$. (Notice that the
index sets are the same.)
\end{theorem}
\begin{theorem}(\cite{loring2011}, Page
347)\lab{thmapp6} Let $\{U_{\alpha}\}_{\alpha\in A}$ be an open
cover of a smooth manifold $M$.
\begin{enumerate}[(i)]
\item There is a $C^\infty$ partition of unity
$\{\varphi_k\}_{k=1}^\infty$ with every $\varphi_k$ \textbf{having
compact support} such that for each $k$,
$\textrm{supp}\,\varphi_k\subseteq U_\alpha$ for some $\alpha\in
A$.
\item If we do not require compact support, then there is a
$C^\infty$ partition of unity $\{\psi_\alpha\}_{\alpha\in A}$
subordinate to $\{U_{\alpha}\}_{\alpha\in A}$.
\end{enumerate}
\end{theorem}
\begin{remark}\lab{winter44}
Let $M$ be a compact smooth manifold. Suppose
$\{U_{\alpha}\}_{\alpha\in A}$ is an open cover of $M$ and
 $\{\psi_\alpha\}_{\alpha\in A}$ is a partition of unity
 subordiante to $\{U_{\alpha}\}_{\alpha\in A}$.
\begin{itemizeXXALI}
\item For all $m\in \mathbb{N}$, $\{\tilde{\psi}_\alpha=\frac{\psi_\alpha^m}{\sum_{\alpha\in
A}\psi_\alpha^m}\}$ is another partition of unity subordinate to
$\{U_{\alpha}\}_{\alpha\in A}$.
\item If $\{V_{\beta}\}_{\beta\in B}$ is an open cover of $M$ and
$\{\xi_\beta\}$ is a partition of unity subordinate to
$\{V_{\beta}\}_{\beta\in B}$, then $\{\psi_\alpha
\xi_\beta\}_{(\alpha,\beta)\in A\times B}$ is a partition of
unity subordinate to the open cover $\{U_\alpha\cap
V_\beta\}_{(\alpha,\beta)\in A\times B}$.
\end{itemizeXXALI}
\end{remark}
\begin{lemma}\lab{lemapp6}
Let $M$ be a compact smooth manifold. Suppose $\{U_\alpha\}_{1\leq
\alpha \leq N}$ is an open cover of $M$. Suppose $C$ is a closed
set in $M$ (so $C$ is compact) which is contained in $U_\beta$
for some $1\leq \beta\leq N$. Then there exists a partition of
unity $\{\psi_{\alpha}\}_{1\leq \alpha\leq N}$ subordinate to
$\{U_\alpha\}_{1\leq \alpha\leq N}$ such that $\psi_\beta=1$ on
$C$.
\end{lemma}
\begin{proof}
We follow the argument in \cite{13}. Without loss of generality we may assume $\beta=1$. We can
construct a partition of unity with the desired property as
follows: Let $A_{\alpha}$ be a collection of open sets that
covers $M$ and such that $C\subseteq A_1\subseteq
\bar{A}_1\subseteq U_1$ and for $\alpha>1$, $A_{\alpha}\subseteq
\bar{A}_{\alpha}\subseteq U_{\alpha}$ (see Lemma \ref{winter41}). Let $\eta_\alpha\in
C_c^{\infty}(U_{\alpha})$ be such that $0\leq \eta_\alpha\leq 1$
and $\eta_\alpha=1$ on a neighborhood of $\bar{A}_{\alpha}$. Of
course $\sum_{\alpha=1}^N \eta_{\alpha}$ is not necessarily equal to $1$ for
all $x\in M$. However, if we define $\psi_1=\eta_1$ and for
$\alpha>1$
\begin{equation*}
\psi_{\alpha}=\eta_{\alpha}(1-\eta_1)\cdots(1-\eta_{\alpha-1})
\end{equation*}
by induction one can easily show that for $1\leq l \leq N$
\begin{equation*}
1-\sum_{\alpha=1}^{l}\psi_{\alpha}=(1-\eta_1)\cdots(1-\eta_l)
\end{equation*}
In particular,
\begin{equation*}
1-\sum_{\alpha=1}^{N}\psi_{\alpha}=(1-\eta_1)\cdots(1-\eta_N)=0
\end{equation*}
since for each $x\in M$ there exists $\alpha$ such that $x\in
A_{\alpha}$ and so $\eta_{\alpha}(x)=1$. Consequently
$\sum_{\alpha=1}^{N}\psi_{\alpha}=1$.
\end{proof}
\subsection{Vector Bundles, Basic Definitions}
Let $M$ be a smooth manifold. A (smooth real) \textbf{vector
bundle} of rank $r$ over $M$ is a smooth manifold $E$ together
with a surjective smooth map $\pi:E\rightarrow M$ such that
\begin{enumerate}
\item for each $x\in M$, $E_x=\pi^{-1}(x)$ is an $r$-dimensional
(real) vector space.
\item for each $x\in M$, there exists a neighborhood $U$ of $x$ in
 $M$ and a smooth map $\rho= (\rho^1,\cdots,\rho^r)$ from
$E|_U:=\pi^{-1}(U)$ onto $\reals^r$ such that
\begin{itemize}
\item for every $x\in U$, $\rho|_{E_x}: E_x\rightarrow \reals^r$
is an isomorphism of vector spaces
\item $\Phi=(\pi|_{E_U},\rho):E_U\rightarrow U\times \reals^r$ is a
diffeomorphism.
\end{itemize}
\end{enumerate}
We denote the projection onto the last $r$ components by $\pi'$.
So $\pi'\circ \Phi=\rho$. The expressions "$E$ is a vector bundle over $M$", or "$E\rightarrow M$ is a vector bundle", or "$\pi:E\rightarrow M$ is a vector bundle"
 are all considered to be equivalent in this manuscript. We refer to both $\Phi: E_U\rightarrow
U\times \reals^r$ and $\rho: E_U\rightarrow \reals^r$ as a
(smooth) \textbf{local trivialization} of $E$ over $U$ (it will
be clear from the context which one we are referring to). We say
that $E|_U$ is trivial. The pair $(U,\rho)$ (or $(U,\Phi)$) is
sometimes called a \textbf{vector bundle chart}. It is easy to
see that if $(U,\rho)$ is a vector bundle chart and
$\emptyset\neq V\subseteq U$ is open, then $(V,\rho|_{E_V})$ is
also a vector bundle chart for $E$. Moreover, if $V$ is any
nonempty open subset of $M$, then $E_V$ is a vector bundle over
the manifold $V$.
 We say that a triple $(U,\varphi, \rho)$ is
a \textbf{total trivialization triple} of the vector bundle $\pi:
E\rightarrow M$ provided that $(U,\varphi)$ is a smooth
coordinate chart and $\rho= (\rho^1,\cdots,\rho^r):
E_U\rightarrow \reals^r$ is a trivialization of $E$ over $U$. A
collection $\{(U_\alpha, \varphi_\alpha,\rho_\alpha)\}$ is called
a \textbf{total trivialization atlas} for the vector bundle
$E\rightarrow M$ provided that for each $\alpha$, $(U_\alpha,
\varphi_\alpha,\rho_\alpha)$ is a total trivialization triple and
$\{(U_\alpha,\varphi_\alpha)\}$ is a smooth atlas for $M$. The
following statements show that any vector bundle has a total
trivialization atlas.
\begin{lemma}(\cite{Wal2004}, Page 77) \lab{winter45}
Let $E$ be a vector bundle over an $n$-dimensional smooth
manifold $M$ ($M$ does not need to be compact). Then $M$ can be
covered by $n+1$ open sets $V_0,\cdots,V_n$ where the restriction
$E|_{V_i}$ is trivial.
\end{lemma}
\begin{theorem}\lab{thmfalltrivatlas}
Let $E$ be a vector bundle of rank $r$ over an $n$-dimensional
smooth manifold $M$. Then $E\rightarrow M$ has a total trivialization atlas. In
particular, if $M$ is compact, then it has a total trivialization
atlas that consists of only finitely many total trivialization
triples.
\end{theorem}
\begin{proof}
Let $V_0,\cdots,V_n$ be an open cover of $M$ such that $E$ is
trivial over $V_\beta$ with the mapping
$\rho_\beta:E_{V_\beta}\rightarrow \reals^r$. Let
$\{(U_\alpha,\varphi_\alpha)\}_{\alpha\in I}$ be a smooth atlas
for $M$ (if $M$ is compact, the index set $I$ can be chosen to be
finite). For all $\alpha\in I$ and $0\leq \beta\leq n$ let
$W_{\alpha \beta}=U_\alpha\cap V_\beta$. Let $J=\{(\alpha,\beta):
W_{\alpha \beta}\neq \emptyset \}$. Clearly $\{(W_{\alpha \beta},
\varphi_{\alpha \beta}, \rho_{\alpha \beta} )\}_{(\alpha,\beta)\in
J }$ where $\varphi_{\alpha
\beta}=\varphi_\alpha|_{W_{\alpha\beta}}$ and
$\rho_{\alpha\beta}=\rho_\beta|_{\pi^{-1}(W_{\alpha\beta})}$ is a total
trivialization atlas for $E\rightarrow M$.
\end{proof}

\begin{definition}{\noindent}\lab{winter46}
\begin{itemizeX}
\leavevmode
\item We say that a total trivialization triple $(U,\varphi,\rho)$
is \textbf{geometrically Lipschitz (GL)} provided that
$\varphi(U)$ is a nonempty bounded open set with Lipschitz
boundary. A total trivialization atlas is called
\textbf{geometrically Lipschitz} if each of its total
trivialization triples is GL.
\item We say that a total trivialization triple $(U,\varphi,\rho)$
is \textbf{nice} provided that $\varphi(U)$ is equal to 
a ball in $\reals^n$. A total trivialization atlas is called nice
if each of its total trivialization triples is nice.
\item We say that a total trivialization triple $(U,\varphi,\rho)$
is \textbf{super nice} provided that $\varphi(U)$ is equal to $\reals^n$. A total trivialization atlas is called super nice if each of its total trivialization triples is super nice.
\item A total trivialization atlas is called
\textbf{generalized geometrically Lipschitz (GGL)} if each of its total
trivialization triples is GL or super nice.
\item We say that two total trivialization atlases
$\{(U_\alpha,\varphi_\alpha,\rho_\alpha)\}_{\alpha\in I}$ and
$\{(\tilde{U}_\beta,\tilde{\varphi}_\beta,\tilde{\rho}_\beta)\}_{\beta\in
J}$ are \textbf{geometrically Lipschitz compatible (GLC)} if the
corresponding atlases\\ $\{(U_\alpha,\varphi_\alpha)\}_{\alpha\in
I}$ and $\{(\tilde{U}_\beta,\tilde{\varphi}_\beta)\}_{\beta\in
J}$ are GLC.
\end{itemizeX}
\end{definition}
\begin{theorem}\lab{thmfallnicetrivatlas2}
Let $E$ be a vector bundle of rank $r$ over an $n$-dimensional
compact smooth manifold $M$. Then $E$ has a nice total
trivialization atlas (and a super nice total trivialization atlas) that consists of only finitely many total
trivialization triples.
\end{theorem}
\begin{proof}
By Theorem \ref{thmfalltrivatlas}, $E\rightarrow M$ has a finite
total trivialization atlas $\{(U_\alpha,
\varphi_\alpha,\rho_\alpha)\}$. By Lemma \ref{lemfallniceatlas1} (and Remark \ref{remwinter41b})
there exists a finite open cover $\{V_\beta \}_{1\leq \beta\leq
 L}$ of $M$ such that
\begin{align*}
&\forall\,\beta\qquad \exists 1\leq \alpha(\beta)\leq
N\,\,\textrm{s.t.}\,\,\quad V_\beta\subseteq
U_{\alpha(\beta)},\quad \varphi_{\alpha(\beta)}(V_\beta)\,
\textrm{is a ball in $\reals^n$}\\
(& \textrm{or}\,\,\forall\,\beta\qquad \exists 1\leq \alpha(\beta)\leq
N\,\,\textrm{s.t.}\,\,\quad V_\beta\subseteq
U_{\alpha(\beta)},\quad \varphi_{\alpha(\beta)}(V_\beta)=\reals^n)
\end{align*}
and thus $\{(V_\beta,\varphi_{\alpha(\beta)}|_{V_\beta})\}_{1\leq
\beta\leq L}$ is a nice (resp. super nice) smooth atlas. Now clearly $\{(V_\beta,\varphi_{\alpha(\beta)}|_{V_\beta},\rho_{\alpha(\beta)}|_{E_{V_\beta}})\}_{1\leq
\beta\leq L}$ is a nice (resp. super nice) total trivialization atlas.
\end{proof}
\begin{theorem}\lab{winter47jan}
Let $E$ be a vector bundle of rank $r$ over an $n$-dimensional
compact smooth manifold $M$. Then $E$ admits a finite total
trivialization atlas that is GL compatible with itself. In fact, there exists a total trivialization atlas $\{(U_\alpha,\varphi_\alpha,\rho_\alpha)\}_{1\leq \alpha\leq N}$ such that
\begin{itemizeX}
\item for all $1\leq \alpha\leq N$, $\varphi_\alpha(U_\alpha)$ is bounded with Lipschitz continuous boundary, and,
\item for all $1\leq \alpha, \beta\leq N$, $U_\alpha\cap U_\beta$ is either empty or else $\varphi_\alpha(U_\alpha\cap U_\beta)$ and $\varphi_\beta(U_\alpha\cap U_\beta)$ are bounded with Lipschitz continuous boundary.
\end{itemizeX}
\end{theorem}
\begin{proof}
The proof of this theorem is based on the argument presented in the proof of Lemma 3.1 in \cite{Inc2013}. Equip $M$ with a smooth Riemannian metric $g$. Let $r_{inj}$ denote the injectivity radius of $M$ which is strictly positive because $M$ is compact. Let $V_0,\cdots,V_n$ be an open cover of $M$ such that $E$ is
trivial over $V_\beta$ with the mapping
$\rho_\beta:E_{V_\beta}\rightarrow \reals^r$. For every $x\in M$ choose $0\leq i(x)\leq n$ such that $x\in V_{i(x)}$.   For all $x\in M$ let $r_{x}$ be a positive number less than $\frac{r_{inj}}{2}$ such that $\textrm{exp}_x(B_{r_{x}})\subseteq V_{i(x)}$ where $B_{r_{x}}$ denotes the open ball in $T_x M$ of radius $r_{x}$ (with respect to the inner product induced by the Riemannian metric $g$) and $\textrm{exp}_x:T_x M\rightarrow M$ denotes the exponential map at $x$. For every $x\in M$ define the normal coordinate chart centered at $x$ , $(U_x,\varphi_x)$, as follows:
\begin{equation*}
U_x=\textrm{exp}_x(B_{r_{x}}),\quad \varphi_x:=\lambda_x^{-1}\circ \textrm{exp}_x^{-1}: U_x\rightarrow  \reals^n,
\end{equation*}
where $\lambda_x:\reals^n\rightarrow T_xM$ is an isomorphism defined by $\lambda_x(y^1,\cdots,y^n)=y^iE_{ix}$; Here $\{E_{ix}\}_{i=1}^n$ is a an arbitrary but fixed orthonormal basis for $T_xM$. It is well-known that (see e.g. \cite{Lee2})
\begin{itemizeX}
\item $\varphi_x(x)=(0,\cdots,0)$
\item $g_{ij}(x)=\delta_{ij}$ where $g_{ij}$ denotes the components of the metric with respect to the normal coordinate chart $(U_x,\varphi_x)$.
\item $E_{ix}=\partial_i|_x$ where $\{\partial_i\}_{1\leq i\leq n}$ is the coordinate basis induced by $(U_x,\varphi_x)$.
\end{itemizeX}
As a consequence of the previous items, it is easy to show that if $X\in T_xM$ $(X=X^i\partial_i|_x)$, then the Euclidean norm of $X$ will be equal to the norm of $X$ with respect to the metric $g$, that is $|X|_g=|X|_{\bar{g}}$ where
\begin{equation*}
|X|_{\bar{g}}=\sqrt{(X^1)^2+\cdots+(X^n)^2}\,\quad |X|_g=\sqrt{g(X,X)}
\end{equation*}
Consequently, for every $x\in M$, $\varphi_x(U_x)$ will be a ball in the Euclidean space, in particular, $\{(U_x,\varphi_x)\}_{x\in M}$ is a GL atlas. The proof of Lemma 3.1 in \cite{Inc2013} in part shows that the atlas $\{(U_x,\varphi_x)\}_{x\in M}$ is GL compatible with itself. Since $M$ is compact there exists $x_1,\cdots,x_N\in M$ such that $\{U_{x_j}\}_{1\leq j\leq N}$ also covers $M$.\\ Now clearly $\{(U_{x_j},\varphi_{x_j},\rho_{i(x_j)}|_{U_{x_j}})\}_{1\leq j\leq N}$ is a total trivialization atlas for $E$ that is GL compatible with itself.
\end{proof}
\begin{corollary}\lab{winter47bjan}
Let $E$ be a vector bundle of rank $r$ over an $n$-dimensional
compact smooth manifold $M$. Then $E$ admits a finite super nice total
trivialization atlas that is GL compatible with itself.
\end{corollary}
\begin{proof}
Let $\{(U_\alpha,\varphi_\alpha,\rho_\alpha)\}_{1\leq \alpha\leq N}$ be the total trivialization atlas that was constructed above. For each $\alpha$, $\varphi_\alpha(U_\alpha)$ is a ball in the Euclidean space and so it is diffeomorphic to $\reals^n$; let $\xi_\alpha: \varphi_{\alpha}(U_\alpha)\rightarrow \reals^n$ be such a diffeomorphism. We let $\tilde{\varphi}_\alpha:=\xi_\alpha\circ \varphi_\alpha: U_\alpha\rightarrow \reals^n$. A composition of diffeomorphisms is a diffeomorphism, so for all $1\leq \alpha,\beta\leq N$,  $\tilde{\varphi}_\alpha\circ \tilde{\varphi}_\beta^{-1}:\tilde{\varphi}_\beta(U_\alpha\cap U_\beta)\rightarrow \tilde{\varphi}_\alpha(U_\alpha\cap U_\beta)$ is a diffeomorphism. So $\{(U_\alpha,\tilde{\varphi}_\alpha,\rho_\alpha)\}_{1\leq \alpha\leq N}$ is clearly a smooth super nice total trivialization atlas. Moreover, if $1\leq \alpha,\beta\leq N$ are such that $U_\alpha\cap
 U_\beta$ is nonempty, then $\tilde{\varphi}_\alpha (U_\alpha\cap U_\beta)$ is $\reals^n$ or a bounded open set with Lipschitz continuous boundary. The reason is that $\tilde{\varphi}_\alpha=\xi_\alpha\circ \varphi_\alpha$, and $\varphi_\alpha(U_\alpha\cap U_\beta)$ is $\reals^n$ or Lipschitz, $\xi_\alpha$ is a diffeomorphism and being equal to $\reals^n$ or Lipschitz is a property that is preserved under diffeomorphisms. Therefore $\{(U_\alpha,\tilde{\varphi}_\alpha,\rho_\alpha)\}_{1\leq \alpha\leq N}$ is a finite super nice total
trivialization atlas that is GL compatible with itself.
\end{proof}
A \textbf{section} of $E$ is a map $u:M\rightarrow E$ such that $\pi\circ u=Id_M$. The collection of all sections of $E$ is denoted by $\Gamma(M,E)$. A section $u\in \Gamma(M,E)$
 is said to be smooth if it is smooth as a map from the smooth manifold $M$ to the smooth manifold $E$. The collection of all smooth sections of $E\rightarrow M$ is denoted by
 $C^\infty(M,E)$. Note that if $\{(U_\alpha, \varphi_\alpha,\rho_\alpha)\}_{\alpha\in
 I}$ is a total trivialization atlas for the vector bundle $E\rightarrow M$ of rank $r$, then
 for $u\in \Gamma(M,E)$ we have
 \begin{equation*}
 u\in C^\infty(M,E)\Longleftrightarrow \forall\,\alpha\in I,\,\forall 1\leq l\leq
 r\quad \rho^l_\alpha\circ u\circ \varphi_\alpha^{-1}\in C^\infty
 (\varphi_\alpha (U_\alpha))
 \end{equation*}
 A local section of $E$ over an open set $U\subseteq M$ is a map $u: U\rightarrow E$ where $u$ has the property that $\pi\circ
 u=Id_U$ (that is, $u$ is a section of the vector bundle $E_U\rightarrow
 U$). We denote the collection of all local sections on $U$ by
 $\Gamma(U,E)$ or $\Gamma(U,E_U)$.
 \begin{remark}\lab{winter47}
As a consequence of $\rho|_{E_x}: E_x\rightarrow \reals^r$ being
an isomorphism, if $u$ is a section of $E|_U\rightarrow U$ and
$f: U\rightarrow \reals$ is a function, then $\rho
(fu)=f\rho(u)$. In particular $\rho(0)=0$.
\end{remark}
Given a total trivialization triple $(U,\varphi,\rho)$ we have the
following commutative diagram:
\begin{center}
\begin{tikzcd}
E|_U \arrow{r}{(\varphi\circ \pi, \rho^j)} \arrow{d}{\pi}
&\varphi(U)\times \reals \arrow{d}{\tilde{\pi}}\\
U \arrow{r}{\varphi} &\varphi(U)\subseteq \reals^n
\end{tikzcd}
\end{center}
If $s$ is a section of $E|_U\rightarrow U$, then by definition the
push forward of $s$ by $\rho^j$ (the $j^{th}$ component of $\rho$)
is a section of $\varphi(U)\times\reals \rightarrow \varphi(U)$
which is defined by
\begin{equation*}
\rho^j_{*}(s)=\rho^j\circ s \circ \varphi^{-1}\quad (\textrm{i.e. $z\in \varphi(U)\mapsto (z,\rho^j\circ s \circ \varphi^{-1}(z) )$})
\end{equation*}

Let $E\rightarrow M$ be a vector bundle of rank $r$ and
$U\subseteq M$ be an open set. A (smooth) \textbf{local frame} for $E$ over
$U$ is an ordered $r$-tuple $(s_1,\cdots,s_r)$ of (smooth) local
sections over $U$ such that for each $x\in U$, $(s_1(x),\cdots,
s_r(x))$ is a basis for $E_x$. Given any vector bundle chart
$(V,\rho)$, we can define the associated local frame on $V$ as
follows:
\begin{equation*}
\forall\,1\leq l\leq r\,\,\forall\,x\in V\qquad
s_l(x)=\rho|_{E_x}^{-1}(e_l)
\end{equation*}
where $(e_1,\cdots,e_r)$ is the standard basis of $\reals^r$. The
following theorem states the converse of this observation is also
true.
\begin{theorem}\lab{thmfalllocalframe1}(\cite{Lee3}, Page 258)
Let $E\rightarrow M$ be a vector bundle of rank $r$ and let $(s_1,\cdots,s_r)$ be a smooth local frame over an open
set $U\subseteq M$. Then $(U,\rho)$ is a vector bundle chart where
the map $\rho: E_U\rightarrow \reals^r$ is defined by
\begin{equation*}
\forall\,x\in U,\forall u\in E_x\qquad
\rho(u)=u^1e_1+\cdots+u^re_r
\end{equation*}
where $u=u^1s_1(x)+\cdots+u^rs_r(x)$.
\end{theorem}

\begin{theorem}\lab{thmfalllocalframe21}(\cite{Lee3}, Page 260)
Let $E\rightarrow M$ be a vector bundle of rank $r$ and let $(s_1,\cdots,s_r)$ be a smooth local frame over an open
set $U\subseteq M$. If $f\in \Gamma(M,E)$, then $f$ is smooth on $U$ if and only if its component functions with respect to $(s_1,\cdots,s_r)$ are smooth.
\end{theorem}

 A (smooth) \textbf{fiber metric} on a vector bundle $E$ is a (smooth)
function which assigns to each $x\in M$ an inner product
\begin{equation*}
\langle .,.\rangle_E: E_x\times E_x\rightarrow \reals
\end{equation*}
Note that the smoothness of the fiber metric means that for all $u,v\in C^\infty(M,E)$ the mapping
\begin{equation*}
M\rightarrow \reals,\qquad x\mapsto \langle u(x),v(x)\rangle_E
\end{equation*}
is smooth. One can show that every (smooth) vector bundle can be equipped
with a (smooth) fiber metric (\cite{Taubes2011}, Page 72).
\begin{remark}\lab{winter48}
If $(M,g)$ is a Riemannian manifold, then $g$ can be viewed as a
fiber metric on the tangent bundle. The metric $g$ induces fiber
metrics on all tensor bundles; it can be shown that (\cite{Lee2}) if
$(M,g)$ is a Riemannian manifold, then there exists a unique
inner product on each fiber of $T^k_l(M)$ with the property that
for all $x\in M$, if $\{e_i\}$ is an orthonormal basis of $T_xM$
with dual basis $\{\eta^i\}$, then the corresponding basis of
$T^k_l(T_xM)$ is orthonormal. We
denote this inner product by $\langle .,.\rangle_F$ and the
corresponding norm by $|.|_F$. If $A$ and $B$ are two tensor
fields, then with respect to any local frame
\begin{equation*}
\langle A,B\rangle_F=g^{i_1r_1}\cdots g^{i_kr_k}g_{j_1s_1}\cdots
g_{j_ls_l}A^{j_1\cdots j_l}_{i_1\cdots i_k}B^{s_1\cdots
s_l}_{r_1\cdots r_k}
\end{equation*}
\end{remark}
\begin{theorem}\lab{thmfalltrivializametric1}
Let $\pi: E\rightarrow M$ be a vector bundle with rank $r$
equipped with a fiber metric $\langle .,.\rangle_E$. Then given
any total trivialization triple $(U,\varphi,\rho)$, there exists a
smooth map $\tilde{\rho}:E_U\rightarrow \reals^r$ such that with
respect to the new total trivialization triple
$(U,\varphi,\tilde{\rho})$ the fiber metric trivializes on $U$,
that is
\begin{equation*}
\forall\,x\in U\,\,\forall\,u,v\in E_x\qquad\langle
u,v\rangle_E=u^1v^1+\cdots+u^rv^r
\end{equation*}
where for each $1\leq l\leq r$, $u^l$ and $v^l$ denote the
$l^{th}$ components of $u$ and $v$, respectively (with respect to
the local frame associated with the bundle chart
$(U,\tilde{\rho})$).
\end{theorem}
\begin{proof}
Let $(t_1,\cdots,t_r)$ be the local frame on $U$ associated with the vector bundle chart $(U,\rho)$. That is
\begin{equation*}
\forall\,x\in U,\,\forall 1\leq l\leq r\qquad t_l(x)=\rho|_{E_x}^{-1}(e_l)
\end{equation*}
Now we apply the Gram-Schmidt algorithm to the local frame $(t_1,\cdots,t_r)$ to construct an orthonormal frame $(s_1,\cdots,s_r)$ where
\begin{equation*}
\forall\,1\leq l\leq r\qquad s_l=\frac{t_l-\sum_{j=1}^{l-1}\langle t_l,s_j\rangle_E s_j}{|t_l-\sum_{j=1}^{l-1}\langle t_l,s_j\rangle_E s_j|}
\end{equation*}
$s_l:U\rightarrow E$ is smooth because
\begin{enumerate}
\item smooth local sections over $U$ form a module over the ring $C^\infty(U)$,
\item the function $x\mapsto \langle t_l(x), s_j(x) \rangle_E$ from $U$ to $\reals$ is smooth,
\item Since $\textrm{Span}\{s_1,\cdots,s_{l-1}\}=\textrm{Span}\{t_1,\cdots,t_{l-1}\}$, $t_l-\sum_{j=1}^{l-1}\langle t_l,s_j\rangle_E s_j$ is nonzero on $U$
 and $x\mapsto |t_l(x)-\sum_{j=1}^{l-1}\langle t_l(x),s_j(x)\rangle_E s_j(x)|$ as a function from $U$ to $\reals$ is nonzero on $U$ and it is a composition of smooth functions.
\end{enumerate}
Thus for each $l$, $s_l$ is a linear combination of elements of the $C^\infty(U)$-module of smooth local sections over $U$, and so it is a smooth local section over $U$.
 Now we let $(U,\tilde{\rho})$ be the associated vector bundle chart described in Theorem \ref{thmfalllocalframe1}. For all $x\in U$ and for all $u,v\in E_x$ we have
 \begin{equation*}
 \langle u,v\rangle_E=\langle u^l s_l, v^js_j\rangle_E=u^lv^j\langle s_l,s_j\rangle_E=u^lv^j\delta_{lj}=u^1v^1+\cdots+u^rv^r\,.
 \end{equation*}
\end{proof}
\begin{corollary}\lab{winter49}
As a consequence of Theorem \ref{thmfalltrivializametric1}, Theorem \ref{winter47jan}, and
Theorem \ref{thmfallnicetrivatlas2}  every vector bundle on a
compact manifold equipped with a fiber metric admits a nice
finite total trivialization atlas (and a super nice finite total trivialization atlas and a finite total trivialization atlas that is GL compatible with itself) such that the fiber metric is
trivialized with respect to each total trivialization triple in
the atlas.
\end{corollary}
\begin{lemma}(\cite{Lee3},Page 252)\lab{lemapp8}
Let $\pi: E\rightarrow M$ be a smooth vector bundle of rank $r$
over $M$. Suppose $\Phi: \pi^{-1}(U)\rightarrow U\times \reals^r$
and $\Psi: \pi^{-1}(V)\rightarrow V\times \reals^r$ are two
smooth local trivializations of $E$ with $U\cap V\neq \emptyset$.
There exists a smooth map $\tau: U\cap V\rightarrow \textrm{GL}(r,\reals)$
such that the composition
\begin{equation*}
\Phi\circ \Psi^{-1}: (U\cap V)\times \reals^r \rightarrow  (U\cap
V)\times \reals^r
\end{equation*}
has the form
\begin{equation*}
\Phi\circ \Psi^{-1}(p,v)=(p,\tau(p)v)
\end{equation*}
\end{lemma}

 \subsection{Standard Total Trivialization Triples}
Let $M^n$ be a smooth manifold and $\pi:E\rightarrow M$ be a
vector bundle of rank $r$. For certain vector bundles there are
standard methods to associate with any given smooth coordinate
chart $(U,\varphi=(x^i))$ a total trivialization triple
$(U,\varphi,\rho)$. We call such a total trivialization triple
 the \textbf{standard total trivialization} associated with
$(U,\varphi)$. Usually this is done by first associating with $(U,\varphi)$ a local frame for $E_U$ and then applying Theorem \ref{thmfalllocalframe1} to construct
 a total trivialization triple.
\begin{itemizeXALI}
\item $E=T^k_l(M)$: The collection of the following tensor fields on $U$ form
a local frame for $E_U$ associated with $(U,\varphi=(x^i))$
\begin{equation*}
\frac{\partial}{\partial x^{i_1}}\otimes \cdots \otimes
\frac{\partial}{\partial x^{i_l}}\otimes dx^{j_1}\otimes
\cdots\otimes dx^{j_k}
\end{equation*}
So given any atlas $\{(U_\alpha,\varphi_\alpha)\}$ of a manifold $M^n$, there is a corresponding total
trivialization atlas for the tensor bundle $T^k_l(M)$, namely
$\{(U_\alpha,\varphi_\alpha,\rho_\alpha)\}$
where for each $\alpha$, $\rho_\alpha$ has $n^{k+l}$ components
which we denote by $(\rho_\alpha)^{j_1\cdots j_l}_{i_1\cdots
i_k}$. For all $F\in \Gamma (M,T^{k}_{l}(M))$, we have
\begin{equation*}
(\rho_\alpha)^{j_1\cdots j_l}_{i_1\cdots
i_k}(F)=(F_\alpha)^{j_1\cdots j_l}_{i_1\cdots i_k}
\end{equation*}
Here $(F_\alpha)^{j_1\cdots j_l}_{i_1\cdots i_k}$ denotes the
components of $F$ with respect to the standard frame for $T^k_l
U_\alpha$ described above. When there is no
possibility of confusion, we may write $F^{j_1\cdots
j_l}_{i_1\cdots i_k}$ instead of $(F_\alpha)^{j_1\cdots
j_l}_{i_1\cdots i_k}$.
\item $E=\Lambda^k(M)$: This is the bundle whose fiber over each $x\in M$ consists of alternating covariant tensors of order $k$. The collection of the following forms on $U$ form
a local frame for $E_U$ associated with $(U,\varphi=(x^i))$
\begin{equation*}
dx^{j_1}\wedge \cdots\wedge dx^{j_k} \quad
(\textrm{$(j_1,\cdots,j_k)$ is increasing})
\end{equation*}
\item $E=\mathcal{D}(M)$ (the density bundle): The density bundle over $M$ is the vector bundle whose fiber over each $x\in M$ is $\mathcal{D}(T_x M)$. More precisely, if we let
\begin{equation*}
\mathcal{D}(M)=\coprod_{x\in M}\mathcal{D}(T_x M)
\end{equation*}
then $\mathcal{D}(M)$ is a smooth vector bundle of rank $1$ over $M$ (\cite{Lee3}, Page 429). Indeed, for every smooth chart $(U,\varphi=(x^i))$, $|dx^1\wedge\cdots\wedge dx^n|$ on $U$ is a local frame for $\mathcal{D}(M)|_U$. We denote the corresponding trivialization by $\rho_{\mathcal{D},\varphi}$, that is, given $\mu\in \mathcal{D}(T_y M)$, there exists a number $a$ such that
\begin{equation*}
\mu= a(|dx^1\wedge\cdots\wedge dx^n|_y)
\end{equation*}
and $\rho_{\mathcal{D},\varphi}$ sends $\mu$ to $a$.
Sometimes we write $\mathcal{D}$ instead of $\mathcal{D}(M)$ if $M$ is clear from the context. Also when there is no possibility of confusion we may write $\rho_\mathcal{D}$ instead of $\rho_{\mathcal{D},\varphi}$.
\end{itemizeXALI}
\begin{remark}[Integration of densities on
manifolds]\lab{winter50}
 Elements of $C_c(M,\mathcal{D})$ can be integrated over $M$. Indeed, for $\mu\in C_c(M,\mathcal{D})$ we may consider two cases
 \begin{itemizeXALI}
 \item \textbf{Case 1:} There exists a smooth chart $(U,\varphi)$ such that $\textrm{supp}\mu\subseteq U$.
 \begin{equation*}
 \int_M \mu:=\int_{\varphi(U)} \rho_{\mathcal{D},\varphi}\circ \mu\circ \varphi^{-1}\,dV
 \end{equation*}
\item \textbf{Case 2:} If $\mu$ is an arbitrary element of $C_c(M,\mathcal{D})$, then we consider a smooth atlas $\{(U_\alpha, \varphi_\alpha)\}_{\alpha\in I}$ and a partition of unity $\{\psi_\alpha\}_{\alpha\in I}$ subordinate to $\{U_\alpha\}$ and we let
    \begin{equation*}
    \int_M \mu:=\sum_{\alpha\in I}\int_M\psi_\alpha \mu
    \end{equation*}
 \end{itemizeXALI}
 It can be shown that the above definitions are independent of the choices (charts and partition of unity) involved (\cite{Lee3}, Pages 431 and 432).
\end{remark}
\subsection{Constructing New Bundles From Old Ones}
\subsubsection{Hom Bundle, Dual Bundle, Functional Dual Bundle}
\begin{itemizeXALI}
\item The construction $\textrm{Hom}(.,.)$ can be applied
fiberwise to a pair of vector bundles $E$ and $\tilde{E}$ over a
manifold $M$ to give a new vector bundle denoted by
$\textrm{Hom}(E,\tilde{E})$. The fiber of
$\textrm{Hom}(E,\tilde{E})$ at any given point $p\in M$ is the
vector space $\textrm{Hom}(E_p,\tilde{E}_p)$. Clearly if
$\textrm{rank}\,E=r$ and  $\textrm{rank}\,\tilde{E}=\tilde{r}$,
then $\textrm{rank}\,\textrm{Hom}(E,\tilde{E})=r\tilde{r}$.\\
If $\{(U_\alpha,\varphi_\alpha,\rho_\alpha)\}$ and
$\{(U_\alpha,\varphi_\alpha,\tilde{\rho}_\alpha)\}$ are total
trivialization atlases for the vector bundles $\pi: E\rightarrow
M$ and $\tilde{\pi}: \tilde{E}\rightarrow M$, respectively, then
$\{U_\alpha,\varphi_\alpha,\hat{\rho}_\alpha\}$ will be a total
trivialization atlas for
$\pi_{\textrm{Hom}}: \textrm{Hom}(E,\tilde{E})\rightarrow M$ where $\hat{\rho}_\alpha: \pi_{\textrm{Hom}}^{-1}(U_\alpha)\rightarrow
\textrm{Hom}(\reals^r,\reals^{\tilde{r}})\cong
\reals^{r\tilde{r}}$ is defined as follows: for $p\in U_\alpha$, $A_p\in \textrm{Hom}(E_p,\tilde{E}_p)$ is mapped to
 $[\tilde{\rho}_\alpha|_{\tilde{E}_{p}}] \circ A \circ
[\rho_\alpha|_{E_{p}}]^{-1}$.
\item Let $\pi: E\rightarrow M$ be a vector bundle. The \textbf{dual
bundle} $E^*$ is defined by $E^*=\textrm{Hom}(E,\tilde{E}=M\times
\reals)$.
\item Let $\pi: E\rightarrow M$ be a vector bundle and let $\mathcal{D}$ denote the density bundle of $M$. The \textbf{functional dual bundle} $E^\vee$ is defined by $E^\vee=\textrm{Hom}(E,\mathcal{D})$(see \cite{Reus1}). Let's describe explicitly
what the standard total trivialization triples of this bundle are.
 Let $(U,\varphi,\rho)$ be a total
  trivialization triple for $E$. We can associate with this triple the total trivialization triple $(U,\varphi,\rho^\vee)$ for $E^\vee$ where $\rho^\vee: E_U
^\vee\rightarrow \reals^{r}$ is defined as follows: for $p\in U$, $L_p\in
\textrm{Hom} (E_p,\mathcal{D}_p)$ is mapped to $\rho_{\mathcal{D},\varphi}\circ L_p\circ
(\rho|_{E_p})^{-1}\in (\reals^r)^{*}\simeq \reals^r$. Note that $(\reals^r)^{*}\simeq \reals^r$ under the following isomorphism
\begin{equation*}
(\reals^r)^{*}\rightarrow \reals^r,\qquad u\mapsto u(e_1)e_1+\cdots+u(e_r)e_r
\end{equation*}
That is, $u$ as an element of $\reals^r$ is the vector whose components are $(u(e_1),\cdots,u(e_r))$. In particular,  if $z=z_1e_1+\cdots+z_re_r$ is an arbitrary vector in $\reals^r$, then
\begin{equation*}
u(z)=u(z_1e_1+\cdots+z_re_r)=z_1u(e_1)+\cdots+z_ru(e_r)=z\cdot u
\end{equation*}
where on the LHS $u$ is viewed as an element of $(\reals^r)^*$ and on the RHS $u$ is viewed as an element of $\reals^r$.
\end{itemizeXALI}
\subsubsection{Tensor Product Of Bundles}
Let $\pi: E\rightarrow
M$ and $\tilde{\pi}: \tilde{E}\rightarrow M$ be two vector
bundles. Then $E\otimes \tilde{E}$ is a new vector bundle whose
fiber at $p\in M$ is $E_p\otimes \tilde{E}_p$. If $\{(U_\alpha,\varphi_\alpha,\rho_\alpha)\}$ and
$\{(U_\alpha,\varphi_\alpha,\tilde{\rho}_\alpha)\}$ are total
trivialization atlases for the vector bundles $\pi: E\rightarrow
M$ and $\tilde{\pi}: \tilde{E}\rightarrow M$, respectively, then
$\{(U_\alpha,\varphi_\alpha,\hat{\rho}_\alpha))\}$ will be a total
trivialization atlas for $\pi_{\textrm{tensor}}:
E\otimes \tilde{E}\rightarrow M$ where $\hat{\rho}_\alpha: \pi_{\textrm{tensor}}^{-1}(U_\alpha)\rightarrow
(\reals^r\otimes\reals^{\tilde{r}})\cong \reals^{r\tilde{r}}$ is defined as follows: for $p\in U_\alpha$, $a_p\otimes \tilde{a}_p\in E_p\otimes \tilde{E}_p$ is mapped to $\rho_\alpha|_{E_{p}} (a_p)\otimes
\tilde{\rho}_\alpha|_{\tilde{E}_{p}} (\tilde{a}_p)$.\\
 It can be shown that $\textrm{Hom}(E,\tilde{E})\cong E^*\otimes \tilde{E}$ (isomorphism of vector bundles over
$M$).
\begin{remark}[Fiber Metric on Tensor Product]\lab{winter51}
Consider the inner product spaces $(U,\langle .,.\rangle_U)$
and $(V,\langle .,.\rangle_V)$. We can turn the tensor product of
$U$ and $V$, $U\otimes V$ into an inner product space by defining
\begin{equation*}
\langle u_1\otimes v_1,u_2\otimes v_2\rangle_{U\otimes V}=\langle
u_1,u_2\rangle _U\langle v_1,v_2\rangle_V
\end{equation*}
and extending by linearity. As a consequence, if $E$ is a vector
bundle (on a Riemannian manifold $(M,g)$) equipped with a fiber
metric $\langle .,.\rangle_E$, then there is a natural fiber
metric on the bundle $(T^*M)^{\otimes k}$ and subsequently on the
bundle $(T^*M)^{\otimes k}\otimes E$. If $F=F^a_{i_1\cdots
i_k}dx^{i_1}\otimes\cdots\otimes dx^{i_k}\otimes s_a$ and $G=G^b_{j_1\cdots
 j_k}dx^{j_1}\otimes\cdots \otimes dx^{j_k}\otimes s_b$ are two local sections of this
bundle on a domain $U$ of a chart, then at any point in $U$ we have
\begin{align*}
\langle F, G\rangle_{(T^*M)^{\otimes k}\otimes E}&=F^a_{i_1\cdots i_k}G^b_{j_1\cdots
 j_k}\langle dx^{i_1},dx^{j_1}\rangle_{T^*M}\cdots \langle
 dx^{i_k},dx^{j_k}\rangle_{T^*M}\langle s_a,
 s_b\rangle_E\\
 &=g^{i_1j_1}\cdots g^{i_kj_k}h_{ab}F^a_{i_1\cdots i_k}G^b_{j_1\cdots
 j_k}
\end{align*}
where $h_{ab}:=\langle s_a,
 s_b\rangle_E$.
\end{remark}


\subsection{Connection on Vector Bundles, Covariant Derivative}
\subsubsection{Basic Definitions}
Let $\pi:E\rightarrow M$ be a vector bundle.
\begin{definition}\lab{winter52}
 A \textbf{connection} in $E$ is a map
\begin{equation*}
\grad: C^\infty(M,TM)\times C^{\infty}(M,E)\rightarrow
C^\infty(M,E),\quad (X,u)\mapsto \grad_X u
\end{equation*}
satisfying the following properties:
\begin{enumerate}
\item $\grad_X u$ is linear over $C^\infty(M)$ in $X$
\begin{equation*}
\forall\, f,g\in C^\infty(M)\qquad
\grad_{fX_1+gX_2}u=f\grad_{X_1}u+g\grad_{X_2}u
\end{equation*}
\item $\grad_X u$ is linear over $\reals$ in $u$:
\begin{equation*}
\forall\, a,b\in \reals\qquad \grad_X(au_1+bu_2)=a\grad_X
u_1+b\grad_X u_2
\end{equation*}
\item $\grad$ satisfies the following product rule
\begin{equation*}
\forall\, f\in C^{\infty}(M)\qquad \grad_X (fu)=f\grad_X u+(Xf)u
\end{equation*}
\end{enumerate}
A \textbf{metric connection} in a real vector bundle $E$ with a
fiber metric is a connection $\grad$ such that
\begin{equation*}
\forall\, X\in C^\infty (M,TM), \forall\, u,v\in C^\infty
(M,E)\qquad X\langle u,v\rangle_E=\langle \grad_X u,
v\rangle_E+\langle u,\grad_X v\rangle_E
\end{equation*}
\end{definition}
Here is a list of useful facts about connections:
\begin{itemizeXALI}
\item (\cite{Moore2009},Page183) Using a partition of unity, one can show that any real vector
bundle with a smooth fiber metric admits a metric connection
\item (\cite{Lee3}, Page 50) If $\grad$ is a
connection in a bundle $E$, $X\in C^{\infty}(M,TM)$, $u\in
C^\infty (M,E)$, and $p\in M$, then $\grad_X u|_p$ depends only
on the values of $u$ in a neighborhood of $p$ and the value of
$X$ at $p$. More precisely, if $u=\tilde{u}$ on a neighborhood of
$p$ and $X_p=\tilde{X}_p$, then $\grad_X
u|_p=\grad_{\tilde{X}}\tilde{u}|_p$.
\item (\cite{Lee3}, Page 53) If $\grad$ is a connection in $TM$, then there exists a
unique connection in each tensor bundle $T^k_l(M)$, also denoted
by $\grad$, such that the following conditions are satisfied:
\begin{enumerate}
\item On the tangent bundle, $\grad$ agrees with the given
connection.
\item On $T^0(M)$, $\grad$ is given by ordinary differentiation of functions, that is, for all real-valued smooth functions $f:M\rightarrow
\reals$: $\grad_X f=Xf$.
\item $\grad_X(F\otimes G)=(\grad_X F)\otimes G+F\otimes (\grad_X
G)$.
\item If ${\rm tr}$ denotes the trace on any pair of indices, then
$\grad_X({\rm tr}F)={\rm tr}(\grad_X F)$.
\end{enumerate}
This connection satisfies the following additional property: for
any $T\in C^\infty(M,T^k_l(M))$, vector fields $Y_i$, and differential
$1$-forms $\omega^j$,
\begin{align*}
(\grad_X
T)(\omega^1,\dots, &\omega^l,Y_1,\dots,Y_k)=X(T(\omega^1,\dots,\omega^l,Y_1,\dots,Y_k))\\
&-\sum_{j=1}^l
T(\omega^1,\dots,\grad_X\omega^j,\dots,\omega^l,Y_1,\dots,Y_k)\\
&-\sum_{i=1}^k
T(\omega^1,\dots,\omega^l,Y_1,\dots,\grad_XY_i,\dots,Y_k)\,.
\end{align*}
\end{itemizeXALI}
\begin{definition}\lab{winter53}
Let $\grad$ be a connection in $\pi: E\rightarrow M$. We define
the corresponding \textbf{covariant derivative} on $E$, also
denoted $\grad$, as follows
\begin{equation*}
\grad: C^\infty (M,E)\rightarrow C^\infty
(M,\textrm{Hom}(TM,E))\cong C^{\infty}(M,T^*M\otimes E),\qquad
u\mapsto \grad u
\end{equation*}
where for all $p\in M$, $\grad u(p): T_p M\rightarrow E_p$ is
defined by
\begin{equation*}
X_p\mapsto \grad_X u|_p
\end{equation*}
where $X$ on the RHS is any smooth vector field whose value at $p$ is $X_p$.
\end{definition}
\begin{remark}\lab{winter54}
Let $\grad$ be a connection in $TM$. As it was discussed $\grad$
induces a connection in any tensor bundle $E=T^k_l(M)$, also
denoted by $\grad$. Some authors (including Lee in \cite{Lee3},
Page 53) define the corresponding covariant derivative on
$E=T^k_l(M)$ as follows:
\begin{equation*}
\grad: C^\infty (M,T^k_l(M))\rightarrow C^\infty
(M,T^{k+1}_l(M)),\qquad F\mapsto \grad F
\end{equation*}
where
\begin{equation*}
\grad F(\omega^1,\cdots,\omega^l, Y_1,\cdots,Y_k,X)=(\grad_X
F)(\omega^1,\cdots,\omega^l,Y_1,\cdots,Y_k)
\end{equation*}
This definition agrees with the previous definition of covariant
derivative that we had for general vector bundles because
\begin{equation*}
T^*M\otimes T^k_l M\cong T^*M\otimes \underbrace{T^*M\otimes
\cdots\otimes T^*M}_{\textrm{$k$ factors}}\otimes
\underbrace{TM\otimes \cdots\otimes TM}_{\textrm{$l$
factors}}\cong T^{k+1}_l M
\end{equation*}
Therefore
\begin{equation*}
C^\infty(M,\textrm{Hom}(TM,T^k_l M))\cong  C^\infty(M,T^*M\otimes
T^k_l M )\cong C^\infty (M,T^{k+1}_l M)
\end{equation*}
More concretely, we have the following one-to-one correspondence
between\\ $C^\infty(M,\textrm{Hom}(TM,T^k_l M))$ and $C^\infty
(M,T^{k+1}_l M)$:
\begin{enumerate}
\item Given $u\in C^\infty (M,T^{k+1}_l M)$, the
corresponding element {\fontsize{8}{9}{$\tilde{u}\in
C^\infty(M,\textrm{Hom}(TM,T^k_l M))$}} is given by
\begin{equation*}
\forall\,p\in M\qquad \tilde{u}(p): T_pM\rightarrow
T^k_l(T_pM),\quad X\mapsto u(p)(\cdots,\cdots,X)
\end{equation*}
\item Given {\fontsize{10}{10}{$\tilde{u}\in C^\infty(M,\textrm{Hom}(TM,T^k_l M))$}}, the
corresponding element {\fontsize{9}{9}{$u\in C^\infty (M,T^{k+1}_l M)$}} is given by
\begin{equation*}
\forall\,p\in M\qquad
u(p)(\omega^1,\cdots,\omega^l,Y_1,\cdots,Y_k,X)=[\tilde{u}(p)(X)](\omega^1,\cdots,\omega^l,Y_1,\cdots,Y_k)
\end{equation*}
\end{enumerate}
\end{remark}
\subsubsection{Covariant Derivative on Tensor Product of Bundles}

 (\cite{Palais65}, Page 87) If $E$ an $\tilde{E}$ are vector bundles over $M$ with
covariant derivatives $\grad^E: C^\infty(M,E)\rightarrow C^\infty
(M, T^*M\otimes E) $ and $\grad^{\tilde{E}}:
C^\infty(M,\tilde{E})\rightarrow C^\infty (M, T^*M\otimes
\tilde{E})$, respectively, then there is a uniquely determined
covariant derivative
\begin{equation*}
\grad^{E\otimes \tilde{E}}: C^\infty(M,E\otimes
\tilde{E})\rightarrow C^{\infty}(M,T^*M\otimes E\otimes \tilde{E})
\end{equation*}
such that
\begin{equation*}
\grad^{E\otimes \tilde{E}}(u\otimes \tilde{u})=\grad^E u\otimes
\tilde{u}+\grad^{\tilde{E}}\tilde{u}\otimes u
\end{equation*}
The above sum makes sense because of the following isomorphisms:
\begin{equation*}
(T^*M\otimes E)\otimes \tilde{E}\cong  T^*M\otimes E\otimes
\tilde{E}\cong T^*M\otimes \tilde{E}\otimes E\cong (T^*M\otimes
\tilde{E})\otimes E
\end{equation*}
\begin{remark}\lab{winter55}
Recall that for tensor fields covariant derivative can be
considered as a map from $C^\infty(M,T^k_l M)\rightarrow
C^{\infty}(M,T^{k+1}_l M)$. Using this, we can give a second
description of covariant derivative on $E\otimes \tilde{E}$ when
$E=T^k_l M$. In this new description we have
\begin{equation*}
\grad^{T^k_l M\otimes \tilde{E}}: C^\infty (M,T^k_l M\otimes
\tilde{E})\rightarrow C^\infty(M,T^{k+1}_l M\otimes \tilde{E})
\end{equation*}
Indeed, for $F\in C^\infty (M,T^k_l M)$ and $u\in
C^\infty(M,\tilde{E})$
\begin{equation*}
\grad^{T^k_l M\otimes \tilde{E}}(F\otimes
u)=\underbrace{(\grad^{T^k_l M}F)}_{T^{k+1}_lM}\otimes
u+\underbrace{\underbrace{F}_{T^k_lM}\otimes
\underbrace{\grad^{\tilde{E}} u}_{T^*M\otimes
\tilde{E}}}_{T^{k+1}_l M\otimes \tilde{E}}
\end{equation*}
In particular, if $f\in C^\infty (M)$ and $u\in C^\infty(M,E)$ we
have $\grad^E (fu)\in C^{\infty}(M,T^*M\otimes E)$ and it is
equal to
\begin{equation*}
\grad^E (fu)=df\otimes u+f\grad^E u
\end{equation*}
\end{remark}

\subsubsection{Higher Order Covariant Derivatives}
Let $\pi: E\rightarrow M$ be a vector bundle. Let $\grad^E$ be a
connection in $E$ and $\grad$ be a connection in $TM$ which
induces a connection in $T^*M$. We have the following chain
{\scriptsize{\begin{align*}
C^\infty(M,E)\xrightarrow{\grad^E}C^\infty (M,T^*M\otimes
E)&\xrightarrow{\grad^{T^*M\otimes E}}C^\infty (M,(T^*M)^{\otimes
2 }\otimes E)\xrightarrow{\grad^{(T^*M)^{\otimes 2}\otimes
E}}\cdots \\
&\quad\cdots\xrightarrow{\grad^{(T^*M)^{\otimes (k-1)}\otimes
E}}C^\infty (M,(T^*M)^{\otimes k }\otimes
E)\xrightarrow{\grad^{(T^*M)^{\otimes k}\otimes E}}\cdots
\end{align*}}}
In what follows we denote all the maps in the above chain by
$\grad^E$. That is, for any $k\in \mathbb{N}_0$ we consider
$\grad^E$ as a map from $C^\infty (M,(T^*M)^{\otimes k }\otimes
E)$ to $C^\infty (M,(T^*M)^{\otimes (k+1) }\otimes E)$. So
\begin{equation*}
(\grad^E)^k: C^\infty(M,E)\rightarrow C^\infty (M,(T^*M)^{\otimes
k }\otimes E)
\end{equation*}
As an example, let's consider $(\grad^E)^k(fu)$ where $f\in
C^\infty (M)$ and $u\in C^\infty(M,E)$. We have
\begin{align*}
& \grad^E (fu)=df\otimes u+f\grad^E u\\
& (\grad^E)^2 (fu)=\grad^{T^*M\otimes E}\big[ df\otimes u+f\grad^E
u\big]\\
&\qquad\quad=[\grad^{T^*M}(df)\otimes u+df\otimes \grad^E
u]+[df\otimes
\grad^E u+f(\grad^E)^2u]\\
&\qquad\quad=\sum_{j=0}^2 {2\choose j} (\grad^{T^*M})^j f\otimes
(\grad^E)^{2-j}u
\end{align*}
In general, we can show by induction that
\begin{equation*}
(\grad^E)^k (fu)=\sum_{j=0}^k {k\choose j} (\grad^{T^*M})^j
f\otimes (\grad^E)^{k-j}u
\end{equation*}
where $(\grad^{T^*M})^0=Id$. Here $(\grad^{T^*M})^j
f$ should be interpreted as applying $\grad$ (in the sense described in Remark \ref{winter54}) $j$ times; so $(\grad^{T^*M})^j
f$ at each point is an element of $T^j_0M=(T^*M)^{\otimes j}$.
\subsubsection{Three Useful Rules, Two Important Observations}

Let $\pi:E\rightarrow M$ and $\tilde{\pi}:\tilde{E}\rightarrow M$
be two vector bundles over $M$ with ranks $r$ and $\tilde{r}$,
respectively. Let $\grad$ be a
 connection in $TM$ (which automatically induces a connection in all tensor bundles), $\grad^{E}$ be
  a connection in $E$ and $\grad^{\tilde{E}}$ be a connection in $\tilde{E}$. Let $(U,\varphi,\rho)$ be a total trivialization triple for $E$.
\begin{enumerate}
\item $\{\partial_i=\varphi_*^{-1}\frac{\partial}{\partial x^i}\}_{1\leq i\leq n}$ is a coordinate frame for $TM$ over $U$.
\item $\{s_a=\rho^{-1}(e_a)\}_{1\leq a\leq r}$ is a local frame for $E$ over
$U$.($\{e_a\}_{1\leq a \leq r}$ is the standard basis for
$\reals^r$ where $r=\textrm{rank}\,E$.)
\item Christoffel Symbols for $\grad$ on $(U,\varphi,\rho)$: $\grad_{\partial_i}\partial_j=\Gamma^k_{ij}\partial_k$
\item Christoffel Symbols for $\grad^E$ on $(U,\varphi,\rho)$: $\grad_{\partial_i}s_a=(\Gamma_E)^b_{ia}s_b$
\end{enumerate}
Also recall that for any 1-form $\omega$
\begin{equation*}
\grad_X\omega=(X^i\partial_i\omega_k-X^i\omega_j\Gamma^j_{ik})dx^k
\end{equation*}
Therefore
\begin{equation*}
\grad_{\partial_i}dx^j=-\Gamma^j_{ik}dx^k
\end{equation*}
\begin{itemizeXALI}
\item \textbf{Rule 1:} For all $u\in C^{\infty}(M,E)$
\begin{equation*}
\grad^E u=dx^i\otimes \grad_{\partial_i}^E u\qquad \textrm{on $U$}
\end{equation*}
The reason is as follows: Recall that for all $p\in M$, $\grad^E u(p)\in T^*M\otimes E$. Since $\{dx^i\otimes s_a\}$ is a local frame for $T^*M\otimes E$ on $U$ we have
\begin{equation*}
\grad^E u=R^a_i dx^i\otimes s_a= dx^i\otimes (R^a_i s_a)
\end{equation*}
According to what was discussed in the study of the isomorphism
$\textrm{Hom}(V,W)\cong V^*\otimes W$ in Section 3 we know that at
any point $p\in M$, $R^a_i$ is the element in column $i$ and row
$a$ of the matrix of $\grad^E u (p)$ as an element of
$\textrm{Hom}(T_pM,E_p)$. Therefore
\begin{equation*}
 \grad_{\partial_i}^E u=R^a_i s_a
\end{equation*}
Consequently we have $\grad^E u=dx^i\otimes (R^a_i
s_a)=dx^i\otimes \grad_{\partial_i}^E u$.
\item \textbf{Rule 2:} For all $v_1\in C^\infty(M,E)$ and $v_2\in C^\infty(M,\tilde{E})$
\begin{equation*}
\grad_{\partial_j}^{E\otimes \tilde{E}}(v_1\otimes
v_2)=(\grad_{\partial_j}^E v_1)\otimes v_2+v_1\otimes
(\grad_{\partial_j}^{\tilde{E}} v_2)
\end{equation*}
\item \textbf{Rule 3:} For all $u\in C^{\infty}(M,E)$ and $f\in C^\infty(M)$
\begin{equation*}
\grad^E(fu)=f\grad^E u+df\otimes u
\end{equation*}
\end{itemizeXALI}
The following two examples are taken from \cite{Gavrilov2008}.
\begin{itemizeXALI}
\item \textbf{Example 1:} Let $u\in C^\infty (M,E)$. On $U$ we may
write $u=u^a s_a$. We have
\begin{align*}
\grad^E u &=\grad^E (u^as_a)\stackrel{\text{Rule 3}}{=}
u^a\grad^Es_a+du^a\otimes s_a=u^a\grad^E s_a+(\partial_i
u^adx^i)\otimes s_a\\
&\stackrel{\text{Rule 1}}{=}u^adx^i\otimes
\grad^E_{\partial_i}s_a+(\partial_i u^adx^i)\otimes s_a\\
&=u^adx^i\otimes \big((\Gamma_E)^b_{ia}s_b\big)+(\partial_i
u^adx^i)\otimes s_a=dx^i\otimes
\big(u^a(\Gamma_E)^b_{ia}s_b\big)+dx^i\otimes (\partial_i u^a
s_a)\\
&= dx^i\otimes \big(u^b(\Gamma_E)^a_{ib}s_a\big)+dx^i\otimes
(\partial_i u^a s_a)\\
&=[\partial_i u^a+(\Gamma_E)^a_{ib}u^b]dx^i\otimes s_a
\end{align*}
That is, $\grad^E u= (\grad^E u)^a_i dx^i\otimes s_a$ where
\begin{equation*}
(\grad^E u)^a_i=\partial_i u^a+(\Gamma_E)^a_{ib}u^b
\end{equation*}
\item \textbf{Example 2:} Let $u\in C^\infty (M,E)$. On $U$ we may
write $u=u^a s_a$. We have
{\fontsize{8}{8}{\begin{align*}
(\grad^E)^2 u &=\grad^{T^*M\otimes E}\big( [\partial_i
u^a+(\Gamma_E)^a_{ib}u^b]dx^i\otimes s_a\big)\\
&\stackrel{\text{Rule 3}}{=}[\partial_i
u^a+(\Gamma_E)^a_{ib}u^b]\grad^{T^*M\otimes E }(dx^i\otimes
s_a)+d [\partial_i u^a+(\Gamma_E)^a_{ib}u^b]\otimes (dx^i\otimes
s_a)\\
&\stackrel{\text{Rule 1}}{=}[\partial_i
u^a+(\Gamma_E)^a_{ib}u^b]dx^j\otimes \grad^{T^*M\otimes E
}_{\partial_j}(dx^i\otimes s_a)+d [\partial_i
u^a+(\Gamma_E)^a_{ib}u^b]\otimes (dx^i\otimes s_a)\\
& \stackrel{\text{Def. of $d$}}{=}[\partial_i
u^a+(\Gamma_E)^a_{ib}u^b]dx^j\otimes \grad^{T^*M\otimes E
}_{\partial_j}(dx^i\otimes s_a)+\partial_{j}[\partial_i
u^a+(\Gamma_E)^a_{ib}u^b]dx^j\otimes dx^i\otimes s_a\\
& \stackrel{\text{Rule 2}}{=}[\partial_i
u^a+(\Gamma_E)^a_{ib}u^b]dx^j\otimes
\big[\grad^{T^*M}_{\partial_j}dx^i\otimes s_a+dx^i\otimes
\grad^E_{\partial_j} s_a\big]+\partial_{j}[\partial_i
u^a+(\Gamma_E)^a_{ib}u^b]dx^j\otimes dx^i\otimes s_a\\
&=[\partial_i u^a+(\Gamma_E)^a_{ib}u^b]dx^j\otimes
\big[-\Gamma_{jk}^idx^k\otimes s_a+dx^i\otimes
 (\Gamma_E)_{ja}^cs_c\big]+\partial_{j}[\partial_i
u^a+(\Gamma_E)^a_{ib}u^b]dx^j\otimes dx^i\otimes s_a\\
& \hspace{-1.25cm}\stackrel{\text{$i\leftrightarrow k$ in the first
summand}}{=}[\partial_k u^a+(\Gamma_E)^a_{kb}u^b]dx^j\otimes
\big[-\Gamma_{ji}^kdx^i\otimes s_a+dx^k\otimes
 (\Gamma_E)_{ja}^cs_c\big]+\partial_{j}[\partial_i
u^a+(\Gamma_E)^a_{ib}u^b]dx^j\otimes dx^i\otimes s_a\\
&\hspace{-1.25cm}= \{\partial_{j}[\partial_i
u^a+(\Gamma_E)^a_{ib}u^b]-\Gamma_{ji}^k[\partial_k
u^a+(\Gamma_E)^a_{kb}u^b]\}dx^j\otimes dx^i\otimes
s_a+[\partial_k u^a+(\Gamma_E)^a_{kb}u^b](\Gamma_E)_{ja}^c
dx^j\otimes dx^k\otimes s_c\\
& \hspace{-1.25cm}\stackrel{\text{$i\leftrightarrow k$ in the last
summand}}{=}\{\partial_{j}[\partial_i
u^a+(\Gamma_E)^a_{ib}u^b]-\Gamma_{ji}^k[\partial_k
u^a+(\Gamma_E)^a_{kb}u^b]\}dx^j\otimes dx^i\otimes s_a\\
&\hspace{8cm}+[\partial_i
u^a+(\Gamma_E)^a_{ib}u^b](\Gamma_E)_{ja}^c
dx^j\otimes dx^i\otimes s_c\\
& \hspace{-1.25cm}\stackrel{\text{$c\leftrightarrow a$ in the last
summand}}{=}\{\partial_{j}[\partial_i
u^a+(\Gamma_E)^a_{ib}u^b]-\Gamma_{ji}^k[\partial_k
u^a+(\Gamma_E)^a_{kb}u^b]\}dx^j\otimes dx^i\otimes s_a\\
&\hspace{8cm}+[\partial_i
u^c+(\Gamma_E)^c_{ib}u^b](\Gamma_E)_{jc}^a dx^j\otimes dx^i\otimes
s_a
\end{align*}}}
\end{itemizeXALI}
Considering the above examples we make the following two useful
observations that can be proved by induction.
\begin{itemizeXALI}
\item \textbf{Observation 1:} In general $(\grad^E)^k u =
\big((\grad^E)^k u \big)^a_{i_1\cdots
 i_k}dx^{i_1}\otimes \cdots\otimes dx^{i_k}\otimes s_a$ $(1\leq a\leq r,\, 1\leq i_1,\cdots,i_k\leq
 n)$ where $((\grad^E)^k u \big)^a_{i_1\cdots
 i_k}\circ \varphi^{-1}$ is a linear combination of $u^1\circ \varphi^{-1},\cdots, u^r\circ
 \varphi^{-1}$ and their partial derivatives up to order $k$ and the coefficients
are polynomials in terms of Christoffel symbols (of the linear
connection on $M$ and connection in $E$) and their derivatives (on
a compact manifold these coefficients are uniformly bounded provided that the metric and the fiber metric are smooth).
That is,
\begin{equation*}
((\grad^E)^k u \big)^a_{i_1\cdots
 i_k}\circ \varphi^{-1}=\sum_{|\eta|\leq k}\sum_{l=1}^r C_{\eta l}\partial^\eta (u^l\circ \varphi^{-1})
\end{equation*}
where for each $\eta$ and $l$, $C_{\eta l}$ is a polynomial in
terms of Christoffel symbols (of the linear connection on $M$ and
connection in $E$) and their derivatives.
\item \textbf{Observation 2:}
The highest order term in $((\grad^E)^k u \big)^a_{i_1\cdots
 i_k}\circ
\varphi^{-1}$ is $\frac{\partial}{x^{i_1}}\cdots
\frac{\partial}{x^{i_k}}(u^a\circ \varphi^{-1})$; that is
{\fontsize{8}{9}{\begin{equation*}
((\grad^E)^k u \big)^a_{i_1\cdots
 i_k}\circ \varphi^{-1}=\frac{\partial}{\partial x^{i_1}}\cdots \frac{\partial}{\partial x^{i_k}}(u^a\circ
\varphi^{-1})+\textrm{terms that contain derivatives of order at
most $k-1$ of $u^l\circ \varphi^{-1}$ $(1\leq l\leq r)$}
\end{equation*}}}
So
\begin{equation*}
((\grad^E)^k u \big)^a_{i_1\cdots
 i_k}\circ \varphi^{-1}=\frac{\partial^k}{\partial x^{i_1}\cdots \partial x^{i_k}}(u^a\circ
\varphi^{-1})+\sum_{|\eta|<k}\sum_{l=1}^rC_{\eta
l}\partial^\eta(u^l\circ \varphi^{-1})
\end{equation*}
\end{itemizeXALI}

\section{Some Results From the Theory of Generalized Functions}
In this section we collect some results from the theory of
distributions that will be needed for our definition of function
spaces on manifolds. Our main reference for this part is the
exquisite exposition by Marcel De Reus (\cite{Reus1}).

\subsection{Distributions on Domains in Euclidean Space}
Let $\Omega$ be a nonempty open set in $\reals^n$.
\begin{enumerateXALI}
\item Recall that
\begin{itemizeX}
\item $\mathcal{K}(\Omega)$ is the collection of all compact
subsets of $\Omega$.
\item $C^\infty (\Omega)=\textrm{the collection of all infinitely differentiable (real-valued) functions on $\Omega$}$.
\item For all $K\in \mathcal{K}(\Omega)$, $C_K^\infty(\Omega)=\{\varphi\in C^\infty(\Omega): \textrm{supp}\,\varphi\subseteq
K\}$.
\item $C^{\infty}_c(\Omega)=\bigcup_{K\in \mathcal{K}(\Omega)}C^\infty_K(\Omega)=\{\varphi\in C^\infty(\Omega): \textrm{supp}\,\varphi\, \textrm{is compact in
$\Omega$}\}$.
\end{itemizeX}
\item For all $\varphi\in C^{\infty}(\Omega)$, $j\in \mathbb{N}$ and $K\in \mathcal{K}(\Omega)$ we define
\begin{equation*}
\|\varphi\|_{j,K}:=\sup \{|\partial^\alpha \varphi(x)|:
|\alpha|\leq j, x\in K\}
\end{equation*}
\item For all $j\in \mathbb{N}$ and $K\in \mathcal{K}(\Omega)$, $\|.\|_{j,K}$
is a seminorm on $C^{\infty}(\Omega)$. We define
$\mathcal{E}(\Omega)$ to be $C^\infty(\Omega)$ equipped with the
natural topology induced by the separating family of seminorms
$\{\|.\|_{j,K}\}_{j\in \mathbb{N},K\in \mathcal{K}(\Omega)}$. It
can be shown that $\mathcal{E}(\Omega)$ is a Frechet space.
\item For all $K\in \mathcal{K}(\Omega)$ we define $\mathcal{E}_K(\Omega)$ to be
$C_K^\infty(\Omega)$ equipped with the subspace topology. Since
$C^\infty_K(\Omega)$ is a closed subset of the Frechet space
$\mathcal{E}(\Omega)$, $\mathcal{E}_K(\Omega)$ is also a Frechet
space.
\item We define $D(\Omega)=\bigcup_{K\in
\mathcal{K}(\Omega)}\mathcal{E}_K(\Omega)$ equipped with the
inductive limit topology with respect to the family of vector subspaces $\{\mathcal{E}_K(\Omega)\}_{K\in
\mathcal{K}(\Omega)}$.  It can be shown that if $\{K_j\}_{j\in \mathbb{N}_0}$ is an exhaustion by compacts sets of $\Omega$, then the inductive limit topology on $D(\Omega)$ with respect to the family $\{\mathcal{E}_{K_j}\}_{j\in \mathbb{N}_0}$ is exactly the same as the inductive limit topology with respect to $\{\mathcal{E}_K(\Omega)\}_{K\in
\mathcal{K}(\Omega)}$.
\end{enumerateXALI}
\begin{remark}\lab{remfallcontintotest1}
Let us mention a trivial but extremely useful consequence of the
above description of the inductive limit topology on $D(\Omega)$. Suppose $Y$ is a topological space and the mapping $T:
Y\rightarrow D(\Omega)$ is such that $T(Y)\subseteq
\mathcal{E}_K(\Omega)$ for some $K\in \mathcal{K}(\Omega)$. Since
$\mathcal{E}_K(\Omega)\hookrightarrow D(\Omega)$, if $T:
Y\rightarrow \mathcal{E}_K(\Omega)$ is continuous, then $T:
Y\rightarrow D(\Omega)$ will be continuous.
\end{remark}
\begin{theorem}[Convergence and Continuity for
$\mathcal{E}(\Omega)$]\lab{winter56} Let $\Omega$ be a nonempty
open set in $\reals^n$. Let $Y$ be a topological vector space
whose topology is induced by a separating family of seminorms $\mathcal{Q}$.
\begin{enumerateXALI}
\item A sequence $\{\varphi_m\}$ converges to $\varphi$ in
$\mathcal{E}(\Omega)$ if and only if
$\|\varphi_m-\varphi\|_{j,K}\rightarrow 0$ for all $j\in
\mathbb{N}$ and $K\in \mathcal{K}(\Omega)$.
\item  Suppose $T:\mathcal{E}(\Omega)\rightarrow Y$ is a linear map.
Then the followings are equivalent
\begin{itemize}
\item $T$ is continuous.
\item For every $q\in\mathcal{Q}$, there exist $j\in \mathbb{N}$ and $K\in\mathcal{K}(\Omega)$, and $C>0$ such that
\begin{equation*}
\forall\,\varphi\in \mathcal{E}(\Omega)\qquad q(T(\varphi))\leq C\|\varphi\|_{j,K}
\end{equation*}
\item If $\varphi_m\rightarrow 0$ in $\mathcal{E}(\Omega)$, then $T(\varphi_m)\rightarrow 0$ in $Y$.
\end{itemize}
\item In particular, a linear map $T:\mathcal{E}(\Omega)\rightarrow
\reals$ is continuous if and only if there exist $j\in \mathbb{N}$
and $K\in\mathcal{K}(\Omega)$, and $C>0$ such that
\begin{equation*}
\forall\,\varphi\in \mathcal{E}(\Omega)\qquad |T(\varphi)|\leq
C\|\varphi\|_{j,K}
\end{equation*}
\item A linear map $T: Y\rightarrow \mathcal{E}(\Omega)$ is continuous if and only if
{\fontsize{10}{10}{\begin{equation*}
\forall\, j\in \mathbb{N},\,\forall\, K\in
\mathcal{K}(\Omega)\qquad \exists\,C>0,\,k\in \mathbb{N}\,,
q_1,\cdots,q_k \in \mathcal{Q}\quad \textrm{such that $\forall\,y$} \quad
\|T(y)\|_{j,K}\leq C \max_{1\leq i\leq k} q_i(y)
\end{equation*}}}
\end{enumerateXALI}
\end{theorem}
\begin{theorem}[Convergence and Continuity for
$\mathcal{E}_K(\Omega)$]\lab{winter57} Let $\Omega$ be a nonempty
open set in $\reals^n$ and $K\in \mathcal{K}(\Omega)$. Let $Y$ be
a topological vector space whose topology is induced by a separating family
of seminorms $\mathcal{Q}$.
\begin{enumerateXALI}
\item A sequence $\{\varphi_m\}$ converges to $\varphi$ in
$\mathcal{E}_K(\Omega)$ if and only if
$\|\varphi_m-\varphi\|_{j,K}\rightarrow 0$ for all $j\in
\mathbb{N}$.
\item Suppose $T:\mathcal{E}_K(\Omega)\rightarrow Y$ is a linear map.
Then the followings are equivalent
\begin{itemize}
\item $T$ is continuous.
\item For every $q\in\mathcal{Q}$, there exists $j\in \mathbb{N}$ and $C>0$ such that
\begin{equation*}
\forall\,\varphi\in \mathcal{E}_K(\Omega)\qquad q(T(\varphi))\leq
C\|\varphi\|_{j,K}
\end{equation*}
\item If $\varphi_m\rightarrow 0$ in $\mathcal{E}_K(\Omega)$, then $T(\varphi_m)\rightarrow 0$ in $Y$.
\end{itemize}
\end{enumerateXALI}
\end{theorem}
\begin{theorem}[Convergence and Continuity for $D(\Omega)$]\lab{thmfallconvcont13}
Let $\Omega$ be a nonempty open set in $\reals^n$. Let $Y$ be a topological vector space whose topology is induced by a separating family of seminorms $\mathcal{Q}$.
\begin{enumerateXALI}
\item A sequence $\{\varphi_m\}$ converges to $\varphi$ in
$D(\Omega)$ if and only if there is a $K\in\mathcal{K}(\Omega)$
such that $\textrm{supp} \varphi_m\subseteq K$ and
$\varphi_m\rightarrow \varphi$ in $\mathcal{E}_K(\Omega)$.
\item Suppose $T:D(\Omega)\rightarrow Y$ is a linear map.
Then the followings are equivalent
\begin{itemize}
\item $T$ is continuous.
\item For all $K\in \mathcal{K}(\Omega)$, $T: \mathcal{E}_K(\Omega)\rightarrow
Y$ is continuous.
\item For every $q\in\mathcal{Q}$ and $K\in \mathcal{K}(\Omega)$, there exists $j\in \mathbb{N}$ and $C>0$ such that
\begin{equation*}
\forall\,\varphi\in \mathcal{E}_K(\Omega)\qquad q(T(\varphi))\leq
C\|\varphi\|_{j,K}
\end{equation*}
\item If $\varphi_m\rightarrow 0$ in $D(\Omega)$, then $T(\varphi_m)\rightarrow 0$ in $Y$.
\end{itemize}
\item In particular, a linear map $T: D(\Omega)\rightarrow \reals$
 is continuous if and only if for every $K\in \mathcal{K}(\Omega)$, there exists $j\in \mathbb{N}$ and $C>0$ such that
\begin{equation*}
\forall\,\varphi\in \mathcal{E}_K(\Omega)\qquad |T(\varphi)|\leq
C\|\varphi\|_{j,K}
\end{equation*}
\end{enumerateXALI}
\end{theorem}
\begin{remark}\lab{remfallcontintotest2}
Let $\Omega$ be a nonempty open set in $\reals^n$. Here are two immediate consequences of the previous theorems and remark:
\begin{enumerateXALI}
\item The identity map
\begin{equation*}
i_{D,\mathcal{E}}:D(\Omega)\rightarrow \mathcal{E}(\Omega)
\end{equation*}
is continuous (that is, $D(\Omega)\hookrightarrow \mathcal{E}(\Omega)$ ).
\item If $T:\mathcal{E}(\Omega)\rightarrow \mathcal{E}(\Omega)$ is a continuous linear map such that $\textrm{supp}(T \varphi)\subseteq \textrm{supp}\varphi$ for all $\varphi\in \mathcal{E}(\Omega)$ (i.e. $T$ is a \textbf{local} continuous linear map), then $T$ restricts to a continuous linear map from $D(\Omega)$ to $D(\Omega)$. Indeed, the assumption $\textrm{supp}(T \varphi)\subseteq \textrm{supp}\varphi$ implies that $T(D(\Omega))\subseteq D(\Omega)$. Moreover $T:D(\Omega)\rightarrow D(\Omega)$ is continuous if and only if for $K\in \mathcal{K}(\Omega)$ $T: \mathcal{E}_K(\Omega)\rightarrow D(\Omega)$ is continuous. Since $T(\mathcal{E}_K(\Omega))\subseteq \mathcal{E}_K(\Omega)$, this map is continuous if and only if $T: \mathcal{E}_K(\Omega)\rightarrow \mathcal{E}_K(\Omega)$ is continuous (see Remark \ref{remfallcontintotest1}). However, since the topology of $\mathcal{E}_K(\Omega)$ is the induced topology from $\mathcal{E}(\Omega)$, the continuity of the preceding map follows from the continuity of $T:\mathcal{E}(\Omega)\rightarrow \mathcal{E}(\Omega)$.
\end{enumerateXALI}
\end{remark}
\begin{theorem}\lab{thmfallprodinductcont11}
Let $\Omega$ be a nonempty open set in $\reals^n$. Let $Y$ be a topological vector space whose topology is induced by a separating family of seminorms $\mathcal{Q}$. Suppose $T: [D(\Omega)]^{\times r}\rightarrow Y$ is a linear map. The following are equivalent: (product spaces are equipped with the product topology)
\begin{enumerate}
\item $T: [D(\Omega)]^{\times r}\rightarrow Y$ is continuous.
\item For all $K\in \mathcal{K}(\Omega)$, $T: [\mathcal{E}_K(\Omega)]^{\times r}\rightarrow Y$ is continuous.
\item For all $q\in \mathcal{Q}$ and $K\in \mathcal{K}(\Omega)$, there exists $j_1,\cdots,j_l\in\mathbb{N}$ such that
\begin{equation*}
\forall\, (\varphi_1,\cdots,\varphi_r)\in [\mathcal{E}_K(\Omega)]^{\times r}\qquad |q\circ T(\varphi_1,\cdots,\varphi_r)|\leq C(\|\varphi_1\|_{j_1,K}+\cdots+\|\varphi_r\|_{j_r,K})
\end{equation*}
\end{enumerate}
\end{theorem}
\begin{theorem}\lab{winter58}
Let $\Omega$ be a nonempty open set in $\reals^n$.
\begin{enumerate}
\item A set $B\subseteq D(\Omega)$ is bounded if and only if there
exists $K\in \mathcal{K}(\Omega)$ such that $B$ is a bounded
susbset of $\mathcal{E}_K(\Omega)$ which is in turn equivalent to
the following statement
\begin{equation*}
\forall\,j\in \mathbb{N}\, \exists r_j\geq 0 \quad \textrm{such
that}\quad \forall\, \varphi \in B \quad \|\varphi\|_{j,K}\leq r_j
\end{equation*}
\item If $\{\varphi_m\}$ is a Cauchy sequence in $D(\Omega)$, then
it converges to a function $\varphi\in D(\Omega)$. We say
$D(\Omega)$ is sequentially complete.
\end{enumerate}
\end{theorem}
\begin{remark}\lab{winter59}
Topological spaces whose topology is determined by knowing the
convergent sequences and their limits exhibit nice properties and
are of particular interest. Let us recall a number of useful
definitions related to this topic:
\begin{itemizeXALI}
\item Let $X$ be a topological space and let $E\subseteq X$. The
\textbf{sequential closure} of $E$, denoted $\textrm{scl}(E)$ is
defined as follows:
\begin{equation*}
\textrm{scl}(E)=\{x\in X: \textrm{there is a sequence $\{x_n\}$
in $E$ such that $x_n\rightarrow x$}\}
\end{equation*}
Clearly $\textrm{scl}(E)$ is contained in the closure if $E$.
\item A topological space $X$ is called a \textbf{Frechet-Urysohn} space if
for every $E\subseteq X$ the sequential closure of $E$ is equal
to the closure of $E$.
\item A subset $E$ of a topological space $X$ is said to be
\textbf{sequentially closed} if $E=\textrm{scl}(E)$.
\item A topological space $X$ is said to be \textbf{sequential} if
for every $E\subseteq X$, $E$ is closed if and only if $E$ is
sequentially closed. If $X$ is a sequential topological space and
$Y$ is any topological space, then a map $f:X\rightarrow Y$ is
continuous if and only if
\begin{equation*}
\lim_{n\rightarrow \infty}f(x_n)=f(\lim_{n\rightarrow \infty}x_n)
\end{equation*}
for each convergent sequence $\{x_n\}$ in $X$.
\end{itemizeXALI}
The following implications hold for a topological space $X$:
\begin{equation*}
\textrm{$X$ is metrizable} \rightarrow \textrm{$X$ is first-countable} \rightarrow \textrm{$X$ is Frechet-Urysohn}
\rightarrow \textrm{$X$ is sequential}
\end{equation*}
As it was stated, $\mathcal{E}$ and $\mathcal{E}_K$ (For all
$K\in \mathcal{K}(\Omega)$) are Frechet and subsequently they are
metrizable. However, it can be shown that $D(\Omega)$ is not
first-countable and subsequently it is not metrizable. In fact,
although according to Theorem \ref{thmfallconvcont13}, the elements of the dual
of $D(\Omega)$ can be determined by knowing the convergent
sequences in $D(\Omega)$, it can be proved that $D(\Omega)$ is not
sequential.
\end{remark}
\begin{definition}\lab{winter60}
Let $\Omega$ be a nonempty open set in $\reals^n$. The topological
dual of $D(\Omega)$, denoted $D'(\Omega)$ $(D'(\Omega)
=[D(\Omega)]^*)$, is called the \textbf{space of distributions} on
$\Omega$. Each element of $D'(\Omega)$ is called a \textbf{distribution} on
$\Omega$.
\end{definition}
\begin{remark}\lab{winter61}
Every function $f\in L^1_{loc}(\Omega)$ defines a distribution
$u_f\in D'(\Omega)$ as follows
\begin{equation}\lab{eqnwinter61}
\forall\, \varphi\in D(\Omega)\qquad u_f(\varphi):=\int_\Omega
f\varphi dx
\end{equation}
In particular, every function $\varphi\in \mathcal{E}(\Omega)$
defines a distribution $u_\varphi$. It can be shown that the map
$j:\mathcal{E}(\Omega)\rightarrow D'(\Omega)$ which sends
$\varphi$ to $u_\varphi$ is an injective linear continuous map
(\cite{Reus1}, Page 11). Therefore we can identify
$\mathcal{E}(\Omega)$ with a subspace of $D'(\Omega)$.
\end{remark}
\begin{remark}\lab{winter61b}
Let $\Omega\subseteq \reals^n$ be a nonempty open set. Recall that $f:\Omega\rightarrow \reals$ is locally integrable ($f\in L^1_{loc}(\Omega)$) if it satisfies any
 of the following equivalent conditions.
 \begin{enumerateX}
 \item $f\in L^1(K)$ for all $K\in \mathcal{K}(\Omega)$.
 \item For all $\varphi\in C_c^{\infty}(\Omega)$, $f\varphi\in L^1(\Omega)$.
 \item For every nonempty open set $V\subseteq \Omega$ such that $\bar{V}$ is compact and contained in $\Omega$, $f\in L^1(V)$.
 \end{enumerateX}
(It can be shown that every locally integrable function is measurable (\cite{debnath2005}, Page 70).)\\
As a consequence, if we define $\textrm{Func}_{reg}(\Omega)$ to be the set
\begin{equation*}
\{f:\Omega\rightarrow \reals: \textrm{$u_f:D(\Omega)\rightarrow \reals$ defined by Equation \ref{eqnwinter61} is well-defined and continuous}\}
\end{equation*}
then $\textrm{Func}_{reg}(\Omega)=L^1_{loc}(\Omega)$.
\end{remark}
\begin{definition}[Calculus Rules for Distributions]\lab{winter62}
Let $\Omega$ be a nonempty open set in $\reals^n$. Let $u\in
D'(\Omega)$.
\begin{itemizeXALI}
\item For all $\varphi \in C^\infty (\Omega)$, $\varphi u$ is
defined by
\begin{equation*}
\forall\,\psi\in C_c^\infty(\Omega)\qquad [\varphi
u](\psi):=u(\varphi \psi)
\end{equation*}
It can be shown that $\varphi u\in D'(\Omega)$.
\item For all multiindices $\alpha$, $\partial^\alpha u$ is
defined by
\begin{equation*}
\forall\,\psi\in C_c^\infty(\Omega)\qquad [\partial^\alpha
u](\psi)=(-1)^{|\alpha|}u(\partial^\alpha \psi)
\end{equation*}
It can be shown that $\partial^\alpha u\in D'(\Omega)$.
\end{itemizeXALI}
\end{definition}
Also it is possible to make sense of "change of coordinates" for
distributions.  Let $\Omega$ and $\Omega'$ be two open sets in
$\reals^n$. Suppose $T:\Omega\rightarrow \Omega'$ is a $C^\infty$
 diffeomorphism. $T$ can be used to move any function on
 $\Omega$ to a function on $\Omega'$ and vice versa.
\begin{align*}
&T^*: \textrm{Func}(\Omega',\reals)\rightarrow \textrm{Func}(\Omega,\reals),\qquad T^*(f)=f\circ T\\
& T_{*}: \textrm{Func}(\Omega,\reals)\rightarrow
\textrm{Func}(\Omega',\reals),\qquad T_{*}(f)=f\circ T^{-1}
\end{align*}
$T^*f$ is called the \textbf{pullback} of the function $f$ under the
mapping $T$ and $T_*f$ is called the \textbf{pushforward} of the function
$f$ under the mapping $T$. Clearly $T^*$ and $T_*$ are inverses of
each other and $T_*=(T^{-1})^*$. One can show that $T_*$ sends
functions in $L^1_{loc}(\Omega)$ to $L^1_{loc}(\Omega')$ and
furthermore $T_*$ restricts to linear topological isomorphisms
$T_*: \mathcal{E}(\Omega)\rightarrow \mathcal{E}(\Omega')$ and
$T_*: D(\Omega)\rightarrow D(\Omega')$. Note that for all $f\in
L^1_{loc}(\Omega)$ and $\varphi\in C_c^{\infty}(\Omega')$
\begin{align*}
<u_{T_*f},\varphi>_{D'(\Omega')\times
D(\Omega')}&=\int_{\Omega'}(T_*f)(y)\varphi(y) dy=\int_{\Omega'}
(f\circ T^{-1})(y)\varphi(y) dy\\
&\stackrel{x=T^{-1}(y)}{=}\int_{\Omega}
f(x)\varphi(T(x))|\textrm{det}T'(x)|
dx\\
&=<u_f,|\textrm{det}T'(x)|\varphi(T(x))>_{D'(\Omega)\times
D(\Omega)}
\end{align*}
The above observation motivates us to define the pushforward of
any distribution $u\in D'(\Omega)$ as follows
\begin{equation*}
\forall \varphi\in D(\Omega')\qquad \langle T_*u,
\varphi\rangle_{D'(\Omega')\times D(\Omega')}:=\langle u,
|\textrm{det}T'(x)|\varphi(T(x))\rangle_{D'(\Omega)\times
D(\Omega)}
\end{equation*}
It can be shown that $T_*u:D(\Omega')\rightarrow \reals$ is
continuous and so it is in fact an element of $D'(\Omega')$.
Similarly, the pullback $T^*:D'(\Omega')\rightarrow D'(\Omega)$ is
defined by
\begin{equation*}
\forall \varphi\in D(\Omega)\qquad \langle T^*u,
\varphi\rangle_{D'(\Omega)\times D(\Omega)}:=\langle u,
|\textrm{det}(T^{-1})'(y)|\varphi(T^{-1}(y))\rangle_{D'(\Omega')\times
D(\Omega')}
\end{equation*}
It can be shown that $T^*u:D(\Omega)\rightarrow \reals$ is
continuous and so it is in fact an element of $D'(\Omega)$.
\begin{definition}[Extension by Zero of a Function]\lab{winter63}
Let $\Omega$ be an open subset of $\reals^n$ and $V$ be an open
susbset of $\Omega$. We define the linear map
$\textrm{ext}_{V,\Omega}^0: \textrm{Func}(V,\reals)\rightarrow
\textrm{Func}(\Omega,\reals)$ as follows
\begin{equation*}
\textrm{ext}_{V,\Omega}^0(f)(x)=
\begin{cases}
f(x)\quad \textrm{if $x\in V$}\\
0\quad \textrm{if $x\in \Omega\setminus V$}
\end{cases}
\end{equation*}
$\textrm{ext}_{V,\Omega}^0$ restricts to a continuous linear map
$D(V)\rightarrow D(\Omega)$.
\end{definition}
\begin{definition}[Restriction of a Distribution]\lab{winter64}
Let $\Omega$ be an open subset of $\reals^n$ and $V$ be an open
susbset of $\Omega$. We define the restriction map
$\textrm{res}_{\Omega,V}: D'(\Omega)\rightarrow D'(V)$ as follows
\begin{equation*}
\langle \textrm{res}_{\Omega,V} u,\varphi \rangle_{D'(V)\times
D(V)}:=\langle u,
\textrm{ext}_{V,\Omega}^0\varphi\rangle_{D'(\Omega)\times
D(\Omega)}
\end{equation*}
This is well-defined; indeed, $\textrm{res}_{\Omega,V}:
D'(\Omega)\rightarrow D'(V)$ is a continuous linear map as it is
the adjoint of the continuous map
$\textrm{ext}_{V,\Omega}^0:D(V)\rightarrow D(\Omega)$. Given
$u\in D'(\Omega)$, we sometimes write $u|_V$ instead of
$\textrm{res}_{\Omega,V} u$.
\end{definition}
\begin{remark}\lab{winter65}
It is easy to see that the restriction of the map
$\textrm{res}_{\Omega,V}: D'(\Omega)\rightarrow D'(V)$ to
$\mathcal{E}(\Omega)$ agrees with the usual restriction of smooth
functions.
\end{remark}
\begin{definition}[Support of a Distribution]\lab{winter66}
Let $\Omega$ be a nonempty open set in $\reals^n$. Let $u\in
D'(\Omega)$.
\begin{itemizeXALI}
\item We say $u$ is equal to zero on some open subset $V$ of
$\Omega$ if $u|_V=0$.
\item Let $\{V_i\}_{i\in I}$ be the collection of all open subsets of
$\Omega$ such that $u$ is equal to zero on $V_i$. Let
$V=\bigcup_{i\in I} V_i$. The support of $u$ is defined as follows
\begin{equation*}
\textrm{supp}\,u:=\Omega\setminus V
\end{equation*}
Note that $\textrm{supp}u$ is closed in $\Omega$ but it is not
necessarily closed in $\reals^n$.
\end{itemizeXALI}
\end{definition}
\begin{theorem}[Properties of the Support]\cite{Reus1, Rudi73,9}\lab{winter67}
Let $\Omega$ and $\Omega'$ be nonempty open sets in $\reals^n$.
\begin{itemizeXALI}
\item If $f\in L^1_{loc}(\Omega)$, then $\textrm{supp} f=\textrm{supp}\,
 u_f$.
\item For all $u\in D'(\Omega)$, $u=0$ on $\Omega\setminus
\textrm{supp}\,u$.
\item  Let $u\in D'(\Omega)$. If $\varphi\in
D(\Omega)$ vanishes on an open neighborhood of $\textrm{supp}\, u$,
then $u(\varphi)=0$.
\item  For every closed subset $A$ of
$\Omega$ and every $u\in D'(\Omega)$, we have $\textrm{supp}\,
u\subseteq A$ if and only if $u(\varphi)=0$ for every $\varphi\in
D(\Omega)$ with $\textrm{supp}\,\varphi\subseteq \Omega\setminus A$.
\item  For every $u\in D'(\Omega)$ and $\psi\in
\mathcal{E}(\Omega)$, $\textrm{supp}(\psi u)\subseteq
\textrm{supp}(\psi)\cap \textrm{supp}(u)$.
\item  Let $u, v\in D'(\Omega)$. If there exists a nonempty open subset $U$ of $\Omega$ such that
 $\textrm{supp} \,u\subseteq U$ and $\textrm{supp}\, v\subseteq U$ and
 \begin{equation*}
 \langle u|_U,\varphi\rangle_{D'(U)\times D(U)}= \langle v|_U,\varphi\rangle_{D'(U)\times D(U)}\qquad \forall\,\varphi\in C_c^\infty(U)
 \end{equation*}
 then $u=v$ as elements of $D'(\Omega)$.
 \item  Let $u, v\in D'(\Omega)$. Then $\textrm{supp}(u+v)\subseteq \textrm{supp}\,u\cup\textrm{supp}\, v$.
 \item  Let $\{u_i\}$ be a sequence in $D'(\Omega)$, $u\in D(\Omega)$, and $K\in\mathcal{K}(\Omega)$ such that $u_i\rightarrow u$ in $D'(\Omega)$ and $\textrm{supp}\,u_i\subseteq K$ for all $i$. Then also $\textrm{supp}\,u\subseteq K$.
 \item For every $u\in D'(\Omega)$ and $\alpha\in \mathbb{N}_0^n$,
$\textrm{supp}(\partial^\alpha u)\subseteq \textrm{supp}(u)$.
\item If $T:\Omega\rightarrow \Omega'$ is a diffeomorphism, then
$\textrm{supp}(T_*u)=T(\textrm{supp}\, u)$. In particular, if $u$
has compact support, then so has $T_*u$.
\end{itemizeXALI}
\end{theorem}
\begin{theorem}(\cite{Reus1}, Pages 10 and 20)\lab{winter68}
Let $\Omega$ be a nonempty open set in $\reals^n$. Let
$\mathcal{E}'(\Omega)$ denote the topological dual of
$\mathcal{E}(\Omega)$ equipped with the strong topology. Then
\begin{itemizeXALI}
\item The map that sends $u\in \mathcal{E}'(\Omega)$ to
$u|_{D(\Omega)}$ is an injective continuous linear map from
$\mathcal{E}'(\Omega)$ into $D'(\Omega)$.
\item The image of the above map consists precisely of those $u\in
D'(\Omega)$ for which $\textrm{supp}\, u$ is compact.
\end{itemizeXALI}
\end{theorem}
Due to the above theorem we may identify $\mathcal{E}'(\Omega)$
with distributions on $\Omega$ with compact support.
\begin{definition}[Extension by Zero of Distributions With Compact
Support]\lab{winter69} Let $\Omega$ be a nonempty open set in
$\reals^n$ and $V$ be a nonempty open subset of $\Omega$. We
define the linear map
$\textrm{ext}^0_{V,\Omega}:\mathcal{E}'(V)\rightarrow
\mathcal{E}'(\Omega)$ as the adjoint of the continuous linear map
$\textrm{res}_{\Omega, V}:\mathcal{E}(\Omega)\rightarrow
\mathcal{E}(V)$; that is
\begin{equation*}
\langle \textrm{ext}^0_{V,\Omega} u,\varphi
\rangle_{\mathcal{E}'(\Omega)\times \mathcal{E}(\Omega)}:=\langle
u, \varphi|_V\rangle_{\mathcal{E}'(V)\times \mathcal{E}(V)}
\end{equation*}
\end{definition}
Suppose $\Omega'$ and $\Omega$ are two nonempty open sets in
$\reals^n$ such that $\Omega'\subseteq \Omega$ and $K\in
\mathcal{K}(\Omega')$. One can easily show that
\begin{itemize}
\item For all $u\in \mathcal{E}_K(\Omega')$, $\textrm{res}_{\reals^n,\Omega}\circ
\textrm{ext}^0_{\Omega',\reals^n}u=\textrm{ext}^0_{\Omega',\Omega}u$.
\item For all $u\in \mathcal{E}_K(\Omega')$, $\textrm{ext}^0_{\Omega,\reals^n}\circ
\textrm{ext}^0_{\Omega',\Omega}u=\textrm{ext}^0_{\Omega',\reals^n}u$.
\item For all $u\in \mathcal{E}_K(\Omega)$, $\textrm{ext}^0_{\Omega',\Omega}\circ
\textrm{res}_{\Omega,\Omega'}u=u$.
\end{itemize}

We summarize the important topological properties of the spaces
of test functions and distributions in the table below.
\begin{table}[H]
\begin{center}
\begin{tabular}{|l|l|l|l|l|l|l|}
\hline & $D(\Omega)$ & $\mathcal{E}(\Omega)$ & \pbox{20cm}{$D'(\Omega)$\\
 Strong}
 & \pbox{20cm}{$\mathcal{E}'(\Omega)$\\ Strong} & \pbox{20cm}{$D'(\Omega)$\\
Weak} &  \pbox{20cm}{$\mathcal{E}'(\Omega)$\\ Weak}\\
\hline Sequential & \quad No   & \quad Yes &\quad No & \quad No &
\quad No & \quad No
\\
\hline First-Countable & \quad
 No & \quad Yes &\quad No &\quad No& \quad No& \quad No\\
 \hline Metrizable & \quad
 No & \quad Yes &\quad No &\quad No& \quad No& \quad No\\
 \hline Second-Countable & \quad
 No & \quad Yes &\quad No &\quad No& \quad No& \quad No\\
 \hline Sequentially Complete & \quad
 Yes & \quad Yes &\quad Yes &\quad Yes& \quad Yes& \quad Yes\\
 \hline Complete & \quad
 Yes & \quad Yes &\quad Yes &\quad Yes& \quad No& \quad No\\
\hline
\end{tabular}
\end{center}
\end{table}

\subsection{Distributions on Vector Bundles}
\subsubsection{Basic Definitions, Notations}
Let $M^n$ be a smooth manifold ($M$ is not necessarily compact).
Let $\pi:E\rightarrow M$ be a vector bundle of rank $r$.
\begin{enumerateXALI}
\item $\mathcal{E}(M,E)$ is defined as $C^{\infty}(M,E)$ equipped with the locally convex
topology induced by the following family of seminorms: let
$\{(U_\alpha,\varphi_\alpha,\rho_\alpha)\}_{\alpha\in I}$ be a total
trivialization atlas. Then for every $\alpha\in I$, $1\leq l\leq r$,
and $f\in C^{\infty}(M,E)$, $\tilde{f}_\alpha^l:=\rho_\alpha^l\circ
f\circ \varphi^{-1}$ is an element of
$C^{\infty}(\varphi_\alpha(U_\alpha))$. For every $4$-tuple
$(l,\alpha, j,K)$ with $1\leq l\leq r$, $\alpha\in I$, $j\in\mathbb{N}$, $K$ a compact subset of
$U_\alpha$ (i.e. $K\in \mathcal{K}(U_\alpha)$) we define
\begin{equation*}
\|.\|_{l,\alpha,j,K}:C^{\infty}(M,E)\rightarrow \reals,\quad
f\mapsto \|\rho_\alpha^l\circ f\circ
\varphi_\alpha^{-1}\|_{j,\varphi_{\alpha}(K)}
\end{equation*}
It is easy to check that $\|.\|_{l,\alpha,j,K}$ is a seminorm on
$C^{\infty}(M,E)$ and the locally convex topology induced by the
above family of seminorms does not depend on the choice of the
total trivialization atlas. Sometimes we may write
$\|.\|_{l,\varphi_\alpha,j,K}$ instead of $\|.\|_{l,\alpha,j,K}$.
\item For any compact subset $K\subseteq M$ we define
\begin{equation*}
\mathcal{E}_K(M,E):=\{f\in \mathcal{E}(M,E):
\textrm{supp}\,f\subseteq K\}\quad \textrm{equipped with the
subspace topology}
\end{equation*}
\item $D(M,E):=C_c^\infty(M,E)=\cup_{K\in \mathcal{K}(M)} \mathcal{E}_K(M,E)$ (union over all compact subsets of
$M$) equipped with the inductive limit topology with respect to the family {\fontsize{6}{7}{$\{\mathcal{E}_K(M,E)\}_{K\in\mathcal{K}(M)}$}}.
 Clearly if $M$ is compact, then $D(M,E)=\mathcal{E}(M,E)$ (as topological vector spaces).
\end{enumerateXALI}
\begin{remark}\lab{winter70}
\leavevmode
\begin{itemizeXALI}
\item If for each $\alpha \in I$,
$\{K^\alpha_m\}_{m\in\mathbb{N}}$ is an exhaustion by compact
sets of $U_\alpha$, then the topology induced by the family of
seminorms
\begin{equation*}
\{\|.\|_{l,\alpha,j,K^\alpha_m}: 1\leq l\leq r, \alpha\in I, j\in
\mathbb{N}, m\in \mathbb{N} \}
\end{equation*}
on $C^{\infty}(M,E)$ is the same as the topology of
$\mathcal{E}(M,E)$. This together with the fact that every manifold
has a countable total trivialization atlas shows that the
topology of $\mathcal{E}(M,E)$ is induced by a countable family
of seminorms. So $\mathcal{E}(M,E)$ is metrizable.
\item If $\{K_j\}_{j\in\mathbb{N}}$ is an exhuastion by compact
sets of $M$, then the inductive limit topology on
$C_c^\infty(M,E)$ with respect to the family
$\{\mathcal{E}_{K_j}(M,E)\}$ is the same as the topology on
$D(M,E)$.
\end{itemizeXALI}
\end{remark}
\begin{definition}\lab{winter71}
The space of distributions on the vector bundle $E$, denoted $D'(M,E)$, is defined as the topological dual of $D(M,E^\vee)$. That is,
\begin{equation*}
D'(M,E)=[D(M,E^\vee)]^*
\end{equation*}
As usual we equip the dual space with the strong topology. Recall that $E^\vee$ denotes the bundle $\textrm{Hom}(E,\mathcal{D}(M))$ where $\mathcal{D}(M)$ is the density bundle of $M$.
\end{definition}
\begin{remark}\lab{winter72}
The reason that space of distributions on the vector bundle $E$ is defined as the dual of $D(M,E^\vee)$ rather than the dual of the seemingly natural choice $D(M,E)$ is well explained in \cite{Grosser2001} and \cite{Reus1}. Of course,  there are other non-equivalent ways to make sense of distributions
on vector bundles (see \cite{Grosser2001} for a detailed discussion). Also see Lemma \ref{lemspringapr13} where it is proved that Riemannian density can be used to identify $D'(M,E)$ with $[D(M,E)]^*$.
\end{remark}
\begin{remark}\lab{winter73}
Let $U$ and $V$ be nonempty open sets in $M$ with $V\subseteq U$.
\begin{itemizeX}
\item As in the Euclidean case, the linear map $\textrm{ext}^0_{V,U}:\Gamma(V,E_V^\vee)\rightarrow
\Gamma(U,E_U^\vee)$ defined by
\begin{equation*}
\textrm{ext}^0_{V,U}f(x)=
\begin{cases}
f(x)\quad x\in V\\
0\quad x\in U\setminus V
\end{cases}
\end{equation*}
restricts to a continuous linear map from $D(V,E_V^\vee)$ to
$D(U,E_U^\vee)$.
\item As in the Euclidean case, the restriction map $\textrm{res}_{U,V}:D'(U,E_U)\rightarrow
D'(V,E_V)$ is defined as the adjoint of $ext^0_{V,U}$:
\begin{equation*}
\langle \textrm{res}_{U,V}u,\varphi\rangle_{D'(V,E_V)\times
D(V,E_V^\vee)}=\langle
u,\textrm{ext}^0_{V,U}\varphi\rangle_{D'(U,E_U)\times
D(U,E_U^\vee)}
\end{equation*}
\item Support of a distribution $u\in D'(M,E)$ is defined in the
exact same way as for distributions in the Euclidean space. It
can be shown that
\begin{enumerateX}
\item (\cite{Reus1}, Page 105) If $u\in D'(M,E)$ and $\varphi\in
D(M,E^\vee)$ vanishes on an open neighborhood of $\textrm{supp}
u$, then $u(\varphi)=0$.
\item (\cite{Reus1}, Page 104) For every closed subset $A$ of
$M$ and every $u\in D'(M,E)$, we have $\textrm{supp} u\subseteq A$
if and only if $u(\varphi)=0$ for every $\varphi\in D(M,E^\vee)$
with $\textrm{supp}\varphi\subseteq M\setminus A$.
\end{enumerateX}
\end{itemizeX}
\end{remark}
The strength of the theory of distributions in the Euclidean case
is largely due to the fact that it is possible to identify a huge
class of ordinary functions with distributions. A question that
arises is that whether there is a natural way to identify regular
sections of $E$ (i.e. elements of $\Gamma(M,E)$) with
distributions. The following theorem provides a partial answer to
this question. Recall that compactly supported continuous
sections of the density bundle can be integrated over $M$.
\begin{theorem}\lab{winter74}
Every $f\in \mathcal{E}(M,E)$ defines the following continuous map:
\begin{equation}\lab{eqnwinter74b}
u_f:D(M,E^\vee)\rightarrow \reals,\quad \psi\mapsto \int_M [
\psi,f]
\end{equation}
where the pairing $[ \psi,f]$ defines a compactly supported
continuous section of the density bundle:
\begin{equation*}
\forall\,x\in M\quad [\psi,f](x):=[\psi(x)][f(x)]\quad
(\textrm{$\psi(x)\in \textrm{Hom}(E_x,\mathcal{D}_x)$ evaluated at $f(x)\in E_x$})
\end{equation*}
\end{theorem}
In general, we define $\Gamma_{reg}(M,E)$ as the set
\begin{equation*}
\{f\in \Gamma(M,E):\textrm{$u_f$ defined by Equation \ref{eqnwinter74b} is well-defined and continuous}\}
\end{equation*}
(Compare this with the definition of
$\textrm{Func}_{reg}(\Omega)$ in Remark \ref{winter61b}.) Theorem
\ref{winter74} tells us that $\mathcal{E}(M,E)$ is contained in
$\Gamma_{reg}(M,E)$. If $u\in D'(M,E)$ is such that $u=u_f$ for some $f\in \Gamma_{reg}(M,E)$, then we say that $u$ is a \textbf{regular distribution}. \\\\
Now let $(U,\varphi,\rho)$ be a total trivialization triple for
$E$ and let $(U,\varphi,\rho_{\mathcal{D}})$ and
$(U,\varphi,\rho^\vee)$ be the corresponding standard total
trivialization triples for $\mathcal{D}(M)$ and $E^\vee$,
respectively. The local representation of the pairing $[\psi,f]$
has a very simple expression in terms of the local
representations of $f$ and $\psi$:
 {\fontsize{10}{10}{\begin{align*}
    & f\in \Gamma_{reg}(M,E)\Longrightarrow (\tilde{f}^1,\cdots,\tilde{f}^r):=(f^1\circ\varphi^{-1},\cdots,f^r\circ\varphi^{-1}):=\rho\circ f\circ\varphi^{-1}\in [\textrm{Func}(\varphi(U),\reals)]^{\times r}\\
    & \quad\textrm{$(\tilde{f}^1,\cdots,\tilde{f}^r)$ is the local representation of $f$}\\
    & \psi\in D(M,E^\vee)\Longrightarrow (\tilde{\psi}^1,\cdots,\tilde{\psi}^r):=(\psi^1\circ\varphi^{-1},\cdots,\psi^r\circ\varphi^{-1}):=\rho^\vee\circ \psi\circ\varphi^{-1}\in [\textrm{Func}(\varphi(U),\reals)]^{\times r} \\
    &\quad\textrm{$(\tilde{\psi}^1,\cdots,\tilde{\psi}^r)$ is the local representation of $\psi$}
    \end{align*}}}
    Our claim is that the local representation of $[\psi,f]$, that is $\rho_{\mathcal{D}}\circ[\psi,f]\circ\varphi^{-1}$, is equal to the Euclidean dot product of the local representations of $f$ and $\psi$:
    \begin{equation*}
    \rho_{\mathcal{D}}\circ[\psi,f]\circ\varphi^{-1}=\sum_i \tilde{f}^i\tilde{\psi} ^i
    \end{equation*}
    The reason is as follows: Let $y\in \varphi(U)$ and $x=\varphi^{-1}(y)$
\begin{align*}
[&\rho_{\mathcal{D}}\circ [\psi, f]\circ \varphi^{-1}](y)=\rho_{\mathcal{D}}\big([\psi(x)][f(x)]\big)=\rho_{\mathcal{D}}\big([\psi(x)][(\rho|_{E_x})^{-1}(\tilde{f}^1(y),\cdots,\tilde{f}^r(y))]\big)\\
&=[\rho_{\mathcal{D}}\circ \psi(x)\circ  (\rho|_{E_x})^{-1}](\tilde{f}^1(y),\cdots,\tilde{f}^r(y))\\
&=[\rho^\vee(\psi(x))][(\tilde{f}^1(y),\cdots,\tilde{f}^r(y))] \quad \textrm{the left bracket is applied to the right bracket}\\
&=\rho^\vee(\psi(x))\cdot (\tilde{f}^1(y),\cdots,\tilde{f}^r(y))\quad \textrm{dot product! $\rho^\vee(\psi(x))$ viewed as an element of $\reals^r$}\\
&=(\tilde{\psi}^1(y),\cdots,\tilde{\psi}^r(y))\cdot
(\tilde{f}^1(y),\cdots,\tilde{f}^r(y))
\end{align*}
\subsubsection{Local Representation of Distributions}

Let $(U,\varphi,\rho)$ be a total trivialization triple for $\pi:
E\rightarrow M$. We know that each $f\in \Gamma(M,E)$ can locally
be represented by $r$ components $\tilde{f}^1,\cdots,\tilde{f}^r$
defined by
\begin{equation*}
\forall\,1\leq l\leq r\quad \tilde{f}^l: \varphi(U)\rightarrow
\reals, \quad \tilde{f}^l=\rho^l\circ f\circ \varphi^{-1}
\end{equation*}
These components play a crucial role in our study of Sobolev
spaces. Now the question is that whether we can similarly use the
total trivialization triple $(U,\varphi,\rho)$ to locally
associate with each distribution $u\in D'(M,E)$, $r$ components
$\tilde{u}^1,\cdots,\tilde{u}^r$ belonging to $D'(\varphi(U))$.
That is, we want to see whether we can define a nice map
\begin{equation*}
D'(U,E_U)=[D(U,E_U^{\vee})]^*\rightarrow
\underbrace{D'(\varphi(U))\times\cdots\times
D'(\varphi(U))}_{\textrm{$r$ times}}
\end{equation*}
(Note that according to Remark \ref{winter73}, if $u\in D'(M,E)$,
then $u|_U\in D'(U,E_U)$.) Such a map, in particular, will be
important when we want to
 make sense of Sobolev spaces
with negative exponents of sections of vector bundles. Also it
would be desirable to ensure that if $u$ is a regular distribution
then the components of $u$ as a distribution agree with the
components
obtained when $u$ is viewed as an element of $\Gamma(M,E)$.\\\\
We begin with the following map at the level of compactly supported smooth
functions:
{\fontsize{10}{10}{\begin{equation*}
\tilde{T}_{E^\vee,U,\varphi}: D(U,E_U^{\vee})\rightarrow
[D(\varphi(U))]^{\times r},\quad \xi\rightarrow \rho^{\vee}\circ
\xi\circ \varphi^{-1}= ((\rho^{\vee})^1\circ \xi\circ
\varphi^{-1},\cdots,(\rho^{\vee})^r\circ \xi\circ \varphi^{-1})
\end{equation*}}}
Note that $\tilde{T}_{E^\vee,U,\varphi}$ has the property that for
all $\psi\in C^{\infty}(U)$ and $\xi\in D(U,E_U^\vee)$
\begin{equation*}
\tilde{T}_{E^\vee,U,\varphi}(\psi \xi)=(\psi\circ
\varphi^{-1})\tilde{T}_{E^\vee,U,\varphi}(\xi)\,.
\end{equation*}
\begin{theorem}\lab{winter75}
The map $\tilde{T}_{E^\vee,U,\varphi}: D(U,E_U^{\vee})\rightarrow
[D(\varphi(U))]^{\times r}$ is a linear topological isomorphism.
($[D(\varphi(U))]^{\times r}$ is equipped with the product
topology.)
\end{theorem}
\begin{proof}
Clearly $\tilde{T}_{E^\vee,U,\varphi}$ is linear. Also the map
$\tilde{T}_{E^\vee,U,\varphi}$ is bijective. Indeed, the inverse
of $\tilde{T}_{E^\vee,U,\varphi}$ (which we denote by
$T_{E^\vee,U,\varphi}$) is given by
\begin{align*}
&T_{E^\vee,U,\varphi}: [D(\varphi(U))]^{\times r}\rightarrow D(U,E_U^{\vee})\\
& \forall\,x\in U\quad
T_{E^\vee,U,\varphi}(\xi_1,\cdots,\xi_r)(x)=\big(\rho^{\vee}|_{E_x^{\vee}}\big)^{-1}\circ(\xi_1,\cdots,\xi_r)\circ\varphi(x)
\end{align*}
Now we show that $\tilde{T}_{E^\vee,U,\varphi}:
D(U,E_U^{\vee})\rightarrow [D(\varphi(U))]^{\times r}$ is
continuous. To this end, it is enough to prove that for each
$1\leq l\leq r$ the map
\begin{equation*}
\pi^l\circ \tilde{T}_{E^\vee,U,\varphi}:D(U,E_U^{\vee})\rightarrow
 D(\varphi(U)),\qquad \xi\mapsto (\rho^\vee)^l\circ
 \xi\circ \varphi^{-1}
\end{equation*}
is continuous. The topology on $D(U,E_U^{\vee})$ is the inductive
limit topology with respect to
$\{\mathcal{E}_K(U,E_U^{\vee})\}_{K\in \mathcal{K}(U)}$, so it is
enough to show that for each $K\in \mathcal{K}(U)$,
$\pi^l\circ\tilde{T}_{E^\vee,U,\varphi}:\mathcal{E}_K(U,E_U^{\vee})\rightarrow
 D(\varphi(U))$ is continuous. Note that $\pi^l\circ\tilde{T}_{E^\vee,U,\varphi}[\mathcal{E}_K(U,E_U^{\vee})]\subseteq
 \mathcal{E}_{\varphi(K)}(\varphi(U))$. Considering that $\mathcal{E}_{\varphi(K)}(\varphi(U))\hookrightarrow
 D(\varphi(U))$, it is enough to show that
 \begin{equation*}
 \pi^l\circ\tilde{T}_{E^\vee,U,\varphi}:
 \mathcal{E}_K(U,E_U^{\vee})\rightarrow \mathcal{E}_{\varphi(K)}(\varphi(U))
 \end{equation*}
 is continuous. For all $\xi\in \mathcal{E}_K(U,E_U^{\vee})$ and $j\in \mathbb{N}$ we
 have
 \begin{equation*}
 \| \pi^l\circ\tilde{T}_{E^\vee,U,\varphi}
 (\xi)\|_{j,\varphi(K)}=\|(\rho^\vee)^l\circ
 \xi\circ\varphi^{-1}\|_{j,\varphi(K)}=\|\xi\|_{l,\varphi,j,K}
 \end{equation*}
 which implies the continuity (note that even an inequality in place of the last equality would have been enough to prove the
 continuity). It remains to prove the continuity of $T_{E^\vee,U,\varphi}: [D(\varphi(U))]^{\times r}\rightarrow D(U,E_U^{\vee})$. By Theorem
 \ref{thmfallprodinductcont11} it is enough to show that for all
 $K\in \mathcal{K}(\varphi(U))$, $T_{E^\vee,U,\varphi}: [\mathcal{E}_K(\varphi(U))]^{\times r}\rightarrow
 D(U,E_U^{\vee})$ is continuous. It is clear that $T_{E^\vee,U,\varphi}([\mathcal{E}_K(\varphi(U))]^{\times r})\subseteq
 \mathcal{E}_{\varphi^{-1}(K)}(U,E_U^{\vee})$. Since $\mathcal{E}_{\varphi^{-1}(K)}(U,E_U^{\vee})\hookrightarrow
 D(U,E_U^{\vee})$, it is sufficient to show that $T_{E^\vee,U,\varphi}: [\mathcal{E}_K(\varphi(U))]^{\times r}\rightarrow
 \mathcal{E}_{\varphi^{-1}(K)}(U,E_U^{\vee})$ is continuous. To
 this end,
 by Theorem \ref{thmfallprodinductcont11}, we just need to show that for all $j\in \mathbb{N}$ and $1\leq l\leq r$ there exists
 $j_1,\cdots,j_r$ such that
 \begin{equation*}
 \|T_{E^\vee,U,\varphi}(\xi_1,\cdots,\xi_r)\|_{l,\varphi,j,\varphi^{-1}(K)
 }\leq C (\|\xi_1\|_{j_1,K}+\cdots\|\xi_r\|_{j_r,K})
 \end{equation*}
 But this obviously holds because
 \begin{equation*}
\|T_{E^\vee,U,\varphi}(\xi_1,\cdots,\xi_r)\|_{l,\varphi,j,\varphi^{-1}(K)
 }=\|\xi_l\|_{j,K}
 \end{equation*}
\end{proof}

 The adjoint of $T_{E^\vee,U,\varphi}$ is
\begin{align*}
&T_{E^\vee,U,\varphi}^{*}: [D(U,E_U^{\vee})]^*\rightarrow \big([D(\varphi(U))]^{\times r}\big)^{*}\\
&\langle T_{E^\vee,U,\varphi}^{*}u,(\xi_1,\cdots,\xi_r)\rangle=\langle u,T_{E^\vee,U,\varphi}(\xi_1,\cdots,\xi_r)\rangle
\end{align*}
Note that, since $T_{E^\vee,U,\varphi}$ is a linear topological
isomorphism, $T_{E^\vee,U,\varphi}^{*}$ is also a linear
topological isomorphism (and in particular it is bijective).
For every $u\in [D(U,E_U^{\vee})]^*$, $T_{E^\vee,U,\varphi}^{*}u$
is in $\big([D(\varphi(U))]^{\times r}\big)^{*}$; we can combine
this with the bijective map
\begin{equation*}
L: \big([D(\varphi(U))]^{\times r}\big)^{*}\rightarrow
[D'(\varphi(U))]^{\times r},\quad L(v)=(v\circ i_1,\cdots,v\circ
i_r)
\end{equation*}
(see Theorem \ref{winter33}) to send $u\in [D(U,E_U^{\vee})]^*$
into an element of $[D'(\varphi(U))]^{\times r}$:
\begin{equation*}
L(T_{E^\vee,U,\varphi}^{*}u)=((T_{E^\vee,U,\varphi}^*u)\circ i_1,\cdots,(T_{E^\vee,U,\varphi}^*u)\circ i_r)
\end{equation*}
where for all $1\leq l\leq r$, $(T_{E^\vee,U,\varphi}^*u)\circ
i_l\in D'(\varphi(U))$ is given by
\begin{align*}
((T_{E^\vee,U,\varphi}^*u)\circ i_l)(\xi)&=(T_{E^\vee,U,\varphi}^*u)(i_l(\xi))=(T_{E^\vee,U,\varphi}^*u)(0,\cdots,0,\underbrace{\xi}_{\textrm{$l^{th}$ position}},0,\cdots,0)\\
&=\langle u,T_{E^\vee,U,\varphi}(0,\cdots,0,\underbrace{\xi}_{\textrm{$l^{th}$ position}},0,\cdots,0)\rangle
\end{align*}
If we define $g_{l,\xi, U,\varphi}\in D(U,E_U^{\vee})$ by
\begin{align*}
g_{l,\xi,U,\varphi}(x)&=T_{E^\vee,U,\varphi}(0,\cdots,0,\underbrace{\xi}_{\textrm{$l^{th}$  position}},0,\cdots,0)(x)\\
&=\big(\rho^{\vee}|_{E_x^{\vee}}\big)^{-1}\circ
(0,\cdots,0,\underbrace{\xi}_{\textrm{$l^{th}$
position}},0,\cdots,0)\circ \varphi(x)
\end{align*}
then we may write
\begin{equation*}
\langle (T_{E^\vee,U,\varphi}^*u)\circ
i_l,\xi\rangle_{D'(\varphi(U))\times D(\varphi(U))}=\langle
u,g_{l,\xi,U,\varphi}\rangle_{[D(U,E_U^\vee)]^*\times
D(U,E_U^\vee)}
\end{equation*}
\textbf{Summary:} We can associate with $u\in
D'(U,E_U)=(D(U,E_U^{\vee}))^{*}$ the following $r$ distributions
in $D'(\varphi(U))$:
\begin{equation*}
\forall\, 1\leq l\leq r\quad
\tilde{u}^l=T_{E^\vee,U,\varphi}^*u\circ i_l
\end{equation*}
that is
\begin{equation*}
\forall\,\xi\in D(\varphi(U))\quad \langle
\tilde{u}^l,\xi\rangle=\langle u,g_{l,\xi,U,\varphi}\rangle
\end{equation*}
where  $g_{l,\xi,U,\varphi}\in D(U,E_U^{\vee})$ is defined by
\begin{equation*}
\big(\rho^{\vee}|_{E_x^{\vee}}\big)^{-1}\circ
(0,\cdots,0,\underbrace{\xi}_{\textrm{$l^{th}$
position}},0,\cdots,0)\circ \varphi(x)
\end{equation*}
In particular,
\begin{equation*}
\rho^\vee\circ g_{l,\xi,U,\varphi}\circ
\varphi^{-1}=(0,\cdots,0,\underbrace{\xi}_{\textrm{$l^{th}$
position}},0,\cdots,0)
\end{equation*}
and so $(\rho^\vee\circ g_{l,\xi,U,\varphi}\circ
\varphi^{-1})^l=\xi$.\\\\
Let's give a name to the composition of $L$ with
$T_{E^\vee,U,\varphi}^*$ that we used above. We set
$H_{E^\vee,U,\varphi}:=L\circ T_{E^\vee,U,\varphi}^*$:
\begin{equation*}
H_{E^\vee,U,\varphi}: [D(U,E_U^\vee)]^*\rightarrow
(D'(\varphi(U)))^{\times r},\qquad u\mapsto
L(T_{E^\vee,U,\varphi}^*u)=(\tilde{u}^1,\cdots,\tilde{u}^r)
\end{equation*}
\begin{remark}\lab{remfall134}
Here we make three observations about the mapping
$H_{E^\vee,U,\varphi}$.
\begin{enumerateX}
\item For every $u\in [D(U,E_U^\vee)]^*$
\begin{equation*}
\textrm{supp}[H_{E^\vee,U,\varphi}\,u]^l=\textrm{supp}\tilde{u}^l\subseteq
\varphi (\textrm{supp}\,u)
\end{equation*}
Indeed, let $A=\varphi(\textrm{supp} u)$. By Theorem
\ref{winter67}, it is enough to show that if $\eta\in D(\varphi
(U))$ is such that $\textrm{supp} \eta \subseteq
\varphi(U)\setminus A$, then $\tilde{u}^l(\eta)=0$. Note that
\begin{align*}
\langle \tilde{u}^l,\eta\rangle=\langle
u,g_{l,\eta,U,\varphi}\rangle
\end{align*}
So by Remark \ref{winter73} we just need to show that
$g_{l,\eta,U,\varphi}=0$ on an open neighborhood of
$\textrm{supp} u$. Let $K=\textrm{supp} \eta$.
clearly $U\setminus \varphi^{-1}(K)$ is an open neighborhood of
$\textrm{supp} u$. We will show that $g_{l,\eta,U,\varphi}$
vanishes on this open neighborhood. Note that
\begin{equation*}
g_{l,\eta,U,\varphi}(x)=(\rho^\vee|_{E_x^\vee})^{-1}(0,\cdots,0,\underbrace{\eta\circ
\varphi(x)}_{\textrm{$l^{th}$ position}},0,\cdots,0 )
\end{equation*}
Since $\rho^\vee|_{E_x^\vee}$ is an isomorphism and $\eta=0$ on
$\varphi(U)\setminus K$, we conclude that $g_{l,\eta,U,\varphi}=0$
on $\varphi^{-1}(\varphi(U)\setminus K)=U\setminus
\varphi^{-1}(K)$.
\item Clearly $H_{E^\vee, U,\varphi}: D'(U,E_U)\rightarrow [D'(\varphi(U))]^{\times
r}$ preserves addition. Moreover if $f\in C^\infty(U)$ and $u\in
D'(U,E_U)$, then $H_{E^\vee, U,\varphi}(fu)=(f\circ \varphi^{-1}
)H_{E^\vee, U,\varphi}(u)$. Recall that $H=L\circ T^*_{E^\vee,
U,\varphi}$.
\begin{align*}
\langle
T_{E^\vee,U,\varphi}^{*}(fu),(\xi_1,\cdots,\xi_r)\rangle&=\langle
fu,T_{E^\vee,U,\varphi}(\xi_1,\cdots,\xi_r)\rangle\\
&=\langle u,fT_{E^\vee,U,\varphi}(\xi_1,\cdots,\xi_r)\rangle\\
&=\langle
u,T_{E^\vee,U,\varphi}[(f\circ\varphi^{-1})(\xi_1,\cdots,\xi_r)]\rangle\\
&= \langle T^*_{E^\vee,U,\varphi}u,(f\circ \varphi^{-1}
)(\xi_1,\cdots,\xi_r)\rangle\\
&=\langle (f\circ \varphi^{-1})T^*_{E^\vee,U,\varphi}u,
(\xi_1,\cdots,\xi_r)\rangle\\
\end{align*}
(the third equality follows directly from the definition of
$T_{E^\vee,U,\varphi}$.) Therefore
\begin{equation*}
T_{E^\vee,U,\varphi}^{*}(fu)=(f\circ
\varphi^{-1})T^*_{E^\vee,U,\varphi}u
\end{equation*}
The fact that $L((f\circ
\varphi^{-1})T^*_{E^\vee,U,\varphi}u)=(f\circ
\varphi^{-1})L(T^*_{E^\vee,U,\varphi}u)$ is an immediate
consequence of the definition of $L$.
\item Since $T_{E^\vee,U,\varphi}$ and $L$ are both linear
topological isomorphisms, $H_{E^\vee,U,\varphi}^{-1}=(L\circ
T_{E^\vee,U,\varphi}^*)^{-1}:(D'(\varphi(U)))^{\times
r}\rightarrow D^*(U,E_U^\vee)$ is also a linear topological
isomorphism. It is useful for our later considerations to find an
explicit formula for this map. Note that
\begin{align*}
H_{E^\vee,U,\varphi}^{-1}&=(L\circ
T_{E^\vee,U,\varphi}^*)^{-1}=(T_{E^\vee,U,\varphi}^*)^{-1}\circ
L^{-1}=(T_{E^\vee,U,\varphi}^{-1})^{*}\circ
L^{-1}\\
&=(\tilde{T}_{E^\vee,U,\varphi})^*\circ
L^{-1}=(\tilde{T}_{E^\vee,U,\varphi})^*\circ \tilde{L}
\end{align*}
Recall that
\begin{align*}
& \tilde{L}: [D^*(\varphi(U))]^{\times r}\rightarrow
[(D(\varphi(U)))^{\times r}]^*,\quad (v^1,\cdots,v^r)\mapsto
v^1\circ \pi_1+\cdots+v^r\circ \pi_r\\
& \tilde{T}_{E^\vee,U,\varphi}^*:[(D(\varphi(U)))^{\times
r}]^*\rightarrow D^*(U,E_U^\vee)
\end{align*}
Therefore for all $\xi\in D(U,E_U^\vee)$
\begin{align*}
H_{E^\vee,U,\varphi}^{-1}(v^1,\cdots,v^r)(\xi)&=\langle
\tilde{T}_{E^\vee,U,\varphi}^*(v^1\circ \pi_1+\cdots+v^r\circ
\pi_r),\xi\rangle\\
&=\langle (v^1\circ
\pi_1+\cdots+v^r\circ \pi_r),\tilde{T}\xi\rangle\\
&=\langle (v^1\circ \pi_1+\cdots+v^r\circ
\pi_r),((\rho^{\vee})^1\circ \xi\circ
\varphi^{-1},\cdots,(\rho^{\vee})^r\circ \xi\circ
\varphi^{-1})\rangle\\
&=\sum_i v^i[(\rho^{\vee})^i\circ \xi\circ \varphi^{-1}]
\end{align*}
\end{enumerateX}
\end{remark}
\begin{remark}\lab{remfall135}
Suppose $u\in D'(M,E)$ is a regular distribution, that is $u=u_f$
where $f\in \Gamma_{reg}(M,E)$. We want to see whether the local
components of such a distribution agree with its components as an
element of $\Gamma(M,E)$. With respect to the total
trivialization triple $(U,\varphi,\rho)$ we have
\begin{enumerate}
\item $f\mapsto (\tilde{f}^1,\cdots,\tilde{f}^r),\quad \tilde{f}^l=\rho^l\circ f\circ\varphi^{-1}$
\item $u_f\mapsto (\tilde{u_f}^1,\cdots,\tilde{u_f}^l)$
\end{enumerate}
The question is whether $u_{\tilde{f}^l}=\tilde{u_f}^l$? Here we
will show that the answer is positive. Indeed, for all $\xi\in
D(\varphi(U))$ we have
\begin{align*}
\langle \tilde{u_f}^l,\xi\rangle&=\langle
u_f,g_{l,\xi,U,\varphi}\rangle=\int_M [
g_{l,\xi,U,\varphi},f]=\int_{\varphi(U)}\sum_i
(\tilde{g}_{l,\xi,U,\varphi})^i\tilde{f}^i dV=\int_{\varphi(U)}
(\tilde{g}_{l,\xi,U,\varphi})^l\tilde{f}^l dV\\
&=\int_{\varphi(U)} \tilde{f}^l \xi dV=\langle
u_{\tilde{f}^l},\xi\rangle
\end{align*}
Note that the above calculation in fact shows that the
restriction of $H_{E^\vee,U,\varphi}$ to $D(U,E_U)$ is $\tilde{T}_{E,U,\varphi}$.
\end{remark}

\section{Spaces of Sobolev and Locally Sobolev Functions in $\reals^n$}
In this section we present a brief overview of the basic
definitions and properties related to Sobolev spaces on Euclidean
spaces. 
\subsection{Basic Definitions}
\begin{definition}\lab{winter76}
Let $s\geq 0$ and $p\in [1,\infty]$. The Sobolev-Slobodeckij space
$W^{s,p}(\mathbb{R}^n)$ is defined as follows:
\begin{itemize}
\item If $s=k\in \mathbb{N}_0$, $p\in[1,\infty]$,
\begin{equation*}
W^{k,p}(\mathbb{R}^n)=\{u\in L^p (\mathbb{R}^n):
\|u\|_{W^{k,p}(\mathbb{R}^n)}:=\sum_{|\nu|\leq
k}\|\partial^{\nu}u\|_p<\infty\}
\end{equation*}
\item If $s=\theta\in(0,1)$, $p\in[1,\infty)$,
\begin{equation*}
W^{\theta,p}(\mathbb{R}^n)=\{u\in L^p (\mathbb{R}^n):
 |u|_{W^{\theta,p}(\mathbb{R}^n)}:=\big(\int\int_{\mathbb{R}^n\times
\mathbb{R}^n}\frac{|u(x)-u(y)|^p}{|x-y|^{n+\theta p}}dx
dy\big)^{\frac{1}{p}} <\infty\}
\end{equation*}
\item If $s=\theta\in(0,1)$, $p=\infty$,
\begin{equation*}
W^{\theta,\infty}(\mathbb{R}^n)=\{u\in L^{\infty} (\mathbb{R}^n):
 |u|_{W^{\theta,\infty}(\mathbb{R}^n)}:=\esssup_{x,y \in
\mathbb{R}^n, x\neq y}\frac{|u(x)-u(y)|}{|x-y|^{\theta}} <\infty\}
\end{equation*}
\item If $s=k+\theta,\, k\in \mathbb{N}_0,\, \theta\in(0,1)$,
$p\in[1,\infty]$,
\begin{equation*}
W^{s,p}(\mathbb{R}^n)=\{u\in
W^{k,p}(\mathbb{R}^n):\|u\|_{W^{s,p}(\mathbb{R}^n)}:=\|u\|_{W^{k,p}(\mathbb{R}^n)}+\sum_{|\nu|=k}
|\partial^{\nu}u|_{W^{\theta,p}(\mathbb{R}^n)}<\infty\}
\end{equation*}
\end{itemize}
\end{definition}
\begin{remark}\lab{winter77}
Clearly for all $s\geq 0$, $W^{s,p}(\reals^n)\subseteq
L^{p}(\reals^n)$. Recall that $L^p(\reals^n)\subseteq
L^1_{loc}(\reals^n)\subseteq D'(\reals^n)$. So we may consider elements
of $W^{s,p}(\reals^n)$ as distributions in
$D'(\reals^n)$. Indeed, for $s\geq 0$, $p\in (1,\infty)$, and $u\in D'(\reals^n)$ we define
\begin{align*}
\begin{cases}
\|u\|_{W^{s,p}(\reals^n)}:=\|f\|_{W^{s,p}(\reals^n)}\quad &\textrm{if $u=u_f$ for some $f\in L^p(\reals^n)$}\\
 \|u\|_{W^{s,p}(\reals^n)}:=\infty \quad &\textrm{otherwise}
\end{cases}
\end{align*}
As a consequence we may write
\begin{equation*}
W^{s,p}(\reals^n)=\{u\in D'(\reals^n): \|u\|_{W^{s,p}(\reals^n)}<\infty\}
\end{equation*}
\end{remark}
\begin{remark}\lab{remfalldecomint1}
Let us make some observations that will be helpful in the proof of a number of important theorems. Let $A$ be a nonempty measurable set in $\reals^n$.
\begin{enumerateXALI}
\item  We may write:
\begin{align*}
&\int\int_{\mathbb{R}^n\times
\mathbb{R}^n}\frac{|\partial^\nu u(x)-\partial^\nu u(y)|^p}{|x-y|^{n+\theta p}}dxdy\\
&= \int\int_{A\times
 A}\cdots dx dy+\int_A\int_{
 \reals^n\setminus A}\cdots dx dy+\int_{\reals^n\setminus A}\int_A\cdots dx dy+\int_{
 \reals^n\setminus A}\int_{
 \reals^n\setminus A}\cdots dx dy
\end{align*}
In particular, if $\textrm{supp} u\subseteq A$, then the last integral vanishes and the sum of the two middle integrals will be equal to $2\int_A\int_{
 \reals^n\setminus A}\frac{|\partial^\nu u(x)|^p}{|x-y|^{n+\theta p}} dy dx$. Therefore in this case
 \begin{align*}
 &\int\int_{\mathbb{R}^n\times
\mathbb{R}^n}\frac{|\partial^\nu u(x)-\partial^\nu u(y)|^p}{|x-y|^{n+\theta p}}dxdy=\\
& \int\int_{A\times
 A}\frac{|\partial^\nu u(x)-\partial^\nu u(y)|^p}{|x-y|^{n+\theta p}} dx dy+2\int_A\int_{
 \reals^n\setminus A}\frac{|\partial^\nu u(x)|^p}{|x-y|^{n+\theta p}} dy dx
 \end{align*}
 \item If $A$ is open, $K\subseteq A$ is compact and $\alpha>n$, then there exists a number $C$ such that for all $x\in K$ we have
 \begin{equation*}
 \int_{\reals^n\setminus A}\frac{1}{|x-y|^{\alpha}}dy\leq C
 \end{equation*}
 ($C$ may depend on $A$, $K$, $n$, and $\alpha$ but is independent of $x$.)
 The reason is as follows: Let $R=\frac{1}{2}\textrm{dist}(K,A^c)>0$. Clearly for all $x\in K$ the ball $B_R(x)$ is inside $A$. Therefore for all $x\in K$, $\reals^n\setminus A\subseteq\reals^n\setminus B_R(x)$ which implies that for all $x\in K$
 {\fontsize{10}{10}{\begin{equation*}
 \int_{\reals^n\setminus A}\frac{1}{|x-y|^{\alpha}}dy\leq  \int_{\reals^n\setminus B_R(x)}\frac{1}{|x-y|^{\alpha}}dy\stackrel{z=y-x}{=} \int_{\reals^n\setminus B_R(0)}\frac{1}{|z|^\alpha}dz=\sigma(S^{n-1})\int_R^\infty \frac{1}{r^\alpha}r^{n-1}dr
 \end{equation*}}}
 which converges because $\alpha>n$. We can let $C=\sigma(S^{n-1})\int_R^\infty \frac{1}{r^\alpha}r^{n-1}dr$.
 \item If $A$ is bounded and $\alpha<n$, then there exists a number $C$ such that for all $x\in A$
 \begin{equation*}
 \int_A\frac{1}{|x-y|^\alpha}dy\leq C
 \end{equation*}
  ($C$ depends on $A$, $n$, and $\alpha$ but is independent of $x$.)  The reason is as follows: Since $A$ is bounded there exists $R>0$ such that for all $x,y\in A$ we have $|x-y|<R$. So for all $x\in A$
  \begin{equation*}
   \int_A\frac{1}{|x-y|^\alpha}dy\leq \sigma(S^{n-1})\int_0^R\frac{1}{r^\alpha}r^{n-1}dr
  \end{equation*}
   which converges because $\alpha<n$.
 \end{enumerateXALI}
\end{remark}
\begin{theorem}\lab{winter78}
Let $s\geq 0$ and $p\in(1,\infty)$. $C_c^\infty(\reals^n)$ is
dense in $W^{s,p}(\reals^n)$. In fact, the identity map
$i_{D,W}: D(\reals^n)\rightarrow W^{s,p}(\reals^n)$ is a linear
continuous map with dense image.
\end{theorem}
\begin{proof}
The fact that $C_c^\infty(\reals^n)$ is
dense in $W^{s,p}(\reals^n)$ follows from Theorem 7.38 and Lemma 7.44 in \cite{Adams75} combined with Remark \ref{remfallequivnorm120}.
Linearity of $i_{D,W}$ is obvious. It remains to prove that this map is continuous. By Theorem \ref{thmfallconvcont13} it is enough to show that
\begin{equation*}
\forall\,K\in \mathcal{K}(\reals^n), \forall\,\varphi\in \mathcal{E}_K(\reals^n)\quad \exists j\in \mathbb{N}\quad \textrm{s.t.}\quad \|\varphi\|_{W^{s,p}(\reals^n)}\preceq \|\varphi\|_{j,K}
\end{equation*}
Let $s=m+\theta$ where $m\in\mathbb{N}_0$ and $\theta\in [0,1)$. If $\theta\neq 0$, by definition $\|\varphi\|_{W^{s,p}(\reals^n)}=\|\varphi\|_{W^{m,p}(\reals^n)}+\sum_{|\nu|=m}|\partial^\nu \varphi|_{W^{\theta,p}(\reals^n)}$. It is enough to show that each summand can be bounded by a constant multiple of $\|\varphi\|_{j,K}$ for some $j$.
\begin{itemizeXALI}
\item \textbf{Step 1:} If $\theta=0$,
\begin{align*}
\|\varphi\|_{W^{m,p}(\reals^n)}&=\sum_{|\nu|\leq m}\|\partial^\nu \varphi\|_{L^p(\reals^n)}=\sum_{|\nu|\leq m}\|\partial^\nu \varphi\|_{L^p(K)}\\
&=\sum_{|\nu|\leq m} (\|\varphi\|_{m,K}|K|^{\frac{1}{p}})\preceq \|\varphi\|_{m,K}
\end{align*}
where the implicit constant depends on $m$, $p$, and $K$ but is independent of $\varphi$.
\item \textbf{Step 2:} Let $A$ be an open ball that contains $K$ (in particular, $A$ is bounded). As it was pointed out in Remark \ref{remfalldecomint1} we may write
\begin{align*}
&\int\int_{\mathbb{R}^n\times
\mathbb{R}^n}\frac{|\partial^\nu \varphi(x)-\partial^\nu \varphi(y)|^p}{|x-y|^{n+\theta p}}dxdy=\\
&
\int\int_{A\times
A}\frac{|\partial^\nu \varphi(x)-\partial^\nu \varphi(y)|^p}{|x-y|^{n+\theta p}}dxdy+
2\int_A\int_{
\mathbb{R}^n\setminus A}\frac{|\partial^\nu \varphi(x)|^p}{|x-y|^{n+\theta p}}dydx
\end{align*}

First note that $\reals^n$ is a convex open set; so by Theorem \ref{thmfallconvexapostol1} every function $f\in \mathcal{E}_K(\reals^n)$ is Lipschitz; indeed, for all $x,y\in \reals^n$ we have $|f(x)-f(y)|\preceq \|f\|_{1,K}\|x-y\|$. Hence
\begin{align*}
\int\int_{A\times
A}\frac{|\partial^\nu \varphi(x)-\partial^\nu \varphi(y)|^p}{|x-y|^{n+\theta p}}dxdy&\leq \int_A \|\partial^\nu \varphi\|^p_{1,K}\int_A \frac{|x-y|^p}{|x-y|^{n+\theta p}}dydx\\
&=\int_A \|\partial^\nu \varphi\|^p_{1,K}\int_A \frac{1}{|x-y|^{n+(\theta-1) p}}dydx
\end{align*}
By part 3 of Remark \ref{remfalldecomint1} $\int_A \frac{1}{|x-y|^{n+(\theta-1) p}}dy$ is bounded by a constant independent of $x$; also clearly $\|\partial^\nu \varphi\|_{1,K}\leq \|\varphi\|_{m+1,K}$. Considering that $|A|$ is finite we get
\begin{equation*}
\int\int_{A\times
A}\frac{|\partial^\nu \varphi(x)-\partial^\nu \varphi(y)|^p}{|x-y|^{n+\theta p}}dxdy\preceq \|\varphi\|^p_{m+1,K}
\end{equation*}
Finally for the remaining integral we have
\begin{align*}
\int_A\int_{
\mathbb{R}^n\setminus A}\frac{|\partial^\nu \varphi(x)|^p}{|x-y|^{n+\theta p}}dydx=\int_K\int_{
\mathbb{R}^n\setminus A}\frac{|\partial^\nu \varphi(x)|^p}{|x-y|^{n+\theta p}}dydx
\end{align*}
because the inner integral is zero for $x\not \in K$. Now we can write
\begin{align*}
\int_K\int_{
\mathbb{R}^n\setminus A}\frac{|\partial^\nu \varphi(x)|^p}{|x-y|^{n+\theta p}}dydx\preceq \int_K \|\varphi\|_{m,K}^p\int_{
\mathbb{R}^n\setminus A}\frac{1}{|x-y|^{n+\theta p}}dydx
\end{align*}
By part 2 of Remark \ref{remfalldecomint1} for all $x\in K$, the inner integral is bounded by a constant. Since $|K|$ is finite we conclude that
\begin{equation*}
\int_A\int_{
\mathbb{R}^n\setminus A}\frac{|\partial^\nu \varphi(x)|^p}{|x-y|^{n+\theta p}}dydx\preceq \|\varphi\|_{m,K}^p
\end{equation*}
\end{itemizeXALI}
Hence
\begin{equation*}
\|u\|_{W^{s,p}(\reals^n)}\preceq \|\varphi\|_{m+1,K}
\end{equation*}
\end{proof}
\begin{definition}\lab{winter79}
Let $s>0$ and $p\in(1,\infty)$. We define
\begin{equation*}
W^{-s,p'}(\mathbb{R}^n)=(W^{s,p}(\mathbb{R}^n))^{*} \quad
(\frac{1}{p}+\frac{1}{p'}=1).
\end{equation*}
\end{definition}
\begin{remark}\lab{winter80}
Note that since the identity map from $D(\reals^n)$ to
$W^{s,p}(\reals^n)$ is continuous with dense image, the dual space
$W^{-s,p'}(\reals^n)$ can be viewed as a subspace of
$D'(\reals^n)$. Indeed, by Theorem \ref{thmfallinjectiveadjoint1}
the adjoint of the identity map, $i_{D,W}^*:
W^{-s,p'}(\reals^n)\rightarrow D'(\reals^n)$ is an injective
linear continuous map and we can use this map to identify
$W^{-s,p'}(\reals^n)$ with a subspace of $D'(\reals^n)$. It is a
direct consequence of the definition of adjoint that for all
$u\in W^{-s,p'}(\reals^n)$, $i_{D,W}^*u=u|_{D(\reals^n)}$. So by
identifying $u: W^{s,p}(\reals^n)\rightarrow \reals$ with
$u|_{D(\reals^n)}:D(\reals^n)\rightarrow \reals$, we can view
$W^{-s,p'}(\reals^n)$ as a subspace of $D'(\reals^n)$.
\end{remark}
\begin{remark}\lab{remfallsobolevbesov1}
\leavevmode
\begin{itemizeXALI}
\item It is a direct consequence of the contents of pages 88 and
178 of \cite{Trie83} that for $m\in \mathbb{Z}$ and $1<p<\infty$
\begin{equation*}
W^{m,p}(\reals^n)=H^m_p(\reals^n)=F^m_{p,2}(\reals^n)
\end{equation*}
\item It is a direct consequence of the contents of pages 38, 51, 90 and
178 of \cite{Trie83} that for $s\not \in \mathbb{Z}$ and
$1<p<\infty$
\begin{equation*}
W^{s,p}(\reals^n)=B^s_{p,p}(\reals^n)
\end{equation*}
\end{itemizeXALI}
\end{remark}
\begin{theorem}\lab{winter81}
For all $s\in \reals$ and $1<p<\infty$, $W^{s,p}(\reals^n)$ is
reflexive.
\end{theorem}
\begin{proof}
See the proof of Theorem \ref{thmfallreflexivity1}. Also see
\cite{36}, Section 2.6, Page 198.
\end{proof}
Note that by definition for all $s>0$ we have
$[W^{s,p}(\reals^n)]^*=W^{-s,p'}(\reals^n)$. Now since
$W^{s,p}(\reals^n)$ is reflexive, $[W^{-s,p'}(\reals^n)]^*$ is
isometrically isomorphic to $W^{s,p}(\reals^n)$ and so they can
be identified with one another. Thus for all $s\in \reals$ and
$1<p<\infty$ we may write
\begin{equation*}
[W^{s,p}(\reals^n)]^*=W^{-s,p'}(\reals^n)
\end{equation*}
Let $s\geq 0$ and $p\in (1,\infty)$. Every function $\varphi \in
C_c^\infty (\reals^n)$ defines a linear functional $L_\varphi:
W^{s,p}(\reals^n)\rightarrow \reals$ defined by
\begin{equation*}
L_\varphi (u)=\int_{\reals^n}u\varphi dx
\end{equation*}
$L_\varphi$ is continuous because by Holder's inequality
\begin{equation*}
|L_\varphi(u)|=|\int_{\reals^n}u\varphi dx|\leq
\|u\|_{L^p(\reals^n)}\|\varphi\|_{L^{p'}(\reals^n)}\leq
\|\varphi\|_{L^{p'}(\reals^n)}\|u\|_{W^{s,p}(\reals^n)}
\end{equation*}
Also the map $L: C_c^\infty (\reals^n)\rightarrow
W^{s,p}(\reals^n)$ which maps $\varphi$ to $L_\varphi$ is
injective because
\begin{align*}
L_\varphi=L_\psi\rightarrow \forall\,u\in W^{s,p}(\reals^n)\quad
\int_{\reals^n}u(\varphi-\psi)dx=0\rightarrow
\int_{\reals^n}|\varphi-\psi|^2 dx=0\rightarrow \varphi=\psi
\end{align*}
Thus we may identify $\varphi$ with $L_\varphi$ and consider
$C_c^{\infty}(\reals^n)$ as a subspace of $W^{-s,p'}(\reals^n)$.
\begin{theorem}\lab{winter82}
For all $s>0$ and $p\in (1,\infty)$, $C_c^{\infty}(\reals^n)$ is
dense in $W^{-s,p'}(\reals^n)$.
\end{theorem}
\begin{proof}
The proof given in Page 65 of \cite{32} for the density of $L^{p'}$ in the integer order Sobolev
space $W^{-m,p'}$, which is based on reflexivity of Sobolev spaces, works
equally well for establishing the density of $C_c^\infty(\reals^n)$ in $W^{-s,p'}(\reals^n)$.
\end{proof}
\begin{remark}\lab{winter83}
As a consequence of the above theorems, for all $s\in \reals$ and
$p\in (1,\infty)$, $W^{s,p}(\reals^n)$ can be considered as a
subspace of $D'(\reals^n)$.
See Theorem \ref{thmfallinjectiveadjoint1} and the discussion
thereafter for further insights. Also see Remark \ref{winter99}.
\end{remark}

Next we list several definitions pertinent to Sobolev spaces on
open subsets of $\reals^n$.
\begin{definition}\lab{winter84}
Let $\Omega$ be a nonempty open set in $\reals^n$. Let $s\in
\mathbb{R}$ and $p\in (1,\infty)$.
\begin{enumerateXALI}
\item
\begin{itemize}
\item If $s=k\in \mathbb{N}_0$,
\begin{equation*}
W^{k,p}(\Omega)=\{u\in L^p (\Omega):
\|u\|_{W^{k,p}(\Omega)}:=\sum_{|\nu|\leq
k}\|\partial^{\nu}u\|_{L^p(\Omega)}<\infty\}
\end{equation*}
\item If $s=\theta\in(0,1)$,
\begin{equation*}
W^{\theta,p}(\Omega)=\{u\in L^p (\Omega):
 |u|_{W^{\theta,p}(\Omega)}:=\big(\int\int_{\Omega\times
\Omega}\frac{|u(x)-u(y)|^p}{|x-y|^{n+\theta p}}dx
dy\big)^{\frac{1}{p}} <\infty\}
\end{equation*}
\item If $s=k+\theta,\, k\in \mathbb{N}_0,\, \theta\in(0,1)$,
\begin{equation*}
W^{s,p}(\Omega)=\{u\in
W^{k,p}(\Omega):\|u\|_{W^{s,p}(\Omega)}:=\|u\|_{W^{k,p}(\Omega)}+\sum_{|\nu|=k}
|\partial^{\nu}u|_{W^{\theta,p}(\Omega)}<\infty\}
\end{equation*}
\item If $s<0$,
\begin{equation*}
W^{s,p}(\Omega)=(W^{-s,p'}_{0}(\Omega))^{*} \quad
(\frac{1}{p}+\frac{1}{p'}=1).
\end{equation*}
where for all $e\geq 0$ and $1<q<\infty$, $W^{e,q}_0(\Omega)$ is
defined as the closure of $C_c^{\infty}(\Omega)$ in
$W^{e,q}(\Omega)$.
\end{itemize}
\item $W^{s,p}(\bar{\Omega})$ is defined as the restriction of
$W^{s,p}(\mathbb{R}^n)$ to $\Omega$. That is,
$W^{s,p}(\bar{\Omega})$ is the collection of all $u\in D'(\Omega)$
such that there is a $v\in W^{s,p}(\reals^n)$ with
$v|_{\Omega}=u$. Here $v|_{\Omega}$ should be interpreted as the
restriction of a distribution in $D'(\reals^n)$ to a distribution
in $D'(\Omega)$. $W^{s,p}(\bar{\Omega})$ is equipped with the
following norm:
\begin{equation*}
\|u\|_{W^{s,p}(\bar{\Omega})}=\inf_{v\in W^{s,p}(\mathbb{R}^n),
v|_{\Omega}=u}\|v\|_{W^{s,p}(\mathbb{R}^n)}.
\end{equation*}
\item
\begin{equation*}
\tilde{W}^{s,p}(\bar{\Omega})=\{u\in W^{s,p}(\reals^n):
\textrm{supp}\, u\subseteq \bar{\Omega}\}
\end{equation*}
$\tilde{W}^{s,p}(\bar{\Omega})$ is equipped with the norm
$\|u\|_{\tilde{W}^{s,p}(\bar{\Omega})}=\|u\|_{W^{s,p}(\reals^n)}$.
\item
\begin{equation}\lab{eqtildefall111}
\tilde{W}^{s,p}(\Omega)=\{u=v|_\Omega, v\in
\tilde{W}^{s,p}(\bar{\Omega})\}
\end{equation}
Again $v|_{\Omega}$ should be interpreted as the restriction of an
element in $D'(\reals^n)$ to $D'(\Omega)$. So
$\tilde{W}^{s,p}(\Omega)$ is a subspace of $D'(\Omega)$. This
space is equipped with the norm $\|u\|_{\tilde{W}^{s,p}}=\inf
 \|v\|_{W^{s,p}(\reals^n)}$ where the infimum is taken over all
$v$ that satisfy (\ref{eqtildefall111}). Note that two elements
$v_1$ and $v_2$ of $\tilde{W}^{s,p}(\bar{\Omega})$ restrict to
the same element in $D'(\Omega)$ if and only if $\textrm{supp}
(v_1-v_2)\subseteq \partial \Omega$. Therefore
\begin{equation*}
\tilde{W}^{s,p}(\Omega)=\frac{\tilde{W}^{s,p}(\bar{\Omega})}{\{v\in
W^{s,p}(\reals^n): \textrm{supp}\,v\subseteq \partial \Omega \}}
\end{equation*}
\item For $s\geq 0$ we define
\begin{equation*}
W^{s,p}_{00}(\Omega)=\{u\in W^{s,p}(\Omega):
\textrm{ext}^0_{\Omega,\reals^n} u\in W^{s,p}(\reals^n)\}
\end{equation*}
We equip this space with the norm
\begin{equation*}
\|u\|_{W^{s,p}_{00}(\Omega)}:=\|\textrm{ext}^0_{\Omega,\reals^n}
u\|_{W^{s,p}(\reals^n)}
\end{equation*}
Note that previously we defined the operator
$\textrm{ext}^0_{\Omega,\reals^n}$ only for distributions with
compact support and functions; this is why the values of $s$ are
restricted to be nonnegative in this definition.
\item For all $K\in \mathcal{K}(\Omega)$ we define
\begin{equation*}
W^{s,p}_K(\Omega)=\{u\in W^{s,p}(\Omega):
\textrm{supp}\,u\subseteq K\}
\end{equation*}
with $\|u\|_{W^{s,p}_K(\Omega)}:=\|u\|_{W^{s,p}(\Omega)}$.
\item
\begin{equation*}
W^{s,p}_{comp}(\Omega)=\bigcup_{K\in\mathcal{K}(\Omega)}W^{s,p}_K(\Omega)
\end{equation*}
This space is normally equipped with the inductive limit topology with respect to the family $\{W^{s,p}_K(\Omega)\}_{K\in \mathcal{K}(\Omega)}$. \textbf{However, in these notes we always consider $W^{s,p}_{comp}(\Omega)$ as a normed space equipped with the norm induced from $W^{s,p}(\Omega)$.}
\end{enumerateXALI}
\end{definition}

\begin{remark}\lab{winter85}
Each of these definitions has its advantages and disadvantages.
For example, the way we defined the spaces $W^{s,p}(\Omega)$ is
well suited for using duality arguments while proving the usual
embedding theorems for these spaces on an arbitrary open set
$\Omega$ is not trivial; on the other hand, duality arguments do
not work as well for spaces $W^{s,p}(\bar{\Omega})$ but the
embedding results for these spaces on an arbitrary open set
$\Omega$ automatically follow from the corresponding results on
$\reals^n$. Various authors adopt different definitions for
Sobolev spaces on domains based on the
applications that they are interested in. Unfortunately the
notations used in the literature for the various spaces
introduced above are not uniform. First note that it is a direct
consequence of Remark \ref{remfallsobolevbesov1} and the
definitions of $B^s_{p,q}(\Omega)$, $H^s_p(\Omega)$ and
$F^s_{p,q}(\Omega)$ in \cite{36} Page 310 and \cite{Trie2002} that
\begin{equation*}
W^{s,p}(\bar{\Omega})=
\begin{cases}
F^s_{p,2}(\Omega)=H^s_p(\Omega)\quad \textrm{if $s\in
\mathbb{Z}$}\\
B^s_{p,p}(\Omega)\quad \textrm{if $s\not\in \mathbb{Z}$}
\end{cases}
\end{equation*}
With this in mind, we have the following table which displays the
connection between the notations used in this work with the notations in a
number of well known references.
\begin{table}[H]
\centering
\begin{tabular}{| >{\centering\arraybackslash}m{1in} | >{\centering\arraybackslash}m{1in} | >{\centering\arraybackslash}m{1in} | >{\centering\arraybackslash}m{1in} |>{\centering\arraybackslash}m{1in} |@{}m{0cm}@{}}
\hline this manuscript& Triebel \cite{36} & Triebel \cite{Trie2002} &
Grisvard \cite{Gris85}
 & Bhattacharyya \cite{33}&\\
\hline $W^{s,p}(\Omega)$ &   &  & \quad $W^s_p(\Omega)$& \quad
$W^{s,p}(\Omega)$&
\\[.5cm]
 \hline $W^{s,p}(\bar{\Omega})$ & \quad $W^s_p(\Omega)$   & \quad
$W^s_p(\Omega)$ &\quad $W^s_p(\bar{\Omega})$ & \quad
$W^{s,p}(\bar{\Omega})$&
\\[.5cm]
\hline $\tilde{W}^{s,p}(\bar{\Omega})$ & \quad
$\tilde{W}^s_p(\Omega)$ & \quad $\tilde{W}^s_p(\bar{\Omega})$ && &
\\[.5cm]
\hline $\tilde{W}^{s,p}(\Omega)$ & & \quad
$\tilde{W}^s_p(\Omega)$ && &
\\[.5cm]
\hline $W^{s,p}_{00}(\Omega)$ & & &$\tilde{W}^s_p(\Omega)$&
$W^{s,p}_{00}(\Omega)$&
\\[.5cm]
\hline
\end{tabular}
\end{table}
\end{remark}

\clearpage

\begin{remark}\lab{remfallequivnorm120}
\leavevmode
\begin{itemizeXALI}
\item Note that
\begin{align*}
\|u\|_{W^{k,p}(\Omega)}&+\sum_{|\nu|=k}|\partial^\nu u|_{W^{\theta,p}(\Omega)}\leq \|u\|_{W^{k,p}(\Omega)}+\sum_{|\nu|=k}\|\partial^\nu u\|_{W^{\theta,p}(\Omega)}\\
&=\|u\|_{W^{k,p}(\Omega)}+\sum_{|\nu|=k}\bigg(\|\partial^\nu u\|_{L^p(\Omega)}+|\partial^\nu u|_{W^{\theta,p}(\Omega)}\bigg)\\
&\preceq \|u\|_{W^{k,p}(\Omega)}+\sum_{|\nu|=k}|\partial^\nu u|_{W^{\theta,p}(\Omega)}\qquad (\textrm{since $\sum_{|\nu|=k}\|\partial^\nu u\|_{L^p(\Omega)}\leq \|u\|_{W^{k,p}(\Omega)}$})
\end{align*}
Therefore the following is an equivalent norm on $W^{s,p}(\Omega)$
\begin{equation*}
\|u\|_{W^{s,p}(\Omega)}:= \|u\|_{W^{k,p}(\Omega)}+\sum_{|\alpha|=k}\|\partial^\alpha u\|_{W^{\theta,p}(\Omega)}
\end{equation*}
\item For $p\in (1,\infty)$ and $a,b>0$ we have
$(a^p+b^p)^{\frac{1}{p}}\simeq a+b$; indeed,
\begin{equation*}
a^p+b^p\leq (a+b)^p\leq (2\max \{a,b\})^p\leq 2^p(a^p+b^p)
\end{equation*}
More generally, if $a_1,\cdots, a_m$ are nonnegative numbers, then $(a_1^p+\cdots+a_m^p)^{\frac{1}{p}}\simeq a_1+\cdots+a_m$. Therefore for any nonempty open set $\Omega$ in $\reals^n$, $s>0$,
the following expressions are both equivalent to the original norm on
$W^{s,p}(\Omega)$
\begin{align*}
&\|u\|_{W^{s,p}(\Omega)}:=\big[\|u\|^p_{W^{k,p}(\Omega)}+\sum_{|\nu|=k}
|\partial^{\nu}u|^p_{W^{\theta,p}(\Omega)}\big]^{\frac{1}{p}}\\
&\|u\|_{W^{s,p}(\Omega)}:=\big[\|u\|^p_{W^{k,p}(\Omega)}+\sum_{|\nu|=k}
\|\partial^{\nu}u\|^p_{W^{\theta,p}(\Omega)}\big]^{\frac{1}{p}}
\end{align*}
where $s=k+\theta,\, k\in \mathbb{N}_0,\, \theta\in(0,1)$.
\end{itemizeXALI}
\end{remark}
\subsection{Properties of Sobolev Spaces on the Whole Space $\reals^n$}

\begin{theorem}[Embedding Theorem I, \cite{36}, Section 2.8.1]
\lab{winter86} Suppose $1< p\leq q<\infty$ and $-\infty< t\leq
s<\infty$ satisfy $s-\frac{n}{p}\geq t-\frac{n}{q}$. Then
$W^{s,p}(\reals^n)\hookrightarrow W^{t,q}(\reals^n)$. In
particular, $W^{s,p}(\reals^n)\hookrightarrow W^{t,p}(\reals^n)$.
\end{theorem}

\begin{theorem}[Multiplication by smooth functions, \cite{Trie92}, Page
203] \lab{winter87} Let $s\in \reals$, $1<p<\infty$, and
$\varphi\in BC^\infty(\reals^n)$. Then the linear map
\begin{equation*}
m_\varphi: W^{s,p}(\reals^n)\rightarrow W^{s,p}(\reals^n),\qquad
u\mapsto \varphi u
\end{equation*}
is well-defined and bounded.
\end{theorem}


A detailed study of the following multiplication theorems can be
found in \cite{holstbehzadan2015b}.
\begin{theorem}\lab{thm4.6}
Let $s_i, s$ and $1 \leq p, p_i < \infty$ ($i=1,2$) be real
numbers satisfying
\begin{enumerate}[(i)]
\item  $s_i \geq s\geq 0$ 
\item  $s\in \mathbb{N}_0$,
\item  $s_i-s\geq n(\dfrac{1}{p_i}-\dfrac{1}{p})$,
\item  $s_1+s_2-s>n(\dfrac{1}{p_1}+\dfrac{1}{p_2}-\dfrac{1}{p})\geq
0$.
\end{enumerate}
where the strictness of the inequalities in items (iii) and (iv)
can be interchanged.\\ If $u\in W^{s_1,p_1}(\mathbb{R}^n)$
and $v\in W^{s_2,p_2}(\mathbb{R}^n)$, then $uv \in
W^{s,p}(\mathbb{R}^n)$ and moreover the pointwise multiplication
of functions is a continuous bilinear map
\begin{equation*}
W^{s_1,p_1}(\mathbb{R}^n)\times
W^{s_2,p_2}(\mathbb{R}^n)\rightarrow W^{s,p}(\mathbb{R}^n).
\end{equation*}
\end{theorem}
\begin{theorem}[Multiplication theorem for Sobolev spaces on the whole space, nonnegative exponents]\lab{thm4.1}
Assume $s_i,s$ and $1 \leq p_i \leq p< \infty$ ($i=1,2$) are real
numbers satisfying
\begin{enumerate}[(i)]
\item  $s_i \geq s$ 
\item  $s\geq 0$,
\item  $s_i-s\geq n(\dfrac{1}{p_i}-\dfrac{1}{p})$,
\item  $s_1+s_2-s>n(\dfrac{1}{p_1}+\dfrac{1}{p_2}-\dfrac{1}{p})$.
\end{enumerate}
If $u\in W^{s_1,p_1}(\mathbb{R}^n)$ and $v\in
W^{s_2,p_2}(\mathbb{R}^n)$, then $uv \in W^{s,p}(\mathbb{R}^n)$
and moreover the pointwise multiplication of functions is a
continuous bilinear map
\begin{equation*}
W^{s_1,p_1}(\mathbb{R}^n)\times
W^{s_2,p_2}(\mathbb{R}^n)\rightarrow W^{s,p}(\mathbb{R}^n).
\end{equation*}
\end{theorem}
\begin{theorem}[Multiplication theorem for Sobolev spaces on the whole space, negative exponents I]\lab{thm4.3}
Assume $s_i, s$ and $1 < p_i \leq p < \infty$ ($i=1,2$) are real
numbers satisfying
\begin{enumerate}[(i)]
\item  $s_i \geq s$, 
\item  $\min\{s_1, s_2\}<0$,
\item  $s_i-s\geq n(\dfrac{1}{p_i}-\dfrac{1}{p})$,
\item  $s_1+s_2-s>n(\dfrac{1}{p_1}+\dfrac{1}{p_2}-\dfrac{1}{p})$.
\item $s_1+s_2\geq n(\dfrac{1}{p_1}+\dfrac{1}{p_2}-1)\geq 0$.
\end{enumerate}
Then the pointwise multiplication of smooth functions extends uniquely
to a continuous bilinear map
\begin{equation*}
W^{s_1,p_1}(\mathbb{R}^n)\times
W^{s_2,p_2}(\mathbb{R}^n)\rightarrow W^{s,p}(\mathbb{R}^n).
\end{equation*}
\end{theorem}
\begin{theorem}[Multiplication theorem for Sobolev spaces on the whole space, negative exponents II]\lab{thm4.5}
Assume $s_i,s$ and $1 < p,  p_i < \infty$ ($i=1,2$) are real
numbers satisfying
\begin{enumerate}[(i)]
\item  $s_i \geq s$, 
\item  $\min\{s_1, s_2\}\geq0$ and $s<0$,
\item  $s_i-s\geq n(\dfrac{1}{p_i}-\dfrac{1}{p})$,
\item  $s_1+s_2-s>n(\dfrac{1}{p_1}+\dfrac{1}{p_2}-\dfrac{1}{p})\geq
0$.
\item $s_1+s_2> n(\dfrac{1}{p_1}+\dfrac{1}{p_2}-1)$. \quad
(\textrm{the inequality is strict})
\end{enumerate}
Then the pointwise multiplication of smooth functions extends uniquely
to a continuous bilinear map
\begin{equation*}
W^{s_1,p_1}(\reals^n)\times W^{s_2,p_2}(\reals^n)\rightarrow
W^{s,p}(\reals^n).
\end{equation*}
\end{theorem}

\begin{remark}\lab{multinterpremark}
Let's discuss further how we should interpret multiplication in the case where negative exponents are involved. Suppose for instance $s_1<0$ ($s_2$ may be positive or negative). A moment's thought shows that the relation
\begin{equation*}
W^{s_1,p_1}(\reals^n)\times W^{s_2,p_2}(\reals^n)\hookrightarrow
W^{s,p}(\reals^n).
\end{equation*}
in the above theorems can be interpreted as follows: for all $u\in W^{s_1,p_1}(\reals^n)$ and $v\in W^{s_2,p_2}(\reals^n)$, if $\{\varphi_i\}$ in $C^\infty(\reals^n)\cap W^{s_1,p_1}(\reals^n)$ is any sequence such that $\varphi_i\rightarrow u$  in $W^{s_1,p_1}(\reals^n)$, then
\begin{enumerate}
\item for all $i$, $\varphi_i v\in W^{s,p}(\reals^n)$ (multiplication of a smooth function and a distribution),
\item  $\varphi_i v$ converges to some element $g$ in $W^{s,p}(\reals^n)$ as $i\rightarrow \infty$,
\item $\|g\|_{W^{s,p}(\reals^n)}\preceq \|u\|_{W^{s_1,p_1}(\reals^n)}\|v\|_{W^{s_2,p_2}(\reals^n)}$ where the implicit constant does not depend on $u$ and $v$,
\item $g\in W^{s,p}(\reals^n)$ is independent of the sequence $\{\varphi_i\}$ and can be regarded as the product of $u$ and $v$.
\end{enumerate}
In particular, $\varphi_i v\rightarrow uv$ in $D'(\reals^n)$ and for all $\psi\in C_c^\infty(\reals^n)$
\begin{equation*}
\langle uv,\psi\rangle_{D'(\reals^n)\times D(\reals^n)}=\lim_{i\rightarrow \infty} \langle \varphi_i v,\psi\rangle_{D'(\reals^n)\times D(\reals^n)}=\langle v,\varphi_i\psi\rangle_{D'(\reals^n)\times D(\reals^n)}\,.
\end{equation*}
\end{remark}

\subsection{Properties of Sobolev Spaces on Smooth Bounded Domains}

In this section we assume that $\Omega$ is an open bounded set in
$\reals^n$ with smooth boundary unless a weaker assumption is
stated. First we list some facts that can be useful in understanding the relationship between various definitions of Sobolev spaces on domains.
\begin{itemizeXALI}
\item (\cite{33}, Page 584)[Theorem 8.10.13 and its proof] Suppose $s> 0$ and $1< p<\infty$.
Then $W^{s,p}(\Omega)=W^{s,p}(\bar{\Omega})$ in the sense of
equivalent normed spaces.
\item For $s>0$ and $1<p<\infty$, $W_{00}^{s,p}(\Omega)$ is isomorphic to
$\tilde{W}^{s,p}(\bar{\Omega})$. Moreover
$[W_{00}^{s,p}(\Omega)]^*=[\tilde{W}^{s,p}(\bar{\Omega})]^*=W^{-s,p'}(\bar{\Omega})$.
\item Let $s\geq 0$ and $1<p<\infty$. Then for $s\neq
\frac{1}{p}, 1+\frac{1}{p}, 2+\frac{1}{p},\cdots$ (that is, when
the fractional part of $s$ is not equal to $\frac{1}{p}$) we have
\begin{enumerate}
\item $W_{00}^{s,p}(\Omega)=W^{s,p}_0(\Omega)$ and so (in the sense of equivalent normed spaces)
\begin{equation*}
[W^{s,p}_0(\bar{\Omega})]^*=[W^{s,p}_0(\Omega)]^*=[W_{00}^{s,p}(\Omega)]^*=W^{-s,p'}(\bar{\Omega})
\end{equation*}
where $W^{s,p}_0(\bar{\Omega})$ is the closure of $C_c^\infty
(\Omega)$ in $W^{s,p}(\bar{\Omega})$. This claim is a direct
consequence of Theorem 1 Page 317 and Theorem 4.8.2 Page 332 of
\cite{36}.
\item
\begin{equation*}
\textrm{ext}_{\Omega,\mathbb{R}^n}^0: \big(C_c^\infty(\Omega),
\|.\|_{s,p}\big)\rightarrow
 W^{s,p}(\mathbb{R}^n)
\end{equation*}
is a well-defined bounded linear
 operator.
\item \begin{equation*}
\textrm{res}_{\mathbb{R}^n,\Omega}:W^{-s,p'}(\mathbb{R}^n)\rightarrow
W^{-s,p'}(\Omega)\qquad u\mapsto u|_\Omega
\end{equation*}
is a well-defined
 bounded linear operator.
\item $W^{-s,p'}(\Omega)=W^{-s,p'}(\bar{\Omega})$.
\end{enumerate}
\item As a consequence of the above items
 $W^{s,p}(\Omega)=W^{s,p}(\bar{\Omega})$ in the sense of
equivalent normed spaces for $1<p<\infty$, $s\in \reals$ with
$s\neq \frac{1}{p}-1, \frac{1}{p}-2, \frac{1}{p}-3,\cdots$. (Note
that if we want the definitions agree for $s<0$, it is enough to
assume that $-s\neq \frac{1}{p'}, 1+\frac{1}{p'},
2+\frac{1}{p'},\cdots$.)
\item (\cite{Trie2002}, Pages 481 and 494) For $s>\frac{1}{p}-1$,
$\tilde{W}^{s,p}(\bar{\Omega})=\tilde{W}^{s,p}(\Omega)$. That is
for $s>\frac{1}{p}-1$
\begin{equation*}
\{v\in W^{s,p}(\reals^n): \textrm{supp}\,v\subseteq \partial
\Omega \}=\{0\}
\end{equation*}
\end{itemizeXALI}
Next we recall some facts about extension operators and embedding properties of Sobolev spaces. The existence of extension operator can be helpful in transferring known results for Sobolev spaces defined on $\reals^n$ to Sobolev spaces defined on bounded domains.
\begin{theorem}[Extension Property I]\lab{thm3.1}(\cite{33}, Page
584) Let $\Omega \subset \mathbb{R}^n$ be a bounded open set with
Lipschitz continuous boundary. Then for all $s>0$ and for $1\leq p
< \infty$, there exists a continuous linear extension operator $P:
W^{s,p}(\Omega)\hookrightarrow W^{s,p}(\mathbb{R}^n)$ such that
$(Pu)|_{\Omega}=u$ and $\parallel P
u\parallel_{W^{s,p}(\mathbb{R}^n)}\leq C\parallel u
\parallel_{W^{s,p}(\Omega)}$ for some constant $C$ that may
depend on $s$, $p$, and $\Omega$ but is independent of $u$.
\end{theorem}
The next theorem states that the claim of Theorem \ref{thm3.1}
holds for all values of $s$ (positive and negative) if we replace
$W^{s,p}(\Omega)$ with $W^{s,p}(\bar{\Omega})$.
\begin{theorem}[Extension Property II]\lab{winter89}(\cite{Trie2002}, Page 487, \cite{Trie83},Page 201)
Let $\Omega \subset \mathbb{R}^n$ be a bounded open set with
Lipschitz continuous boundary, $p\in(1,\infty)$ and  $s\in
\reals$. Let $R: W^{s,p}(\reals^n)\rightarrow
W^{s,p}(\bar{\Omega})$ be the restriction operator
$(R(u)=u|_{\Omega})$. Then there exists a continuous linear
operator $S: W^{s,p}(\bar{\Omega})\rightarrow W^{s,p}(\reals^n)$
such that $R\circ S=Id$.
\end{theorem}
\begin{corollary}\lab{winter90}
As it was pointed out earlier for $s\neq
\frac{1}{p}-1,\frac{1}{p}-2,\cdots$
$W^{s,p}(\Omega)=W^{s,p}(\bar{\Omega})$. Therefore it follows
from the above theorems that if $s\neq
\frac{1}{p}-1,\frac{1}{p}-2,\cdots$, then there exists a
continuous linear extension operator $P:
W^{s,p}(\Omega)\hookrightarrow W^{s,p}(\mathbb{R}^n)$ such that
$(Pu)|_{\Omega}=u$ and $\parallel P
u\parallel_{W^{s,p}(\mathbb{R}^n)}\leq C\parallel u
\parallel_{W^{s,p}(\Omega)}$ for some constant $C$ that may
depend on $s$, $p$, and $\Omega$ but is independent of $u$.
\end{corollary}
\begin{corollary}\lab{winter91}
One can easily show that the results of Sobolev multiplication
theorems in the previous section (Theorems \ref{thm4.6},
\ref{thm4.1}, \ref{thm4.3}, and \ref{thm4.5}) hold also for
Sobolev spaces on any Lipschitz domain as long as all the Sobolev spaces involved satisfy $W^{e,q}(\Omega)=W^{e,q}(\bar{\Omega})$ (and so, in particular, existence of an extension operator is guaranteed). Indeed, if $P_1:
W^{s_1,p_1}(\Omega)\rightarrow W^{s_1,p_1}(\reals^n)$ and $P_2:
W^{s_2,p_2}(\Omega)\rightarrow W^{s_2,p_2}(\reals^n)$ are
extension operators, then $(P_1 u)(P_2 v)|_{\Omega}=uv$ and
therefore
\begin{align*}
\parallel uv\parallel_{W^{s,p}(\Omega)}=\parallel uv\parallel_{W^{s,p}(\bar{\Omega})}\leq \parallel (P_1 u)(P_2
v)\parallel_{W^{s,p}(\reals^n)} &\preceq \parallel P_1
u\parallel_{W^{s_1,p_1}(\reals^n)}\parallel P_2
v\parallel_{W^{s_2,p_2}(\reals^n)}\\
&\preceq \parallel u\parallel_{W^{s_1,p_1}(\Omega)}\parallel
v\parallel_{W^{s_2,p_2}(\Omega)}\,.
\end{align*}
\end{corollary}
\begin{remark}
In the above Corollary, we presumed that $(P_1 u)(P_2 v)|_{\Omega}=uv$. Clearly if $s_1$ and $s_2$ are both nonnegative, the equality holds. But what if at least one of the exponents, say $s_1$, is negative? In order to prove this equality, we may proceed as follows: let $\{\varphi_i\}$ be a sequence in $C^\infty(\reals^n)\cap W^{s_1,p_1}(\reals^n)$ such that $\varphi_i\rightarrow P_1 u$ in $W^{s_1,p_1}(\reals^n)$. By assumption $W^{s_1,p_1}(\Omega)=W^{s_1,p_1}(\bar{\Omega})$, therefore the restriction operator is continuous and $\{\varphi_i|_{\Omega}\}$ is a sequence in $C^\infty(\Omega)\cap W^{s_1,p_1}(\Omega)$ that converges to $u$ in $W^{s_1,p_1}(\Omega)$. For all $\psi\in C_c^\infty(\Omega)$ we have
\begin{align*}
\langle [(P_1 u)(P_2 v)]|_{\Omega},\psi\rangle_{D'(\Omega)\times D(\Omega)}&=\langle (P_1 u)(P_2 v),\textrm{ext}^0_{\Omega,\reals^n}\,\psi\rangle_{D'(\reals^n)\times D(\reals^n)}\\
&\stackrel{\textrm{Remark \ref{multinterpremark}}}{=}\lim_{i\rightarrow\infty}
\langle \varphi_i(P_2 v),\textrm{ext}^0_{\Omega,\reals^n}\,\psi\rangle_{D'(\reals^n)\times D(\reals^n)}\\
&= \lim_{i\rightarrow\infty}
\langle (P_2 v),\varphi_i\textrm{ext}^0_{\Omega,\reals^n}\,\psi\rangle_{D'(\reals^n)\times D(\reals^n)}\\
&= \lim_{i\rightarrow\infty}
\langle (P_2 v),\textrm{ext}^0_{\Omega,\reals^n}\,(\varphi_i|_\Omega\psi)\rangle_{D'(\reals^n)\times D(\reals^n)}\\
&= \lim_{i\rightarrow\infty}
\langle (P_2 v)|_\Omega,\varphi_i|_\Omega\psi\rangle_{D'(\Omega)\times D(\Omega)}\\
&= \lim_{i\rightarrow\infty}
\langle \varphi_i|_\Omega v,\psi\rangle_{D'(\Omega)\times D(\Omega)}\\
&=\langle u v, \psi\rangle_{D'(\Omega)\times D(\Omega)}\\
\end{align*}
\end{remark}
\begin{theorem}[Embedding Theorem II]\cite{Gris85}\lab{thm3.3}
Let $\Omega$ be a nonempty bounded open subset of $\mathbb{R}^n$ with
Lipschitz continuous boundary or $\Omega=\mathbb{R}^n$. If
$sp>n$, then $W^{s,p}(\Omega)\hookrightarrow
L^{\infty}(\Omega)\cap C^{0}(\Omega)$ and $W^{s,p}(\Omega)$ is a
Banach algebra.
\end{theorem}
\begin{theorem}[Embedding Theorem III]\cite{holstbehzadan2015b}\lab{thm3.4} Let $\Omega$ be a nonempty
bounded open subset of $\mathbb{R}^n$ with Lipschitz continuous
boundary. Suppose $1\leq p, q<\infty$ ($p$ does NOT need to be
less than or equal to $q$) and $0\leq t\leq s$ satisfy
$s-\frac{n}{p}\geq t-\frac{n}{q}$. Then
$W^{s,p}(\Omega)\hookrightarrow W^{t,q}(\Omega)$. In particular,
$W^{s,p}(\Omega)\hookrightarrow W^{t,p}(\Omega)$.
\end{theorem}
\begin{theorem}\lab{sobislip31}
Let $\Omega$ be a nonempty bounded open subset of $\mathbb{R}^n$ with
Lipschitz continuous boundary. Then $u: \Omega\rightarrow \reals$ is Lipschitz continuous if and only if $u\in W^{1,\infty}(\Omega)$. In particular, every function in $BC^1(\Omega)$ is Lipschitz continuous.
\end{theorem}
\begin{proof}
The above theorem is proved in Chapter 5 of \cite{Evans2010} for open sets with $C^1$ boundary. The exact same proof works for open sets with Lipschitz continuous boundary.
\end{proof}
The following theorem (and its corollary) will play an important role in our study of Sobolev spaces on manifolds.
\begin{theorem}[Multiplication by smooth
functions]\lab{thmfallmultsmooth20} Let $\Omega$ be a nonempty
bounded open set in $\reals^n$ with Lipschitz continuous boundary.
\begin{enumerateX}
\item Let $k\in \mathbb{N}_0$ and $1<p<\infty$. If $\varphi\in
BC^k(\Omega)$, then the linear map $W^{k,p}(\Omega)\rightarrow
W^{k,p}(\Omega)$ defined by $u\mapsto \varphi u$ is well-defined
and bounded.
\item Let $s\in (0,\infty)$ and $1<p<\infty$. If $\varphi\in
BC^{\floor{s},1}(\Omega)$ (all partial derivatives of $\varphi$
up to and including order $\floor{s}$ exist and are bounded and
Lipschitz continuous), then the linear map
$W^{s,p}(\Omega)\rightarrow W^{s,p}(\Omega)$ defined by $u\mapsto
\varphi u$ is well-defined and bounded.
\item Let $s\in (-\infty,0)$ and $1<p<\infty$. If $\varphi\in
BC^{\infty,1}(\Omega)$, then the linear map
$W^{s,p}(\Omega)\rightarrow W^{s,p}(\Omega)$ defined by $u\mapsto
\varphi u$ is well-defined and bounded.\\
\textbf{Note:} According to Theorem \ref{sobislip31}, when $\Omega$ is an open bounded set with Lipschitz continuous boundary, every function in $BC^1(\Omega)$ is Lipschitz continuous. As a consequence, $BC^{\infty,1}(\Omega)=BC^\infty(\Omega)$. Of course, as it was discussed after Theorem \ref{thmfallconvexapostol1}, for a general bounded open set $\Omega$ whose boundary is not Lipschitz, functions in $BC^\infty (\Omega)$ are not necessarily Lipschitz.
\end{enumerateX}
\end{theorem}
\begin{proof}
\leavevmode
\begin{itemizeX}
\item \textbf{Step 1:} $s=k\in \mathbb{N}_0$.
The claim is proved in (\cite{13}, Page 995).

\item \textbf{Step 2:} $0<s<1$.
The proof in Page 194 of \cite{Demengel2012}, with obvious
modifications, shows the validity of the claim for the case where
$s\in (0,1)$.

\item \textbf{Step 3:} $1<s\not \in \mathbb{N}$.
In this case we can proceed as follows: Let $k=\floor{s}$,
$\theta=s-k$.
\begin{align*}
\|\varphi u\|_{s,p}&=\|\varphi u\|_{k,p}+\sum_{|\nu|=k}\|\partial^{\nu}(\varphi u)\|_{\theta,p}\\
&\preceq \|\varphi u\|_{k,p}+\sum_{|\nu|=k}\sum_{\beta\leq
\nu}\|\partial^{\nu-\beta}\varphi  \partial^{\beta}u\|_{\theta,p}\\
&\preceq \|u\|_{k,p}+\sum_{|\nu|=k}\sum_{\beta\leq \nu}\|
\partial^{\beta}u\|_{\theta,p} \quad ({\fontsize{9}{9}{\textrm{by Step1 and
Step2; the implicit constant may depend on $\varphi$}}})\\
&=\|u\|_{s,p}+\sum_{|\nu|=k}\sum_{\beta<\nu}\|
\partial^{\beta}u\|_{\theta,p}\\
&\preceq \|u\|_{s,p}+\sum_{|\nu|=k}\sum_{\beta<\nu}\|u\|_{\theta+|\beta|,p}\quad (\partial^{\beta}:W^{\theta+|\beta|,p}(\Omega)\rightarrow W^{\theta,p}(\Omega) \textrm{is continuous} )\\
&\preceq \|u\|_{s,p}+\sum_{|\nu|=k}\sum_{\beta<\nu}\|u\|_{s,p}
\quad (\theta+|\beta|<s \Rightarrow
W^{s,p}(\Omega)\hookrightarrow W^{\theta+|\beta|,p}(\Omega))\\
&\preceq \|u\|_{s,p}.
\end{align*}
Note that the embedding $W^{s,p}(\Omega)\hookrightarrow
W^{\theta+|\beta|,p}(\Omega)$ is valid due to the extra assumption
that $\Omega$ is bounded with Lipschitz continuous boundary. (See Theorem \ref{thm3.2} and Remark \ref{winter96}).
\item \textbf{Step 4:} $s<0$.
For this case we use a duality argument. Note that since $\varphi\in C^\infty(\Omega)$,  $\varphi u$ is defined as an element of $D'(\Omega)$. Also recall that $W^{s,p}(\Omega)$ is isometrically isomorphic to $[C_c^\infty(\Omega),\|.\|_{-s,p'}]^*$ (see the discussion after Remark \ref{winter34}). So, in order to prove the claim, it is enough to show that multiplication by $\varphi$ is a well-defined continuous operator from $W^{s,p}(\Omega)$ to $A=[C_c^\infty(\Omega),\|.\|_{-s,p'}]^*$. We have
\begin{align*}
\|\varphi u\|_{A}&=\sup_{v\in
C_c^\infty\setminus\{0\}}\frac{|\langle \varphi
u,v\rangle_{D'(\Omega)\times D(\Omega)}|}{\|v\|_{-s,p'}} =\sup_{v\in
C_c^\infty\setminus\{0\}}\frac{|\langle u,\varphi
v\rangle_{D'(\Omega)\times D(\Omega)}|}{\|v\|_{-s,p'}}
\\
&\stackrel{\textrm{Remark \ref{winter99}}}{=}\sup_{v\in
C_c^\infty\setminus\{0\}}\frac{|\langle u,\varphi
v\rangle_{W^{s,p}(\Omega)\times W^{-s,p'}_0(\Omega)}|}{\|v\|_{-s,p'}}
\\
&\leq \sup_{v\in
C_c^\infty\setminus\{0\}}\frac{\|u\|_{s,p}\|\varphi
v\|_{-s,p'}}{\|v\|_{-s,p'}} \preceq \sup_{v\in
C_c^\infty\setminus\{0\}}\frac{\|u\|_{s,p}\|v\|_{-s,p'}}{\|v\|_{-s,p'}}=\|u\|_{s,p}.
\end{align*}
\end{itemizeX}
\end{proof}
\begin{corollary}\lab{corollarywinter92a}
Let $\Omega$ be a nonempty bounded open set in $\reals^n$ with
Lipschitz continuous boundary. Let $K\in \mathcal{K}(\Omega)$.
Suppose $s\in \reals$ and $p\in (1,\infty)$. If $\varphi\in
C^\infty(\Omega)$, then the linear map
$W^{s,p}_K(\Omega)\rightarrow W^{s,p}_K(\Omega)$ defined by
$u\mapsto \varphi u$ is well-defined and bounded.
\end{corollary}
\begin{proof}
Let $U$ be an open set such that $K\subset U\subseteq
\bar{U}\subseteq \Omega$. Let $\psi\in C_c^\infty (\Omega)$ be
such that $\psi=1$ on $K$ and $\psi=0$ outside $U$. Clearly
$\psi\varphi\in C_c^{\infty}(\Omega)$ and thus $\psi\varphi\in
BC^{\infty,1}(\Omega)$ (see the paragraph above Theorem \ref{thmfalllipbound1}). So it follows from Theorem
\ref{thmfallmultsmooth20} that $\|\psi\varphi u\|_{s,p}\preceq
\|u\|_{s,p}$ where the implicit constant in particular may depend
on $\varphi$ and $\psi$. Now the claim follows from the obvious
observation that for all $u\in W^{s,p}_K(\Omega)$, we have $\psi
\varphi u=\varphi u$.
\end{proof}
\begin{theorem}\lab{lemwinter2003}
Let $\Omega=\reals^n$ or $\Omega$ be a nonempty bounded open set in $\reals^n$ with Lipschitz continuous boundary. Let $K\subseteq \Omega$ be compact, $s\in \reals$ and $p\in (1,\infty)$. Then
\begin{enumerateX}
\item $W^{s,p}_K(\Omega)\subseteq W^{s,p}_0(\Omega)$. That is, every element of $W^{s,p}_K(\Omega)$ is a limit of a sequence in $C_c^\infty (\Omega)$;
\item  if  $K\subseteq V\subseteq K'\subseteq \Omega$ where
and $K'$ is compact and $V$ is open, then for every $u\in W^{s,p}_K(\Omega)$, there exists a sequence in
$C^\infty_{K'}(\Omega)$ that converges to $u$ in $W^{s,p}(\Omega)$.
\end{enumerateX}
\end{theorem}
\begin{proof}
\begin{enumerateX}
\item Let $u\in W^{s,p}_K(\Omega)$. By Theorem \ref{winter93} and Theorem \ref{winter94}, there exists a sequence $\{\varphi_i\}$ in $C^\infty(\Omega)$ such that $\varphi_i\rightarrow u$ in $W^{s,p}(\Omega)$. Let $\psi\in C_c^\infty(\Omega)$ be such that $\psi=1$ on $K$. Since $C_c^\infty(\Omega)\subseteq BC^{\infty,1}(\Omega)$, it follows from Theorem \ref{winter87} and Theorem \ref{thmfallmultsmooth20} that $\psi \varphi_i\rightarrow \psi u$ in $W^{s,p}(\Omega)$. This proves the claim because $\psi\varphi_i\in C_c^\infty(\Omega)$ and $\psi u=u$.
\item In the above argument, choose $\psi\in C_c^\infty(\Omega)$ such that $\psi=1$ on $K$ and $\psi=0$ outside $V$.
 \end{enumerateX}
\end{proof}
\begin{theorem}[(\cite{Trie2002}, Page 496), (\cite{36}, Pages 317, 330, and
332)]\lab{winter92} Let $\Omega$ be a bounded Lipschitz domain in
$\reals^n$. Suppose $1<p<\infty$, $0\leq s<\frac{1}{p}$. Then
$C_c^{\infty}(\Omega)$ is dense in $W^{s,p}(\Omega)$ (thus
$W^{s,p}(\Omega)=W^{s,p}_0(\Omega)$).
\end{theorem}
\subsection{Properties Of Sobolev Spaces on General Domains}
In this section $\Omega$ and $\Omega'$ are arbitrary nonempty
open sets in $\reals^n$. We begin with some facts about the relationship between various Sobolev spaces defined on bounded domains.
\begin{itemizeXALI}
\item Suppose $s\geq 0$ and $\Omega'\subseteq \Omega$. Then for all $u\in
W^{s,p}(\Omega)$, we have $\textrm{res}_{\Omega,\Omega'} u\in
W^{s,p}(\Omega')$. Moreover $\|\textrm{res}_{\Omega,\Omega'}
u\|_{W^{s,p}(\Omega')}\leq \|u\|_{W^{s,p}(\Omega)}$. Indeed, if we let $s=k+\theta$
\begin{align*}
\|u\|_{W^{s,p}(\Omega')}&=\|u\|_{W^{k,p}(\Omega')}+\sum_{|\nu|=k}\big(\int\int_{\Omega'\times
\Omega' }\frac{|\partial^\nu
u(x)-\partial^\nu u(y)|^p}{|x-y|^{n+\theta p}}\,dxdy\big)^{\frac{1}{p}}\\
&=\sum_{|\alpha|\leq
k}\|\partial^\alpha u\|_{L^p(\Omega')}+\sum_{|\nu|=k}\big(\int\int_{\Omega'\times
\Omega' }\frac{|\partial^\nu u(x)-\partial^\nu
u(y)|^p}{|x-y|^{n+\theta
p}}\,dxdy\big)^{\frac{1}{p}}\\
&\leq \sum_{|\alpha|\leq
k}\|\partial^\alpha u\|_{L^p(\Omega)}+\sum_{|\nu|=k}\big(\int\int_{\Omega\times
\Omega}\frac{|\partial^\nu u(x)-\partial^\nu
u(y)|^p}{|x-y|^{n+\theta
p}}\,dxdy\big)^{\frac{1}{p}}=\|u\|_{W^{s,p}(\Omega)}
\end{align*}
So $\textrm{res}_{\Omega,\Omega'}: W^{s,p}(\Omega)\rightarrow
W^{s,p}(\Omega')$ is a continuous linear map. Also as a
consequence for every real number $s\geq 0$
\begin{equation*}
W^{s,p}(\bar{\Omega})\hookrightarrow W^{s,p}(\Omega)
\end{equation*}
Indeed, if $u\in W^{s,p}(\bar{\Omega})$, then there exists $v\in
W^{s,p}(\reals^n)$ such that $\textrm{res}_{\reals^n,\Omega}v=u$
and thus $u\in W^{s,p}(\Omega)$. Moreover, for every such $v$,
$\|u\|_{W^{s,p}(\Omega)}=\|\textrm{res}_{\reals^n,\Omega}v\|_{W^{s,p}(\Omega)}\leq
\|v\|_{W^{s,p}(\reals^n)} $. This implies that
\begin{equation*}
\|u\|_{W^{s,p}(\Omega)}\leq \inf_{v\in W^{s,p}(\reals^n),
v|_\Omega=u}\|v\|_{W^{s,p}(\reals^n)}=\|u\|_{W^{s,p}(\bar{\Omega})}
\end{equation*}
\item Clearly for all $s\geq 0$
\begin{equation*}
W^{s,p}_{00}(\Omega)\hookrightarrow W^{s,p}(\bar{\Omega})
\end{equation*}
\item (\cite{Gris85}, Page 18) For every integer $m>0$
\begin{equation*}
W^{m,p}_0(\Omega)\subseteq W^{m,p}_{00}(\Omega)\subseteq
W^{m,p}(\bar{\Omega})\subseteq W^{m,p}(\Omega)
\end{equation*}
\item Suppose $s\geq 0$. Clearly the restriction map $\textrm{res}_{\reals^n, \Omega}: W^{s,p}(\reals^n)\rightarrow
W^{s,p}(\bar{\Omega})$ is a continuous linear map. This combined
with the fact that $C_c^\infty(\reals^n)$ is dense in
$W^{s,p}(\reals^n)$ implies that $C_c^\infty
(\bar{\Omega}):=\textrm{res}_{\reals^n,
\Omega}(C_c^\infty(\reals^n))$ is dense in
$W^{s,p}(\bar{\Omega})$ for all $s\geq 0$.
\item $\tilde{W}^{s,p}(\bar{\Omega})$ is a closed subspace of
$W^{s,p}(\reals^n)$. Closed subspaces of reflexive spaces are
reflexive, hence $\tilde{W}^{s,p}(\bar{\Omega})$ is a reflexive
space.
\end{itemizeXALI}
\begin{theorem}\lab{thmfallreflexivity1}
Let $\Omega$ be a nonempty open set in $\reals^n$ and
$1<p<\infty$.
\begin{enumerateXALI}
\item For all $s\geq 0$, $W^{s,p}(\Omega)$ is reflexive.
\item For all $s\geq 0$, $W^{s,p}_0(\Omega)$ is reflexive.
\item For all $s<0$, $W^{s,p}(\Omega)$ is reflexive.
\end{enumerateXALI}
\end{theorem}
\begin{proof}
\leavevmode
\begin{enumerateXALI}
\item The proof for $s\in \mathbb{N}_0$ can be found in \cite{32}. Let $s=k+\theta$ where $k\in \mathbb{N}_0$ and
$0< \theta<1$.
\begin{equation*}
r=\#\{\nu\in \mathbb{N}_0^n: |\nu|=k\}
\end{equation*}
Define $P: W^{s,p}(\Omega)\rightarrow W^{k,p}(\Omega)\times
[L^p(\Omega\times \Omega)]^{\times r}$ by
\begin{equation*}
P(u)= (u, \bigg(\frac{|\partial^\nu
 u(x)-\partial^\nu u(y)|}{|x-y|^{\frac{n}{p}+\theta}}\bigg)_{|\nu|=k})
\end{equation*}
The space $W^{k,p}(\Omega)\times [L^p(\Omega\times
\Omega)]^{\times r}$ equipped with the norm
\begin{equation*}
\|(f,v_1,\cdots,v_r)\|:=\|f\|_{W^{k,p}(\Omega)}+\|v_1\|_{L^p(\Omega\times
\Omega )}+\cdots+\|v_r\|_{L^p(\Omega\times \Omega )}
\end{equation*}
is a product of reflexive spaces and so it is reflexive (see
Theorem \ref{thmfallnormedreflexive1}). Clearly the operator $P$
is an isometry from $W^{s,p}(\Omega)$ to $W^{k,p}(\Omega)\times
[L^p(\Omega\times \Omega)]^{\times r}$. Since $W^{s,p}(\Omega)$
is a Banach space, $P(W^{s,p}(\Omega))$ is a closed subspace of
the reflexive space $W^{k,p}(\Omega)\times [L^p(\Omega\times
\Omega)]^{\times r}$ and thus it is reflexive. Hence
$W^{s,p}(\Omega)$ itself is reflexive.
\item $W_0^{s,p}(\Omega)$ is the closure of $C_c^\infty(\Omega)$
in $W^{s,p}(\Omega)$. Closed subspaces of reflexive spaces are
reflexive. Therefore $W^{s,p}_0(\Omega)$ is reflexive.
\item A normed space $X$ is reflexive if and only if $X^*$ is
reflexive (see Theorem \ref{thmfallnormedreflexive1}). Since for
$s<0$ we have $W^{s,p}(\Omega)=[W_0^{-s,p'}(\Omega)]^*$, the
reflexivity of $W^{s,p}(\Omega)$ follows from the reflexivity of
$W^{-s,p'}_0(\Omega)$.
\end{enumerateXALI}
\end{proof}
\begin{theorem}\lab{winter93}
For all $s<0$ and $1<p<\infty$, $C_c^\infty(\Omega)$  is dense in
$W^{s,p}(\Omega)$.
\end{theorem}
\begin{proof}
The proof of the density of $L^p$ in $W^{m,p}$ in page 65 of \cite{32} for integer order Sobolev
spaces, which is based on the reflexivity of $W^{-m,p'}_0(\Omega)$,
works in the exact same way for establishing the density of $C_c^\infty(\Omega)$ in $W^{s,p}(\Omega)$.
\end{proof}
\begin{theorem}[Meyers-Serrin]\lab{winter94}
For all $s\geq 0$ and $p\in (1,\infty)$, $C^{\infty}(\Omega)\cap
W^{s,p}(\Omega)$ is dense in $W^{s,p}(\Omega)$.
\end{theorem}
Next we consider \emph{extension by zero} and its properties.
\begin{lemma}(\cite{33}, Page 201)\lab{lemfallextension1}
Let $\Omega$ be a nonempty open set in $\reals^n$ and $u\in
W^{m,p}_0(\Omega)$ where $m\in \mathbb{N}_0$ and $1<p<\infty$.
Then
\begin{enumerate}
\item $\forall\, |\alpha|\leq m$, $\partial^\alpha \tilde{u}=\widetilde{(\partial^\alpha
u)}$ as elements of $D'(\reals^n)$.
\item $\tilde{u}\in W^{m,p}(\reals^n)$ with $\|\tilde{u}\|_{W^{m,p}(\reals^n)}=\|u\|_{W^{m,p}(\Omega)}$
\end{enumerate}
Here, $\tilde{u}:=\textrm{ext}^0_{\Omega,\reals^n} u$ and
$\widetilde{(\partial^\alpha
u)}:=\textrm{ext}^0_{\Omega,\reals^n}(\partial^\alpha u)$.
\end{lemma}
\begin{lemma}[\cite{12},Page 546]\lab{lemfallextension2}
Let $\Omega$ be a nonempty open set in $\reals^n$, $K\in
\mathcal{K}(\Omega)$, $u\in W^{s,p}_K(\Omega)$ where $s\in (0,1)$
and $1<p<\infty$. Then $\textrm{ext}^0_{\Omega,\reals^n}u\in W^{s,p}(\reals^n)$ and
\begin{equation*}
\|\textrm{ext}^0_{\Omega,\reals^n}\|_{W^{s,p}(\reals^n)}\preceq \|u\|_{W^{s,p}(\Omega)}
\end{equation*}
where the implicit constant depends on $n, p, s, K$ and $\Omega$.
\end{lemma}
\begin{theorem}[Extension by Zero]\lab{winter95}
Let $s\geq 0$ and $p\in (1,\infty)$. Let $\Omega$ be a nonempty
open set in $\reals^n$ and let $K\in \mathcal{K}(\Omega)$. Suppose
$u\in W_K^{s,p}(\Omega)$. Then
\begin{enumerate}
\item $\textrm{ext}_{\Omega,\reals^n}^0 u\in W^{s,p}(\reals^n)$.
Indeed,
$\|\textrm{ext}^0_{\Omega,\reals^n}u\|_{W^{s,p}(\reals^n)}\preceq
\|u\|_{W^{s,p}(\Omega)}$ where the implicit constant may depend
on $s, p, n, K, \Omega$ but it is independent of $u\in
W^{s,p}_K(\Omega)$.
\item Moreover,
\begin{equation*}
\|\textrm{ext}_{\Omega,\reals^n}^0 u\|_{W^{s,p}(\reals^n)}\geq
\|u\|_{W^{s,p}(\Omega)}
\end{equation*}
\end{enumerate}
In short
$\|\textrm{ext}^0_{\Omega,\reals^n}u\|_{W^{s,p}(\reals^n)}\simeq
\|u\|_{W^{s,p}(\Omega)}$.
\end{theorem}
\begin{proof}
Let $\tilde{u}=\textrm{ext}^0_{\Omega,\reals^n}u$. If $s\in
\mathbb{N}_0$ then both items follow from Lemma \ref{lemfallextension1}. So let $s=m+\theta$ where $m\in \mathbb{N}_0$ and $\theta\in(0,1)$.
We have
\begin{align*}
\|\tilde{u}\|_{W^{s,p}(\reals^n)}&=\|\tilde{u}\|_{W^{m,p}(\reals^n)}+\sum_{|\nu|=m}|\partial^\nu \tilde{u}|_{W^{\theta,p}(\reals^n)}\\
&=\|u\|_{W^{m,p}(\Omega)}+\sum_{|\nu|=m}|\widetilde{\partial^\nu u}|_{W^{\theta,p}(\reals^n)}\\
&\stackrel{\textrm{Lemma \ref{lemfallextension2}}}{\preceq} \|u\|_{W^{m,p}(\Omega)}+\sum_{|\nu|=m}\|\partial^\nu u\|_{W^{\theta,p}(\Omega)}\\
&\preceq \|u\|_{W^{s,p}(\Omega)}
\end{align*}
The fact that $\|\tilde{u}\|_{W^{s,p}(\reals^n)}\geq \|u\|_{W^{s,p}(\Omega)}$ is a direct consequence of the decomposition stated in item 1. of Remark \ref{remfalldecomint1}.
\end{proof}
\begin{corollary}\lab{corofallextrensionzeropos1}
Let $s\geq 0$ and $p\in (1,\infty)$. Let $\Omega$ and $\Omega'$ be
nonempty open sets in $\reals^n$ with $\Omega'\subseteq \Omega$
and let $K\in \mathcal{K}(\Omega')$. Suppose $u\in
W_K^{s,p}(\Omega')$. Then
\begin{enumerate}
\item $\textrm{ext}_{\Omega',\Omega}^0 u\in W^{s,p}(\Omega)$
\item $\|\textrm{ext}^0_{\Omega',\Omega}u\|_{W^{s,p}(\Omega)}\simeq
\|u\|_{W^{s,p}(\Omega')}$
\end{enumerate}
\end{corollary}
\begin{proof}
\begin{align*}
u\in W_K^{s,p}(\Omega')\Longrightarrow
\textrm{ext}^0_{\Omega',\reals^n}u\in
W^{s,p}(\reals^n)\Longrightarrow
\textrm{ext}^0_{\Omega',\reals^n}u|_\Omega \in
W^{s,p}(\bar{\Omega})
\end{align*}
As it was shown for any arbitrary domain $\Omega$,
$W^{s,p}(\bar{\Omega})\hookrightarrow W^{s,p}(\Omega)$. Also it
is easy to see that
$\textrm{ext}^0_{\Omega',\reals^n}u|_\Omega=\textrm{ext}^0_{\Omega',\Omega}u$.
 Therefore $\textrm{ext}^0_{\Omega',\Omega}u\in W^{s,p}(\Omega)$.
 Moreover
\begin{equation*}
\|\textrm{ext}^0_{\Omega',\Omega}u\|_{W^{s,p}(\Omega)}\simeq
\|\textrm{ext}^0_{\Omega,\reals^n}\circ\textrm{ext}^0_{\Omega',\Omega}u\|_{W^{s,p}(\reals^n)}=\|\textrm{ext}^0_{\Omega',\reals^n}u\|_{W^{s,p}(\reals^n)}\simeq
\|u\|_{W^{s,p}(\Omega')}
\end{equation*}
\end{proof}
Extension by zero for Sobolev spaces with negative exponents will be discussed in Theorem \ref{thmfallextrensionzeroneg1}.
\begin{theorem}[Embedding Theorem IV]\lab{thm3.2}
Let $\Omega\subseteq \reals^n$ be an arbitrary nonempty open set.
\begin{enumerateXALI}
\item  Suppose $1\leq p\leq q<\infty$
 and $0\leq t\leq s$
satisfy $s-\frac{n}{p}\geq t-\frac{n}{q}$. Then
$W^{s,p}(\bar{\Omega})\hookrightarrow W^{t,q}(\bar{\Omega})$.
\item Suppose $1\leq p\leq q<\infty$
 and $0\leq t\leq s$
satisfy $s-\frac{n}{p}\geq t-\frac{n}{q}$. Then
$W^{s,p}_{K}(\Omega)\hookrightarrow W^{t,q}_{K}(\Omega)$ for all
$K\in \mathcal{K}(\Omega)$.
\item For all $k_1, k_2\in \mathbb{N}_0$  with $k_1\leq k_2$ and $1<p<\infty$,
$W^{k_2,p}(\Omega)\hookrightarrow W^{k_1,p}(\Omega)$.
\item If $0\leq t \leq s <1$ and $1<p<\infty$, then $W^{s,p}(\Omega)\hookrightarrow
W^{t,p}(\Omega)$.
\item If $0\leq t \leq s <\infty$ are such that $\floor{s}=\floor{t}$ and
$1<p<\infty$, then $W^{s,p}(\Omega)\hookrightarrow
W^{t,p}(\Omega)$.
\item If $0\leq t \leq s <\infty$, $t\in \mathbb{N}_0$, and
$1<p<\infty$, then $W^{s,p}(\Omega)\hookrightarrow
W^{t,p}(\Omega)$.
\end{enumerateXALI}
\end{theorem}
\begin{proof}
\leavevmode
\begin{enumerateXALI}
\item This item can be found in (\cite{36}, Section 4.6.1).
\item For all $u\in W^{s,p}_K(\Omega)$ we have
\begin{align*}
\|u\|_{W^{t,q}(\Omega)}\simeq
\|\textrm{ext}^0_{\Omega,\reals^n}u\|_{W^{t,q}(\reals^n)}\preceq
\|\textrm{ext}^0_{\Omega,\reals^n}u\|_{W^{s,p}(\reals^n)}\simeq
\|u\|_{W^{s,p}(\Omega)}
\end{align*}
\item This item is a direct consequence of the definition of
integer order Sobolev spaces.
\item Proof can be found in \cite{12}, Page 524.
\item This is a direct consequence of the previous two items.
\item This is true because $W^{s,p}(\Omega)\hookrightarrow W^{\floor{s},p}(\Omega)\hookrightarrow
W^{t,p}(\Omega)$.
\end{enumerateXALI}
\end{proof}
\begin{remark}\lab{winter96}
For an arbitrary open set $\Omega$ in $\reals^n$ and $0<t<1$, the
 embedding $W^{1,p}(\Omega)\hookrightarrow W^{t,p}(\Omega)$ does
 NOT necessarily hold (see e.g. \cite{12}, Section 9.). Of course,
 as it was discussed, under the extra assumption that $\Omega$ is
 Lipschitz, the latter embedding holds true. So, if $\floor{s}\neq
 \floor{t}$ and $t\not\in\mathbb{N}_0$, then in order to ensure
 that $W^{s,p}(\Omega)\hookrightarrow W^{t,p}(\Omega)$ we need to
 assume some sort of regularity for the domain $\Omega$ (for instance it is enough to assume $\Omega$ is
 Lipschitz).
\end{remark}
\begin{theorem}[Multiplication by smooth functions] \lab{thmfallmultsmooth21}
Let $\Omega$ be any nonempty open set in $\reals^n$. Let $p\in
(1,\infty)$.
\begin{enumerateX}
\item If $0\leq s<1$ and $\varphi\in BC^{0,1}(\Omega)$ (that is, $\varphi\in L^\infty(\Omega)$ and $\varphi$ is Lipschitz), then
\begin{equation*}
m_\varphi: W^{s,p}(\Omega)\rightarrow W^{s,p}(\Omega),\qquad
u\mapsto \varphi u
\end{equation*}
is a well-defined bounded linear map.
\item If $k\in \mathbb{N}_0$ and $\varphi\in BC^k(\Omega)$, then
\begin{equation*}
m_\varphi: W^{k,p}(\Omega)\rightarrow W^{k,p}(\Omega),\qquad
u\mapsto \varphi u
\end{equation*}
is a well-defined bounded linear map.
\item If $-1<s<0$ and $\varphi\in BC^{\infty,1}(\Omega)$ or $s\in
\mathbb{Z}^-$ and $\varphi\in BC^{\infty}(\Omega)$, then
\begin{equation*}
m_\varphi: W^{s,p}(\Omega)\rightarrow W^{s,p}(\Omega),\qquad
u\mapsto \varphi u
\end{equation*}
is a well-defined bounded linear map. ($\varphi u$ is interpreted as the product of a smooth function and a distribution.)
\end{enumerateX}
\end{theorem}
\begin{proof}
\leavevmode
\begin{enumerateX}
\item Proof can be found in \cite{12}, Page 547.
\item Proof can be found in \cite{13}, Page 995.
\item The duality argument in Step 4. of the proof of Theorem
\ref{thmfallmultsmooth20} works for this item too.
\end{enumerateX}
\end{proof}
\begin{remark}\lab{winter97}
Suppose $\varphi\in BC^{\infty,1}(\Omega)$. Note that the above
theorem says nothing about the boundedness of the mapping
$m_\varphi: W^{s,p}(\Omega)\rightarrow W^{s,p}(\Omega)$ in the
case where $s$ is noninteger such that $|s|>1$. Of course, if we
assume $\Omega$ is Lipschitz, then the continuity of $m_\varphi$
follows from Theorem \ref{thmfallmultsmooth20}. It is important
to note that the proof of that theorem for the case $s>1$
(noninteger) uses the embedding
$W^{k+\theta,p}(\Omega)\hookrightarrow W^{k'+\theta,p}(\Omega)$
with $k'<k$ which as we discussed does not hold for an arbitrary
open set $\Omega$. The proof for the case $s<-1$ (noninteger) uses
duality to transfer the problem to $s>1$ and thus again we need
the extra assumption of regularity of the boundary of $\Omega$.
\end{remark}
\begin{theorem}\lab{lemapp3} 
Let $\Omega$ be a nonempty open set in $\reals^n$,
$K\in\mathcal{K}(\Omega)$, $p\in (1,\infty)$, and $-1<s<0$ or $s\in\mathbb{Z}^{-}$ or $s\in [0,\infty)$. If $\varphi\in C^{\infty}(\Omega)$, then the linear
map
\begin{equation*}
W^{s,p}_{K}(\Omega)\rightarrow W^{s,p}_{K}(\Omega),\qquad u\mapsto
\varphi u
\end{equation*}
is well-defined and bounded.
\end{theorem}
\begin{proof}
There exists $\psi\in C_c^{\infty}(\Omega)$ such that $\psi=1$ on
$K$. Clearly $\psi \varphi\in C_c^{\infty}(\Omega)$ and if $u\in
W^{s,p}_{K}(\Omega)$, $\psi \varphi u=\varphi u$ on $\Omega$. Thus
without loss of generality we may assume that $\varphi\in
C_c^{\infty}(\Omega)$. Since $C_c^\infty(\Omega)\subseteq BC^{\infty}(\Omega)$ and $C_c^\infty(\Omega)\subseteq BC^{\infty,1}(\Omega)$, the cases where $-1<s<0$ or $s\in\mathbb{Z}^{-}$ follow from Theorem \ref{thmfallmultsmooth21}. For $s\geq 0$, the proof of Theorem
\ref{thmfallmultsmooth20} works for this theorem as well. The
only place in that proof that the regularity of the boundary of
$\Omega$ was used was for the validity of the embedding
$W^{s,p}(\Omega)\hookrightarrow W^{\theta+|\beta|,p}(\Omega)$.
However, as we know (see Theorem \ref{thm3.2}), this embedding holds for Sobolev spaces with
support in a fixed compact set inside $\Omega$ for a general open
set $\Omega$, that is for $W^{s,p}_K(\Omega)\hookrightarrow
W^{\theta+|\beta|,p}_K(\Omega)$ to be true we do not need to
assume $\Omega$ is Lipschitz.
\end{proof}
\begin{remark}\lab{winter98}
Note that our proofs for $s<0$ are based on duality. As a result
it seems that for the case where $s$ is a noninteger less than $-1$ we cannot have a multiplication by smooth
functions result for $W^{s,p}_K(\Omega)$ similar to the one stated in the above theorem. (Note that there is no fixed compact set $K$ such that every $v\in C_c^\infty(\Omega)$ has compact support in $K$. Thus the technique used in Step 4 of the proof of Theorem \ref{thmfallmultsmooth20} does not work in this case.)
\end{remark}
\begin{theorem}\lab{thmfallextrensionzeroneg1}
Let $s<0$ and $p\in (1,\infty)$. Let $\Omega$ and $\Omega'$ be
nonempty open sets in $\reals^n$ with $\Omega'\subseteq \Omega$
and let $K\in \mathcal{K}(\Omega')$. Suppose $u\in
W_K^{s,p}(\Omega')$. Then
\begin{enumerateX}
\item If $\textrm{ext}^0_{\Omega',\Omega}u\in W^{s,p}(\Omega)$,
then $\|u\|_{W^{s,p}(\Omega')}\preceq
\|\textrm{ext}^0_{\Omega',\Omega}u\|_{W^{s,p}(\Omega)}$ (the
implicit constant may depend on $K$).
\item If $s\in (-\infty,-1]\cap \mathbb{Z}$ or $-1<s<0$, then $\textrm{ext}^0_{\Omega',\Omega}u\in
W^{s,p}(\Omega)$ and\\
$\|\textrm{ext}^0_{\Omega',\Omega}u\|_{W^{s,p}(\Omega)}\simeq
\|u\|_{W^{s,p}(\Omega')}$. This result holds for all $s<0$ if we
further assume that $\Omega$ is Lipschitz or $\Omega=\reals^n$.
\end{enumerateX}
\end{theorem}
\begin{proof}
To be completely rigorous, let $i_{D,W}:D(\Omega')\rightarrow
W^{-s,p'}_0(\Omega')$ be the identity map and let $i^*_{D,W}:
W^{s,p}(\Omega')\rightarrow D'(\Omega')$ be its dual with which
we identify $W^{s,p}(\Omega')$ with a subspace of $D'(\Omega')$.
Previously we defined $ext^0_{\Omega',\Omega}$ for distributions
with compact support in $\Omega'$. For any $u\in
W^{s,p}_K(\Omega')$ we let
\begin{equation*}
\textrm{ext}^0_{\Omega',\Omega}u:=\textrm{ext}^0_{\Omega',\Omega}\circ
i_{D,W}^* u
\end{equation*}
which by definition will be an element of $D'(\Omega)$. Note that (see Remark \ref{winter99} and the discussion right after Remark \ref{winter34})
\begin{align*}
&
\|\textrm{ext}^0_{\Omega',\Omega}u\|_{W^{s,p}(\Omega)}=\sup_{0\neq\psi\in
D(\Omega)}\frac{|\langle
\textrm{ext}^0_{\Omega',\Omega}u,\psi\rangle_{D'(\Omega)\times
D(\Omega)} |}{\|\psi\|_{W^{-s,p'}(\Omega)}}\\
& \|u\|_{W^{s,p}(\Omega')}=\sup_{0\neq\varphi\in
D(\Omega')}\frac{|\langle
 u,\varphi\rangle_{D'(\Omega')\times D(\Omega')}
|}{\|\varphi\|_{W^{-s,p'}(\Omega')}}
\end{align*}
So in order to prove the first item we just need to show that
\begin{equation*}
\forall\, 0\neq \varphi\in D(\Omega')\quad \exists\,\psi\in
D(\Omega)\,\,\textrm{s.t.}\,\,\frac{|\langle
 u,\varphi\rangle_{D'(\Omega')\times D(\Omega')}
|}{\|\varphi\|_{W^{-s,p'}(\Omega')}}\preceq \frac{|\langle
\textrm{ext}^0_{\Omega',\Omega}u,\psi\rangle_{D'(\Omega)\times
D(\Omega)} |}{\|\psi\|_{W^{-s,p'}(\Omega)}}
\end{equation*}
Let $\varphi\in D(\Omega')$. Define
$\psi=\textrm{ext}^0_{\Omega',\Omega}\varphi$. Clearly $\psi\in
D(\Omega)$ and $\psi=\varphi$ on $\Omega'$. Therefore
\begin{equation*}
 \langle \textrm{ext}^0_{\Omega',\Omega}u,\psi\rangle_{D'(\Omega)\times
 D(\Omega)}=\langle u,\psi|_{\Omega'}\rangle_{D'(\Omega')\times
 D(\Omega')}=\langle u,\varphi\rangle_{D'(\Omega')\times
 D(\Omega')}
\end{equation*}
Moreover, since $-s>0$
\begin{equation*}
\|\psi\|_{W^{-s,p'}(\Omega)}=\|\textrm{ext}^0_{\Omega',\Omega}\varphi\|_{W^{-s,p'}(\Omega)}\preceq\|\varphi\|_{W^{-s,p'}(\Omega')}
\end{equation*}
This completes the proof of the first item. For the second item
we just need to prove that under the given hypotheses
\begin{equation*}
\forall\, 0\neq \psi\in D(\Omega)\quad \exists\, \varphi\in
D(\Omega')\,\,\textrm{s.t.}\,\, \frac{|\langle
\textrm{ext}^0_{\Omega',\Omega}u,\psi\rangle_{D'(\Omega)\times
D(\Omega)} |}{\|\psi\|_{W^{-s,p'}(\Omega)}}\preceq \frac{|\langle
 u,\varphi\rangle_{D'(\Omega')\times D(\Omega')}
|}{\|\varphi\|_{W^{-s,p'}(\Omega')}}
\end{equation*}
To this end suppose $\psi\in D(\Omega)$. Choose a compact set
$\tilde{K}$ such that $K\subset \mathring{\tilde{K}}\subset
\tilde{K}\subset \Omega'$. Fix $\chi\in D(\Omega)$ such that
$\chi=1$ on $\tilde{K}$ and $\textrm{supp}\,\chi\subset \Omega'$.
Clearly $\psi=\chi\psi$ on a neighborhood of $K$ and if we set $\varphi=\chi\psi|_{\Omega'}$, then $\varphi\in
D(\Omega')$. Therefore
\begin{equation*}
\langle
 \textrm{ext}^0_{\Omega',\Omega}u,\psi\rangle_{D'(\Omega)\times D(\Omega)}=
 \langle
 \textrm{ext}^0_{\Omega',\Omega}u,\chi\psi\rangle_{D'(\Omega)\times D(\Omega)}=
 \langle
  u,\chi\psi|_{\Omega'}\rangle_{D'(\Omega')\times D(\Omega')}=
  \langle u,\varphi\rangle_{D'(\Omega')\times D(\Omega')}
\end{equation*}
 Also since $-s>0$, we have
\begin{equation*}
\|\varphi\|_{W^{-s,p'}(\Omega')}\leq
\|\textrm{ext}^0_{\Omega',\Omega}\varphi\|_{W^{-s,p'}(\Omega)}=\|\chi
\psi\|_{W^{-s,p'}(\Omega)}\preceq \|\psi\|_{W^{-s,p'}(\Omega)}
\end{equation*}
The latter inequality is the place where we used the assumption
that $s\in (-\infty,-1]\cap \mathbb{Z}$ or $-1<s<0$ or $\Omega$
is Lipschitz or $\Omega=\reals^n$. This completes the proof of the second item.
\end{proof}
\begin{corollary}\lab{corofallusef1}
Let $p\in (1,\infty)$. Let $\Omega$ and $\Omega'$ be nonempty
open sets in $\reals^n$ with $\Omega'\subseteq \Omega$ and let
$K\in \mathcal{K}(\Omega')$. Suppose $u\in W_K^{s,p}(\Omega)$. It
follows from Corollary \ref{corofallextrensionzeropos1} and
Theorem \ref{thmfallextrensionzeroneg1} that
\begin{itemizeX}
\item if $s\in \reals$ is not a noninteger less than $-1$, then
\begin{equation*}
\|u\|_{W^{s,p}(\Omega)}\simeq \|u\|_{W^{s,p}(\Omega')}
\end{equation*}
\item if $\Omega$ is Lipschitz or $\Omega=\reals^n$, then for all $s\in \reals$
\begin{equation*}
\|u\|_{W^{s,p}(\Omega)}\simeq \|u\|_{W^{s,p}(\Omega')}
\end{equation*}
\end{itemizeX}
Note that on the right hand sides of the above expressions, $u$
stands for $\textrm{res}_{\Omega,\Omega'}u$. Clearly
$\textrm{ext}^0_{\Omega',\Omega}\circ
\textrm{res}_{\Omega,\Omega'}u = u$.
\end{corollary}
\begin{theorem}\lab{thmfall312}
Let $\Omega$ be any nonempty open set in $\reals^n$, $K\subseteq
\Omega$ be compact, $s>0$, and $p\in (1,\infty)$. Then the
following norms on $W^{s,p}_{K}(\Omega)$ are equivalent:
\begin{align*}
&\|u\|_{W^{s,p}(\Omega)}:=\|u\|_{W^{k,p}(\Omega)}+\sum_{|\nu|=k}
|\partial^{\nu}u|_{W^{\theta,p}(\Omega)}\\
& [u]_{W^{s,p}(\Omega)}:=\|u\|_{W^{k,p}(\Omega)}+\sum_{1\leq
|\nu|\leq k} |\partial^{\nu}u|_{W^{\theta,p}(\Omega)}
\end{align*}
where $s=k+\theta,\, k\in \mathbb{N}_0,\, \theta\in(0,1)$.
Moreover, if we further assume $\Omega$ is Lipschitz, then the
above norms are equivalent on $W^{s,p}(\Omega)$.
\end{theorem}
\begin{proof}
Clearly for all $u\in W^{s,p}(\Omega)$,
$\|u\|_{W^{s,p}(\Omega)}\leq [u]_{W^{s,p}(\Omega)}$. So it is
enough to show that there is a constant $C>0$ such that for all
$u\in W^{s,p}_{K}(\Omega)$ (or $u\in W^{s,p}(\Omega)$ if $\Omega$
is Lipschitz)
\begin{equation*}
[u]_{W^{s,p}(\Omega)}\leq C \|u\|_{W^{s,p}(\Omega)}
\end{equation*}
For each $1\leq i\leq k$ we have
\begin{equation*}
\sum_{|\nu|=i}
|\partial^{\nu}u|_{W^{\theta,p}(\Omega)}=\|u\|_{W^{i+\theta,p}(\Omega)}-\|u\|_{W^{i,p}(\Omega)}
\end{equation*}
Thus
\begin{align*}
[u]_{W^{s,p}(\Omega)}&=\|u\|_{W^{s,p}(\Omega)}+\sum_{1\leq i< k}
\sum_{|\nu|=i}|\partial^{\nu}u|_{W^{\theta,p}(\Omega)}\\
&=\|u\|_{W^{s,p}(\Omega)}+\sum_{1\leq i< k}
\bigg(\|u\|_{W^{i+\theta,p}(\Omega)}-\|u\|_{W^{i,p}(\Omega)}\bigg)
\end{align*}
Therefore it is enough to show that there exists a constant
$C\geq 1$ such that
\begin{equation*}
\sum_{1\leq i< k} \|u\|_{W^{i+\theta,p}(\Omega)}\leq
(C-1)\|u\|_{W^{s,p}(\Omega)}+\sum_{1\leq i< k}
\|u\|_{W^{i,p}(\Omega)}
\end{equation*}
By Theorem \ref{thm3.2}, for each $1\leq i<k$,
$W^{s,p}_{K}(\Omega)\hookrightarrow W_{K}^{i+\theta,p}(\Omega)$
(also we have $W^{s,p}(\Omega)\hookrightarrow
 W^{i+\theta,p}(\Omega)$ with the extra assumption that $\Omega$
is Lipschitz); so there is a constant $C_i$ such that
$\|u\|_{W^{i+\theta,p}(\Omega)}\leq C_i \|u\|_{W^{s,p}(\Omega)}$.
Clearly with $C=1+\sum_{i=1}^{k-1}C_i$ the desired inequality
holds.
\end{proof}

\begin{remark}\lab{winter99}
Let $s\geq 0$ and $1<p<\infty$. Here we summarize the connection
between Sobolev spaces and space of distributions.
\begin{enumerateXALI}
\item \textbf{Question 1:} What does it mean to say $u\in
D'(\Omega)$ belongs to $W^{-s,p'}(\Omega)$?\\
 \textbf{Answer:}
\begin{align*}
&\textrm{$u\in D'(\Omega)$ is in
$W^{-s,p'}(\Omega)$}\Longleftrightarrow \textrm{$u:
(D(\Omega),\|.\|_{s,p})\rightarrow \reals$ is continuous}\\
& \Longleftrightarrow\textrm{$u:D(\Omega)\rightarrow \reals$ has
a unique continuous extension to
$\hat{u}:W^{s,p}_0(\Omega)\rightarrow \reals$ }
\end{align*}
\item \textbf{Question 2:} How should we interpret $W^{-s,p'}(\Omega)\subseteq
D'(\Omega)$?\\
\textbf{Answer:} $i: D(\Omega)\rightarrow W^{s,p}_0(\Omega)$
 is continuous with dense image. Therefore $i^*: W^{-s,p'}(\Omega)\rightarrow
 D'(\Omega)$ is an injective continuous linear map. If $u\in
 W^{-s,p'}(\Omega)$, then $i^* u\in D'(\Omega)$ and
 \begin{equation*}
 \langle i^*u,\varphi\rangle_{D'(\Omega)\times D(\Omega)}=\langle
 u,i\varphi\rangle_{W^{-s,p'}(\Omega)\times
 W^{s,p}_0(\Omega)}=\langle
 u,\varphi\rangle_{W^{-s,p'}(\Omega)\times
 W^{s,p}_0(\Omega)}
 \end{equation*}
 So $i^*u=u|_{D(\Omega)}$ and if we identify with $i^*u$ with $u$ we
 can write
 {\fontsize{10}{10}{\begin{equation*}
 \langle u,\varphi\rangle_{D'(\Omega)\times D(\Omega)}=\langle
 u,\varphi\rangle_{W^{-s,p'}(\Omega)\times
 W^{s,p}_0(\Omega)},\qquad \|u\|_{W^{-s,p'}(\Omega)}=\sup_{0\neq\varphi\in C_c^\infty(\Omega)}\frac{|\langle u,\varphi\rangle_{D'(\Omega)\times D(\Omega)}|}{\|\varphi\|_{W^{s,p}(\Omega)}}
 \end{equation*}}}
 \item \textbf{Question 3:} What does it mean to say $u\in
D'(\Omega)$ belongs to $W^{s,p}(\Omega)$?\\
\textbf{Answer:} It means there exists $f\in W^{s,p}(\Omega)$
such that $u=u_f$.
\item \textbf{Question 4:} How should we interpret $W^{s,p}(\Omega)\subseteq
D'(\Omega)$?\\
\textbf{Answer:} It is a direct consequence of the definition of
$W^{s,p}(\Omega)$ that $W^{s,p}(\Omega)\hookrightarrow
L^p(\Omega)$ for any open set $\Omega$. So any $f\in
W^{s,p}(\Omega)$ can be identified with the regular distribution
$u_f\in D'(\Omega)$ where
\begin{equation*}
\langle u_f,\varphi\rangle=\int f\varphi\qquad
\forall\,\varphi\in D(\Omega)
\end{equation*}
\end{enumerateXALI}
\end{remark}
\begin{remark}\lab{remjanduality1}
Let $\Omega$ be a nonempty open set in $\reals^n$ and $f,g\in C_c^\infty(\Omega)$. Suppose $s\in\reals$ and $p\in (1,\infty)$.
\begin{itemizeX}
\item If $s\geq 0$, then
\begin{equation*}
\|f\|_{W^{-s,p'}(\Omega)}=\sup_{0\neq\varphi\in C_c^\infty(\Omega)}\frac{|\langle f,\varphi\rangle_{D'(\Omega)\times D(\Omega)}|}{\|\varphi\|_{W^{s,p}(\Omega)}}=\sup_{0\neq\varphi\in C_c^\infty(\Omega)}\frac{|\int_\Omega f\varphi\,dx|}{\|\varphi\|_{W^{s,p}(\Omega)}}
\end{equation*}
So for all $\varphi\in C_c^\infty(\Omega)$
\begin{equation*}
|\int_\Omega f\varphi\,dx|\leq \|f\|_{W^{-s,p'}(\Omega)}\|\varphi\|_{W^{s,p}(\Omega)}
\end{equation*}
In particular, for $g$, we have
\begin{equation*}
|\int_\Omega fg\,dx|\leq \|f\|_{W^{-s,p'}(\Omega)}\|g\|_{W^{s,p}(\Omega)}
\end{equation*}
\item If $s<0$, we may replace the roles of $f$ and $g$, and also $(s,p)$ and $(-s,p')$ in the above argument to get the exact same inequality: $|\int_\Omega fg\,dx|\leq \|f\|_{W^{-s,p'}(\Omega)}\|g\|_{W^{s,p}(\Omega)}$.
\end{itemizeX}
\end{remark}


\subsection{Invariance Under Change of Coordinates, Composition}
\begin{theorem}[\cite{Trie92}, Section 4.3]\lab{thmfallcoc1}
Let $s\in \reals$ and $1<p<\infty$. Suppose that  $T:
\reals^n\rightarrow \reals^n$ is a $C^\infty$-diffeomorphism (i.e.
$T$ is bijective and $T$ and $T^{-1}$ are $C^{\infty}$) with the
property that the partial derivatives (of any order) of the
components of $T$ are bounded on $\reals^n$ (the bound may depend
on the order of the partial derivative) and
$\inf_{\reals^n}|\textrm{det}\,T'|>0$. Then the linear map
\begin{equation*}
W^{s,p}(\reals^n)\rightarrow W^{s,p}(\reals^n),\qquad u\mapsto
u\circ T
\end{equation*}
is well-defined and is bounded.
\end{theorem}
Now let $U$ and $V$ be two nonempty open sets in $\reals^n$. Suppose $T:U\rightarrow V$ is a bijective map.
Similar to \cite{32} we say $T$ is \textbf{$k$-smooth} if all the components of $T$ belong to $BC^k(U)$ and all the components of $T^{-1}$ belong to $BC^{k}(V)$.
\begin{remark}\lab{winter100}
It is useful to note that if $T$ is $1$-smooth, then
\begin{equation*}
\inf_{U}|\textrm{det}\,T'|>0 \quad \textrm{and}\quad
\inf_{V}|\textrm{det}\,(T^{-1})'|>0
\end{equation*}
Indeed, since the first order partial derivatives of the
components of $T$ and $T^{-1}$ are bounded, there exist postive
numbers $M$ and $\tilde{M}$ such that for all $x\in U$ and $y\in
V$
\begin{equation*}
|\textrm{det}\,T'(x)|<M,\qquad
|\textrm{det}\,(T^{-1})'(y)|<\tilde{M}
\end{equation*}
Since $|\textrm{det}\,T'(x)|\times
|\textrm{det}\,(T^{-1})'(T(x))|=1$, we can conclude that for all
$x\in U$ and $y\in V$
\begin{equation*}
|\textrm{det}\,T'(x)|>\frac{1}{\tilde{M}},\qquad
|\textrm{det}\,(T^{-1})'(y)|>\frac{1}{M}
\end{equation*}
which proves the claim.
\end{remark}
\begin{remark}\lab{winter101}
Also it is interesting to note that as a consequence of the
inverse function theorem, if $T:U\rightarrow V$ is a bijective
map that is $C^k$ ($k\in\mathbb{N}$) with the property that $\textrm{det}\,T'(x)\neq
0$ for all $x\in U$, then the inverse of $T$ will be $C^k$ as
well, that is $T$ will automatically be a $C^k$-diffeomorphism
(see e.g. Appendix C in \cite{Lee3} for more details).
\end{remark}
\begin{remark}\lab{winter102}
Note that since we do not assume that $U$ and $V$ are necessarily
convex or Lipschitz, the continuity and boundedness of the partial derivatives
of the components of $T$ do not imply that the components of $T$
are Lipschitz. (See the "Warning" immediately after Theorem
\ref{thmfallconvexapostol1}.)
\end{remark}
\begin{theorem}\lab{lemfall1}[(\cite{13}, Page 1003), (\cite{32}, Pages 77 and 78 )] Let $p\in (1,\infty)$ and $k\in \mathbb{N}$. Suppose that $U$ and $V$ are nonempty
open subsets of $\reals^n$.
\begin{enumerateX}
\item If $T:U\rightarrow V$ is a $1$-smooth map, then the map
\begin{equation*}
L^p(V)\rightarrow L^p(U),\qquad u\mapsto u\circ T
\end{equation*}
is well-defined and is bounded.
\item If $T: U\rightarrow V$ is a $k$-smooth map, then the map
\begin{equation*}
W^{k,p}(V)\rightarrow W^{k,p}(U),\qquad u\mapsto u\circ T
\end{equation*}
is well-defined and is bounded.
\end{enumerateX}
\end{theorem}
\begin{theorem}\lab{thmfallcoc2}
Let $p\in (1,\infty)$ and $k\in \mathbb{Z}^{-}$ ($k$ is a
negative integer). Suppose that $U$ and $V$ are nonempty open
subsets of $\reals^n$, and $T:U\rightarrow V$ is $\infty$-smooth.
Then the map
\begin{equation*}
W^{k,p}(V)\rightarrow W^{k,p}(U),\qquad u\mapsto u\circ T
\end{equation*}
is well-defined and is bounded.
\end{theorem}
\begin{proof}
By definition we have ($T^*u$ denotes the pullback of $u$ by $T$)
\begin{align*}
\|T^*u\|_{W^{k,p}(U)}&=\sup_{\varphi\in
C_c^\infty(U)}\frac{|\langle T^* u,\varphi \rangle_{D'(U)\times
D(U)}
|}{\|\varphi\|_{W^{-k,p'}(U)}}\\
&=\sup_{\varphi\in C_c^\infty(U)}\frac{|\langle
u,|\textrm{det}(T^{-1})'|\varphi\circ T^{-1} \rangle_{D'(V)\times
D(V)}
|}{\|\varphi\|_{W^{-k,p'}(U)}}\\
& \preceq
\sup_{\varphi\in C_c^\infty(U)}\frac{\|u\|_{W^{k,p}(V)}\||\textrm{det}(T^{-1})'|\varphi\circ
T^{-1} \|_{W^{-k,p'}(V)}}{\|\varphi\|_{W^{-k,p'}(U)}}\\
&\stackrel{|\textrm{det}(T^{-1})'|\in BC^{\infty}}{\preceq} \sup_{\varphi\in C_c^\infty(U)}
\frac{\|u\|_{W^{k,p}(V)}\|\varphi\circ T^{-1}
\|_{W^{-k,p'}(V)}}{\|\varphi\|_{W^{-k,p'}(U)}}
\end{align*}
Since $-k$ is a positive integer, by Theorem \ref{lemfall1} we have
{\fontsize{10}{10}{$\|\varphi\circ T^{-1} \|_{W^{-k,p'}(V)}\preceq
\|\varphi\|_{W^{-k,p'}(U)}$}}. Consequently
\begin{equation*}
\|T^*u\|_{W^{k,p}(U)}\preceq \|u\|_{W^{k,p}(V)}
\end{equation*}
\end{proof}
\begin{theorem}\lab{thmfallcoc3}
Let $p\in (1,\infty)$ and $0<s<1$. Suppose that $U$ and $V$ are
nonempty open subsets of $\reals^n$, $T:U\rightarrow V$ is
$1$-smooth, and $T$ is Lipschitz continuous on $U$. Then the map
\begin{equation*}
W^{s,p}(V)\rightarrow W^{s,p}(U),\qquad u\mapsto u\circ T
\end{equation*}
is well-defined and is bounded.
\end{theorem}
\begin{proof}
Note that
\begin{equation*}
\|u\circ T\|_{W^{s,p}(U)}=\|u\circ T\|_{L^{p}(U)}+|u\circ
T|_{W^{s,p}(U)}\stackrel{Theorem \ref{lemfall1}}{\preceq}
\|u\|_{L^{p}(V)}+|u\circ T|_{W^{s,p}(U)}
\end{equation*}
So it is enough to show that $|u\circ T|_{W^{s,p}(U)}\preceq
|u|_{W^{s,p}(V)}$
\begin{align*}
|u\circ T|_{W^{s,p}(U)}&=\big(\int\int_{U\times
 U}\frac{|(u\circ T)(x)-(u\circ T)(y)|^p}{|x-y|^{n+s p}}dx
 dy\big)^{\frac{1}{p}}\\
 & \stackrel{\stackrel{z=T(x)}{w=T(y)}}{\preceq}\big(\int\int_{V\times
 V}\frac{|u(z)-u(w)|^p}{|T^{-1}(z)-T^{-1}(w)|^{n+s p}} \frac{1}{|\textrm{det}T'(x)|}
 \frac{1}{|\textrm{det}T'(y)|}dz
 dw\big)^{\frac{1}{p}}\\
 &\preceq\big(\int\int_{V\times
 V}\frac{|u(z)-u(w)|^p}{|T^{-1}(z)-T^{-1}(w)|^{n+s p}} dz
 dw\big)^{\frac{1}{p}}
\end{align*}
$T$ is Lipschitz continuous on $U$; so there exists a constant
$C>0$ such that
\begin{equation*}
|T(x)-T(y)|\leq C|x-y|\Longrightarrow |z-w|\leq
C|T^{-1}(z)-T^{-1}(w)|
\end{equation*}
Therefore
\begin{equation*}
|u\circ T|_{W^{s,p}(U)}\preceq \big(\int\int_{V\times
 V}\frac{|u(z)-u(w)|^p}{|z-w|^{n+s p}} dz
 dw\big)^{\frac{1}{p}}=|u|_{W^{s,p}(V)}
\end{equation*}
\end{proof}
\begin{theorem}\lab{thmfallcoc4}
Let $p\in (1,\infty)$ and $-1<s<0$. Suppose that $U$ and $V$ are
nonempty open subsets of $\reals^n$, $T:U\rightarrow V$ is
$\infty$-smooth, $T^{-1}$ is Lipschitz continuous on $V$, and
$|\textrm{det}(T^{-1})'|$ is in $BC^{0,1}(V)$. Then the map
\begin{equation*}
W^{s,p}(V)\rightarrow W^{s,p}(U),\qquad u\mapsto u\circ T
\end{equation*}
is well-defined and is bounded.
\end{theorem}
\begin{proof}
The proof of Theorem \ref{thmfallcoc2}, with obvious
modifications, shows the validity of the above claim.
\end{proof}
\begin{remark}\lab{winter103}
By assumption the first order partial derivatives of the
components of $T^{-1}$ are continuous and bounded. Also it is true
that absolute value of a Lipschitz continuous function and the sum
and product of bounded Lipschitz continuous functions will be
Lipschitz continuous. Consequently, in order to ensure that
$|\textrm{det}(T^{-1})'|$ is in $BC^{0,1}(V)$, it is enough to
make sure that the first order partial derivatives of the
components of $T^{-1}$ are bounded and Lipschitz continuous.
\end{remark}
\begin{theorem}\lab{thmfallcoc5}
Let $s=k+\theta$ where $k\in \mathbb{N}$, $\theta\in (0,1)$, and
let $p\in (1,\infty)$. Suppose that $U$ and $V$ are two nonempty
open sets in $\reals^n$. Let $T: U\rightarrow V $ be a Lipschitz continuous
 $k$-smooth map on $U$ such that the
partial derivatives up to and including order $k$ of all the
components of $T$ are Lipschitz continuous on $U$ as well. Then
\begin{enumerate}
\item for each $K\in \mathcal{K}(V)$ the linear map
\begin{equation*}
T^*: W^{s,p}_{K}(V)\rightarrow W^{s,p}_{T^{-1}(K)}(U),\qquad
u\mapsto u\circ T
\end{equation*}
is well-defined and is bounded.
\item if we further assume that $V$ is Lipschitz (and so $U$ is Lipschitz), the linear map
\begin{equation*}
T^*: W^{s,p}(V)\rightarrow W^{s,p}(U),\qquad u\mapsto u\circ T
\end{equation*}
is well-defined and is bounded.\\
\textbf{Note:} When $U$ is a Lipschitz domain, the fact that $T$ is $k$-smooth automatically implies that all the partial derivatives of the components of $T$ up to and including order $k-1$ are Lipschitz continuous (see Theorem \ref{sobislip31}). So in this case, the only extra assumption, in addition to $T$ being $k$-smooth, is that the partial derivatives of the components of $T$ of order $k$ are Lipschitz continuous on $U$.
\end{enumerate}
\end{theorem}
\begin{proof}
Recall that $C^{\infty}(V)\cap W^{s,p}(V)$ is dense in
$W^{s,p}(V)$. Our proof consists of two steps: in the first step
we addditionally assume that $u\in C^\infty(V)$. Then in the
second step we prove the validity of the claim for $u\in
W^{s,p}_K(V)$ (or $u\in W^{s,p}(V)$ with the assumption that $V$
is Lipschitz).
\begin{itemizeX}
\item \textbf{Step 1:}
We have
\begin{align*}
\|u\circ T\|_{W^{s,p}(U)}&=\|u\circ T\|_{W^{k,p}(U)}+\sum_{
|\nu|= k}|\partial^\nu(u\circ T)|_{W^{\theta,p}(U)}\\
&\stackrel{\textrm{Theorem \ref{lemfall1}}}{\preceq}
\|u\|_{W^{k,p}(V)}+\sum_{|\nu|= k}|\partial^\nu(u\circ
T)|_{W^{\theta,p}(U)}
\end{align*}
Since $u$ and $T$ are both $C^k$, it can be proved by induction
that (see e.g. \cite{32})
\begin{equation*}
\partial^\nu(u\circ T)(x)=\sum_{\beta\leq \nu, 1\leq |\beta|} M_{\nu
\beta}(x)[(\partial^\beta u)\circ T](x)
\end{equation*}
where $M_{\nu \beta}(x)$ are polynomials of degree at most
$|\beta|$ in derivatives of order at most $|\nu|$ of the
components of $T$. In particular, $M_{\nu\beta}\in BC^{0,1}(U)$ . Therefore
{\fontsize{10}{10}{\begin{align*}
|\partial^\nu &(u\circ T)|_{W^{\theta,p}(U)}\leq
\|\partial^\nu(u\circ T)\|_{W^{\theta,p}(U)}=\|\sum_{\beta\leq
\nu, 1\leq |\beta|} M_{\nu \beta}(x)[(\partial^\beta u)\circ
T](x)\|_{W^{\theta,p}(U)}\\
& \stackrel{\textrm{Theorem \ref{thmfallmultsmooth21}}}{\preceq}
\sum_{\beta\leq \nu, 1\leq |\beta|}\|(\partial^\beta u)\circ
T\|_{W^{\theta,p}(U)}=\sum_{\beta\leq \nu, 1\leq
|\beta|}\|(\partial^\beta u)\circ T\|_{L^p(U)}+|(\partial^\beta
u)\circ T|_{W^{\theta,p}(U)}\\
&\stackrel{\textrm{Theorem \ref{lemfall1} and \ref{thmfallcoc3}}}{\preceq}
\sum_{\beta\leq \nu, 1\leq |\beta|}\|\partial^\beta
u\|_{L^p(V)}+|\partial^\beta u|_{W^{\theta,p}(V)} \leq
\|u\|_{W^{k,p}(V)}+\sum_{\beta\leq \nu, 1\leq
|\beta|}|\partial^\beta u|_{W^{\theta,p}(V)}
\end{align*}}}
(The fact that $\partial^\beta u$ belongs to
$W^{\theta,p}(V)\hookrightarrow L^p(V)$ is a consequence of the definition of the
Slobodeckij norm combined with our
embedding theorems for Sobolev spaces of functions with fixed compact support in an arbitrary domain or
 embedding theorems for Sobolev spaces of functions on a Lipschitz domain). Hence
\begin{align*}
\|u\circ T\|_{W^{s,p}(U)}&\preceq \|u\|_{W^{k,p}(V)}+\sum_{1\leq
|\nu|\leq k}\sum_{\beta\leq \nu, 1\leq |\beta|}|\partial^\beta
u|_{W^{\theta,p}(V)}\\
&\preceq \|u\|_{W^{k,p}(V)}+\sum_{1\leq |\alpha|\leq
k}|\partial^\alpha u|_{W^{\theta,p}(V)}\stackrel{Theorem \ref{thmfall312}}{\simeq}\|u\|_{W^{s,p}(V)}
\end{align*}
Note that the last equivalence is due to the assumption that $u\in W^{s,p}_K(V)$ ( or $u\in W^{s,p}(V)$ with $V$ being Lipschitz).
\item \textbf{Step 2:} Now suppose $u$ is an arbitrary element of
$W^{s,p}_{K}(V)$ (or $W^{s,p}(V)$ with $V$ being Lipschitz). There exists a sequence $\{u_m\}_{m\geq 1}$ in
$C^\infty(V)$ such that $u_m\rightarrow u$ in $W^{s,p}(V)$. In
particular, $\{u_m\}$ is Cauchy. By the previous steps we have
\begin{equation*}
\|T^* u_m-T^*u_l\|_{W^{s,p}(U)}\preceq
\|u_m-u_l\|_{W^{s,p}(V)}\rightarrow 0\qquad (\textrm{as
$m,l\rightarrow \infty$})
\end{equation*}
Therefore $\{T^*u_m\}$ is a Cauchy sequence in the Banach space
$W^{s,p}(U)$ and subsequently there exists $v\in W^{s,p}(U)$ such
that $T^*u_m\rightarrow v$ as $m\rightarrow \infty$. It remains
to show that $v=T^*u$ as elements of $W^{s,p}(U)$. As a direct
consequence of the definition of $W^{s,p}$-norm $(s\geq 0)$ we
have
\begin{align*}
& \|T^*u_m-v\|_{L^p(U)}\leq \|T^*u_m-v\|_{W^{s,p}(U)}\rightarrow 0\\
&\|u_m-u\|_{L^p(V)}\leq \|u_m-u\|_{W^{s,p}(V)}\rightarrow 0
\end{align*}
Note that by Theorem \ref{lemfall1}, $u_m\rightarrow u$ in
$L^p(V)$ implies that $T^* u_m\rightarrow T^*u$ in $L^p(U)$. Thus
$T^*u=v$ as elements of $L^p(U)$ and hence as elements of
$W^{s,p}(U)$.
\end{itemizeX}
\end{proof}
\begin{theorem}\lab{thmfallcoc6}
Let $p\in (1,\infty)$ and $s<-1$ be a \textbf{noninteger} number.
Suppose that $U$ and $V$ are two nonempty \textbf{Lipschitz} open
sets in $\reals^n$ and $T: U\rightarrow V$ is a $\infty$-smooth map
 such that $T^{-1}$ is Lipschitz continuous on
$V$ and the partial derivatives up to and including order $k$ of
all the components of $T^{-1}$ are Lipschitz  continuous on $V$.
Then the linear map
\begin{equation*}
T^*: W^{s,p}(V)\rightarrow W^{s,p}(U),\qquad u\mapsto u\circ T
\end{equation*}
is well-defined and is bounded.\\
\textbf{Note:} Since $V$ is a Lipschitz domain, the fact that $T$ is $\infty$-smooth automatically implies that $T^{-1}$ and all the partial derivatives of the components of $T^{-1}$ are Lipschitz continuous (see Theorem \ref{sobislip31}).
\end{theorem}
\begin{proof}
The proof is completely analogous to the proof of Theorem
\ref{thmfallcoc2}. We have
\begin{align*}
\|T^*u\|_{W^{s,p}(U)}&=\sup_{\varphi\in
C_c^\infty(U)}\frac{|\langle T^* u,\varphi \rangle_{D'(U)\times
D(U)}
|}{\|\varphi\|_{W^{-s,p'}(U)}}\\
&=\sup_{\varphi\in C_c^\infty(U)}\frac{|\langle
u,|\textrm{det}(T^{-1})'|\varphi\circ T^{-1} \rangle_{D'(V)\times
D(V)}
|}{\|\varphi\|_{W^{-s,p'}(U)}}\\
& \preceq
\frac{\|u\|_{W^{s,p}(V)}\||\textrm{det}(T^{-1})'|\varphi\circ
T^{-1} \|_{W^{-s,p'}(V)}}{\|\varphi\|_{W^{-s,p'}(U)}}\\
&\stackrel{|\textrm{det}(T^{-1})'|\in
BC^{\infty}(V)}{\preceq}
\frac{\|u\|_{W^{s,p}(V)}\|\varphi\circ T^{-1}
\|_{W^{-s,p'}(V)}}{\|\varphi\|_{W^{-s,p'}(U)}}
\end{align*}
Since $-s>0$, it follows from the hypotheses of this theorem and
the result of Theorem \ref{thmfallcoc5} that $\|\varphi\circ
T^{-1} \|_{W^{-s,p'}(V)}\preceq \|\varphi\|_{W^{-s,p'}(U)}$.
Consequently
\begin{equation*}
\|T^*u\|_{W^{s,p}(U)}\preceq \|u\|_{W^{s,p}(V)}
\end{equation*}
\end{proof}
\begin{lemma}\lab{winter104}
Let $U$ and $V$ be two nonempty open sets in $\reals^n$. Suppose
$T:U\rightarrow V$ ($T=(T^1,\cdots,T^n)$) is a
$C^{k+1}$-diffeomorphism for some $k\in\mathbb{N}_0$ and let
$B\subseteq U$ be a nonempty bounded open set such that
$B\subseteq \bar{B}\subseteq U$. Then
\begin{enumerate}
\item $T:B\rightarrow T(B)$ is a $(k+1)$-smooth map.
\item $T: B\rightarrow T(B)$ and $T^{-1}:T(B)\rightarrow B$ are
Lipschitz (the Lipschitz constant may depend on $B$).
\item For all $1\leq i\leq n$ and $|\alpha|\leq k$, $\partial^\alpha T^i\in
BC^{k,1}(B)$ and $\partial^\alpha (T^{-1})^i\in BC^{k,1}(T(B))$.
\end{enumerate}
\end{lemma}
\begin{proof}
Item 1. is true because $\bar{B}$ is compact and so $T(\bar{B})$
is compact and continuous functions are bounded on compact sets.
Items 2. and 3. are direct consequences of Theorem
 \ref{thmfalllipbound1}.
\end{proof}
\begin{theorem}\lab{winter105}
Let $s\in \reals$ and $p\in(1,\infty)$. Suppose that $U$ and $V$
are two nonempty open sets in $\reals^n$ and $T: U\rightarrow V$
is a $C^\infty$-diffeomorphism (if $s\geq 0$ it is enough to assume $T$ is a $C^{\floor{s}+1}$-diffeomorphism). Let $B\subseteq U$ be a
nonempty bounded open set such that $B\subseteq \bar{B}\subseteq
U$. Let $u\in W^{s,p}(V)$ be such that $\textrm{supp} u\subseteq
T(B) $. (Note that if $\textrm{supp} u$ is compact in $V$, then such a $B$ exists.)
\begin{enumerate}
\item If $s$ is NOT a noninteger less than $-1$, then
\begin{equation*}
\|u\circ T\|_{W^{s,p}(U)}\preceq \|u\|_{W^{s,p}(V)}
\end{equation*}
(the implicit constant may depend on $B$ but otherwise is independent of $u$)
\item If $U$ and $V$ are Lipschitz or $\reals^n$, then the above result holds
for all $s\in\reals$.
\end{enumerate}
\end{theorem}
\begin{proof}
If $s$ is an integer or $-1<s<1$, or if $U$ and $V$ are Lipschitz or $\reals^n$
and $s\in \reals$ then as a consequence of the above lemma and the
preceding theorems we may write
\begin{align*}
\|u\circ T\|_{W^{s,p}(U)}\stackrel{\textrm{Corollary
\ref{corofallusef1}}}{\simeq} \|u\circ T\|_{W^{s,p}(B)}\preceq
\|u\|_{W^{s,p}(T(B))}\stackrel{\textrm{Corollary
\ref{corofallusef1}}}{\simeq} \|u\|_{W^{s,p}(V)}
\end{align*}
For general $U$ and $V$, if $s=k+\theta$, we let $\hat{B}$ be an
open set such that $\bar{\hat{B}}$ is a compact subset of $U$ and
$\bar{B}\subseteq{\hat{B}}$. We can apply the previous lemma to
$\hat{B}$ and write
\begin{align*}
\|u\circ T\|_{W^{s,p}(U)}\stackrel{\textrm{Corollary
\ref{corofallusef1}}}{\simeq} \|u\circ
T\|_{W^{s,p}_{\bar{B}}(\hat{B})}\stackrel{\textrm{Theorem
\ref{thmfallcoc5}}}{\preceq}
\|u\|_{W^{s,p}_{T(\bar{B})}(T(\hat{B}))}\stackrel{\textrm{Corollary
\ref{corofallusef1}}}{\simeq} \|u\|_{W^{s,p}(V)}
\end{align*}
\end{proof}
\begin{theorem}\cite{38}\lab{winter106}
Let $s\in [1,\infty)$, $1<p<\infty$, and let
\begin{align*}
m=\begin{cases}
&s,\quad \textrm{if $s$ is an integer}\\
& \floor{s}+1,\quad \textrm{otherwise}
\end{cases}
\end{align*}
If $F\in C^m(\reals)$ is such that $F(0)=0$ and $F,F',\cdots,F^{(m)}\in L^{\infty}(\reals)$ (in particular, note that every $F\in C_c^\infty(\reals)$ with $F(0)=0$ satisfies these conditions), then the map $u\mapsto F(u)$ is well-defined and continuous from $W^{s,p}(\reals^n)\cap W^{1,sp}(\reals^n)$ into $W^{s,p}(\reals^n)$.
\end{theorem}
\begin{corollary}\lab{winter107}
Let $s$, $p$, and $F$ be as in the previous theorem. Moreover suppose $sp>n$. Then the map $u\mapsto F(u)$ is well-defined and continuous from $W^{s,p}(\reals^n)$ into $W^{s,p}(\reals^n)$. The reason is that when $sp>n$, we have $W^{s,p}(\reals^n)\hookrightarrow W^{1,sp}(\reals^n)$.
\end{corollary}
\subsection{Differentiation}
\begin{theorem}[(\cite{33}, Pages 598-605), (\cite{Gris85}, Section 1.4)]
\lab{winter88} Let $s\in \reals$, $1<p<\infty$, and $\alpha\in
\mathbb{N}_0^n$. Suppose $\Omega$ is a nonempty open set in $\reals^n$. Then
\begin{enumerateX}
\item the linear operator $\partial^\alpha: W^{s,p}(\reals^n)\rightarrow
W^{s-|\alpha|,p}(\reals^n)$ is well-defined and bounded;
\item for $s<0$, the linear operator $\partial^\alpha: W^{s,p}(\Omega)\rightarrow
W^{s-|\alpha|,p}(\Omega)$ is well-defined and bounded;
\item for $s\geq 0$ and $|\alpha|\leq s$, the linear operator $\partial^\alpha: W^{s,p}(\Omega)\rightarrow
W^{s-|\alpha|,p}(\Omega)$ is well-defined and bounded;
\item if $\Omega$ is bounded with Lipschitz continuous boundary, and if $s\geq 0$, $s-\frac{1}{p}\neq \textrm{integer}$ (i.e. the fractional part of $s$ is not equal to $\frac{1}{p}$), then the linear operator $\partial^\alpha: W^{s,p}(\Omega)\rightarrow W^{s-|\alpha|,p}(\Omega)$ for $|\alpha|>s$ is well-defined and bounded.
\end{enumerateX}
\end{theorem}
\begin{remark}
Comparing the first and last items of the previous theorem, we see that not all the properties of Sobolev-Slobodeckij spaces on $\reals^n$ are fully inherited by Sobolev-Slobodeckij spaces on bounded domains even when the domain has Lipschitz continuous boundary. (Note that the above difference is related to the more fundamental fact that for $s> 0$, even when $\Omega$ is Lipschitz, $C_c^\infty(\Omega)$ is not necessarily dense in $W^{s,p}(\Omega)$ and subsequently $W^{-s,p'}(\Omega)$ is defined as the dual of $W_0^{s,p}(\Omega)$ rather than the dual of $W^{s,p}(\Omega)$ itself.) For this reason, when working with Sobolev spaces on manifolds, we prefer super nice atlases (i.e. we prefer to work with coordinate charts whose image under the coordinate map is the entire $\reals^n$). The next best choice would be GGL or GL atlases.
\end{remark}
\subsection{Spaces of Locally Sobolev Functions}
Material of this section are taken from \cite{holstbehzadan2018c}.
\begin{definition}
Let $s\in \reals$, $1<p<\infty$. Let $\Omega$ be a nonempty open set in
$\reals^n$. We define
\begin{align*}
 W^{s,p}_{loc}(\Omega):=\{u\in D'(\Omega): \forall \varphi\in
C_c^\infty(\Omega)\quad \varphi u\in W^{s,p}(\Omega)\}
\end{align*}
$W^{s,p}_{loc}(\Omega)$ is equipped with the natural topology
induced by the separating family of seminorms $\{|.|_{\varphi}\}_{\varphi\in
C_c^{\infty}(\Omega)}\}$ where
\begin{equation*}
\forall\, u\in W^{s,p}_{loc}(\Omega)\quad \varphi\in
C_c^{\infty}(\Omega)\qquad |u|_\varphi:=\|\varphi
u\|_{W^{s,p}(\Omega)}
\end{equation*}
\end{definition}

\begin{theorem}\lab{thmapp8a}
 Let $s\in \reals$, $1<p<\infty$, and $\alpha\in
\mathbb{N}_0^n$. Suppose $\Omega$ is a nonempty bounded open set in $\reals^n$ with Lipschitz continuous boundary. Then
\begin{enumerateX}
\item the linear operator $\partial^\alpha: W^{s,p}_{loc}(\reals^n)\rightarrow
W^{s-|\alpha|,p}_{loc}(\reals^n)$ is well-defined and continuous;
\item for $s<0$, the linear operator $\partial^\alpha: W^{s,p}_{loc}(\Omega)\rightarrow
W^{s-|\alpha|,p}_{loc}(\Omega)$ is well-defined and continuous;
\item for $s\geq 0$ and $|\alpha|\leq s$, the linear operator $\partial^\alpha: W^{s,p}_{loc}(\Omega)\rightarrow
W^{s-|\alpha|,p}_{loc}(\Omega)$ is well-defined and continuous;
\item if $s\geq 0$, $s-\frac{1}{p}\neq \textrm{integer}$ (i.e. the fractional part of $s$ is not equal to $\frac{1}{p}$), then the linear operator $\partial^\alpha: W^{s,p}_{loc}(\Omega)\rightarrow W^{s-|\alpha|,p}_{loc}(\Omega)$ for $|\alpha|>s$ is well-defined and continuous.
\end{enumerateX}
\end{theorem}


The following statements play a key role in our study of Sobolev spaces on Riemannian manifolds with rough metrics.
\begin{theorem}\lab{thmapp8b}
Let $\Omega$ be a nonempty bounded open set in $\reals^n$ with Lipschitz continuous boundary or $\Omega=\reals^n$. Suppose $u\in W^{s,p}_{loc}(\Omega)$ where $sp>n$. Then $u$ has a continuous
version.
\end{theorem}

\begin{lemma}\lab{lemapp1}
Let $\Omega=\reals^n$ or $\Omega$ be a bounded open set in $\reals^n$ with Lipschitz continuous boundary. Suppose $s_1, s_2, s\in \reals$ and $1<p_1,p_2,p<\infty$ are such that
\begin{equation*}
W^{s_1,p_1}(\Omega)\times W^{s_2,p_2}(\Omega)\hookrightarrow W^{s,p}(\Omega)\,.
\end{equation*}
Then
\begin{enumerate}
\item $W^{s_1,p_1}_{loc}(\Omega)\times W^{s_2,p_2}_{loc}(\Omega)\hookrightarrow
W^{s,p}_{loc}(\Omega)$.
\item For all $K\in \mathcal{K}(\Omega)$, $W^{s_1,p_1}_{loc}(\Omega)\times W^{s_2,p_2}_{K}(\Omega)\hookrightarrow
W^{s,p}(\Omega)$. In particular, if $f\in W^{s_1,p_1}_{loc}(\Omega)$, then the mapping $u\mapsto fu$ is a well-defined continuous linear map from $W^{s_2,p_2}_{K}(\Omega)$ to $W^{s,p}(\Omega)$.
\end{enumerate}
\end{lemma}
\begin{remark}
It can be shown that the locally Sobolev spaces on $\Omega$ are metrizable, so the
continuity of the mapping
\begin{equation*}
W^{s_1,p_1}_{loc}(\Omega)\times W^{s_2,p_2}_{loc}(\Omega)\rightarrow
W^{s,p}_{loc}(\Omega),\quad (u,v)\mapsto uv
\end{equation*}
in the above lemma can be interpreted as follows: if $u_i\rightarrow u$ in
$W^{s_1,p_1}_{loc}(\Omega)$ and $v_i\rightarrow v$ in
$W^{s_2,p_2}_{loc}(\Omega)$, then $u_iv_i\rightarrow uv$ in
$W^{s,p}_{loc}(\Omega)$. Also since $W^{s_2,p_2}_{K}(\Omega)$ is considered as a normed
subspace of $W^{s_2,p_2}(\Omega)$, we have a similar interpretation of the continuity of the mapping in item 2.
\end{remark}

\begin{lemma}\lab{lemapp7}
Let $\Omega=\reals^n$ or let $\Omega$ be a nonempty bounded open set in $\reals^n$ with Lipschitz continuous boundary. Let $s\in \reals$ and
$p\in (1,\infty)$ be such that $sp>n$. Let
$B:\Omega\rightarrow \textrm{GL}(k,\reals)$. Suppose for all $x\in \Omega$
and $1\leq i,j\leq k$, $B_{ij}(x)\in W^{s,p}_{loc}(\Omega)$. Then
\begin{enumerate}
\item $\textrm{det}\,B\in W^{s,p}_{loc}(\Omega)$.
\item Moreover if for each $m\in \mathbb{N}$ $B_m: \Omega\rightarrow
\textrm{GL}(k,\reals)$ and for all $1\leq i,j\leq k$
$(B_m)_{ij}\rightarrow B_{ij}$ in $W^{s,p}_{loc}(\Omega)$, then
$\textrm{det}\, B_m\rightarrow \textrm{det}\, B$ in
$W^{s,p}_{loc}(\Omega)$.
\end{enumerate}
\end{lemma}

\begin{theorem}\lab{thmapp9}
Let $\Omega=\reals^n$ or let $\Omega$ be a nonempty bounded open set in $\reals^n$ with Lipschitz continuous boundary. Let $s\geq 1$ and
$p\in (1,\infty)$ be such that $sp>n$.
\begin{enumerate}
\item Suppose that $u\in W^{s,p}_{loc}(\Omega)$ and that $u(x)\in I$ for
all $x\in \Omega$ where $I$ is some interval in $\reals$. If
$F:I\rightarrow \reals$ is a smooth function, then $F(u)\in
W^{s,p}_{loc}(\Omega)$.
\item Suppose that $u_m\rightarrow u$ in $W^{s,p}_{loc}(\Omega)$
and that for all $m\geq 1$ and $x\in \Omega$, $u_m(x),u(x)\in I$
where $I$ is some open interval in $\reals$. If $F:I\rightarrow
\reals$ is a smooth function, then $F(u_m)\rightarrow F(u)$ in
 $W^{s,p}_{loc}(\Omega)$.
\item If $F:\reals\rightarrow \reals$ is a smooth function, then
the map taking $u$ to $F(u)$ is continuous from
$W^{s,p}_{loc}(\Omega)$ to $W^{s,p}_{loc}(\Omega)$.
\end{enumerate}
\end{theorem}

\section{Lebesgue Spaces on Compact Manifolds}
Let $M^n$ be a compact smooth manifold and $E\rightarrow M$
be a smooth vector bundle of rank $r$.
\begin{definition}
A collection $\{(U_\alpha,\varphi_\alpha,\rho_\alpha,\psi_\alpha)\}_{1\leq \alpha\leq N}$ of $4$-tuples is called an \textbf{augmented total trivialization atlas} for $E\rightarrow M$ provided that $\{(U_\alpha,\varphi_\alpha,\rho_\alpha)\}_{1\leq \alpha\leq N}$
is a total trivialization atlas for $E\rightarrow M$ and
$\{\psi_\alpha\}$ is a partition of unity subordinate to the open
cover $\{U_\alpha\}$.
\end{definition}
Let $\{(U_\alpha,\varphi_\alpha,\rho_\alpha,\psi_\alpha)\}_{1\leq \alpha\leq N}$ be an augmented total trivialization atlas for $E\rightarrow M$. Let $g$ be a continuous Riemannian metric on
$M$ and $\langle .,.\rangle_E$ be a fiber metric on $E$ (we
denote the corresponding norm by $|.|_E$). Suppose $1\leq q< \infty$.
\begin{enumerateX}
\item \textbf{Definition 1:} The space $L^q(M,E)$ is the
completion of $C^\infty (M,E)$ with respect to the following norm
\begin{equation*}
\|u\|_{L^q(M,E)}:=\sum_{\alpha=1}^N\sum_{l=1}^r
\|\rho^l_\alpha\circ (\psi_\alpha u ) \circ\varphi_\alpha^{-1}
\|_{L^q(\varphi_\alpha(U_\alpha))}
\end{equation*}
Note that for this definition to make sense it is not necessary
to have metric on $M$ or fiber metric on $E$.
\item \textbf{Definition 2:} The space $L^q(M,E)$ is the
completion of $C^\infty (M,E)$ with respect to the following norm
\begin{equation*}
|u|_{L^q(M,E)}^q:=\int_M|u|_E^q dV_g\qquad 
\end{equation*}
\item \textbf{Definition 3:} The metric $g$ defines a Lebesgue
measure on $M$. Define the following
equivalence relation on $\Gamma(M,E)$:
\begin{equation*}
u\sim v\Longleftrightarrow u=v\,\,a.e.
\end{equation*}
We define
\begin{equation*}
L^q(M,E):=\frac{\{u\in \Gamma(M,E): \|u\|_{L^q(M,E)}^q:=\int_M
|u|_E^q dV_g<\infty\}}{\sim}
\end{equation*}
For $q=\infty$ we define
\begin{equation*}
L^\infty(M,E):=\frac{\{u\in \Gamma(M,E):
\|u\|_{L^\infty(M,E)}:=\textrm{esssup} |u|_E <\infty\}}{\sim}
\end{equation*}
\end{enumerateX}
\textbf{Note:} We may define negligible sets (sets of measure
zero) on a compact manifold using charts (see Chapter 6 in \cite{Lee1}); it can be shown that
this definition is independent of the charts and equivalent to
the one that is obtained using the metric $g$. So it is meaningful
to write $u=v\,\,a.e$ even without using a metric.
\begin{theorem}
Definition 1 is equivalent to Definition 2 (i.e. the norms are
equivalent).
\end{theorem}
\begin{proof}
Our proof consists of four steps:
\begin{itemizeXALI}
\item \textbf{Step 1:} In the next section it will be proved that different total
trivialization atlases and partitions of unity result in
equivalent norms (note that $L^q=W^{0,q}$). Therefore WLOG we may
assume that $\{(U_\alpha, \varphi_\alpha,\rho_\alpha )\}_{1\leq
\alpha\leq N}$ is a total trivialization atlas that trivializes
the fiber metric $\langle .,.\rangle_E$ (see Theorem \ref{thmfalltrivializametric1} and Corollary \ref{winter49}). So on any bundle chart
$(U,\varphi,\rho)$ and for any section $u$ we have
\begin{equation*}
|u|_E^2\circ\varphi^{-1}=\langle u,u\rangle_E\circ \varphi^{-1}
=\sum_{l=1}^r(\rho^l\circ u\circ \varphi^{-1} )^2
\end{equation*}
\item \textbf{Step 2:} In this step we show that if there is $1\leq \beta\leq
N$ such that $\textrm{supp} u\subseteq U_\beta$, then
\begin{equation*}
|u|_{L^q(M,E)}^q=\int_M |u|_E^q dV_g\simeq
\sum_{l=1}^r\|\rho_\beta^l\circ u\circ \varphi_\beta^{-1}
\|^q_{L^q(\varphi_\beta(U_\beta))}
\end{equation*}
We have
\begin{align*}
\int_M |u|_E^q dV_g &=\int_{\varphi_\beta(U_\beta)}(|u|_E\circ
\varphi_\beta^{-1})^p\sqrt{\det
(g_{ij}\circ\varphi_\beta^{-1})(x)}\,dx^1\cdots dx^n\\
&\simeq \int_{\varphi_\beta(U_\beta)}(|u|_E\circ
\varphi_\beta^{-1})^q\,dx^1\cdots dx^n\qquad
{\fontsize{7}{7}{(\textrm{$\sqrt{\det
(g_{ij}\circ\varphi_\beta^{-1})(x)}$ is bounded by positive constants})}}\\
&=\int_{\varphi_\beta(U_\beta)}\bigg(\,\sqrt{\sum_{l=1}^r(\rho^l_\beta\circ
u\circ \varphi^{-1}_\beta )^2}\,\bigg)^q\,dx^1\cdots dx^n\\
&\simeq
\int_{\varphi_\beta(U_\beta)}[\sum_{l=1}^r|\rho^l_\beta\circ
u\circ \varphi^{-1}_\beta |]^q\,dx^1\cdots dx^n\qquad (\sqrt{\sum
a_l^2}\simeq \sum|a_l|)\\
&\simeq
\int_{\varphi_\beta(U_\beta)}\sum_{l=1}^r|\rho^l_\beta\circ
u\circ \varphi^{-1}_\beta |^q\,dx^1\cdots dx^n\qquad ((\sum
 a_l)^q\simeq \sum a_l^q)\\
 &=\sum_{l=1}^r\int_{\varphi_\beta(U_\beta)} |\rho^l_\beta\circ
u\circ \varphi^{-1}_\beta |^q\,dx^1\cdots
dx^n=\sum_{l=1}^r\|\rho_\beta^l\circ u\circ \varphi_\beta^{-1}
\|^q_{L^q(\varphi_\beta(U_\beta))}
\end{align*}
\item \textbf{Step 3:} In this step we will prove that for all $u\in C^\infty(M,E)$
\begin{equation*}
|u|_{L^q(M,E)}^q\simeq\sum_\alpha|\psi_\alpha u|_{L^q(M,E)}^q
\end{equation*}
We have
\begin{align*}
|u|_{L^q(M,E)}^q &=\int_M|u|_E^q dV_g=\sum_\alpha  \int_M
\frac{\psi_\alpha^q}{\sum_\beta \psi_\beta^q}|u|_E^q dV_g\quad
{\fontsize{7}{7}{(\textrm{$\{\frac{\psi_\alpha^q}{\sum_\beta
\psi_\beta^q}\}$ is a
partition of unity subordinate to $\{U_\alpha\}$})}}\\
& \simeq \sum_\alpha \int_{U_\alpha} \psi_\alpha^q |u|_E^q
dV_g\qquad (\textrm{$\frac{1}{\sum_\beta \psi_\beta^q}$ is bounded
by positive constants})\\
&=\sum_\alpha \int_{U_\alpha}|\psi_\alpha u|_E^q dV_g=\sum_\alpha
\int_{M}|\psi_\alpha u|_E^q dV_g\\
&=\sum_\alpha|\psi_\alpha u|_{L^q(M,E)}^q
\end{align*}
\item \textbf{Step 4:} Let $u$ be an arbitrary element of $C^\infty(M,E)$. We have
{\fontsize{11}{11}{\begin{align*}
|u|_{L^q(M,E)}^q\stackrel{\text{Step
3}}{\simeq}\sum_\alpha|\psi_\alpha
u|_{L^q(M,E)}^q\stackrel{\text{Step 2}}{\simeq}\sum_\alpha\sum_l
\|\rho^l_\alpha\circ (\psi_\alpha u ) \circ \varphi_\alpha^{-1}
\|_{L^q(\varphi_\alpha(U_\alpha))}^q\simeq \|u\|_{L^q(M,E)}^q
\end{align*}}}
\end{itemizeXALI}
\end{proof}
\section{Sobolev Spaces on Compact Manifolds and Alternative Characterizations}
\subsection{The Definition}
Let $M^n$ be a compact smooth manifold. Let $\pi: E\rightarrow M$
be a smooth vector bundle of rank $r$. Let $\Lambda=\{(U_\alpha,\varphi_\alpha,\rho_\alpha,\psi_\alpha)\}_{1\leq \alpha\leq
N}$ be an augmented total trivialization atlas for $E\rightarrow M$. For each $1\leq \alpha \leq N$, let $H_\alpha$ denote the map
$H_{E^\vee,U_\alpha, \varphi_\alpha}$ which was introduced in
 Section 6.
\begin{definition}\lab{defwintermainsobolev}
{\fontsize{11}{11}{\begin{equation*}
W^{e,q}(M,E;\Lambda)=\{u\in D'(M,E): \|u\|_{W^{e,q}(M,E;\Lambda)}=\sum_{\alpha=1}^N\sum_{l=1}^r
\|[H_\alpha (\psi_\alpha u
)]^l\|_{W^{e,q}(\varphi_\alpha(U_\alpha))}<\infty\}
\end{equation*}}}
\end{definition}
\begin{remark}
\leavevmode
\begin{enumerateXALI}
\item If $u\in W^{e,q}(M,E;\Lambda)$ is a regular distribution, it follows
from Remark \ref{remfall135} that
\begin{equation*}
\|u\|_{W^{e,q}(M,E;\Lambda)}=\sum_{\alpha=1}^N\sum_{l=1}^r \|(\rho_\alpha)^l\circ
 (\psi_\alpha u)\circ\varphi_\alpha^{-1}\|_{W^{e,q}(\varphi_\alpha(U_\alpha))}
\end{equation*}
\item It is clear that the collection of functions from $M$ to $\reals$ can be identified with sections of the vector bundle $E=M\times \reals$. For this reason
 $W^{e,q}(M;\Lambda)$ is defined as $W^{e,q}(M,M\times \reals;\Lambda)$. Note that in this case, for each $\alpha$, $\rho_\alpha$ is the identity map. So we may consider an augmented total trivialization atlas $\Lambda$ as a collection of $3$-tuples $\{(U_\alpha,\varphi_\alpha,\psi_\alpha)\}_{1\leq \alpha\leq N}$. In
 particular, if $u\in W^{e,q}(M;\Lambda)$ is a regular distribution, then
\begin{equation*}
\| u\|_{W^{e,q}(M;\Lambda)}=\sum_{\alpha=1}^N \| (\psi_\alpha u)\circ
\varphi_\alpha^{-1} \|_{W^{e,q}(\varphi_\alpha(U_\alpha))}
\end{equation*}
\item Sometimes, when the underlying manifold $M$ and the augmented total trivialization atlas are clear from the context (or when they are irrelevant), we may write $W^{e,q}(E)$ instead of $W^{e,q}(M,E;\Lambda)$. In particular, for tensor bundles, we may write $W^{e,q}(T^k_l M)$ instead of $W^{e,q}(M,T^k_l M;\Lambda)$.
 \end{enumerateXALI}
\end{remark}
\begin{remark}\lab{remvarcharac1}
Here is a list of some alternative, not necessarily equivalent, characterizations of Sobolev
spaces.
\begin{enumerate}
\item Suppose $e\geq 0$.
{\fontsize{10}{10}{\begin{equation*}
W^{e,q}(M,E;\Lambda)=\{u\in L^q(M,E): \|u\|_{W^{e,q}(M,E;\Lambda)}=\sum_{\alpha=1}^N\sum_{l=1}^r
\|(\rho_\alpha)^l\circ
 (\psi_\alpha u)\circ\varphi_\alpha^{-1}\|_{W^{e,q}(\varphi_\alpha(U_\alpha))}<\infty\}
\end{equation*}}}
\item 
{\fontsize{10}{10}{\begin{equation*}
W^{e,q}(M,E;\Lambda)=\{u\in D'(M,E): \|u\|_{W^{e,q}(M,E;\Lambda)}=\sum_{\alpha=1}^N\sum_{l=1}^r
\|\textrm{ext}^0_{\varphi_\alpha(U_\alpha),\reals^n}[H_\alpha
(\psi_\alpha u )]^l\|_{W^{e,q}(\reals^n)}<\infty\}
\end{equation*}}}
\item
{\fontsize{11}{11}{\begin{equation*}
W^{e,q}(M,E;\Lambda)=\{u\in D'(M,E): [H_\alpha (u|_{U_\alpha}
)]^l\in W^{e,q}_{loc}(\varphi_\alpha(U_\alpha)),\,\,\forall\, 1\leq \alpha\leq N,\,\forall\, 1\leq l\leq r\}
\end{equation*}}}
\item 
$W^{e,q}(M,E;\Lambda)$ is the completion of
$C^\infty (M,E)$ with respect to the norm
\begin{equation*}
\|u\|_{W^{e,q}(M,E;\Lambda)}=\sum_{\alpha=1}^N\sum_{l=1}^r \|(\rho_\alpha)^l\circ (\psi_\alpha u) \circ
\varphi_\alpha^{-1}\|_{W^{e,q}(\varphi_\alpha(U_\alpha))}
\end{equation*}
\item 
\begin{itemize}
\item Let $g$ be a smooth Riemannian metric (i.e a fiber metric on $TM$).
So $g^{-1}$ is a fiber metric on $T^*M$.
\item Let $\langle .,.\rangle_E$ be a smooth fiber metric on $E$.
\item Let $\grad^E$ be a metric connection in the vector bundle $\pi:E\rightarrow
M$. 
\end{itemize}
For $k\in \mathbb{N}_0$, $W^{k,q}(M,E;g,\grad^E)$ is the completion of
$C^{\infty}(M,E)$ with respect to the following norm
\begin{equation*}
\|u\|_{W^{k,q}(M,E;g,\grad^E)}^q=\sum_{i=0}^k |(\grad^E)^i
u|_{L^q}^q=\sum_{i=0}^k \int_M
|\underbrace{\grad^E\cdots\grad^E}_{\textrm{$i$ times}} u|_{(T^*M)^{\otimes i}\otimes E}^q
 dV_g
\end{equation*}
In particular, if we denote the Levi Civita connection
corresponding to the smooth Riemannian metric $g$ by $\grad$,
then $W^{k,q}(M;g)$ is the completion of $C^{\infty}(M)$ with
respect to the following norm
\begin{align*}
&\|u\|_{W^{k,q}(M;g)}^q=\sum_{i=0}^k | \grad^i
u|_{L^q}^q=\sum_{i=0}^k\int_M |\underbrace{\grad\cdots\grad}_{\textrm{$i$ times}} u|_{T^iM}^q dV_g
\end{align*}
\end{enumerate}
In the subsequent discussions we will study the relation between
each of these alternative descriptions of Sobolev spaces and Definition \ref{defwintermainsobolev}.
\end{remark}
An important question is whether our definition of Sobolev spaces
(as topological spaces) depends on the augmented total trivialization atlas $\Lambda$. We will answer
this question at 3 levels. Although each level can be considered
as a generalization of the preceding level, the proofs will be
independent of each other. The following theorems show that at
least when $e$ is not a noninteger less than $-1$, the
space $W^{e,q}(M,E;\Lambda)$ and its topology are independent
of the choice of augmented total trivialization atlas.
\begin{remark}
In the following theorems, by the equivalence of two norms
$\|.\|_1$ and $\|.\|_2$ we mean there exist constants $C_1$ and
$C_2$ such that
\begin{equation*}
C_1\|.\|_1\leq \|.\|_2\leq C_2\|.\|_1
\end{equation*}
where $C_1$ and $C_2$ may depend on
\begin{equation*}
n,e,q,\varphi_\alpha,U_\alpha,\tilde{\varphi}_\beta,\tilde{U}_\beta,\psi_\alpha,
\tilde{\psi}_\beta
\end{equation*}
\end{remark}
\begin{theorem}[Equivalence of norms for functions]
Let $e\in \reals$ and $q\in (1,\infty)$. Let $\Lambda=\{(U_\alpha,\varphi_\alpha,\psi_\alpha)\}_{1\leq
\alpha\leq N}$ and $\Upsilon=\{(\tilde{U}_\beta,\tilde{\varphi}_\beta,\tilde{\psi}_\beta)\}_{1\leq \beta\leq
\tilde{N}}$ be two augmented total trivialization atlases for the trivial bundle $M\times \reals\rightarrow M$. Also let $\mathcal{W}$ be any
vector subspace of $W^{e,q}(M;\Upsilon)$ whose elements are regular
distributions (e.g $C^\infty(M)$).
\begin{enumerateX}
\item If $e$ is not a noninteger less than $-1$, then $W$ is a subspace of $W^{e,q}(M;\Lambda)$ as well, and the norms
produced by $\Lambda$ and $\Upsilon$ are equivalent on
$\mathcal{W}$.
\item If $e$ is a noninteger less than $-1$, further assume that the total trivialization atlases corresponding to  $\Lambda$ and $\Upsilon$ are GLC. Then $W$ is a subspace of $W^{e,q}(M;\Lambda)$ as well, and the norms
produced by $\Lambda$ and $\Upsilon$ are equivalent on
$\mathcal{W}$.
\end{enumerateX}
\end{theorem}
\begin{proof}
Let $u\in \Gamma_{reg}(M)$. Our goal is to show that the following expressions are comparable:
\begin{align*}
& \sum_{\alpha=1}^N\|(\psi_\alpha u)\circ \varphi_\alpha^{-1}\|_{W^{e,q}(\varphi_\alpha(U_\alpha))}\\
& \sum_{\beta=1}^{\tilde{N}}\|(\tilde{\psi}_\beta u )\circ \tilde{\varphi}_{\beta}^{-1}\|_{W^{e,q}(\tilde{\varphi}_{\beta}(\tilde{U}_\beta))}
\end{align*}
To this end it suffices to show that for each $1\leq \alpha\leq N$
\begin{equation*}
\|(\psi_\alpha u)\circ \varphi_\alpha^{-1}\|_{W^{e,q}(\varphi_\alpha(U_\alpha))}\preceq \sum_{\beta=1}^{\tilde{N}}\|(\tilde{\psi}_\beta u )\circ \tilde{\varphi}_{\beta}^{-1}\|_{W^{e,q}(\tilde{\varphi}_{\beta}(\tilde{U}_\beta))}
\end{equation*}
We have
\begin{align*}
\|(\psi_\alpha u)\circ \varphi_\alpha^{-1}\|_{W^{e,q}(\varphi_\alpha(U_\alpha))}&=\|\sum_{\beta=1}^{\tilde{N}}\tilde{\psi}_\beta(\psi_\alpha u)\circ \varphi_\alpha^{-1}\|_{W^{e,q}(\varphi_\alpha(U_\alpha))}\\
&\leq \sum_{\beta=1}^{\tilde{N}}\|\tilde{\psi}_\beta(\psi_\alpha u)\circ \varphi_\alpha^{-1}\|_{W^{e,q}(\varphi_\alpha(U_\alpha))}\\
&\simeq \sum_{\beta=1}^{\tilde{N}}\|(\tilde{\psi}_\beta\psi_\alpha u)\circ \varphi_\alpha^{-1}\|_{W^{e,q}(\varphi_\alpha(U_\alpha\cap \tilde{U}_\beta))}\\
\end{align*}
The last equality follows from Corollary \ref{corofallusef1} because $(\tilde{\psi}_\beta\psi_\alpha u)\circ \varphi_\alpha^{-1}$ has support in the compact set $\varphi_\alpha (\textrm{supp}\,\psi_\alpha\cap \textrm{supp}\,\tilde{\psi}_\beta)\subseteq \varphi_\alpha(U_\alpha\cap \tilde{U}_\beta)$. Note that here we used the assumption that if $e$ is a noninteger less than $-1$, then $\varphi_\alpha(U_\alpha)$ is Lipschitz or the entire $\reals^n$. Clearly
\begin{equation*}
\sum_{\beta=1}^{\tilde{N}}\|(\tilde{\psi}_\beta\psi_\alpha u)\circ \varphi_\alpha^{-1}\|_{W^{e,q}(\varphi_\alpha(U_\alpha\cap \tilde{U}_\beta))}=\sum_{\beta=1}^{\tilde{N}}\|(\tilde{\psi}_\beta\psi_\alpha u)\circ \tilde{\varphi}_\beta^{-1}\circ \tilde{\varphi}_\beta \circ \varphi_\alpha^{-1}\|_{W^{e,q}(\varphi_\alpha(U_\alpha\cap \tilde{U}_\beta))}
\end{equation*}
Since $\tilde{\varphi}_\beta \circ \varphi_\alpha^{-1}: \varphi_\alpha(U_\alpha\cap \tilde{U}_\beta)\rightarrow \tilde{\varphi}_\beta(U_\alpha\cap\tilde{U}_\beta)$ is a $C^\infty$-diffeomorphism and $(\tilde{\psi}_\beta\psi_\alpha u)\circ \tilde{\varphi}_\beta^{-1}$ has compact support in the compact set $\tilde{\varphi}_\beta (\textrm{supp}\,\psi_\alpha\cap \textrm{supp}\,\tilde{\psi}_\beta)\subseteq \tilde{\varphi}_\beta(U_\alpha\cap \tilde{U}_\beta)$, it follows from Theorem \ref{winter105} that
\begin{equation*}
\sum_{\beta=1}^{\tilde{N}}\|(\tilde{\psi}_\beta\psi_\alpha u)\circ \tilde{\varphi}_\beta^{-1}\circ \tilde{\varphi}_\beta \circ \varphi_\alpha^{-1}\|_{W^{e,q}(\varphi_\alpha(U_\alpha\cap \tilde{U}_\beta))}\preceq \sum_{\beta=1}^{\tilde{N}}\|(\tilde{\psi}_\beta\psi_\alpha u)\circ \tilde{\varphi}_\beta^{-1}\|_{W^{e,q}(\tilde{\varphi}_\beta(U_\alpha\cap \tilde{U}_\beta))}
\end{equation*}
Note that here we used the assumption that if $e$ is a noninteger less than $-1$, then the two total trivialization atlases are GL compatible. As a direct consequence of Corollary \ref{corofallextrensionzeropos1} and Theorem \ref{thmfallextrensionzeroneg1} we have
\begin{align*}
\|(\tilde{\psi}_\beta\psi_\alpha u)\circ \tilde{\varphi}_\beta^{-1}\|_{W^{e,q}(\tilde{\varphi}_\beta(U_\alpha\cap \tilde{U}_\beta))}&\simeq \|(\tilde{\psi}_\beta\psi_\alpha u)\circ \tilde{\varphi}_\beta^{-1}\|_{W^{e,q}(\tilde{\varphi}_\beta(\tilde{U}_\beta))}\\
& =\|(\psi_\alpha\circ \tilde{\varphi}_\beta^{-1}) [(\tilde{\psi}_\beta u)\circ \tilde{\varphi}_\beta^{-1}]\|_{W^{e,q}(\tilde{\varphi}_\beta(\tilde{U}_\beta))}
\end{align*}
Now note that $\psi_\alpha\circ \tilde{\varphi}_\beta^{-1}\in C^\infty (\tilde{\varphi}_\beta(\tilde{U}_\beta))$ and $(\tilde{\psi}_\beta u)\circ \tilde{\varphi}_\beta^{-1}$ has support in the compact set $\tilde{\varphi}_\beta(\textrm{supp}\,\tilde{\psi}_\beta)$. Therefore by Theorem \ref{lemapp3} (for the case where $e$ is not a noninteger less than $-1$) and Corollary \ref{corollarywinter92a} (for the case where $e$ is a noninteger less than $-1$) we have
\begin{equation*}
\|(\psi_\alpha\circ \tilde{\varphi}_\beta^{-1}) [(\tilde{\psi}_\beta u)\circ \tilde{\varphi}_\beta^{-1}]\|_{W^{e,q}(\tilde{\varphi}_\beta(\tilde{U}_\beta))}\preceq \| (\tilde{\psi}_\beta u)\circ \tilde{\varphi}_\beta^{-1}\|_{W^{e,q}(\tilde{\varphi}_\beta(\tilde{U}_\beta))}
\end{equation*}
Hence
\begin{equation*}
\|(\psi_\alpha u)\circ \varphi_\alpha^{-1}\|_{W^{e,q}(\varphi_\alpha(U_\alpha))}\preceq \sum_{\beta=1}^{\tilde{N}} \| (\tilde{\psi}_\beta u)\circ \tilde{\varphi}_\beta^{-1}\|_{W^{e,q}(\tilde{\varphi}_\beta(\tilde{U}_\beta))}
\end{equation*}
\end{proof}
\begin{theorem}[Equivalence of norms for regular sections]\lab{thmwinter20181}
Let $e\in \reals$ and $q\in (1,\infty)$. Let $\Lambda=\{(U_\alpha,\varphi_\alpha,\rho_\alpha,\psi_\alpha)\}_{1\leq
\alpha\leq N}$ and $\Upsilon=\{(\tilde{U}_\beta,\tilde{\varphi}_\beta,\tilde{\rho}_\beta,\tilde{\psi}_\beta)\}_{1\leq \beta\leq
\tilde{N}}$ be two augmented total trivialization atlases for the vector bundle $E\rightarrow M$. Also let $\mathcal{W}$ be any
vector subspace of $W^{e,q}(M,E;\Upsilon)$ whose elements are regular
distributions (e.g $C^\infty(M,E)$).
\begin{enumerateX}
\item If $e$ is not a noninteger less than $-1$, then $W$ is a subspace of $W^{e,q}(M,E;\Lambda)$ as well, and the norms
produced by $\Lambda$ and $\Upsilon$ are equivalent on
$\mathcal{W}$.
\item If $e$ is a noninteger less than $-1$, further assume that the total trivialization atlases corresponding to  $\Lambda$ and $\Upsilon$ are GLC. Then $W$ is a subspace of $W^{e,q}(M,E;\Lambda)$ as well, and the norms
produced by $\Lambda$ and $\Upsilon$ are equivalent on
$\mathcal{W}$.
\end{enumerateX}
\end{theorem}
\begin{proof}
Let $u\in \Gamma_{reg}(M,E)$. Our goal is to show that the following expressions are comparable:
\begin{align*}
& \sum_{\alpha=1}^N\sum_{l=1}^r\|\rho_\alpha^l\circ (\psi_\alpha u)\circ \varphi_\alpha^{-1}\|_{W^{e,q}(\varphi_\alpha(U_\alpha))}\\
& \sum_{\beta=1}^{\tilde{N}}\sum_{l=1}^r\|\tilde{\rho}_\beta^l\circ (\tilde{\psi}_\beta u)\circ \tilde{\varphi}_\beta^{-1}\|_{W^{e,q}(\tilde{\varphi}_{\beta}(\tilde{U}_\beta))}
\end{align*}
To this end, it is enough to show that for each $1\leq \alpha\leq N$ and $1\leq l\leq r$
\begin{equation*}
\|\rho_\alpha^l\circ (\psi_\alpha u)\circ \varphi_\alpha^{-1}\|_{W^{e,q}(\varphi_\alpha(U_\alpha))}\preceq \sum_{\beta=1}^{\tilde{N}}\sum_{t=1}^r\|\tilde{\rho}_\beta^t\circ (\tilde{\psi}_\beta u)\circ \tilde{\varphi}_\beta^{-1}\|_{W^{e,q}(\tilde{\varphi}_{\beta}(\tilde{U}_\beta))}
\end{equation*}
We have
\begin{align*}
\|\rho_\alpha^l\circ (\psi_\alpha u)\circ \varphi_\alpha^{-1}\|_{W^{e,q}(\varphi_\alpha(U_\alpha))}&=\|\rho_\alpha^l\circ (\sum_{\beta=1}^{\tilde{N}}\tilde{\psi}_\beta\psi_\alpha u)\circ \varphi_\alpha^{-1}\|_{W^{e,q}(\varphi_\alpha(U_\alpha))}\\
& \leq \sum_{\beta=1}^{\tilde{N}}\|\rho_\alpha^l\circ (\tilde{\psi}_\beta\psi_\alpha u)\circ \varphi_\alpha^{-1}\|_{W^{e,q}(\varphi_\alpha(U_\alpha))}\\
&\simeq \sum_{\beta=1}^{\tilde{N}}\|\rho_\alpha^l\circ (\tilde{\psi}_\beta\psi_\alpha u)\circ \varphi_\alpha^{-1}\|_{W^{e,q}(\varphi_\alpha(U_\alpha\cap\tilde{U}_\beta))}
\end{align*}
The last equality follows from Corollary \ref{corofallusef1} because $\rho_\alpha^l\circ (\tilde{\psi}_\beta\psi_\alpha u)\circ \varphi_\alpha^{-1}$ has support in the compact set $\varphi_\alpha (\textrm{supp}\,\psi_\alpha\cap \textrm{supp}\,\tilde{\psi}_\beta)\subseteq \varphi_\alpha(U_\alpha\cap \tilde{U}_\beta)$. Note that here we used the assumption that if $e$ is a noninteger less than $-1$, then $\varphi_\alpha(U_\alpha)$ is either Lipschitz or equal to the entire $\reals^n$. Note that
\begin{align*}
\sum_{\beta=1}^{\tilde{N}}\|\rho_\alpha^l\circ (\tilde{\psi}_\beta\psi_\alpha u)&\circ \varphi_\alpha^{-1}\|_{W^{e,q}(\varphi_\alpha(U_\alpha\cap\tilde{U}_\beta))}\\
&=\sum_{\beta=1}^{\tilde{N}}\|\rho_\alpha^l\circ (\tilde{\psi}_\beta\psi_\alpha u)\circ \tilde{\varphi}_\beta^{-1}\circ \tilde{\varphi}_\beta\circ \varphi_\alpha^{-1}\|_{W^{e,q}(\varphi_\alpha(U_\alpha\cap\tilde{U}_\beta))}\\
&\stackrel{\textrm{Theorem \ref{winter105}}}{\preceq} \sum_{\beta=1}^{\tilde{N}}\|\rho_\alpha^l\circ (\tilde{\psi}_\beta\psi_\alpha u)\circ \tilde{\varphi}_\beta^{-1}\|_{W^{e,q}(\tilde{\varphi}_\beta(U_\alpha\cap\tilde{U}_\beta))}\\
&=\sum_{\beta=1}^{\tilde{N}}\|(\psi_\alpha\circ\tilde{\varphi}_\beta^{-1})[\rho_\alpha^l\circ (\tilde{\psi}_\beta u)\circ \tilde{\varphi}_\beta^{-1}]\|_{W^{e,q}(\tilde{\varphi}_\beta(U_\alpha\cap\tilde{U}_\beta))}\\
&=\sum_{\beta=1}^{\tilde{N}}\|(\psi_\alpha\circ\tilde{\varphi}_\beta^{-1})\big[\pi_l\circ \underbrace{\pi'\circ\Phi_\alpha}_{\rho_\alpha}\circ (\tilde{\psi}_\beta  u)\circ \tilde{\varphi}_\beta^{-1}\big]\|_{W^{e,q}(\tilde{\varphi}_\beta(U_\alpha\cap\tilde{U}_\beta))}\\
&=\sum_{\beta=1}^{\tilde{N}}\|(\psi_\alpha\circ\tilde{\varphi}_\beta^{-1})\big[\pi_l\circ \pi'\circ\Phi_\alpha\circ \Phi_\beta^{-1}\circ\Phi_\beta\circ (\tilde{\psi}_\beta  u)\circ \tilde{\varphi}_\beta^{-1}\big]\|_{W^{e,q}(\tilde{\varphi}_\beta(U_\alpha\cap\tilde{U}_\beta))}
\end{align*}
Let $v_\beta:\tilde{\varphi}_\beta(\tilde{U}_\beta)\rightarrow E$ be defined by $v_\beta(x)=(\tilde{\psi}_\beta  u)\circ \tilde{\varphi}_\beta^{-1}$. Clearly $\pi(v_\beta(x))=\tilde{\varphi}_\beta^{-1}(x)$. Therefore
\begin{equation*}
\Phi_\beta(v_\beta(x))=\big(\pi(v_\beta(x)),\tilde{\rho}_\beta(v_\beta(x)) \big)=\big(\tilde{\varphi}_\beta^{-1}(x),\tilde{\rho}_\beta(v_\beta(x))\big)
\end{equation*}
For all $x\in \tilde{\varphi}_\beta(U_\alpha\cap\tilde{U}_\beta)$ we have
\begin{align*}
\pi'&\circ\Phi_\alpha\circ \Phi_\beta^{-1}\big(\Phi_\beta(v_\beta(x))\big)\\
&=\pi'\circ\Phi_\alpha\circ\Phi_\beta^{-1}\big(\tilde{\varphi}_\beta^{-1}(x),\tilde{\rho}_\beta(v_\beta(x))\big)\\
&\stackrel{\textrm{Lemma \ref{lemapp8}}}{=} \pi'\circ\big(\tilde{\varphi}_\beta^{-1}(x),\tau_{\alpha\beta}(\tilde{\varphi}_\beta^{-1}(x))\tilde{\rho}_\beta(v_\beta(x))\big)\\
&=\underbrace{\tau_{\alpha\beta}(\tilde{\varphi}_\beta^{-1}(x))}_{\textrm{an $r\times r$ matrix}}\tilde{\rho}_\beta(v_\beta(x))
\end{align*}
Let $A_{\alpha\beta}=\tau_{\alpha\beta}\circ\tilde{\varphi}_\beta^{-1}$ on $\tilde{\varphi}_\beta(U_\alpha\cap\tilde{U}_\beta)$. So we can write
\begin{align*}
\|\rho_\alpha^l&\circ (\psi_\alpha u)\circ \varphi_\alpha^{-1}\|_{W^{e,q}(\tilde{\varphi}_\beta(U_\alpha\cap\tilde{U}_\beta))}\\
&\preceq \sum_{\beta=1}^{\tilde{N}}\|(\psi_\alpha\circ\tilde{\varphi}_\beta^{-1})(x)\big[\pi_l\circ A_{\alpha\beta}(x)\tilde{\rho}_\beta(v_\beta(x))\big]\|_{W^{e,q}(\tilde{\varphi}_\beta(U_\alpha\cap\tilde{U}_\beta))}\\
&=\sum_{\beta=1}^{\tilde{N}}\|(\psi_\alpha\circ\tilde{\varphi}_\beta^{-1})(x)\big[\sum_{t=1}^r (A_{\alpha\beta}(x))_{lt}\tilde{\rho}_\beta^t(v_\beta(x))\big]\|_{W^{e,q}(\tilde{\varphi}_\beta(U_\alpha\cap\tilde{U}_\beta))}\\
&\leq \sum_{\beta=1}^{\tilde{N}}\sum_{t=1}^r \| (\psi_\alpha\circ\tilde{\varphi}_\beta^{-1})(x) (A_{\alpha\beta}(x))_{lt}\tilde{\rho}_\beta^t(v_\beta(x))\|_{W^{e,q}(\tilde{\varphi}_\beta(U_\alpha\cap\tilde{U}_\beta))}
\end{align*}
Now note that $(A_{\alpha\beta}(x))_{lt}$ are in $C^\infty(\tilde{\varphi}_\beta(U_\alpha\cap\tilde{U}_\beta))$
 and $(\psi_\alpha\circ\tilde{\varphi}_\beta^{-1})(x)\tilde{\rho}_\beta^t(v_\beta(x))$ has support inside the compact set $\tilde{\varphi}_{\beta}(\textrm{supp}\,\tilde{\psi_\beta}\cap \textrm{supp}\,\psi_\alpha)$. Therefore by Theorem \ref{lemapp3} (for the case where $e$ is not a noninteger less than $-1$) and Corollary \ref{corollarywinter92a} (for the case where $e$ is a noninteger less than $-1$) we have
\begin{equation*}
\sum_{t=1}^r\| (\psi_\alpha\circ\tilde{\varphi}_\beta^{-1})(x) (A_{\alpha\beta}(x))_{lt}\tilde{\rho}_\beta^t(v_\beta(x))\|_{W^{e,q}(\tilde{\varphi}_\beta(U_\alpha\cap\tilde{U}_\beta))}\preceq \sum_{t=1}^r\| (\psi_\alpha\circ\tilde{\varphi}_\beta^{-1})(x)\tilde{\rho}_\beta^t(v_\beta(x))\|_{W^{e,q}(\tilde{\varphi}_\beta(U_\alpha\cap\tilde{U}_\beta))}
\end{equation*}
Therefore
\begin{align*}
\|\rho_\alpha^l&\circ (\psi_\alpha u)\circ \varphi_\alpha^{-1}\|_{W^{e,q}(\varphi_\alpha(U_\alpha))}\\
&\preceq \sum_{\beta=1}^{\tilde{N}}\sum_{t=1}^r \| (\psi_\alpha\circ\tilde{\varphi}_\beta^{-1})(x)  \tilde{\rho}_\beta^t(v_\beta(x))\|_{W^{e,q}(\tilde{\varphi}_\beta(U_\alpha\cap\tilde{U}_\beta))}\\
&\simeq \sum_{\beta=1}^{\tilde{N}}\sum_{t=1}^r \| (\psi_\alpha\circ\tilde{\varphi}_\beta^{-1})(x)  \tilde{\rho}_\beta^t(v_\beta(x))\|_{W^{e,q}(\tilde{\varphi}_\beta(\tilde{U}_\beta))}\\
&\textrm{(Here we used Corollary \ref{corofallextrensionzeropos1} and Theorem \ref{thmfallextrensionzeroneg1})}\\
&\preceq \sum_{\beta=1}^{\tilde{N}}\sum_{t=1}^r \|   \tilde{\rho}_\beta^t(v_\beta(x))\|_{W^{e,q}(\tilde{\varphi}_\beta(\tilde{U}_\beta))}\\
&\textrm{(Here we used Theorem \ref{lemapp3} and Corollary \ref{corollarywinter92a})}\\
&=\sum_{\beta=1}^{\tilde{N}}\sum_{t=1}^r \| \tilde{\rho}_\beta^t\circ (\tilde{\psi}_\beta u)\circ \tilde{\varphi}_\beta^{-1}\|_{W^{e,q}(\tilde{\varphi}_\beta(\tilde{U}_\beta))}
\end{align*}
\end{proof}
\begin{theorem}[Equivalence of norms for distributional sections]
Let $e\in \reals$ and $q\in (1,\infty)$. Let $\Lambda=\{(U_\alpha,\varphi_\alpha,\rho_\alpha,\psi_\alpha)\}_{1\leq
\alpha\leq N}$ and $\Upsilon=\{(\tilde{U}_\beta,\tilde{\varphi}_\beta,\tilde{\rho}_\beta,\tilde{\psi}_\beta)\}_{1\leq \beta\leq
\tilde{N}}$ be two augmented total trivialization atlases for the vector bundle $E\rightarrow M$.
\begin{enumerateX}
\item If $e$ is not a noninteger less than $-1$, then $W^{e,q}(M,E;\Lambda)$ and $W^{e,q}(M,E;\Upsilon)$ are equivalent normed spaces.
\item If $e$ is a noninteger less than $-1$, further assume that the total trivialization atlases corresponding to  $\Lambda$ and $\Upsilon$ are GLC. Then $W^{e,q}(M,E;\Lambda)$ and $W^{e,q}(M,E;\Upsilon)$ are equivalent normed spaces.
\end{enumerateX}
\end{theorem}
\begin{proof}
 Let $u\in D'(M,E)$. We want to show the following expressions are comparable:
\begin{align*}
& \sum_{\alpha=1}^N \sum_{l=1}^r\|[H_\alpha (\psi_\alpha u)
]^l\|_{W^{e,q}(\varphi_\alpha(U_\alpha))}\\
& \sum_{\beta=1}^{\tilde{N}} \sum_{i=1}^r \|[\tilde{H}_\beta(\tilde{\psi}_\beta u
)]^i\|_{W^{e,q}(\tilde{\varphi}_\beta(\tilde{U}_\beta))}
\end{align*}
To this end it is enough to show that for each $1\leq \alpha \leq
 N$ and $1\leq l\leq r$
 \begin{equation*}
\|[H_\alpha (\psi_\alpha u)
]^l\|_{W^{e,q}(\varphi_\alpha(U_\alpha))}\preceq \sum_{\beta=1}^{\tilde{N}}
\sum_{i=1}^r \|[\tilde{H}_\beta(\tilde{\psi}_\beta u
)]^i\|_{W^{e,q}(\tilde{\varphi}_\beta(\tilde{U}_\beta))}
 \end{equation*}
 We have
 \begin{equation*}
 [H_\alpha (\psi_\alpha u)
]^l=[H_\alpha (\sum_{\beta=1}^{\tilde{N}} \tilde{\psi}_\beta\psi_\alpha u)
]^l\stackrel{\textrm{Remark \ref{remfall134}}}{=}\sum_{\beta=1}^{\tilde{N}}[H_\alpha (\tilde{\psi}_\beta\psi_\alpha u) ]^l
 \end{equation*}
 In what follows we will prove that
 \begin{equation}\lab{winterprovlate}
 [H_\alpha (\tilde{\psi}_\beta\psi_\alpha u)
]^l=\sum_{i=1}^r \big((A_{\alpha\beta})_{il}[\tilde{H}_\beta(\tilde{\psi}_\beta\psi_\alpha u)]^i \big)\circ \tilde{\varphi}_\beta\circ\varphi_\alpha^{-1}
 \end{equation}
 for some functions $(A_{\alpha\beta})_{il},\,(1\leq i\leq r)$ in $C^\infty(\tilde{\varphi}_\beta(U_\alpha\cap\tilde{U}_\beta))$. For now let's assume the validity of Equation \ref{winterprovlate} to prove the claim.
 \begin{align*}
\| [H_\alpha(\psi_\alpha
u)&]^l\|_{W^{e,q}(\varphi_\alpha(U_\alpha))}=\| \sum_{\beta=1}^{\tilde{N}}
[H_\alpha(\tilde{\psi}_\beta\psi_\alpha u)]^l
\|_{W^{e,q}(\varphi_\alpha(U_\alpha))}\\
&\leq \sum_{\beta=1}^{\tilde{N}} \|
[H_\alpha(\tilde{\psi}_\beta\psi_\alpha u)]^l
\|_{W^{e,q}(\varphi_\alpha(U_\alpha))}\\
&\stackrel{\textrm{Corollary \ref{corofallusef1}}}{\simeq}\sum_{\beta=1}^{\tilde{N}}\|
[H_\alpha(\tilde{\psi}_\beta\psi_\alpha u)]^l
\|_{W^{e,q}(\varphi_\alpha(U_\alpha\cap \tilde{U}_\beta))}\\
& {}\hspace{-1 cm}{\fontsize{10}{10}{\textrm{(note that by Remark \ref{remfall134} $[H_\alpha(\tilde{\psi}_\beta\psi_\alpha u)]^l$ has support in the compact set $\varphi_\alpha(\textrm{supp}\,\psi_\alpha\cap\textrm{supp}\,\tilde{\psi}_\beta)$)}}}\\
&=\sum_{\beta=1}^{\tilde{N}} \| \sum_{i=1}^r \big((A_{\alpha\beta})_{il}[\tilde{H}_\beta(\tilde{\psi}_\beta\psi_\alpha u)]^i \big)\circ \tilde{\varphi}_\beta\circ\varphi_\alpha^{-1} \|_{W^{e,q}(\varphi_\alpha(U_\alpha\cap \tilde{U}_\beta))}\\
&\leq \sum_{\beta=1}^{\tilde{N}} \sum_{i=1}^r\|  \big((A_{\alpha\beta})_{il}[\tilde{H}_\beta(\tilde{\psi}_\beta\psi_\alpha u)]^i \big)\circ \tilde{\varphi}_\beta\circ\varphi_\alpha^{-1} \|_{W^{e,q}(\varphi_\alpha(U_\alpha\cap \tilde{U}_\beta))}
\end{align*}
\begin{align*}
&\stackrel{\textrm{Theorem \ref{winter105}}}{\preceq} \sum_{\beta=1}^{\tilde{N}} \sum_{i=1}^r\|  (A_{\alpha\beta})_{il}[\tilde{H}_\beta(\tilde{\psi}_\beta\psi_\alpha u)]^i \|_{W^{e,q}(\tilde{\varphi}_\beta(U_\alpha\cap \tilde{U}_\beta))}\\
&=\sum_{\beta=1}^{\tilde{N}} \sum_{i=1}^r\|  (A_{\alpha\beta})_{il}(\psi_\alpha\circ\tilde{\varphi}_\beta^{-1})[\tilde{H}_\beta(\tilde{\psi}_\beta u)]^i \|_{W^{e,q}(\tilde{\varphi}_\beta(U_\alpha\cap \tilde{U}_\beta))}\\
&\preceq\sum_{\beta=1}^{\tilde{N}} \sum_{i=1}^r\|  (\psi_\alpha\circ\tilde{\varphi}_\beta^{-1})[\tilde{H}_\beta(\tilde{\psi}_\beta u)]^i \|_{W^{e,q}(\tilde{\varphi}_\beta(U_\alpha\cap \tilde{U}_\beta))}\\
&\simeq\sum_{\beta=1}^{\tilde{N}} \sum_{i=1}^r\|  (\psi_\alpha\circ\tilde{\varphi}_\beta^{-1})[\tilde{H}_\beta(\tilde{\psi}_\beta u)]^i \|_{W^{e,q}(\tilde{\varphi}_\beta(\tilde{U}_\beta))}\\
&\textrm{(Here we used Corollary \ref{corofallextrensionzeropos1} and Theorem \ref{thmfallextrensionzeroneg1})}\\
&\preceq \sum_{\beta=1}^{\tilde{N}} \sum_{i=1}^r\| [\tilde{H}_\beta(\tilde{\psi}_\beta u)]^i \|_{W^{e,q}(\tilde{\varphi}_\beta(\tilde{U}_\beta))}\\
&\textrm{(Here we used Theorem \ref{lemapp3} and Corollary \ref{corollarywinter92a})}\\
\end{align*}
So it remains to prove Equation \ref{winterprovlate}. Since $\textrm{supp}[H_\alpha (\tilde{\psi}_\beta\psi_\alpha u) ]^l$ is inside the compact set $\varphi_\alpha(\textrm{supp}\psi_\alpha\cap \textrm{supp}\tilde{\psi}_\beta)\subseteq \varphi_\alpha(U_\alpha\cap\tilde{U}_\beta)$, it is enough to consider the action of $[H_\alpha (\tilde{\psi}_\beta\psi_\alpha u) ]^l$ on elements of $C_c^\infty(\varphi_\alpha(U_\alpha\cap\tilde{U}_\beta))$. $\tilde{\varphi}_\beta\circ \varphi_\alpha^{-1}:\varphi_\alpha (U_\alpha\cap \tilde{U}_\beta) \rightarrow \tilde{\varphi}_\beta (U_\alpha\cap
\tilde{U}_\beta)$ is a $C^\infty$-diffeomorphism. Therefore the
map
\begin{equation*}
C_c^{\infty}[\tilde{\varphi}_\beta (U_\alpha\cap
\tilde{U}_\beta)]\rightarrow C_c^{\infty}[\varphi_\alpha
(U_\alpha\cap \tilde{U}_\beta)],\qquad \eta\mapsto \eta\circ
\tilde{\varphi}_\beta\circ \varphi_\alpha^{-1}
\end{equation*}
is bijective. In particular, an arbitrary element of
$C_c^{\infty}[\varphi_\alpha (U_\alpha\cap \tilde{U}_\beta)]$ has
the form $\eta\circ \tilde{\varphi}_\beta\circ
\varphi_\alpha^{-1}$ where $\eta$ is an element of
$C_c^{\infty}[\tilde{\varphi}_\beta (U_\alpha\cap
\tilde{U}_\beta)]$.\\
For all $\eta \in C_c^{\infty}[\tilde{\varphi}_\beta (U_\alpha\cap
\tilde{U}_\beta)]$ we have (see Section 6.2.2)
\begin{equation}\lab{eqnwinter41a}
\langle [H_\alpha (\tilde{\psi}_\beta\psi_\alpha u) ]^l, \eta
\circ \tilde{\varphi}_\beta\circ
\varphi_\alpha^{-1}\rangle=\langle \tilde{\psi}_\beta\psi_\alpha
u, g^\alpha_{l,\eta\circ\tilde{\varphi}_\beta\circ
\varphi_\alpha^{-1}} \rangle
\end{equation}
where $g^\alpha_{l,\eta\circ\tilde{\varphi}_\beta\circ
\varphi_\alpha^{-1}}$ stands for $g_{l,\eta\circ\tilde{\varphi}_\beta\circ
\varphi_\alpha^{-1},U_\alpha,\varphi_\alpha}$.\\
For all $y\in \varphi_\alpha (U_\alpha\cap \tilde{U}_\beta)$
we have $(x=\varphi_\alpha^{-1}(y))$
\begin{align*}
& \rho_{\alpha}^\vee|_{E_x^\vee}\circ
g^\alpha_{l,\eta\circ \tilde{\varphi}_\beta\circ
\varphi_{\alpha}^{-1}} \circ
\underbrace{\varphi_\alpha^{-1}(y)}_{x}
 =(0,\cdots,0,\underbrace{\eta\circ \tilde{\varphi}_\beta\circ \varphi_{\alpha}^{-1}(y)}_{\textrm{$l^{th}$ position}},0,\cdots,0)\\
& \tilde{\rho}_{\beta}^\vee\circ
\tilde{g}_{l,\eta}^\beta\circ
\underbrace{\tilde{\varphi}_\beta^{-1}(\tilde{\varphi}_\beta\circ\varphi_\alpha^{-1}(y))}_{x}=(0,\cdots,0,\underbrace{\eta\circ
\tilde{\varphi}_\beta\circ \varphi_{\alpha}^{-1}(y)}_{\textrm{$l^{th}$ position}},0,\cdots,0)
\end{align*}
Therefore for all $y\in \varphi_\alpha (U_\alpha\cap
\tilde{U}_\beta)$
\begin{equation*}
\rho_{\alpha}^\vee|_{E_x^\vee}\circ g^\alpha_{l,\eta\circ
\tilde{\varphi}_\beta\circ \varphi_{\alpha}^{-1}} \circ
\varphi_\alpha^{-1}(y)=\tilde{\rho}_{\beta}^\vee\circ
\tilde{g}_{l,\eta}^\beta\circ \varphi_\alpha^{-1}(y)
\end{equation*}
which implies that on $U_\alpha\cap \tilde{U}_\beta$
\begin{equation}\lab{eqnwinter41b}
g^\alpha_{l,\eta\circ \tilde{\varphi}_\beta\circ
\varphi_{\alpha}^{-1}}=[\rho_{\alpha}^\vee|_{E_x^\vee}]^{-1}\circ
[\tilde{\rho}_{\beta}^\vee|_{E_x^\vee}]\circ
\tilde{g}^\beta_{l,\eta}
\end{equation}
It follows from Lemma \ref{lemapp8} that for all $a\in E_x^\vee$
\begin{equation*}
[\tilde{\rho}_{\beta}^\vee|_{E_x^\vee}]\circ
[\rho_{\alpha}^\vee|_{E_x^\vee}]^{-1}\circ
[\tilde{\rho}_{\beta}^\vee|_{E_x^\vee}](a)=\underbrace{\tau^{\tilde{\beta}\alpha}(x)}_{r\times
r }(\tilde{\rho}_{\beta}^\vee|_{E_x^\vee}(a))
\end{equation*}
That is,
\begin{equation*}
[\rho_{\alpha}^\vee|_{E_x^\vee}]^{-1}\circ
[\tilde{\rho}_{\beta}^\vee|_{E_x^\vee}](a)=[\tilde{\rho}_{\beta}^\vee|_{E_x^\vee}]^{-1}
[\tau^{\tilde{\beta}\alpha}(x)(\tilde{\rho}_{\beta}^\vee|_{E_x^\vee}(a))]
\end{equation*}
For $a=\tilde{g}^\beta_{l,\eta}(x)$ we have
\begin{equation*}
\tilde{\rho}_{\beta}^\vee|_{E_x^\vee}(a)=\tilde{\rho}_{\beta}^\vee|_{E_x^\vee}(\tilde{g}^\beta_{l,\eta}(x))=(0,\cdots,0,\underbrace{\eta\circ
\tilde{\varphi}_\beta(x)}_{\textrm{$l^{th}$ position}},0,\cdots,0 )
\end{equation*}
So
{\fontsize{11}{11}{\begin{align}
[\rho_{\alpha}^\vee|_{E_x^\vee}]^{-1}\circ
[\tilde{\rho}_{\beta}^\vee|_{E_x^\vee}]\circ
\tilde{g}^\beta_{l,\eta}&=[\tilde{\rho}_{\beta}^\vee|_{E_x^\vee}]^{-1}
[\tau^{\tilde{\beta}\alpha}(x)(\tilde{\rho}_{\beta}^\vee|_{E_x^\vee}(\tilde{g}^\beta_{l,\eta}(x)))]
 =[\tilde{\rho}_{\beta}^\vee|_{E_x^\vee}]^{-1}\big((\eta\circ \tilde{\varphi}_\beta)\begin{bmatrix}
  \tau^{\tilde{\beta}\alpha}_{1l}\\\vdots\\\tau^{\tilde{\beta}\alpha}_{rl}\end{bmatrix}\big)\notag\\
  &=[\tilde{\rho}_{\beta}^\vee|_{E_x^\vee}]^{-1}\big(\begin{bmatrix}(\eta\circ \tilde{\varphi}_\beta) \tau^{\tilde{\beta}\alpha}_{1l}\\0\\\vdots\\0\end{bmatrix}
   +\cdots+\begin{bmatrix}0\\\vdots\\0\\(\eta\circ \tilde{\varphi}_\beta)
   \tau^{\tilde{\beta}\alpha}_{rl}\end{bmatrix}\big)\notag\\
   &=
   \tilde{g}^\beta_{1,(\tau^{\tilde{\beta}\alpha}_{1l}\circ\tilde{\varphi}_\beta^{-1})\eta}+\cdots+
   \tilde{g}^\beta_{r,(\tau^{\tilde{\beta}\alpha}_{rl}\circ\tilde{\varphi}_\beta^{-1})\eta}\lab{eqnwinter41c}
\end{align}}}
It follows from (\ref{eqnwinter41a}),  (\ref{eqnwinter41b}), and (\ref{eqnwinter41c}) that for all $\eta\in C_c^{\infty}[\tilde{\varphi}_\beta (U_\alpha\cap
\tilde{U}_\beta)]$
\begin{align*}
\langle [ &H_\alpha(\tilde{\psi}_\beta\psi_\alpha u
)]^l,\eta\circ \tilde{\varphi}_\beta\circ
\varphi_\alpha^{-1}\rangle =\langle \tilde{\psi}_\beta\psi_\alpha u,
[\rho_{\alpha}^\vee|_{E_x^\vee}]^{-1}\circ
[\tilde{\rho}_{\beta}^\vee|_{E_x^\vee}]\circ
\tilde{g}^\beta_{l,\eta}\rangle\\
&=\langle
\tilde{\psi}_\beta\psi_\alpha u,\sum_{i=1}^r
\tilde{g}^\beta_{i,(\tau^{\tilde{\beta}\alpha}_{il}\circ\tilde{\varphi}_\beta^{-1})\eta}\rangle\\
&= \sum_{i=1}^r \langle [\tilde{H}_\beta(\tilde{\psi}_\beta\psi_\alpha
u)]^i,
(\tau^{\tilde{\beta}\alpha}_{il}\circ\tilde{\varphi}_\beta^{-1})\eta\rangle\\
&=\sum_{i=1}^r
 \langle (\tau^{\tilde{\beta}\alpha}_{il}\circ\tilde{\varphi}_\beta^{-1})[\tilde{H}_\beta(\tilde{\psi}_\beta\psi_\alpha
u)]^i, \eta\rangle\\
&=\sum_{i=1}^r
 \langle (\tau^{\tilde{\beta}\alpha}_{il}\circ\tilde{\varphi}_\beta^{-1})[\tilde{H}_\beta(\tilde{\psi}_\beta\psi_\alpha
u)]^i, \eta\circ \tilde{\varphi}_\beta\circ
\varphi_\alpha^{-1}\circ(\varphi_\alpha\circ\tilde{\varphi}_\beta^{-1})\rangle\\
&=\sum_{i=1}^r
 \langle \frac{1}{\textrm{det}(\varphi_\alpha\circ\tilde{\varphi}_\beta^{-1})}(\tau^{\tilde{\beta}\alpha}_{il}\circ\tilde{\varphi}_\beta^{-1})[\tilde{H}_\beta(\tilde{\psi}_\beta\psi_\alpha
u)]^i\circ\tilde{\varphi}_\beta\circ\varphi_\alpha^{-1}, \eta\circ \tilde{\varphi}_\beta\circ
\varphi_\alpha^{-1}\rangle
\end{align*}
For the last equality we used the following identity
\begin{equation*}
\langle \frac{1}{\textrm{det}T^{-1}}(u\circ T), \varphi\rangle=\langle u, \varphi\circ T^{-1}\rangle
\end{equation*}
Hence
\begin{equation*}
[H_\alpha(\tilde{\psi}_\beta\psi_\alpha u
)]^l=\sum_{i=1}^r
 \frac{1}{\textrm{det}(\varphi_\alpha\circ\tilde{\varphi}_\beta^{-1})}(\tau^{\tilde{\beta}\alpha}_{il}\circ\tilde{\varphi}_\beta^{-1})[\tilde{H}_\beta(\tilde{\psi}_\beta\psi_\alpha
u)]^i\circ\tilde{\varphi}_\beta\circ\varphi_\alpha^{-1}
\end{equation*}
and consequently letting
\begin{equation*}
(A_{\alpha\beta})_{il}=\frac{1}{\textrm{det}(\varphi_\alpha\circ\tilde{\varphi}_\beta^{-1})}(\tau^{\tilde{\beta}\alpha}_{il}\circ\tilde{\varphi}_\beta^{-1})
\end{equation*}
leads to (\ref{winterprovlate}).
\end{proof}
\begin{remark}
Note that the above theorems establish the full independence of {\fontsize{8}{8}{$W^{e,q}(M,E;\Lambda)$}} from $\Lambda$ at least when $e$ is not a noninteger less than $-1$. So it is justified to write $W^{e,q}(M,E)$ instead of $W^{e,q}(M,E;\Lambda)$ at least when $e$ is not a noninteger less than $-1$. Also see Remark \ref{remwinterdualequiv}.
\end{remark}
\subsection{The Properties}
\subsubsection{Multiplication Properties}
\begin{theorem}\lab{thmwinter1100}
Let $M^n$ be a compact smooth manifold and $E\rightarrow M$ be a vector
bundle with rank $r$. Let $\Lambda=\{(U_\alpha,\varphi_\alpha,\rho_\alpha,\psi_\alpha)\}_{1\leq \alpha
\leq N }$ be an augmented total trivialization atlas for $E$. Suppose $e\in \reals$, $q\in (1,\infty)$,
$\eta\in C^\infty(M)$. If $e$ is a noninteger less than $-1$, further assume that the total trivialization atlas of $\Lambda$ is GGL. Then the linear map
\begin{equation*}
m_\eta: W^{e,q}(M,E;\Lambda)\rightarrow W^{e,q}(M,E;\Lambda),\quad u\mapsto \eta u
\end{equation*}
is well-defined and bounded.
\end{theorem}
\begin{proof}
\begin{align*}
\|\eta u\|_{W^{e,q}(M,E;\Lambda)}:&=\sum_{\alpha=1}^N\sum_{l=1}^r
\|(H_\alpha(\psi_\alpha \eta
u))^l\|_{W^{e,q}(\varphi_\alpha(U_\alpha))}\\
&\stackrel{Remark
\ref{remfall134}}{=}\sum_{\alpha=1}^N\sum_{l=1}^r \|(\eta\circ
\varphi_\alpha^{-1} )(H_\alpha(\psi_\alpha
u))^l\|_{W^{e,q}(\varphi_\alpha(U_\alpha))}\\
&\preceq \sum_{\alpha=1}^N\sum_{l=1}^r \|(H_\alpha(\psi_\alpha
u))^l\|_{W^{e,q}(\varphi_\alpha(U_\alpha))}=\| u\|_{W^{e,q}(M,E;\Lambda)}
\end{align*}
For the case where $e$ is not a noninteger less than $-1$, the
last inequality follows from Theorem  \ref{lemapp3}. If $e$ is a
noninteger less than $-1$, then by assumption
$\varphi_\alpha(U_\alpha)$ is either entire $\reals^n$ or is Lipschitz, and the last inequality
is due to Theorem \ref{winter87} and Corollary \ref{corollarywinter92a}.
\end{proof}
\begin{theorem}
Let $M^n$ be a compact smooth manifold and $E\rightarrow M$ be a vector
bundle with rank $r$. Let $\Lambda$ be an augmented total trivialization atlas for $E$. Let $s_1, s_2, s\in \reals$ and $p_1,p_2,p\in (1,\infty)$. If any of $s_1$, $s_2$, or $s$ is a noninteger less than $-1$, further assume that the total trivialization atlas of $\Lambda$ is GL compatible with itself.
\begin{enumerateX}
\item If $s_1$, $s_2$, and $s$ are not nonintegers less than $-1$, and if $W^{s_1,p_1}(\reals^n)\times
W^{s_2,p_2}(\reals^n)\hookrightarrow W^{s,p}(\reals^n)$, then
\begin{equation*}
W^{s_1,p_1}(M;\Lambda)\times W^{s_2,p_2}(M,E;\Lambda)\hookrightarrow W^{s,p}(M,E;\Lambda)
\end{equation*}
\item If $s_1$, $s_2$, and $s$ are not nonintegers less than $-1$, and if $W^{s_1,p_1}(\Omega)\times
W^{s_2,p_2}(\Omega)\hookrightarrow W^{s,p}(\Omega)$, for any open ball $\Omega$,  then
\begin{equation*}
W^{s_1,p_1}(M;\Lambda)\times W^{s_2,p_2}(M,E;\Lambda)\hookrightarrow W^{s,p}(M,E;\Lambda)
\end{equation*}
\item If any of $s_1$, $s_2$, or $s$ is a noninteger less than $-1$, and if $W^{s_1,p_1}(\Omega)\times
W^{s_2,p_2}(\Omega)\hookrightarrow W^{s,p}(\Omega)$ for $\Omega=\reals^n$ \textbf{and} for any bounded open set $\Omega$ with Lipschitz continuous boundary, then
\begin{equation*}
W^{s_1,p_1}(M;\Lambda)\times W^{s_2,p_2}(M,E;\Lambda)\hookrightarrow W^{s,p}(M,E;\Lambda)
\end{equation*}
\end{enumerateX}
\end{theorem}
\begin{proof}
\begin{enumerateX}
\item Let $\Lambda_1=\{(U_\alpha,\varphi_\alpha,\rho_\alpha,\psi_\alpha)\}_{1\leq \alpha
\leq N }$ be any augmented total trivialization atlas which is super nice. Let $\Lambda_2=\{(U_\alpha,\varphi_\alpha,\rho_\alpha,\tilde{\psi}_\alpha)\}_{1\leq \alpha
\leq N }$ where for each $1\leq \alpha\leq N$, $\tilde{\psi}_{\alpha}=\frac{\psi_\alpha^2}{\sum_{\beta=1}^N
\psi_\beta^2}$. Note that $\frac{1}{\sum_{\beta=1}^N
\psi_\beta^2}\circ \varphi_\alpha^{-1}\in BC^{\infty}(\varphi_\alpha(U_\alpha))$. For
$f\in W^{s_1,p_1}(M;\Lambda)$ and $u\in W^{s_2,p_2}(M,E;\Lambda)$ we have
{\fontsize{10}{10}{\begin{align*}
\|fu\|_{W^{s,p}(M,E;\Lambda)}&\simeq \|fu\|_{W^{s,p}(M,E;\Lambda_2)}=\sum_{\alpha=1}^N \sum_{j=1}^r\|
[H_\alpha(\tilde{\psi}_\alpha (fu))]^j
\|_{W^{s,p}(\varphi_\alpha(U_\alpha))}\\
&\preceq\sum_{\alpha=1}^N \sum_{j=1}^r\| ((\psi_\alpha
f)\circ \varphi_\alpha^{-1})[H_\alpha(\psi_\alpha u)]^j
\|_{W^{s,p}(\varphi_\alpha(U_\alpha))}\\
&\preceq \big(\sum_{\alpha=1}^N \| (\psi_\alpha f)\circ
\varphi_\alpha^{-1}
\|_{W^{s_1,p_1}(\varphi_\alpha(U_\alpha))}\big)\big(\sum_{\alpha=1}^N \sum_{j=1}^r\| [H_\alpha(\psi_\alpha u)]^j
\|_{W^{s_2,p_2}(\varphi_\alpha(U_\alpha))}\big)\\
&=\|f\|_{W^{s_1,p_1}(M;\Lambda_1)}\|u\|_{W^{s_2,p_2}(M,E;\Lambda_1)}\simeq \|f\|_{W^{s_1,p_1}(M;\Lambda)}\|u\|_{W^{s_2,p_2}(M,E;\Lambda)}
\end{align*}}}
\item We can use the exact same argument as item 1. Just choose $\Lambda_1$ to be "nice" instead of "super nice".
\item The exact same argument as item 1. works. Just choose $\Lambda_1=\Lambda$. (The equality $\|fu\|_{W^{s,p}(M,E;\Lambda)}\simeq \|fu\|_{W^{s,p}(M,E;\Lambda_2)}$ holds due to the assumption that $\Lambda=\Lambda_1$ is GL compatible with itself.)
\end{enumerateX}
\end{proof}
\begin{remark}\lab{remspring1}
Suppose $e$ is a noninteger less than $-1$ and $q\in (1,\infty)$. We will prove that if $\Lambda$ and $\tilde{\Lambda}$ are two augmented total trivialization atlases and each of $\Lambda$ and $\tilde{\Lambda}$ is GL compatible with itself, then $W^{e,q}(M,E;\Lambda)=W^{e,q}(M,E;\tilde{\Lambda})$ (see Remark \ref{remwinterdualequiv}). Considering this and the fact that we can choose $\Lambda_1$ to be super nice (or nice) and GL compatible with itself (see Theorem \ref{winter47jan} and Corollary \ref{winter47bjan}), we can remove the assumption "$s_1$, $s_2$, and $s$ are not nonintegers less than $-1$" from part 1 and part 2 of the preceding theorem.
\end{remark}
\subsubsection{Embedding Properties}
\begin{theorem}\lab{thmwinterembedding1}
Let $M^n$ be a compact smooth manifold. Let $\pi:E\rightarrow M$
be a smooth vector bundle of rank $r$ over $M$. Let $\Lambda$ be an augmented total trivialization atlas for $E$. Let $e_1, e_2\in \reals$ and $q_1, q_2\in (1,\infty)$. If any of $e_1$ or $e_2$ is a noninteger less than $-1$, further assume that the total trivialization atlas in $\Lambda$ is GGL.
\begin{enumerateX}
\item If $e_1$ and $e_2$ are not nonintegers less than $-1$ and if $W^{e_1,q_1}(\reals^n)\hookrightarrow W^{e_2,q_2}(\reals^n)$, then $W^{e_1,q_1}(M,E;\Lambda)\hookrightarrow W^{e_2,q_2}(M,E;\Lambda)$.
\item If $e_1$ and $e_2$ are not nonintegers less than $-1$ and if $W^{e_1,q_1}(\Omega)\hookrightarrow W^{e_2,q_2}(\Omega)$ for all open balls $\Omega\subseteq \reals^n$, then $W^{e_1,q_1}(M,E;\Lambda)\hookrightarrow W^{e_2,q_2}(M,E;\Lambda)$.
\item If any of $e_1$ or $e_2$ is a noninteger less than $-1$ and if $W^{e_1,q_1}(\Omega)\hookrightarrow W^{e_2,q_2}(\Omega)$ for $\Omega=\reals^n$ and for any bounded domain $\Omega\subseteq \reals^n$ with
Lipschitz continuous boundary, then $W^{e_1,q_1}(M,E;\Lambda)\hookrightarrow W^{e_2,q_2}(M,E;\Lambda)$.
\end{enumerateX}
\end{theorem}
\begin{proof}
\begin{enumerateX}
\item Let $\Lambda_1=\{(U_\alpha,\varphi_\alpha,\rho_\alpha,\psi_\alpha)\}_{1\leq \alpha
\leq N }$ be any augmented total trivialization atlas for $E$ which is super nice. We have
\begin{align*}
\|u\|_{W^{e_2,q_2}(M,E;\Lambda)}&\simeq \|u\|_{W^{e_2,q_2}(M,E;\Lambda_1)}=\sum_{\alpha=1}^N\sum_{l=1}^r  \|[H_\alpha (
\psi_\alpha u)]^l
\|_{W^{e_2,q_2}(\varphi_\alpha(U_\alpha))}\\
&\preceq
\sum_{\alpha=1}^N\sum_{l=1}^r \|[H_\alpha (
\psi_\alpha u)]^l
\|_{W^{e_1,q_1}(\varphi_\alpha(U_\alpha))}\\
&=
\|u\|_{W^{e_1,q_1}(M,E;\Lambda_1)}\simeq \|u\|_{W^{e_1,q_1}(M,E;\Lambda)}
\end{align*}
\item We can use the exact same argument as item 1. Just choose $\Lambda_1$ to be "nice" instead of "super nice".
\item The exact same argument as item 1. works. Just choose $\Lambda_1=\Lambda$.
\end{enumerateX}
\end{proof}
\begin{remark}
If we further assume that $\Lambda$ is GL compatible with itself, then we can remove the assumption "$e_1$ and $e_2$ are not nonintegers less than $-1$" from part 1 and part 2 of the preceding theorem. (See the explanation in Remark \ref{remspring1}.)
\end{remark}
\begin{theorem}
Let $M^n$ be a compact smooth manifold. Let $\pi:E\rightarrow M$
be a smooth vector bundle of rank $r$ over $M$ equipped with
fiber metric $\langle .,.\rangle_E$ (so it is meaningful to talk
about $L^\infty(M,E)$). Suppose $s\in\reals$ and
$p\in (1,\infty)$ are such that $sp>n$. Then
$W^{s,p}(M,E)\hookrightarrow L^\infty(M,E)$. Moreover, every
element $u$ in $W^{s,p}(M,E)$ has a continuous version. (Note that since $s$ is not a noninteger less than $-1$, the choice of the augmented total trivialization atlas is immaterial.)
\end{theorem}
\begin{proof}
Let $\{(U_\alpha,\varphi_\alpha,\rho_\alpha)\}_{1\leq \alpha\leq
N }$ be a nice total trivialization atlas for $E\rightarrow M$
that trivializes the fiber metric. Let $\{\psi_\alpha\}_{1\leq
\alpha \leq N}$ be a partition of unity subordinate to
$\{U_\alpha\}$. We need to show that for every $u\in W^{s,p}(M,E)$
\begin{equation*}
|u|_{L^{\infty}(M,E)}\preceq \|u\|_{W^{s,p}(M,E)}
\end{equation*}
Note that since $s>0$, $W^{s,p}(M,E)\hookrightarrow L^p(M,E)$ and
we can treat $u$ as an ordinary section of $E$. We prove the above inequality in
two steps:
\begin{itemizeX}
\item \textbf{Step 1:} Suppose there exists $1\leq \beta\leq N$
such that $\textrm{supp} u\subseteq U_\beta$. We have
\begin{align*}
|u|_{L^{\infty}(M,E)}&=\esssup_{x\in M}|u|_E=\esssup_{x\in
U_\beta}|u|_E\\
&=\esssup_{y\in
\varphi_\beta(U_\beta)}\sqrt{\sum_{l=1}^r|\rho^l_\beta\circ u\circ
\varphi_\beta^{-1}|^2}\qquad ({\fontsize{10}{10}{\textrm{by assumption the triples
trivialize the metric}}})\\
& \leq \esssup_{y\in \varphi_\beta(U_\beta)}
\sum_{l=1}^r|\rho^l_\beta\circ u\circ \varphi_\beta^{-1}|\leq\sum_{l=1}^r
\esssup_{y\in \varphi_\beta(U_\beta)} |\rho^l_\beta\circ u\circ
\varphi_\beta^{-1}|\\
&=\sum_{l=1}^r \|\rho^l_\beta\circ u\circ
\varphi_\beta^{-1}\|_{L^{\infty}(\varphi_\beta(U_\beta))}\\
&\preceq \sum_{l=1}^r \|\rho^l_\beta\circ u\circ
\varphi_\beta^{-1}\|_{W^{s,p}(\varphi_\beta(U_\beta))}\qquad
({\fontsize{10}{10}{\textrm{$sp>n$ so $W^{s,p}(\varphi_\beta(U_\beta))\hookrightarrow
L^{\infty}(\varphi_\beta(U_\beta))$}}})
\end{align*}
\item \textbf{Step 2:} Now suppose $u$ is an arbitrary element of
$W^{s,p}(M,E)$. We have
\begin{align*}
|u|_{L^{\infty}(M,E)}&=|\sum_{\alpha=1}^N \psi_\alpha
u|_{L^{\infty}(M,E)}\leq  \sum_{\alpha=1}^N |\psi_\alpha
u|_{L^{\infty}(M,E)}\\
&\stackrel{\text{Step 1}}{\preceq} \sum_{\alpha=1}^N
\sum_{l=1}^r \|\rho^l_\alpha\circ \psi_\alpha u\circ
\varphi_\alpha^{-1}\|_{W^{s,p}(\varphi_\alpha(U_\alpha))}\simeq\|u\|_{W^{s,p}(M,E)}
\end{align*}
\end{itemizeX}
Next we prove that every element $u$ of $W^{s,p}(M,E)$ has a
continuous version. Note that for all $x\in U_\alpha$
\begin{equation*}
\psi_\alpha u(x)=\Phi_\alpha^{-1}(x,\rho_\alpha^1\circ \psi_\alpha
u,\cdots, \rho_\alpha^r\circ \psi_\alpha u)
\end{equation*}
Also for all $1\leq l\leq r$ and $1\leq \alpha \leq N$ we have
\begin{equation*}
\rho_\alpha^l\circ \psi_\alpha u\circ \varphi_\alpha^{-1}\in
W^{s,p}(\varphi_\alpha (U_\alpha))
\end{equation*}
Therefore $\rho_\alpha^l\circ \psi_\alpha u\circ
\varphi_\alpha^{-1}$ has a continuous version which we denote by
$v_\alpha^l$. Suppose $A_\alpha^l$ is the set of measure zero on
which $v_\alpha^l\neq \rho_\alpha^l\circ \psi_\alpha u\circ
\varphi_\alpha^{-1}$. Let $A_\alpha=\cup_{1\leq l\leq r}
A_\alpha^l$. Clearly $A_\alpha$ is a set of measure zero. Since
$\varphi_\alpha: U_\alpha\rightarrow \varphi_\alpha(U_\alpha)$ is
a diffeomorphism, $B_\alpha:=\varphi_\alpha^{-1}(A_\alpha)$ is a
set of measure zero in $U_\alpha$. (In general, if $M$ and $N$ are smooth $n$-manifolds, $F:M\rightarrow N$ is a
smooth map, and $A\subseteq M$ is a subset of measure zero, then
$F(A)$ has measure zero in $N$. See Page 128 in \cite{Lee3}.)\\
Clearly
\begin{equation*}
(x,v_\alpha^1\circ \varphi_\alpha,\cdots, v_\alpha^r\circ
\varphi_\alpha)=(x,\rho_\alpha^1\circ \psi_\alpha u,\cdots,
\rho_\alpha^r\circ \psi_\alpha u)
\end{equation*}
on $U_\alpha\setminus B_\alpha$. So
\begin{equation*}
w_\alpha:=\Phi_\alpha^{-1}(x,v_\alpha^1\circ
\varphi_\alpha,\cdots, v_\alpha^r\circ
\varphi_\alpha)=\Phi_\alpha^{-1}(x,\rho_\alpha^1\circ \psi_\alpha
u,\cdots, \rho_\alpha^r\circ \psi_\alpha u)=\psi_\alpha u
\end{equation*}
on $U_\alpha\setminus B_\alpha$. Note that $w_\alpha:
U_\alpha\rightarrow E$ is a composition of continuous functions
and so it is continuous on $U_\alpha$. Let $\xi_\alpha\in
C_c^{\infty}(U_\alpha)$ be such that $\xi_\alpha=1$ on
$\textrm{supp} \psi_\alpha$. So $\xi_\alpha w_\alpha=\psi_\alpha
u$ on $M\setminus B_\alpha$. Consequently if we let
$w=\sum_{\alpha=1}^N \xi_\alpha w_\alpha$, then $w$ is a continuous
function that agrees with $u=\sum_{\alpha=1}^N \psi_\alpha u$ on
$M\setminus B$ where $B=\cup_{1\leq \alpha\leq N} B_\alpha$.
\end{proof}
\subsubsection{Observations Concerning the Local Representation of Sobolev Functions}
Let $M^n$ be a compact smooth manifold. Let $E\rightarrow M$
be a smooth vector bundle of rank $r$ over $M$. As it was discussed in Section 6, Given a total trivialization triple $(U_\alpha,\varphi_\alpha,\rho_\alpha)$, we can associate with every $u\in D'(M,E)$ and every $f\in \Gamma(M,E)$, a local representation with respect to $(U_\alpha,\varphi_\alpha,\rho_\alpha)$:
\begin{align*}
&u\mapsto (\tilde{u}^1,\cdots,\tilde{u}^r)\in [D'(\varphi_\alpha(U_\alpha))]^{\times r},\qquad \tilde{u}^l=[H_\alpha(u|_{U_\alpha})]^l\\
&f\mapsto (\tilde{f}^1,\cdots,\tilde{f}^r)\in [\textrm{Func}(\varphi_\alpha(U_\alpha),\reals)]^{\times r},\qquad \tilde{f}^l=\rho_\alpha^l\circ (f|_{U_\alpha})\circ\varphi_\alpha^{-1}
\end{align*}
and of course as it was pointed out in Remark \ref{remfall135}, the two representations agree when $u$ is a regular distribution. The goal of this section is to list some useful facts about the local representations of elements of Sobolev spaces. In what follows, when there is no possibility of confusion, we may write $H_\alpha(u)$ instead of $H_\alpha(u|_{U_\alpha})$, or $\rho_\alpha^l\circ f\circ\varphi_\alpha^{-1}$ instead of $\rho_\alpha^l\circ (f|_{U_\alpha})\circ\varphi_\alpha^{-1}$.
\begin{theorem}\lab{thmapp12}
Let $M^n$ be a compact smooth manifold and $E\rightarrow M$ be a
vector bundle of rank $r$. Suppose
$\Lambda=\{(U_\alpha,\varphi_\alpha,\rho_\alpha,\psi_\alpha)\}_{\alpha=1}^N$ is an augmented
total trivialization atlas for $E\rightarrow M$. Let $u\in D'(M,E)$, $e\in \reals$, and $q\in (1,\infty)$. If for
all $1\leq \alpha\leq N$ and $1\leq j\leq r$, $[H_\alpha(u)]^j\in W^{e,q}_{loc}(\varphi_{\alpha}(U_{\alpha}))$,
then $u\in W^{e,q}(M,E;\Lambda)$.
\end{theorem}
\begin{proof}
\begin{align*}
\parallel u\parallel_{W^{e,q}(M,E;\Lambda)}&=\sum_{\alpha=1}^N\sum_{j=1}^r\parallel[H_\alpha(\psi_\alpha u)]^j\parallel_{W^{e,q}(\varphi_{\alpha}(U_{\alpha}))}\\
&=
\sum_{\alpha=1}^N\sum_{j=1}^r\parallel(\psi_\alpha\circ\varphi_{\alpha}^{-1})\cdot
([H_\alpha(u)]^j)\parallel_{W^{e,q}(\varphi_{\alpha}(U_{\alpha}))}
\end{align*}
Now note that $\psi_\alpha\circ\varphi_{\alpha}^{-1}:
\varphi_{\alpha}(U_{\alpha}) \rightarrow \reals$ is smooth with
compact support (its support is in the compact set
$\varphi_{\alpha}(\textrm{supp}\,\psi_{\alpha})$). Therefore it follows from the assumption that each term on the right hand side of the above equality is finite.
\end{proof}
\begin{remark}
Note that, as opposed to what is claimed in some references, it is NOT true in
general that if $u\in W^{e,q}(M,E;\Lambda)$, then the components of the local representations of $u$ will be in the corresponding Euclidean Sobolev space; that is  $u\in W^{e,q}(M,E;\Lambda)$ does not imply that for
all $1\leq \alpha\leq N$ and $1\leq j\leq r$, $[H_\alpha(u)]^j\in W^{e,q}(\varphi_{\alpha}(U_{\alpha}))$.  Consider the following example:\\
$M=S^1$, $e=0$, $q=1$, and $f: M\rightarrow \reals$ defined by
$f\equiv 1$. Clearly $f\in W^{0,1}(M)=L^1(S^1)$. Now consider the
atlas $\mathcal{A}=\{(U_1,\varphi_1),(U_2,\varphi_2)\}$ where
\begin{align*}
& U_1= S^1\setminus \{(0,1)\},\qquad \varphi_1(x,y)=\frac{x}{1-y}
\\
& U_2= S^1\setminus \{(0,-1)\},\qquad
\varphi_2(x,y)=\frac{x}{1+y} \quad(\textrm{stereographic
projection})
\end{align*}
Clearly $f\circ\varphi_1^{-1}=f\circ\varphi_2^{-1}=1$ and
$\varphi_1(U_1)=\varphi_2(U_2)=\reals$. So $f\circ\varphi_1^{-1}$
and $f\circ\varphi_2^{-1}$ do not belong to $L^1(\varphi_1(U_1))$
or  $L^1(\varphi_2(U_2))$.
\end{remark}
However, the following theorem holds true.
\begin{theorem}\lab{thmapp13}
Let $M^n$ be a compact smooth manifold and $E\rightarrow M$ be a
vector bundle of rank $r$. Let $e\in\reals$ and $q\in (1,\infty)$. Suppose
$\Lambda=\{(U_\alpha,\varphi_\alpha,\rho_\alpha,\psi_\alpha)\}_{\alpha=1}^N$ is an augmented
total trivialization atlas for $E\rightarrow M$. If $e$ is a noninteger less than $-1$ further assume that $\Lambda$ is GL compatible with itself. Let $u\in W^{e,q}(M,E;\Lambda)$ be such that
$\textrm{supp}\,u\subseteq V\subseteq \bar{V}\subseteq U_{\beta}$
for some open set $V$ and some $1\leq \beta\leq N$. Then for all $1\leq i\leq r$, $[H_\beta(u)]^i\in
W^{e,q}(\varphi_\beta(U_{\beta}))$. Indeed,
\begin{equation*}
\parallel [H_\beta(u)]^i\parallel_{W^{e,q}(\varphi_\beta(U_\beta))}\leq
\parallel u\parallel_{W^{e,q}(M,E;\Lambda )}
\end{equation*}
\end{theorem}
\begin{proof}
Let $\Lambda_1=\{(U_\alpha,\varphi_\alpha,\rho_\alpha,\tilde{\psi}_\alpha)\}_{\alpha=1}^N$ where $\{\tilde{\psi}_{\alpha}\}_{1\leq \alpha\leq N}$ is a partition of
unity subordinate to the cover $\{U_{\alpha}\}_{1\leq \alpha\leq N}$ such that
$\tilde{\psi}_\beta=1$ on a neighborhood of $\bar{V}$ (see Lemma
~\ref{lemapp6}). We have
\begin{align*}
\parallel[H_\beta(u)]^i\parallel_{W^{e,q}(\varphi_\beta(U_\beta))}&=\parallel
[H_\beta(\tilde{\psi}_\beta u)]^i\parallel_{W^{e,q}(\varphi_\beta(U_\beta))}\\
&\leq \sum_{\alpha=1}^N  \sum_{j=1}^r\parallel
[H_\alpha(\tilde{\psi}_\alpha u)]^j\parallel_{W^{e,q}(\varphi_\alpha(U_\alpha))}\\
&=\parallel
u\parallel_{W^{e,q}(M,E;\Lambda_1 )}\simeq \parallel
u\parallel_{W^{e,q}(M,E;\Lambda )}
\end{align*}
\end{proof}
\begin{corollary}\lab{corapp2}
Let $M^n$ be a compact smooth manifold and $E\rightarrow M$ be a
vector bundle of rank $r$. Let $e\in\reals$ and $q\in (1,\infty)$. Suppose
$\Lambda=\{(U_\alpha,\varphi_\alpha,\rho_\alpha,\psi_\alpha)\}_{\alpha=1}^N$ is an augmented
total trivialization atlas for $E\rightarrow M$. If $e$ is a noninteger less than $-1$ further assume that $\Lambda$ is GL compatible with itself. If $u\in W^{e,q}(M,E;\Lambda)$, then for all $1\leq \alpha \leq N$ and $1\leq i\leq r$,  $[H_{\alpha}(u)]^i$ (i.e. each component of the local representation of $u$ with respect to $(U_\alpha,\varphi_\alpha,\rho_\alpha)$) belongs to $W^{e,q}_{loc}(\varphi_\alpha(U_\alpha))$. Moreover, if $\xi\in C_c^\infty(\varphi_\alpha(U_\alpha))$, then
\begin{equation*}
\|\xi [H_{\alpha}(u)]^i\|_{W^{e,q}(\varphi_\alpha(U_\alpha))}\preceq
\|u\|_{W^{e,q}(M,E;\Lambda)}
\end{equation*}
where the implicit constant may depend on $\xi$.
\end{corollary}
\begin{proof}
Define $G:M\rightarrow \reals$ by
\begin{equation*}
G(p)=\begin{cases}
      \hfill \xi\circ \varphi_\alpha    \hfill & \text{ if $p \in U_\alpha$ } \\
      \hfill 0 \hfill & \text{ if $p\not \in U_\alpha$}
  \end{cases}
\end{equation*}
Clearly $G\in C^{\infty}(M)$. So, by Theorem \ref{thmwinter1100}, $Gu\in W^{e,q}(M,E;\Lambda)$. Also since
$\xi\in C_c^{\infty}(\varphi_\alpha(U_\alpha))$, there exists a
compact set $K$ such that
\begin{equation*}
\textrm{supp}\,\xi\subseteq \mathring{K}\subseteq K\subseteq
\varphi_\alpha(U_\alpha)
\end{equation*}
Consequently there exists an open set $V_\alpha$ (e.g.
$V_\alpha=\varphi_\alpha^{-1}(\mathring{K})$) such that
\begin{equation*}
\textrm{supp}\,(Gu)\subseteq
\textrm{supp}(\xi\circ\varphi_\alpha)\subseteq V_\alpha\subseteq
\bar{V}_\alpha\subseteq U_\alpha
\end{equation*}
So by Theorem \ref{thmapp13}, $[H_\alpha(Gu)]^i\in W^{e,q}(\varphi_\alpha(U_\alpha))$ and
\begin{align*}
\| [H_\alpha(Gu)]^i\|_{W^{e,q}(\varphi_\alpha(U_\alpha))}\preceq
\|Gu\|_{W^{e,q}(M,E;\Lambda)}\preceq \|u\|_{W^{e,q}(M,E;\Lambda)}
\end{align*}
 Now we just need to notice that on $\varphi_\alpha(U_\alpha)$,
\begin{equation*}
[H_\alpha(Gu)]^i=(G\circ \varphi_\alpha^{-1})[H_\alpha(u)]^i=\xi[H_\alpha(u)]^i
\end{equation*}
\end{proof}
\subsubsection{Observations Concerning the Riemannian Metric}
The sobolev spaces that appear in this section all have nonnegative smoothness exponents; therefore the choice of the augmented total trivialization atlas is immaterial and will not appear in the notation.
\begin{corollary}\lab{corappmetric12}
Let $(M^n,g)$ be a compact Riemannian manifold with $g\in
W^{s,p}(T^2M)$, $sp>n$. Let
$\{(U_\alpha,\varphi_\alpha,\rho_\alpha)\}_{1\leq \alpha \leq N}$ be a standard total trivialization
atlas for $T^2M\rightarrow M$. Fix some $\alpha$ and denote the components of the metric with respect to $(U_\alpha,\varphi_\alpha,\rho_\alpha)$ by $g_{ij}:U_\alpha\rightarrow \reals$
$(g_{ij}=(\rho_{\alpha})_{ij}\circ g)$. As an immediate consequence
of Corollary ~\ref{corapp2} we have
\begin{align*}
g_{ij}\circ \varphi_\alpha^{-1}\in W^{s,p}_{loc}(\varphi_\alpha(U_\alpha))
\end{align*}
\end{corollary}

\begin{theorem}\lab{thmapp14}
Let $(M^n,g)$ be a compact Riemannian manifold with $g\in
W^{s,p}(T^2M)$, $sp>n$, $s\geq 1$. Let
$\{(U_\alpha,\varphi_\alpha,\rho_\alpha)\}_{1\leq \alpha \leq N}$ be a GGL standard total trivialization
atlas for $T^2M\rightarrow M$. Fix some $\alpha$ and denote the components of the metric with respect to $(U_\alpha,\varphi_\alpha,\rho_\alpha)$ by $g_{ij}:U_\alpha\rightarrow \reals$
$(g_{ij}=(\rho_{\alpha})_{ij}\circ g)$. Then
\begin{enumerate}
\item $\textrm{det}\,g_\alpha\in
W^{s,p}_{loc}(\varphi_{\alpha}(U_\alpha))$ where $g_\alpha (x)$
is the matrix whose $(i,j)$-entry is $g_{ij}\circ
\varphi_\alpha^{-1}$.
\item $\sqrt{\textrm{det}\,g}\circ \varphi_{\alpha}^{-1}=\sqrt{\textrm{det}\,g_\alpha}\in
W^{s,p}_{loc}(\varphi_\alpha(U_\alpha))$.
\item $\frac{1}{\sqrt{\textrm{det}\,g}\circ \varphi_{\alpha}^{-1}}\in
W^{s,p}_{loc}(\varphi_\alpha(U_\alpha))$.
\end{enumerate}
\end{theorem}
\begin{proof}
\leavevmode
\begin{enumerateXALI}
\item By Corollary ~\ref{corapp2}, $g_{ij}\circ
\varphi_\alpha^{-1}$ is in
$W^{s,p}_{loc}(\varphi_\alpha(U_\alpha))$. So it follows from
Lemma ~\ref{lemapp7} that $\textrm{det}\,g_\alpha\in
W^{s,p}_{loc}(\varphi_{\alpha}(U_\alpha))$.
\item This is a direct consequence of item 1 and Theorem
~\ref{thmapp9}.
\item This is a direct consequence of item 1 and Theorem
~\ref{thmapp9}.
\end{enumerateXALI}
\end{proof}
\begin{theorem}\lab{thmapp15}
Let $(M^n,g)$ be a compact Riemannian manifold with $g\in
W^{s,p}(T^2M)$, $sp>n$, $s\geq 1$. Then the inverse metric tensor $g^{-1}$
(which is a $0\choose 2$ tensor field) is in $W^{s,p}(T_2M)$.
\end{theorem}
\begin{proof}
Let
$\{(U_\alpha,\varphi_\alpha,\rho_\alpha)\}_{1\leq \alpha \leq N}$ be a GGL standard total trivialization
atlas for $T^2M\rightarrow M$. Let $\{\psi_\alpha\}_{1\leq\alpha\leq N}$ be a partition of unity subordinate to $\{U_\alpha\}_{1\leq \alpha\leq N}$. We have
\begin{equation*}
\|g^{-1}\|_{W^{s,p}(T_2M)}=\sum_{\alpha=1}^N\sum_{i,j}\|\psi_\alpha
g^{ij}\circ\varphi_\alpha^{-1}\|_{W^{s,p}(\varphi_{\alpha}(U_\alpha))}
\end{equation*}
So it is enough to show that for all $i,j$ and $\alpha$,
$g^{ij}\circ \varphi_{\alpha}^{-1}$ is in
$W^{s,p}_{loc}(\varphi_{\alpha}(U_\alpha))$. Let $B=(B_{ij})$
where $B_{ij}=g_{ij}\circ \varphi_{\alpha}^{-1}$. By assumption
$g\in W^{s,p}(T^2M)$; so it follows from Corollary ~\ref{corapp2}
that $B_{ij}\in W^{s,p}_{loc}(\varphi_{\alpha}(U_\alpha))$. Our
goal is to show that the entries of the inverse of $B$ are in
$W^{s,p}_{loc}(\varphi_{\alpha}(U_\alpha))$. Recall that
\begin{equation*}
(B^{-1})_{ij}=\frac{(-1)^{i+j}}{\textrm{det}\,B}M_{ij}
\end{equation*}
where $M_{ij}$ is the determinant of the $(n-1)\times (n-1)$
matrix formed by removing the $j^{th}$ row and $i^{th}$ column of
$B$. Since the entries of $B$ are in
$W^{s,p}_{loc}(\varphi_{\alpha}(U_\alpha))$, it follows from
Lemma ~\ref{lemapp7} and Theorem
~\ref{thmapp9} that $\frac{1}{\textrm{det}\,B}$ and $M_{ij}$ are in
$W^{s,p}_{loc}(\varphi_{\alpha}(U_\alpha))$. Also $sp>n$, so
$W^{s,p}_{loc}(\varphi_{\alpha}(U_\alpha))$ is closed under multiplication. Consequently
$(B^{-1})_{ij}$ is in $W^{s,p}_{loc}(\varphi_{\alpha}(U_\alpha))$.
\end{proof}
\begin{corollary}\lab{corgamma1}
Let $(M^n,g)$ be a compact Riemannian manifold with $g\in
W^{s,p}(T^2M)$, $sp>n$, $s\geq 1$. $\{(U_\alpha,\varphi_\alpha)\}_{1\leq \alpha \leq N}$ be a GGL smooth
atlas for $M$. Denote the standard components of the inverse
metric with respect to this chart by $g^{ij}:U_\alpha\rightarrow
\reals$. As an immediate consequence of Theorem ~\ref{thmapp15}
and Corollary ~\ref{corapp2} we have
\begin{align*}
g^{ij}\circ \varphi_\alpha^{-1}\in
W^{s,p}_{loc}(\varphi_\alpha(U_\alpha))
\end{align*}
Also since
\begin{align*}
\Gamma^k_{ij}\circ\varphi_{\alpha}^{-1}=\frac{1}{2}g^{kl}(\partial_ig_{jl}+\partial_jg_{il}-\partial_lg_{ij})\circ \varphi_{\alpha}^{-1}\\
\end{align*}
it follows from Corollary ~\ref{corappmetric12}, Lemma
~\ref{lemapp1}, Theorem ~\ref{thmapp8a}, and the fact that
$W^{s,p}(\varphi_\alpha(U_\alpha))\times W^{s-1,p}(\varphi_\alpha(U_\alpha))\hookrightarrow W^{s-1,p}(\varphi_\alpha(U_\alpha))$ that
\begin{equation*}
\Gamma^k_{ij}\circ \varphi_\alpha^{-1}\in
W^{s-1,p}_{loc}(\varphi_\alpha(U_\alpha))
\end{equation*}
\end{corollary}

\subsubsection{A Useful Isomorphism}
Let $M^n$ be a compact smooth manifold and $E\rightarrow M$ be a
vector bundle of rank $r$. Let $e\in\reals$ and $q\in (1,\infty)$. Suppose
$\Lambda=\{(U_\alpha,\varphi_\alpha,\rho_\alpha,\psi_\alpha)\}_{\alpha=1}^N$ is an augmented
total trivialization atlas for $E\rightarrow M$. Given a closed subset $A\subseteq M$, $W^{e,q}_A(M,E;\Lambda)$ is defined
to be the subspace of $W^{e,q}(M,E;\Lambda)$ consisting of $u\in
W^{e,q}(M,E;\Lambda)$ with $\textrm{supp}u\subseteq A$.  Fix $1\leq \beta\leq N$ and suppose $K\subseteq U_\beta$ is compact. Then each
element of $W^{e,q}_K(M,E;\Lambda)$ can be identified with an element of
$D'(U_\beta,E_{U_\beta})$ under the injective map $u\in W^{e,q}_K(M,E;\Lambda)\subseteq
D'(M,E)\mapsto u|_U\in D'(U_\beta,E_{U_\beta})$. So we can restrict the domain
of $H_\beta: [D(U_\beta,E_{U_\beta}^\vee)]^*\rightarrow (D'(\varphi_\beta(U_\beta)))^{\times r}$
to $W^{e,q}_K(M,E;\Lambda)$ which associates with each element $u\in
W^{e,q}_K(M,E;\Lambda)$, the $r$ components of
$H_\beta(u)=(\tilde{u}_\beta^1,\cdots,\tilde{u}_\beta^r)$. (Here $H_\beta$ stands for
$H_{E^\vee, U_\beta,\varphi_\beta}$.)
\begin{lemma}\lab{lemwinter1103}
Consider the above setting and further assume that if $e$ is a noninteger less than $-1$, then the total trivialization atlas in $\Lambda$ is GL compatible with itself. Then the linear topological isomorphism $H_\beta:
[D(U_\beta,E_{U_\beta}^\vee)]^*=D'(U_\beta,E_{U_\beta})\rightarrow (D'(\varphi_\beta(U_\beta)))^{\times r}$
restricts to a linear topological isomorphism
\begin{equation*}
\hat{H}_\beta: W^{e,q}_K(M,E;\Lambda)\rightarrow
[W^{e,q}_{\varphi_\beta(K)}(\varphi_\beta(U_\beta))]^{\times r}
\end{equation*}
\end{lemma}
\begin{proof}
In order to simplify the notations we will use $(U,\varphi,\rho)$, $H$, $\hat{H}$, and $\tilde{u}^l$ instead of $(U_\beta,\varphi_\beta,\rho_\beta)$, $H_\beta$, $\hat{H}_\beta$, and $\tilde{u}^l_\beta$. In order to prove this claim, we proceed as follows:
\begin{enumerateX}
\item First we show that $\textrm{supp}\tilde{u}^l\subseteq
\varphi(K)$.
\item Next we show that if $u\in W^{e,q}_K(M,E;\Lambda)$, then {\fontsize{10}{10}{$\|u\|_{W^{e,q}(M,E;\Lambda)}\simeq \sum_{l=1}^r
\|\tilde{u}^l\|_{W^{e,q}(\varphi(U))}$}} which proves that
\begin{enumerate}[(i.)]
\item $\tilde{u}^l$ is indeed an element of $W^{e,q}(\varphi(U))$
and
\item $\hat{H}$ is continuous.
\end{enumerate}
Note that (i) together with the fact that
$\textrm{supp}\tilde{u}^l\subseteq \varphi(K)$ shows that
$\tilde{u}^l$ is indeed an element of
$W^{e,q}_{\varphi(K)}(\varphi(U))$ so $\hat{H}$ is well-defined.
\item We prove that $\hat{H}$ is injective.
\item In order to prove that $\hat{H}$ is surjective we use our explicit
formula for $H^{-1}$ (see Remark \ref{remfall134}).
\end{enumerateX}
Note that the fact that $\hat{H}$ is bijective combined with the equality $\|u\|_{W^{e,q}(M,E;\Lambda)}\simeq \sum_{l=1}^r
\|\tilde{u}^l\|_{W^{e,q}(\varphi(U))}$ implies that
$\hat{H}^{-1}$ is continuous as well.\\
Here are the proofs:
\begin{enumerateX}
\item This item is a direct consequence of item 1. in Remark
\ref{remfall134}.
\item  Define the augmented total trivialization atlas $\Lambda_1$ by $\Lambda_1=\{(U_\alpha,\varphi_\alpha,\rho_\alpha,\tilde{\psi}_\alpha)\}_{\alpha=1}^N$ where $\{\tilde{\psi}_\alpha\}_{1\leq \alpha\leq N}$ is a partition of unity subordinate
 to $\{U_\alpha\}_{1\leq \alpha\leq N}$ such that $\tilde{\psi}_\beta=1$ on a neighborhood of $K$.
 Note that for each $\alpha$, $\tilde{\psi}_\alpha \geq 0$ and $\sum_{\alpha=1}^N
 \tilde{\psi}_\alpha=1$. Thus the assumption $\tilde{\psi}_\beta=1$ on $K$ implies
 that $\tilde{\psi}_\alpha=0$ on $K$ for all $\alpha\neq\beta$. We have
 \begin{align*}
\|u\|_{W^{e,q}(M,E;\Lambda)}&\simeq \|u\|_{W^{e,q}(M,E;\Lambda_1)}\simeq\sum_{\alpha=1}^N\sum_{l=1}^r \|(H_\alpha(\tilde{\psi}_\alpha
u))^l\|_{W^{e,q}(\varphi_\alpha(U_\alpha))}\\
&=\sum_{l=1}^r \|(H(\tilde{\psi}_\beta
u))^l\|_{W^{e,q}(\varphi_\alpha(U_\alpha))}=\sum_{l=1}^r \|[H(
u)]^l\|_{W^{e,q}(\varphi_\alpha(U_\alpha))}
 \end{align*}
Note that $\textrm{supp}u\subseteq K$ and $\tilde{\psi}_\beta=1$ on $K$, so
$\tilde{\psi}_\beta u=u|_U$ as elements of $D'(U,E_U)$. Therefore $H(\tilde{\psi}_\beta
u)=H(u)=(\tilde{u}^1,\cdots,\tilde{u}^r)$.
\item $\hat{H}$ is injective because it is a restriction of the
injective map $H$.
\item Let $(v^1,\cdots,v^r)\in [W^{e,q}_{\varphi(K)}(\varphi(U))]^{\times
r}$. Our goal is to show that $H^{-1}(v^1,\cdots,v^r)\in
W^{e,q}_K(M,E;\Lambda)\simeq W^{e,q}_K(M,E;\Lambda_1)$ (this implies that $\hat{H}$ is surjective). By Remark \ref{remfall134}, for all
$\xi\in D(U,E_U^\vee)$
\begin{align*}
H^{-1}(v^1,\cdots,v^r)(\xi)=\sum_i v^i[(\rho^{\vee})^i\circ
\xi\circ \varphi^{-1}]
\end{align*}
First note it follows from Remark \ref{winter73} that $\textrm{supp}H^{-1}(v^1,\cdots,v^r)\subseteq K$; indeed,
 if $\textrm{supp}\xi\subseteq U\setminus K$, then
$\xi\circ\varphi^{-1}=0$ on $\varphi(K)$. So
$(\rho^{\vee})^i\circ \xi\circ \varphi^{-1}=0$ on $\varphi(K)$.
That is $\textrm{supp}[(\rho^{\vee})^i\circ \xi\circ
\varphi^{-1}]\subseteq \varphi(U)\setminus \varphi(K)$. Thus for
all $i$, $v^i[(\rho^{\vee})^i\circ \xi\circ \varphi^{-1}]=0$
(because by assumption $\textrm{supp}v^i\subseteq
\varphi(K)$). This shows that if $\textrm{supp}\xi\subseteq
U\setminus K$, then $H^{-1}(v^1,\cdots,v^r)(\xi)=0$. Consequently
$\textrm{supp}H^{-1}(v^1,\cdots,v^r)\subseteq K$.\\
Also we have
\begin{equation*}
\|H^{-1}(v^1,\cdots,v^r)\|_{W^{e,q}(M,E;\Lambda_1)}\simeq\sum_{l=1}^r
\|v^l\|_{W^{e,q}(\varphi(U))}<\infty
\end{equation*}
So $H^{-1}(v^1,\cdots,v^r)\in W^{e,q}(M,E;\Lambda)$.
\end{enumerateX}
\end{proof}
It is clear that
$u\in W^{e,q}(M,E;\Lambda)$ if and only if for all $\alpha$, $\psi_\alpha
u\in W^{e,q}_{K_\alpha}(M,E;\Lambda)$ where $K_\alpha$ can be taken as any compact set
such that $\textrm{supp}\psi_\alpha\subseteq K_\alpha\subseteq
U_\alpha$. In fact as a direct consequence of the definition of Sobolev
spaces and the above mentioned isomorphism we have
\begin{align*}
u\in W^{e,q}(M,E;\Lambda)&\Longleftrightarrow \forall\,1\leq \alpha\leq
N\quad H_\alpha(\psi_\alpha u)\in
[W^{e,q}_{\varphi_\alpha(\textrm{supp}\psi_\alpha)}(\varphi_\alpha(U_\alpha))]^{\times
r}\\
&\Longleftrightarrow \forall\,1\leq \alpha\leq N\quad
\psi_\alpha u\in W^{e,q}_{\textrm{supp}\psi_\alpha}(M,E;\Lambda)
\end{align*}

\subsubsection{Completeness; Density of Smooth Functions}
Our proofs for completeness of Sobolev spaces and density of smooth functions are based on the ideas presented in \cite{Reus1}.
 \begin{lemma}\lab{lemwinterh20}
 Let $M^n$ be a compact smooth manifold and $E\rightarrow M$ be a
vector bundle of rank $r$. Let $e\in\reals$ and $q\in (1,\infty)$. Suppose
$\Lambda=\{(U_\alpha,\varphi_\alpha,\rho_\alpha,\psi_\alpha)\}_{\alpha=1}^N$ is an augmented
total trivialization atlas for $E\rightarrow M$. If $e$ is a noninteger less than $-1$ further assume that $\Lambda$ is GL compatible with itself. Let $K_\alpha$ be a compact subset of
 $U_\alpha$ that contains the support of $\psi_\alpha$. Let $S: W^{e,q}(M,E;\Lambda)\rightarrow \prod_{\alpha=1}^N W^{e,q}_{K_\alpha}(M,E;\Lambda)
 $ be the linear map defined by $S(u)=(\psi_1 u,\cdots,\psi_N u)$. Then $S: W^{e,q}(M,E;\Lambda)\rightarrow
 S(W^{e,q}(M,E;\Lambda))\subseteq \prod_{\alpha=1}^N
 W^{e,q}_{K_\alpha}(M,E;\Lambda)$ is a linear topological isomorphism. Moreover
 $S(W^{e,q}(M,E;\Lambda))$ is closed in $\prod_{\alpha=1}^N
 W^{e,q}_{K_\alpha}(M,E;\Lambda)$.
 \end{lemma}
 \begin{proof}${}$\\
Each component of $S$ is continuous (see Theorem
\ref{thmwinter1100}), therefore $S$ is continuous.
 Define $P: \prod_{\alpha=1}^N
W^{e,q}_{K_\alpha}(M,E)\rightarrow W^{e,q}(M,E)$ by
\begin{equation*}
P(v_1,\cdots,v_N)=\sum_i v_i
\end{equation*}
Clearly $P$ is continuous. Also $P\circ S =id$.
Now the claim follows from Theorem \ref{thmwinter1101}.
 \end{proof}
 \begin{theorem}\lab{thmwinterh20}
 Let $M^n$ be a compact smooth manifold and $E\rightarrow M$ be a
vector bundle of rank $r$. Let $e\in\reals$ and $q\in (1,\infty)$. Suppose
$\Lambda=\{(U_\alpha,\varphi_\alpha,\rho_\alpha,\psi_\alpha)\}_{\alpha=1}^N$ is an augmented
total trivialization atlas for $E\rightarrow M$. If $e$ is a noninteger less than $-1$ further assume that $\Lambda$ is GL compatible with itself. Then $W^{e,q}(M,E;\Lambda)$ is a Banach space.
 \end{theorem}
 \begin{proof}
 According to Lemma \ref{lemwinter1103}, for each $1\leq \alpha\leq N$, $W^{e,q}_{K_\alpha}(M,E;\Lambda)$ is isomorphic to the Banach space $[W^{e,q}_{\varphi_{\alpha}(K_\alpha)}(\varphi_\alpha(U_\alpha))]^{\times r}$.
 So $\prod_{\alpha=1}^N W^{e,q}_{K_\alpha}(M,E;\Lambda)$ is a Banach
 space. A closed subspace of a Banach space is Banach. Therefore
 $S(W^{e,q}(M,E;\Lambda))$ is a Banach space. Since $S$ is a linear
 topological isomorphism onto its image, $W^{e,q}(M,E;\Lambda)$ is also a Banach space.
 \end{proof}

\begin{theorem}\lab{thmwinterh22}
Let $M^n$ be a compact smooth manifold and $E\rightarrow M$ be a
vector bundle of rank $r$. Let $e\in\reals$ and $q\in (1,\infty)$. Suppose
$\Lambda=\{(U_\alpha,\varphi_\alpha,\rho_\alpha,\psi_\alpha)\}_{\alpha=1}^N$ is an augmented
total trivialization atlas for $E\rightarrow M$. If $e$ is a noninteger less than $-1$ further assume that $\Lambda$ is GL compatible with itself. Then $D(M,E)$ is dense in $W^{e,q}(M,E;\Lambda)$.
\end{theorem}
\begin{proof}
Let $K_\alpha=\textrm{supp}\psi_\alpha$. For each $1\leq \alpha\leq N$, let
$V_\alpha$ be an open set such that
\begin{equation*}
K_\alpha\subseteq V_\alpha\subseteq \bar{V}_\alpha\subseteq
U_\alpha
\end{equation*}
Suppose $u\in W^{e,q}(M,E;\Lambda)$ and let $u_\alpha=\psi_\alpha u$.
Clearly $\textrm{supp}u_\alpha\subseteq K_\alpha$. Also according
to Lemma \ref{lemwinter1103}, for each $\alpha$
there exits a linear topological isomorphism
\begin{equation*}
\hat{H}_\alpha: W^{e,q}_{\bar{V}_\alpha}(M,E)\rightarrow
[W^{e,q}_{\varphi_\alpha(\bar{V}_\alpha)}(\varphi_\alpha(U_\alpha))]^{\times
r}
\end{equation*}
Note that $\hat{H}_\alpha(u_\alpha)\in
[W^{e,q}_{\varphi_\alpha(K_\alpha)}(\varphi_\alpha(U_\alpha))]^{\times
r}$. Therefore by Lemma \ref{lemwinter2003} there exists a sequence
$\{(\eta_\alpha)_i\}$ in
$[C^\infty_{\varphi_\alpha(\bar{V}_\alpha)}(\varphi_\alpha(U_\alpha))]^{\times
r}$ (of course we view each component of $(\eta_\alpha)_i$ as a
distribution) that converges to $\hat{H}_\alpha(u_\alpha)$ in
$W^{e,q}$ norm as $i\rightarrow \infty$. Since $\hat{H}_\alpha$
is a linear topological isomorphism, we can conclude that
\begin{equation*}
\hat{H}_\alpha^{-1}((\eta_\alpha)_i)\rightarrow u_\alpha,\quad
(\textrm{in $W^{e,q}_{\bar{V}_\alpha}(M,E;\Lambda)$ as $i\rightarrow
\infty$})
\end{equation*}
(Note that if  a sequence converges in $W^{e,q}_A(M,E;\Lambda)$  where $A$
is a closed subset of $M$, it also obviously converges in
$W^{e,q}(M,E;\Lambda)$.) Let $\xi_i=\sum_{\alpha=1}^N
\hat{H}_\alpha^{-1}((\eta_\alpha)_i)$. This sum makes sense because, as we will shortly prove, each summand is in
$C_c^{\infty}(U_\alpha,E_\alpha)$ and so by extension by zero can be viewed as an element of $C^{\infty}(M,E)$.  Clearly $\xi_i\rightarrow
\sum_\alpha u_\alpha=u$ in $W^{e,q}(M,E;\Lambda)$. It remains to show
that for each $i$, $\xi_i$ is in $C^{\infty}(M,E)$. To this end, it suffices to show that if $\chi=(\chi^1,\cdots,\chi^r)\in
[C_c^{\infty}(\varphi_\alpha(U_\alpha))]^{\times r}$, then
$\hat{H}_\alpha^{-1}(\chi)$ is in
$C_c^{\infty}(U_\alpha,E_\alpha)$ and so can be considered as an
element of $C^{\infty}(M,E)$ (by extension by zero). Note that
$\hat{H}_\alpha^{-1}(\chi)$ is compactly supported in $U_\alpha$
because by definition of $\hat{H}_\alpha$ any distribution in the codomain of
$\hat{H}_\alpha^{-1}$ has compact support in $\bar{V}_\alpha$. So
we just need to prove the smoothness of
$\hat{H}_\alpha^{-1}(\chi)$. That is, we need to show that there
is a smooth section $f\in C^\infty(U_\alpha,E_{U_\alpha})$ such that
$u_f=\hat{H}_\alpha^{-1}(\chi)$. It seems that the natural
candidate for $f(x)$ should be $(\rho_\alpha|_{E_x})^{-1}\circ \chi\circ
\varphi_\alpha(x)$. In fact, if we define $f$ by this formula, then
$\hat{H}_\alpha(u_f)=H_\alpha(u_f)$ and by Remark \ref{remfall135} $H_\alpha(u_f)$ is a
distribution that corresponds to the regular function
$(\tilde{f}^1,\cdots,\tilde{f}^r)=\rho_\alpha\circ f\circ \varphi_\alpha^{-1}$.
Obviously
\begin{equation*}
\rho_\alpha\circ f\circ \varphi_\alpha^{-1}|_{\varphi_\alpha(x)}=\rho_\alpha\circ
(\rho_\alpha|_{E_x})^{-1}\circ \chi\circ \varphi_\alpha\circ
\varphi_\alpha^{-1}|_{\varphi_\alpha(x)}=\chi|_{\varphi_\alpha(x)}
\end{equation*}
So the regular section $f(x)=\rho_\alpha|_{E_x}^{-1}\circ \chi\circ
\varphi_\alpha(x)$ corresponds to $\hat{H}_\alpha^{-1}(\chi)$ and we just
need to show that $f$ is smooth; this is true because $f$ is a
composition of smooth functions. Indeed,
\begin{equation*}
f(x)=\rho_\alpha|_{E_x}^{-1}\circ \chi\circ
\varphi_\alpha(x)=\Phi_\alpha^{-1}(x,\chi\circ \varphi_\alpha(x))\Longrightarrow
f=\Phi_\alpha^{-1}\circ (Id,\chi\circ \varphi_\alpha)
\end{equation*}
and all the maps involved in the above expression are smooth.
\end{proof}

\subsubsection{Dual of Sobolev Spaces}
\begin{lemma}\lab{lemspringapr13}
Let $M^n$ be a compact smooth manifold and let $\pi:E\rightarrow M$ be a
vector bundle of rank $r$ equipped with a fiber metric
$\langle .,.\rangle_E$. Let $e\in\reals$ and $q\in (1,\infty)$. Suppose
$\Lambda=\{(U_\alpha,\varphi_\alpha,\rho_\alpha,\psi_\alpha)\}_{\alpha=1}^N$ is an augmented
total trivialization atlas for $E\rightarrow M$ which trivializes the fiber metric. If $e$ is a noninteger less than $-1$ further assume that the total trivialization atlas in $\Lambda$ is GGL.\\
Fix a positive smooth density $\mu$ on
$M$ (for instance we can equip $M$ with a smooth Riemannian
metric and consider the corresponding Riemannian density). Let
{\fontsize{10}{10}{$T:D(M,E)\rightarrow D(M,E^\vee)$}} be the map that sends $\xi$ to
$T_\xi$ where $T_\xi$ is defined by
\begin{equation*}
\forall\,x\in M\quad T_\xi(x):E_x\rightarrow \mathcal{D}_x,\quad
a\mapsto\langle a,\xi(x)\rangle_E \,\mu(x)
\end{equation*}
Then $T$ is a linear bijective continuous map. (So the adjoint of
$T$ is a well-defined bijective continuous map that can be used
to identify $D'(M,E)=[D(M,E^\vee)]^*$ with $[D(M,E)]^*$.) Moreover,
{\fontsize{10}{10}{$T: (C^{\infty}(M,E),\|.\|_{W^{e,q}(M,E;\Lambda)})\rightarrow
(C^{\infty}(M,E^\vee),\|.\|_{W^{e,q}(M,E^\vee;\Lambda^\vee)})$}} is a topological isomorphism.
\end{lemma}
\begin{proof}
The fact that $T$ is linear is obvious.
\begin{itemizeX}
\item \textbf{T is one-to-one:} Suppose $\xi\in D(M,E)$ is such that
$T_\xi=0$. Then
\begin{align*}
\forall\,x\in M\quad T_\xi(x)=0 &\Longrightarrow \forall\,x\in
M,\, \forall\, a\in E_x\quad [T_\xi(x)](a)=0\\
&\Longrightarrow
\forall\,x\in M,\, \forall\, a\in E_x\quad \langle
a,\xi(x)\rangle_E=0\\
&\Longrightarrow \forall\,x\in M\quad \langle
\xi(x),\xi(x)\rangle_E=0\Longrightarrow \forall\,x\in M\quad
\xi(x)=0
\end{align*}
\item \textbf{T is onto:} Let $u\in D(M,E^\vee)$. Our goal is to
show that there exists $\xi\in D(M,E)$ such that $u=T_\xi$. Note
that
\begin{equation*}
\forall\,x\in M\,\quad u(x)=T_\xi(x)\Longleftrightarrow
\forall\,x\in M\,\forall\,a\in E_x\quad \langle a,\xi(x)\rangle_E\,
\mu(x)=[u(x)](a)
\end{equation*}
Since $\mathcal{D}_x$ is $1$-dimensional and both $\mu(x)$ (which is a positive smooth density) and $[u(x)][a]$
belong to $\mathcal{D}_x,$, there exists a number $b(x,a)$ such that
\begin{equation*}
[u(x)](a)=b(x,a)\mu(x)
\end{equation*}
So we need to show that there exists $\xi\in D(M,E)$ such that
\begin{equation*}
\forall\,x\in M\, \forall\,a\in E_x\qquad \langle
a,\xi(x)\rangle_E=b(x,a)
\end{equation*}
The above equality uniquely defines a functional on $E_x$ which
gives us a unique element $\xi(x)\in E_x$ by the Riesz
representation theorem. It remains to prove that $\xi$ is smooth.
To this end, we will
show that for each $\alpha$, $\xi|_{U_\alpha}$ is smooth. Let
$(s_1,\cdots,s_r)$ be a smooth orthonormal frame for
$E_{U_\alpha}$.
\begin{equation*}
\forall\,x\in U_\alpha\quad \xi(x)=\xi^1(x)s_1(x)+\cdots+\xi^r(x)s_r(x)
\end{equation*}
It suffices to show that $\xi^1,\cdots,\xi^r$ are smooth
functions (see Theorem \ref{thmfalllocalframe21}). We have
\begin{equation*}
\xi^i(x)=\langle \xi(x),s_i(x)\rangle_E
\end{equation*}
It follows from the definition of $\xi(x)$ that
\begin{equation*}
[u(x)][s_i(x)]=\langle s_i(x),\xi(x)\rangle_E\, \mu(x)
\end{equation*}
Therefore $\xi^i(x)$ satisfies the following equality
\begin{equation*}
[u(x)][s_i(x)]=\xi^i(x)\mu(x)
\end{equation*}
That is, if we define a section of $\mathcal{D}\rightarrow U_\alpha$ by
\begin{equation*}
[u,s_i]:U_\alpha\rightarrow \mathcal{D},\quad x\mapsto [u(x)][s_i(x)]
\end{equation*}
then $\xi^i$ is the component of this section with respect to the
smooth frame $\{\mu(x)\}$ on $U_\alpha$. The smoothness of $\xi^i$ follows from the fact that if
$N$ is any manifold, $E\rightarrow N$ is a vector bundle and $u$
and $v$ are in $\mathcal{E}(N,E^\vee)$ and $\mathcal{E}(N,E)$,
respectively, then $[u,v]$ is in $\mathcal{E}(N,\mathcal{D})$; indeed, the
local representation of $[u,v]$ is $\sum_l\tilde{u}^l\tilde{v}^l$
which is a smooth function because $\tilde{u}^l$ and
$\tilde{v}^l$ are smooth functions.
\item \textbf{$T: D(M,E)\rightarrow D(M,E^\vee)$ is continuous:}\\
We make use of Theorem \ref{thmapptvconvergence1}.
 Recall that
\begin{enumerate}
\item The topology on $D(M,E)$ is induced by the seminorms:
{\fontsize{9}{10}{\begin{equation*}
\forall\,1\leq l\leq r,\forall\,1\leq \alpha \leq N,\forall\, k\in
\mathbb{N},\forall\, K\subseteq U_\alpha
\textrm{(compact)}\qquad
p_{l,\alpha,k,K}(\xi)=\|\rho_\alpha^l\circ\xi\circ\varphi_\alpha^{-1}\|_{\varphi_\alpha(K),k}
\end{equation*}}}
\item The topology on $D(M,E^\vee)$ is induced by the seminorms:
{\fontsize{9}{10}{\begin{equation*}
\forall\,1\leq l\leq r,\forall\,1\leq \alpha \leq N,\forall\, k\in
\mathbb{N},\forall\, K\subseteq U_\alpha
\textrm{(compact)}\qquad
q_{l,\alpha,k,K}(\eta)=\|(\rho_{\alpha}^\vee)^l\circ\eta\circ\varphi_\alpha^{-1}\|_{\varphi_\alpha(K),k}
\end{equation*}}}
\end{enumerate}
For all $\xi\in D(M,E)$ we have
{\fontsize{10}{10}{\begin{align*}
q_{l,\alpha,k,K}(T_\xi)=\|(\rho_{\alpha}^\vee)^l\circ
 T_\xi \circ\varphi_\alpha^{-1}\|_{\varphi_\alpha(K),k}
 =\|(\rho_{\mathcal{D},\varphi_\alpha})\circ (T_\xi\circ\varphi_\alpha^{-1}) \circ\underbrace{(\rho|_{E_x})^{-1}(e_l)}_{s_l(x)}\|_{\varphi_\alpha(K),k}
\end{align*}}}
where $(e_1,\cdots,e_r)$ is the standard basis for $\reals^r$. Let $y=\varphi_\alpha(x)$. Note that
\begin{equation*}
[T_\xi(\varphi_\alpha^{-1}(y))][s_l(x)]=\langle
s_l(x),\xi(x)\rangle_E\,\mu(x)
\end{equation*}
Therefore if we define the smooth function $f_\alpha$ on
$U_\alpha$ by $\mu(x)=f_\alpha(x)|dx^1\wedge\cdots\wedge dx^n|$,
then
{\fontsize{9}{10}{\begin{equation}\lab{eqnwinterdual13}
(\rho_{\mathcal{D},\varphi_\alpha})\circ (T_\xi\circ\varphi_\alpha^{-1})
\circ s_l(x)=\langle s_l(x),\xi(x)\rangle_E
f_\alpha(x)=\xi^l(x)f_\alpha(x)=(\rho_\alpha^l\circ \xi \circ
\varphi_\alpha^{-1}(y)) (f_\alpha \circ \varphi_\alpha^{-1}(y))
\end{equation}}}
So if we let
\begin{equation*}
C=\max_{y\in \varphi_\alpha(K),|\beta|\leq k
}|\partial^{\beta}(f_\alpha \circ \varphi_\alpha^{-1}(y))|
\end{equation*}
Then
{\fontsize{9}{10}{\begin{equation*}
q_{l,\alpha,k,K}(T_\xi)=\|(\rho_\alpha^l\circ \xi \circ
\varphi_\alpha^{-1}(y)) (f_\alpha \circ
\varphi_\alpha^{-1}(y))\|_{\varphi_\alpha(K),k}\leq C
\|\rho_\alpha^l\circ \xi \circ \varphi_\alpha^{-1}(y))
\|_{\varphi(K),k}=C\,p_{l,\alpha,k,K}(\xi)
\end{equation*}}}
\item \textbf{$T: (C^{\infty}(M,E),\|.\|_{e,q})\rightarrow
(C^{\infty}(M,E^\vee),\|.\|_{e,q})$ is a topological isomorphism:}
\begin{equation*}
\|\xi\|_{W^{e,q}(M,E;\Lambda)}=\sum_{\alpha=1}^N\sum_{l=1}^r \|\rho_\alpha^l\circ \psi_\alpha\xi\circ \varphi_\alpha^{-1}\|_{W^{e,q}(\varphi_\alpha(U_\alpha))}
\end{equation*}
\begin{align*}
\|T_\xi\|_{W^{e,q}(M,E^\vee;\Lambda^\vee)}&=\sum_{\alpha=1}^N\sum_{l=1}^r
\|(\rho_{\alpha}^\vee)^l\circ
 \psi_\alpha T_\xi
 \circ\varphi_\alpha^{-1}\|_{W^{e,q}(\varphi_\alpha(U_\alpha))}
\end{align*}
By Equation \ref{eqnwinterdual13}, we have
\begin{equation*}
(\rho_{\alpha}^\vee)^l\circ \psi_\alpha T_\xi \circ\varphi_\alpha^{-1}=\rho_{\mathcal{D},\varphi_\alpha}\circ (\psi_\alpha T_\xi\circ\varphi_\alpha^{-1})\circ s_l(x)=(\rho_\alpha^l\circ \psi_\alpha\xi \circ \varphi_\alpha^{-1})
(f_\alpha \circ
\varphi_\alpha^{-1})
\end{equation*}
Therefore
\begin{equation*}
\|T_\xi\|_{W^{e,q}(M,E^\vee;\Lambda^\vee)}=\sum_{\alpha=1}^N\sum_{l=1}^r
\|(\rho_\alpha^l\circ \psi_\alpha\xi \circ \varphi_\alpha^{-1})
(f_\alpha \circ
\varphi_\alpha^{-1})\|_{W^{e,q}(\varphi_\alpha(U_\alpha))}
\end{equation*}
Now we just need to notice that $f_\alpha \circ
\varphi_\alpha^{-1}$ is a positive function and belongs to
$C^{\infty}(\varphi_\alpha(U_\alpha))$ (so $\frac{1}{f_\alpha \circ
\varphi_\alpha^{-1}}$ is also smooth) and $\rho_\alpha^l\circ \psi_\alpha\xi \circ \varphi_\alpha^{-1}$ has support in the compact set $\varphi_\alpha(\textrm{supp}(\psi_\alpha))$ to conclude that
\begin{equation*}
\|\xi\|_{W^{e,q}(M,E;\Lambda)}\simeq \|T_\xi\|_{W^{e,q}(M,E^\vee;\Lambda^\vee)}
\end{equation*}
\end{itemizeX}
\end{proof}
\begin{lemma}\lab{lemwinterdual11}
Let $M^n$ be a compact smooth manifold and let $\pi:E\rightarrow M$ be a
vector bundle of rank $r$ equipped with a fiber metric
$\langle .,.\rangle_E$. Let $e\in\reals$ and $q\in (1,\infty)$. Suppose
$\Lambda=\{(U_\alpha,\varphi_\alpha,\rho_\alpha,\psi_\alpha)\}_{\alpha=1}^N$ is an augmented
total trivialization atlas for $E\rightarrow M$. If $e$ is a noninteger less than $-1$ further assume that the total trivialization atlas in $\Lambda$ is GGL. Then
$D(M,E)\hookrightarrow W^{e,q}(M,E)\hookrightarrow D'(M,E)$.
\end{lemma}
\begin{proof}
For $e\in \mathbb{Z}$ the claim is proved in \cite{Reus1}. For $e\in
\reals\setminus \mathbb{Z}$ we have
\begin{align*}
& W^{e,q}(M,E;\Lambda)\hookrightarrow W^{\floor{e},q}(M,E;\Lambda)\hookrightarrow
D'(M,E)\\
& D(M,E)\hookrightarrow W^{\floor{e}+1,q}(M,E;\Lambda)\hookrightarrow W^{e,q}(M,E;\Lambda)
\end{align*}
\end{proof}
\begin{theorem}\lab{thmwinterdual22}
Let $M^n$ be a compact smooth manifold and let $\pi:E\rightarrow M$ be a
vector bundle of rank $r$ equipped with a fiber metric
$\langle .,.\rangle_E$. Let $e\in\reals$ and $q\in (1,\infty)$. Suppose
$\Lambda=\{(U_\alpha,\varphi_\alpha,\rho_\alpha,\psi_\alpha)\}_{\alpha=1}^N$ is an augmented
total trivialization atlas for $E\rightarrow M$ which trivializes the fiber metric. If $e$ is a noninteger whose magnitude is greater than $1$ further assume that the total trivialization atlas in $\Lambda$ is GL compatible with itself. Fix a positive smooth density $\mu$ on
$M$. \\Consider the $L^2$ inner product on
$D(M,E)$ defined by
\begin{equation*}
\langle u,v\rangle_2=\int_M \langle u,v\rangle_{E} \mu
\end{equation*}
Then
\begin{enumerate}[(i)]
\item $\langle .,.\rangle_2$ extends uniquely to a continuous bilinear pairing
$\langle .,.\rangle_2:W^{-e,q'}(M,E;\Lambda)\times W^{e,q}(M,E;\Lambda)\rightarrow
\reals$. (We are using the same notation (i.e. $\langle
.,.\rangle_2$) for the extended bilinear map!)
\item The map $S: W^{-e,q'}(M,E;\Lambda)\rightarrow
[W^{e,q}(M,E;\Lambda)]^*$ defined by $S(u)=l_u$ where
\begin{equation*}
l_u: W^{e,q}(M,E;\Lambda)\rightarrow \reals,\quad l_u(v)=\langle u,v\rangle_2
\end{equation*}
is a well-defined topological isomorphism.
\end{enumerate}
In particular,  $[W^{e,q}(M,E;\Lambda)]^*$ can be identified with $W^{-e,q'}(M,E;\Lambda)$.
\end{theorem}
\begin{proof} ${}$
\begin{enumerateX}
\item By Theorem \ref{thmwinterdualab}, in order to prove (i) it is enough to show that
\begin{equation*}
\langle .,.\rangle_2: (C^{\infty}(M,E),\|.\|_{-e,q'})\times
(C^{\infty}(M,E),\|.\|_{e,q}) \rightarrow \reals
\end{equation*}
is a \textbf{continuous} bilinear map. Denote the corresponding
standard trivialization map for the density bundle $\mathcal{D}\rightarrow M$ by
$\rho_{\mathcal{D},\varphi_\alpha}$.  Let $\Lambda_1=\{(U_\alpha,\varphi_\alpha,\rho_\alpha,\tilde{\psi}_\alpha)\}_{\alpha=1}^N$ be an augmented total trivialization atlas for $E$ where
$\tilde{\psi}_{\alpha}=\frac{\psi_\alpha^3}{\sum_{\beta=1}^N
\psi_\beta^3}$. Note that $\frac{1}{\sum_{\beta=1}^N
\psi_\beta^3}\circ \varphi_\alpha^{-1}\in BC^{\infty}(\varphi_\alpha(U_\alpha))$. Let
$K_\alpha=\textrm{supp}\psi_\alpha$.  Recall that on $U_\alpha$ we may write $\mu=h_\alpha |dx^1\wedge \cdots \wedge dx^n|$ where $h_\alpha=\rho_{\mathcal{D},\varphi_\alpha}\circ \mu$ is smooth. Moreover, for any continuous function $f:M\rightarrow \reals$
\begin{align*}
\int_M f\mu &=\sum_{\alpha=1}^N\int_M \tilde{\psi}_\alpha f \mu\\
&=\sum_{\alpha=1}^N \int_{\varphi_\alpha(U_\alpha)}(\varphi_\alpha^{-1})^*(\tilde{\psi}_\alpha f\mu)\\
&=\sum_{\alpha=1}^N \int_{\varphi_\alpha(U_\alpha)}(\tilde{\psi}_\alpha f\circ\varphi_\alpha^{-1})(\varphi_\alpha^{-1})^*\mu\\
&=\sum_{\alpha=1}^N \int_{\varphi_\alpha(U_\alpha)}(\tilde{\psi}_\alpha f\circ\varphi_\alpha^{-1})(h_\alpha\circ\varphi_\alpha^{-1})\,dV\\
&\preceq \sum_{\alpha=1}^N \int_{\varphi_\alpha(U_\alpha)}(\psi_\alpha^2 f\circ\varphi_\alpha^{-1})(\psi_\alpha h_\alpha\circ\varphi_\alpha^{-1})\,dV \quad ({\fontsize{10}{10}{\textrm{$\frac{1}{\sum_{\beta=1}^N
\psi_\beta^3}\circ \varphi_\alpha^{-1}\in BC^{\infty}(\varphi_\alpha(U_\alpha))$}}})
\end{align*}
Therefore we have
\begin{align*}
|\int_M \langle u,v\rangle_E \mu|&=|\sum_{\alpha=1}^N \int_M
\tilde{\psi}_\alpha \langle u,v\rangle_E \mu |\\
&\preceq|\sum_{\alpha=1}^N
\int_{\varphi_\alpha(U_\alpha)} (\psi_\alpha^2 \langle u,v
\rangle_E\circ\varphi_\alpha^{-1})
(\psi_\alpha h_\alpha\circ\varphi_\alpha^{-1})dV |
\end{align*}
Since by assumption the total trivialization atlas in $\Lambda$ trivializes the metric, we get
\begin{align*}
&|\int_M \langle u,v\rangle_E \mu|\preceq \sum_{\alpha=1}^N\sum_{i=1}^r
|\int_{\varphi_\alpha(U_\alpha)} (\psi_\alpha\circ
\varphi_\alpha^{-1} \tilde{u}_i)(\psi_\alpha\circ
\varphi_\alpha^{-1}\tilde{v}_i)(\psi_\alpha h_\alpha\circ\varphi_\alpha^{-1})dV|\\
&\stackrel{\textrm{Remark \ref{remjanduality1}}}{\preceq}
\sum_{\alpha=1}^N\sum_{i=1}^r\|(\psi_\alpha\circ \varphi_\alpha^{-1}
\tilde{u}_i)\|_{W^{-e,q'}(\varphi_\alpha(U_\alpha))}\|(\psi_\alpha\circ
\varphi_\alpha^{-1}\tilde{v}_i)(\psi_\alpha h_\alpha\circ\varphi_\alpha^{-1})\|_{W^{e,q}(\varphi_\alpha(U_\alpha))}\\
&\preceq
\sum_{\alpha=1}^N\sum_{i=1}^r\|(\psi_\alpha\circ \varphi_\alpha^{-1}
\tilde{u}_i)\|_{W^{-e,q'}(\varphi_\alpha(U_\alpha))}\|(\psi_\alpha\circ
\varphi_\alpha^{-1}\tilde{v}_i)\|_{W^{e,q}(\varphi_\alpha(U_\alpha))}\\
&\preceq \big[\sum_{\alpha=1}^N\sum_{i=1}^r\|(\psi_\alpha\circ
\varphi_\alpha^{-1}
\tilde{u}_i)\|_{W^{-e,q'}(\varphi_\alpha(U_\alpha))}\big]\big[\sum_{\alpha=1}^N\sum_{i=1}^r\|(\psi_\alpha\circ
\varphi_\alpha^{-1}\tilde{v}_i)\|_{W^{e,q}(\varphi_\alpha(U_\alpha))}\big]\\
&=\|u\|_{W^{-e,q'}(M,E;\Lambda)}\|v\|_{W^{e,q}(M,E;\Lambda)}
\end{align*}
\item For each $u\in W^{-e,q'}(M,E;\Lambda)$, $l_u$ is continuous because
$\langle.,.\rangle_2$ is continuous. So $S$ is well-defined.
\item $S$ is a  continuous linear map because
\begin{align*}
\forall\,u\in W^{-e,q'}(M,E;\Lambda)\quad &\|S(u)\|_{(W^{e,q}(M,E;\Lambda))^*}=\sup_{0\neq v\in
W^{e,q}(M,E;\Lambda)} \frac{|S(u)v|}{\|v\|_{W^{e,q}(M,E;\Lambda)}}\\
&=\sup_{0\neq v\in W^{e,q}(M,E;\Lambda)}
\frac{|\langle u,v\rangle_2|}{\|v\|_{W^{e,q}(M,E;\Lambda)}}\leq C\|u\|_{W^{-e,q'}(M,E;\Lambda)}
\end{align*}
where $C$ is the norm of the continuous bilinear form $\langle .,.
\rangle_2$.
\item $S$ is injective: suppose $u\in W^{-e,q'}(M,E;\Lambda)$ is such that $S(u)=0$, then
\begin{equation*}
\forall\, v\in W^{e,q}(M,E;\Lambda)\quad  l_u(v)=\langle u,v\rangle_2=0
\end{equation*}
We need to show that $u=0$.
\begin{itemizeX}
\item \textbf{Step 1:} For $\xi$ and $\eta$ in $D(M,E)$
we have
\begin{equation*}
\langle \xi,\eta\rangle_2=\langle
u_\xi,T\eta\rangle_{[D(M,E^\vee)]^*\times D(M,E^\vee)}
\end{equation*}
where $T$ is the map introduced in Lemma \ref{lemspringapr13}. (Note that if we identify $D(M,E)$ with a subset of
$[D(M,E^\vee)]^*$, then we may write $\xi$ instead of $u_\xi$ on
the right hand side of the above equality.) The reason is as
follows
\begin{align*}
\langle u_\xi,T\eta\rangle_{[D(M,E^\vee)]^*\times
D(M,E^\vee)}&=\int_M [T_\eta(x)][\xi(x)]\qquad \textrm{by
definition of $u_\xi$}
\end{align*}
Recall that by definition of $T_\eta$ we have
\begin{equation*}
\forall\,x\in M\quad \forall a\in E_x\qquad
[T_\eta(x)][a]=\langle a,\eta(x)\rangle_E\,\mu
\end{equation*}
In particular
\begin{equation*}
[T_\eta(x)][\xi_(x)]=\langle \xi(x),\eta(x)\rangle_E\,\mu
\end{equation*}
Therefore
\begin{equation*}
\langle u_\xi,T\eta\rangle_{[D(M,E^\vee)]^*\times
D(M,E^\vee)}=\int_M \langle \xi(x),\eta(x)\rangle_E\mu=\langle
\xi,\eta\rangle_2
\end{equation*}
\item \textbf{Step 2:} For $w\in W^{-e,q'}(M,E;\Lambda)$ and $\eta\in
D(M,E)\subseteq W^{e,q}(M,E;\Lambda)$ we have
\begin{equation*}
\langle w,\eta\rangle_2=\langle
w,T\eta\rangle_{[D(M,E^\vee)]^*\times D(M,E^\vee)}
\end{equation*}
Indeed, let $\{\xi_m\}$ be a sequence in {\fontsize{9}{9}{$D(M,E)$}} that
converges to $w$ in {\fontsize{8}{8}{$W^{-e,q'}(M,E;\Lambda)$}}. Note that
$W^{-e,q'}(M,E;\Lambda)\hookrightarrow [D(M,E^\vee)]^*$, so the sequence
converges to $w$ in $[D(M,E^\vee)]^*$ as well. By what was proved in the
first step, for all $m$
\begin{equation*}
\langle \xi_m,\eta\rangle_2=\langle
\xi_m,T\eta\rangle_{[D(M,E^\vee)]^*\times D(M,E^\vee)}
\end{equation*}
Taking the limit as $m\rightarrow \infty$ proves the claim.
\item \textbf{Step 3:}
Finally note that for all $v\in D(M,E)\subseteq W^{e,q}(M,E;\Lambda)$
\begin{equation*}
\langle T^*u,v\rangle_{[D(M,E)]^*\times D(M,E)}=\langle u,Tv
\rangle_{[D(M,E^\vee)]^*\times D(M,E^\vee)}=\langle u,v\rangle_2=0
\end{equation*}
Therefore $T^*u=0$ as an element of $[D(M,E)]^*$. $T$ is a
continuous bijective map, so $T^*$ is injective. It follows that
$u=0$ as an element of $[D(M,E^\vee)]^*$ and so $u=0$ as an element
of $W^{-e,q'}(M,E;\Lambda)$.
\end{itemizeX}
\item $S$ is surjective. Let $F\in [W^{e,q}(M,E;\Lambda)]^*$. We need to
show that there is an element $u\in W^{-e,q'}(M,E;\Lambda)$ such that
$S(u)=F$. Since $D(M,E)$ is dense in $W^{e,q}(M,E;\Lambda)$, it is
enough to show that there exists an element {\fontsize{8}{8}{$u\in W^{-e,q'}(M,E;\Lambda)$}}
with the property that
\begin{equation*}
\forall\, \xi\in D(M,E)\quad F(\xi)=\langle
u,\xi\rangle_2
\end{equation*}
Note that according to what was proved in Step 2
\begin{equation*}
\langle u,\xi\rangle_2=\langle u,T\xi\rangle_{[D(M,E^\vee)]^*\times
D(M,E^\vee)}=\langle T^*u,\xi\rangle_{[D(M,E)]^*\times D(M,E)}
\end{equation*}
So we need to show that there exists an element $u\in
W^{-e,q'}(M,E;\Lambda)$ such that
\begin{equation*}
\forall\, \xi\in D(M,E)\quad F(\xi)=\langle
T^*u,\xi\rangle_{[D(M,E)]^*\times D(M,E)}
\end{equation*}
Since $D(M,E)\hookrightarrow W^{e,q}(M,E;\Lambda)$, $F|_{D(M,E)}$ is an
element of $[D(M,E)]^*$. We let
\begin{equation*}
u:=[T^{-1}]^{*}(F|_{D(M,E)})\in [D(M,E^\vee)]^*
\end{equation*}
Clearly $u$ satisfies the desired equality (note that
$[T^{-1}]^{*}=[T^*]^{-1}$). So we just need to show that $u$ is
indeed an element of $W^{-e,q'}(M,E;\Lambda)$. Note that
\begin{equation*}
u\in W^{-e,q'}(M,E;\Lambda)\Longleftrightarrow \forall\,1\leq \alpha\leq
N\quad H_\alpha(\psi_\alpha u)\in
[W^{-e,q'}_{\varphi_\alpha(\textrm{supp}\psi_\alpha)}(\varphi_\alpha(U_\alpha))]^{\times
r }
\end{equation*}
Since $\textrm{supp}(\psi_\alpha u)\subseteq
\textrm{supp}\psi_\alpha$, it follows from Remark \ref{remfall134} that
\begin{equation*}
\forall\,1\leq l\leq r\quad \textrm{supp}([H_\alpha(\psi_\alpha
u)]^l)\subset \varphi_\alpha (\textrm{supp}\psi_\alpha)
\end{equation*}
It remains to prove that $[H_\alpha(\psi_\alpha u)]^l\in
W^{-e,q'}(\varphi_\alpha(U_\alpha))$. Note that
{\fontsize{10}{10}{\begin{align*}
& \textrm{for}\,e\geq 0\qquad  [W_{0}^{e,q}(\varphi_\alpha(U_\alpha))]^*=W^{-e,q'}(\varphi_\alpha(U_\alpha))\\
& \textrm{for}\,e< 0\qquad
[W_{0}^{e,q}(\varphi_\alpha(U_\alpha))]^*=[W^{e,q}(\varphi_\alpha(U_\alpha))]^*=
W_{0}^{-e,q'}(\varphi_\alpha(U_\alpha))\subseteq
W^{-e,q'}(\varphi_\alpha(U_\alpha))
\end{align*}}}
Consequently for all $e$
\begin{equation*}
[W^{e,q}_0(\varphi_\alpha(U_\alpha))]^*\subseteq
W^{-e,q'}(\varphi_\alpha(U_\alpha))
\end{equation*}
Therefore it is enough to show that
\begin{equation*}
[H_\alpha(\psi_\alpha u)]^l\in [W^{e,q}_0(\varphi_\alpha(U_\alpha))]^*
\end{equation*}
To this end, we need to prove that
\begin{equation*}
[H_\alpha(\psi_\alpha u)]^l:
(C_c^\infty(\varphi_\alpha(U_\alpha)),\|.\|_{e,q})\rightarrow
\reals
\end{equation*}
is continuous. For all $\xi\in
C_c^{\infty}(\varphi_\alpha(U_\alpha))$ we have
{\fontsize{10}{10}{\begin{align*}
[H_\alpha(\psi_\alpha u)]^l(\xi)&=\langle \psi_\alpha
u,g_{l,\xi,U_\alpha,\varphi_\alpha}\rangle_{[D(U_\alpha,E_{U_\alpha}^\vee)]^*\times D(U_\alpha,E_{U_\alpha}^\vee) }=\langle
 u,\psi_\alpha g_{l,\xi,U_\alpha,\varphi_\alpha}\rangle_{[D(M,E^\vee)]^*\times
D(M,E^\vee) }\\
&=\langle
 [T^{-1}]^*F|_{D(M,E)},\psi_\alpha g_{l,\xi,U_\alpha,\varphi_\alpha}\rangle_{[D(M,E^\vee)]^*\times
D(M,E^\vee)}\\
&=\langle F|_{D(M,E)},T^{-1}(\psi_\alpha
g_{l,\xi,U_\alpha,\varphi_\alpha})\rangle_{D^*(M,E)\times D(M,E)}=F(T^{-1}(\psi_\alpha
g_{l,\xi,U_\alpha,\varphi_\alpha}))
\end{align*}}}
Thus $[H_\alpha(\psi_\alpha u)]^l$ is the composition of the following
maps
{\tiny{\begin{align*}
& (C_c^\infty(\varphi_\alpha(U_\alpha)),\|.\|_{e,q})\rightarrow
[W^{e,q}_{\varphi_\alpha(\textrm{supp}\psi_\alpha)}(\varphi_\alpha(U_\alpha))]^{\times
r }\cap [C_c^{\infty}(\varphi_\alpha(U_\alpha))]^{\times
r}\rightarrow
W^{e,q}_{\textrm{supp}\psi_\alpha}(M,E^\vee;\Lambda^\vee)\cap
C^{\infty}(M,E^\vee)\\
 &\hspace{8cm}\rightarrow (C^{\infty}(M,E),\|\|_{e,q})\rightarrow \reals\\
 &\\
& \xi \mapsto (0,\cdots,0 ,\underbrace{(\psi_\alpha\circ
\varphi_\alpha^{-1})\xi}_{\textrm{$l^{th}$
position}},0,\cdots,0)\mapsto H^{-1}_{E^\vee,U_\alpha,\varphi_\alpha}(0,\cdots,0
,(\psi_\alpha\circ
\varphi_\alpha^{-1})\xi,0,\cdots,0)=\psi_\alpha g_{l,\xi,U_\alpha,\varphi_\alpha}\\
&\hspace{8cm}\mapsto
T^{-1}(\psi_\alpha g_{l,\xi,U_\alpha,\varphi_\alpha})\mapsto F(T^{-1}(\psi_\alpha
g_{l,\xi,U_\alpha,\varphi_\alpha}))
\end{align*}}}
which is a composition of continuous maps.
\item $S: W^{-e,q'}(M,E;\Lambda)\rightarrow [W^{e,q}(M,E;\Lambda)]^*$ is a
continuous bijective map, so by the Banach isomorphism theorem,
it is a topological isomorphism.
\end{enumerateX}
\end{proof}
\begin{remark}\lab{remwinterdualequiv}${}$
\begin{enumerateX}
\item The result of Theorem \ref{thmwinterdual22} remains valid even if $\Lambda=\{(U_\alpha,\varphi_\alpha,\rho_\alpha,\psi_\alpha)\}$ does not trivialize the fiber metric. Indeed, if $e$ is not a noninteger whose magnitude is greater than $1$, then the Sobolev spaces $W^{e,q}$ and $W^{-e,q'}$ are independent of the choice of augmented total trivialization atlas. If $e$ is a noninteger whose magnitude is greater than $1$, then by Theorem \ref{thmfalltrivializametric1} there exists an augmented total trivialization atlas $\tilde{\Lambda}=\{(U_\alpha,\varphi_\alpha,\tilde{\rho}_\alpha,\psi_\alpha)\}$ that trivializes the metric and has the same base atlas as $\Lambda$ (so it is GL compatible with $\Lambda$ because by assumption $\Lambda$ is GL compatible with itself). So we can replace $\Lambda$ by $\tilde{\Lambda}$.
\item Let $\Lambda$ be an augmented total trivialization atlas that is GL compatible with itself. Let $e$ be a noninteger less than $-1$ and $q\in (1,\infty)$. By Theorem  \ref{thmwinterdual22} and the above observation, $W^{e,q}(M,E;\Lambda)$ is topologically isomorphic to $[W^{-e,q'}(M,E;\Lambda)]^*$. However, the space $W^{-e,q'}(M,E;\Lambda)$ is independent of $\Lambda$. So we may conclude that even when $e$ is a noninteger less than $-1$, the space $W^{e,q}(M,E;\Lambda)$ is independent of the choice of the augmented total trivialization atlas as long as the corresponding total trivialization atlas is \textbf{GL compatible} with itself.
\end{enumerateX}
\end{remark}

\subsection{On the Relationship Between Various Characterizations}
Here we discuss the relationship between the characterizations of Sobolev spaces given in Remark \ref{remvarcharac1} and our original definition (Definition \ref{defwintermainsobolev}).
\begin{enumerate}
\item Suppose $e\geq 0$.
{\fontsize{10}{10}{\begin{equation*}
W^{e,q}(M,E;\Lambda)=\{u\in L^q(M,E): \|u\|_{W^{e,q}(M,E;\Lambda)}=\sum_{\alpha=1}^N\sum_{l=1}^r
\|(\rho_\alpha)^l\circ
 (\psi_\alpha u)\circ\varphi_\alpha^{-1}\|_{W^{e,q}(\varphi_\alpha(U_\alpha))}<\infty\}
\end{equation*}}}
As a direct consequence of Theorem \ref{thmwinterembedding1}, for $e\geq 0$, $W^{e,q}(M,E;\Lambda)\hookrightarrow L^q(M,E)$. Therefore the above characterization is completely consistent with the original definition.
\item 
{\fontsize{9.5}{10}{\begin{equation*}
W^{e,q}(M,E;\Lambda)=\{u\in D'(M,E): \|u\|_{W^{e,q}(M,E;\Lambda)}=\sum_{\alpha=1}^N\sum_{l=1}^r
\|\textrm{ext}^0_{\varphi_\alpha(U_\alpha),\reals^n}[H_\alpha
(\psi_\alpha u )]^l\|_{W^{e,q}(\reals^n)}<\infty\}
\end{equation*}}}
It follows from Corollary \ref{corofallusef1} that
\begin{itemizeX}
\item if $e$ is not a noninteger less than $-1$, then
\begin{equation*}
\|[H_\alpha
(\psi_\alpha u )]^l\|_{W^{e,q}(\varphi_\alpha(U_\alpha))}\simeq\|\textrm{ext}^0_{\varphi_\alpha(U_\alpha),\reals^n}[H_\alpha
(\psi_\alpha u )]^l\|_{W^{e,q}(\reals^n)}\, ,
\end{equation*}
\item if $e$ is a noninteger less than $-1$ and $\varphi_\alpha(U_\alpha)$ is $\reals^n$ or a bounded open set with Lipschitz continuous boundary, then again the above equality holds.
\end{itemizeX}
Therefore when $e$ is not a noninteger less than $-1$, the above characterization completely agrees with the original definition. If $e$ is a noninteger less than $-1$ and the total trivialization atlas corresponding to $\Lambda$ is GGL, then again the two definitions agree.
\item
{\fontsize{10}{10}{\begin{equation*}
W^{e,q}(M,E;\Lambda)=\{u\in D'(M,E): [H_\alpha (u|_{U_\alpha}
)]^l\in W^{e,q}_{loc}(\varphi_\alpha(U_\alpha)),\,\,\forall\, 1\leq \alpha\leq N,\,\forall\, 1\leq l\leq r\}
\end{equation*}}}
It follows immediately from Theorem \ref{thmapp12} and Corollary \ref{corapp2} that the above characterization of the set of Sobolev functions is equivalent to the set given in the original definition provided we assume that if $e$ is a noninteger less than $-1$, then $\Lambda$ is GL compatible with itself.
\item 
$W^{e,q}(M,E;\Lambda)$ is the completion of
$C^\infty (M,E)$ with respect to the norm
\begin{equation*}
\|u\|_{W^{e,q}(M,E;\Lambda)}=\sum_{\alpha=1}^N\sum_{l=1}^r \|(\rho_\alpha)^l\circ (\psi_\alpha u) \circ
\varphi_\alpha^{-1}\|_{W^{e,q}(\varphi_\alpha(U_\alpha))}
\end{equation*}
It follows from Theorem \ref{thmwinterh22} that if $e$ is not a noninteger less than $-1$ the above characterization of Sobolev spaces is equivalent to the original definition. Also if $e$ is a noninteger less than $-1$ and $\Lambda $ is GL compatible with itself the two characterizations are equivalent.
\end{enumerate}
Now we will focus on proving the equivalence of the original definition and the fifth characterization of Sobolev spaces. In what follows instead of $\|.\|_{W^{k,q}(M,E;g,\grad^E)}$ we just write $|.|_{W^{k,q}(M,E)}$. Also note that since $k$ is a nonnegative integer, the choice of the augmented total trivialization atlas in Definition \ref{defwintermainsobolev} is immaterial. Our proof follows the argument presented in \cite{grosse1} and is based on the following five facts:
\begin{itemizeX}
\item \textbf{Fact 1:} Let $u\in C^\infty (M,E)$ be such that $\textrm{supp} u\subseteq
U_\beta$ for some $1\leq \beta\leq N$. Then
\begin{equation*}
|u|_{L^q(M,E)}^q=\int_M|u|_E^q dV_g\simeq \sum_l \|
\underbrace{\rho_\beta^l\circ u}_{u^l}\circ \varphi_\beta^{-1}
\|_{L^q(\varphi_\beta(U_\beta))}^q
\end{equation*}
\item \textbf{Fact 2:} Let $u\in C^\infty (M,E)$ be such that $\textrm{supp} u\subseteq
U_\beta$ for some $1\leq \beta\leq N$. Then
\begin{equation*}
|u|_{W^{k,q}(M,E)}^q\simeq \sum_{s=0}^k\sum_{a=1}^r\sum_{1\leq
j_1,\cdots,j_s\leq n}\|\big( (\grad^E)^s u\big)^a_{j_1\cdots j_s
}\circ \varphi_\beta^{-1}\|_{L^q(\varphi_\beta(U_\beta))}^q
\end{equation*}
\begin{proof}
\begin{align*}
|u|_{W^{k,q}(M,E)}^q&\simeq
\sum_{s=0}^k|(\grad^E)^s u|^q_{L^q(M,(T^*M)^{\otimes i}\otimes
E)}\\
&\stackrel{\text{Fact
1}}{\simeq}\sum_{s=0}^k\sum_{a=1}^r\sum_{1\leq j_1,\cdots,j_s\leq
n}\|\underbrace{\big( (\grad^E)^s u\big)^a_{j_1\cdots j_s
}}_{\textrm{components w.r.t
$(U_\beta,\varphi_\beta,\rho_\beta)$}}\circ
\varphi_\beta^{-1}\|_{L^q(\varphi_\beta(U_\beta))}^q
\end{align*}
\end{proof}
\item \textbf{Fact 3:} Let $u\in C^\infty(M,E)$ be such that $\textrm{supp}\,u\subseteq
U_\beta$ for some $1\leq \beta\leq N$. Then
\begin{equation*}
\|u\|_{W^{e,q}(M,E)}\simeq \sum_{l=1}^r \|\rho_\beta^l\circ
u\circ \varphi_\beta^{-1}\|_{W^{e,q}(\varphi_\beta(U_\beta))}
\end{equation*}
\begin{proof}
Let $\{\psi_\alpha\}$ be a partition of unity such that
$\psi_\beta=1$ on $\textrm{supp}\,u$ (note that since elements of
a partition of unity are nonnegative and their sum is equal to
$1$, we can conclude that if $\alpha\neq \beta$ then
$\psi_\alpha=0$ on $\textrm{supp}\,u$). We have
\begin{align*}
\|u\|_{W^{e,q}(M,E)}&\simeq\sum_{\alpha=1}^N\sum_{l=1}^r\|\rho_\alpha^l\circ
(\psi_\alpha u)\circ \varphi_\alpha^{-1}
\|_{W^{e,q}(\varphi_\alpha(U_\alpha))}\\
&=\sum_{l=1}^r\|\rho_\beta^l\circ
(\psi_\beta u)\circ \varphi_\beta^{-1}
\|_{W^{e,q}(\varphi_\beta(U_\beta))}=\sum_{l=1}^r\|\rho_\beta^l\circ
u\circ \varphi_\beta^{-1} \|_{W^{e,q}(\varphi_\beta(U_\beta))}
\end{align*}
\end{proof}
\item \textbf{Fact 4:} Let $u\in C^\infty(M,E)$. Then for any
multi-index $\gamma$ and all $1\leq l \leq r$ we have (on any
total trivialization triple $(U,\varphi,\rho)$):
\begin{equation*}
|\partial^\gamma[\rho^l\circ u\circ \varphi^{-1}]|\preceq
\sum_{s\leq |\gamma|}\underbrace{\sum_{a=1}^r\sum_{1\leq
j_1,\cdots,j_s\leq n}}_{\textrm{sum over all components of
$(\grad^E)^s u$ }}|\big( (\grad^E)^s u\big)^a_{j_1\cdots j_s
}\circ \varphi^{-1}|
\end{equation*}
\begin{proof}
For any multi-index $\gamma=(\gamma_1,\cdots,\gamma_n)$ we define
$\textrm{seq}\,\gamma$ to be the following list of numbers
\begin{equation*}
\textrm{seq}\,\gamma=\underbrace{1\cdots 1}_{\textrm{$\gamma_1$
times}}\underbrace{2\cdots 2}_{\textrm{$\gamma_2$ times}}\cdots
\underbrace{n\cdots n}_{\textrm{$\gamma_n$ times}}
\end{equation*}
Note that there are exactly $|\gamma|=\gamma_1+\cdots+\gamma_n$
numbers in $\textrm{seq}\,\gamma$. By Observation 2 in Section 5.5.4 we have
\begin{equation*}
\big( (\grad^E)^{|\gamma|}u\big)^l_{\textrm{seq}\,\gamma}\circ
\varphi^{-1}=\partial^\gamma[\rho^l\circ u\circ
\varphi^{-1}]+\sum_{a=1}^r\sum_{\alpha:|\alpha|<|\gamma|}
C_{\alpha a}\partial^\alpha[\rho^a\circ u\circ \varphi^{-1}]
\end{equation*}
Thus
\begin{align*}
&\partial^\gamma[\rho^l\circ u\circ \varphi^{-1}]=\big(
(\grad^E)^{|\gamma|}u\big)^l_{\textrm{seq}\,\gamma}\circ
\varphi^{-1}-\sum_{a=1}^r\sum_{\alpha:|\alpha|<|\gamma|}
C_{\alpha a}\partial^\alpha[\rho^a\circ u\circ \varphi^{-1}]\\
& \partial^\alpha[\rho^a\circ u\circ \varphi^{-1}]=\big(
(\grad^E)^{|\alpha|}u\big)^a_{\textrm{seq}\,\alpha}\circ
\varphi^{-1}-\sum_{b=1}^r\sum_{\beta:|\beta|<|\alpha|} C_{\beta
b}\partial^\beta[\rho^b\circ u\circ \varphi^{-1}]\\
& \vdots
\end{align*}
where the coefficients $C_{\alpha a}$, $C_{\beta b}$, etc. are
polynomials in terms of christoffel symbols and the metric and so
they are all bounded on the compact manifold $M$. Consequently
\begin{equation*}
|\partial^\gamma[\rho^l\circ u\circ \varphi^{-1}]|\preceq
\sum_{s\leq |\gamma|}\underbrace{\sum_{a=1}^r\sum_{1\leq
j_1,\cdots,j_s\leq n}}_{\textrm{sum over all components of
$(\grad^E)^s u$ }}|\big( (\grad^E)^s u\big)^a_{j_1\cdots j_s
}\circ \varphi_\beta^{-1}|
\end{equation*}
\end{proof}
\item \textbf{Fact 5:} Let $f\in C^{\infty}(M,E)$ and $u\in
W^{k,q}(M,\tilde{E})$ where $\tilde{E}$ is another vector bundle over $M$. Then
\begin{equation*}
\|f\otimes u\|_{W^{k,q}(M,E\otimes \tilde{E})}\preceq
\|u\|_{W^{k,q}(M,\tilde{E})}
\end{equation*}
where the implicit constant may depend on $f$ but it does not depend on $u$.
\begin{proof}
Let $\{(U_\alpha,\varphi_\alpha,\rho_\alpha)\}_{1\leq\alpha\leq N}$ and $\{(U_\alpha,\varphi_\alpha,\tilde{\rho}_\alpha)\}_{1\leq \alpha\leq N}$ be total trivialization atlases for $E$ and $\tilde{E}$, respectively. Let $\{s_{\alpha,a}=\rho_\alpha^{-1}(e_a)\}_{a=1}^r$ be the corresponding local frame for $E$ on $U_\alpha$ and
$\{t_{\alpha,b}=\tilde{\rho}_\alpha^{-1}(e_b)\}_{b=1}^{\tilde{r}}$
be the corresponding local frame for $\tilde{E}$ on $U_\alpha$. Let $G:\{1,\cdots,r\}\times \{1,\cdots,\tilde{r}\}\rightarrow \{1,\cdots,r\tilde{r}\}$ be an arbitrary but fixed bijective function. Then $\{(U_\alpha,\varphi_\alpha,\hat{\rho}_\alpha)\}$ is a total trivialization atlas for $E\otimes \tilde{E}$ where
\begin{equation*}
\hat{\rho}_\alpha(s_{\alpha,a}\otimes t_{\alpha,b})=e_{G(a,b)}\,\, (\textrm{as an element of $\reals^{r\tilde{r}}$})
\end{equation*}
and it is extended by linearity to the $E\otimes \tilde{E}|_{U_\alpha}$. Now we have
\begin{align*}
\|f\otimes u\|_{W^{k,q}(M,E\otimes \tilde{E})}&=\sum_{\alpha=1}^N\sum_{a=1}^r\sum_{b=1}^{\tilde{r}}\|\hat{\rho}_{\alpha}^{a,b}\circ (\psi_\alpha f\otimes u)\circ \varphi_\alpha^{-1}\|_{W^{k,q}(\varphi_\alpha(U_\alpha))}\\
&=\sum_{\alpha=1}^N\sum_{a=1}^r\sum_{b=1}^{\tilde{r}}\|(\psi_\alpha\circ \varphi_\alpha^{-1})(f^a_\alpha\circ \varphi_\alpha^{-1})(u^b_\alpha\circ\varphi_\alpha^{-1})\|_{W^{k,q}(\varphi_\alpha(U_\alpha))}
\end{align*}
where $f=f^a_\alpha s_{\alpha,a}$ and $u=u^b_\alpha t_{\alpha,b}$ on $U_\alpha$. Clearly $f^a_\alpha\circ \varphi_\alpha^{-1}\in C^\infty(\varphi_\alpha(U_\alpha))$. Therefore
\begin{align*}
\|f\otimes u\|_{W^{k,q}(M,E\otimes \tilde{E})}\preceq \sum_{\alpha=1}^N\sum_{b=1}^{\tilde{r}}\|(\psi_\alpha\circ \varphi_\alpha^{-1})(u^b_\alpha\circ\varphi_\alpha^{-1})\|_{W^{k,q}(\varphi_\alpha(U_\alpha))}\simeq \|u\|_{W^{k,q}(M,\tilde{E})}
\end{align*}
\end{proof}
\end{itemizeX}
\begin{itemizeX}
\item \textbf{Part I:} First we prove that $\|u\|_{W^{k,q}(M,E)}\preceq |u|_{W^{k,q}(M,E)}$.\\
${}$\\
\begin{enumerate}
\item \textbf{Case 1:} Suppose there exists $1\leq \beta\leq N$
such that $\textrm{supp}\, u\subseteq U_\beta$. We have
\begin{align*}
\|u\|_{W^{k,q}(M,E)}^q&\stackrel{\text{Fact
3}}{\simeq}\sum_{l=1}^r\|\rho^l_\beta\circ u\circ
\varphi_\beta^{-1}
\|_{W^{k,q}(\varphi_\beta(U_\beta))}^q\simeq\sum_{l=1}^r\sum_{|\gamma|\leq
 k}\|\partial^\gamma(\rho^l_\beta\circ u\circ \varphi_\beta^{-1})
\|_{L^{q}(\varphi_\beta(U_\beta))}^q\\
& \stackrel{\text{Fact 4}}{\preceq}\sum_{l=1}^r\sum_{|\gamma|\leq
 k}\sum_{s\leq
|\gamma|}\sum_{a=1}^r\sum_{1\leq j_1,\cdots,j_s\leq n}\|\big(
(\grad^E)^s u\big)^a_{j_1\cdots j_s }\circ
\varphi_\beta^{-1}\|^q_{L^q(\varphi_\beta(U_\beta))}\\
& \preceq \sum_{s=0}^k \sum_{a=1}^r\sum_{1\leq j_1,\cdots,j_s\leq
n}\|\big( (\grad^E)^s u\big)^a_{j_1\cdots j_s }\circ
\varphi_\beta^{-1}\|^q_{L^q(\varphi_\beta(U_\beta))}\\
& \stackrel{\text{Fact 2}}{\simeq} |u|^q_{W^{k,q}(M,E)}
\end{align*}
\item \textbf{Case 2:} Now let $u$ be an arbitrary element of
$C^\infty(M,E)$. We have
\begin{align*}
\|u\|_{W^{k,q}(M,E)}&=\|\sum_{\alpha=1}^N \psi_\alpha u
\|_{W^{k,q}(M,E)}\leq \sum_{\alpha=1}^N \|\psi_\alpha u
\|_{W^{k,q}(M,E)}\\
& \preceq \sum_{\alpha=1}^N |\psi_\alpha u|_{W^{k,q}(M,E)}\qquad
(\textrm{by what was proved in Case 1})\\
& \stackrel{\text{see the Box}}{\preceq}
\sum_{\alpha=1}^N|u|_{W^{k,q}(M,E)}\simeq |u|_{W^{k,q}(M,E)}
\end{align*}
\end{enumerate}
\end{itemizeX}
{\fontsize{9}{9}{\begin{mdframed}
\begin{align*}
|\psi_\alpha u|_{W^{k,q}(M,E)}^q&=\sum_{i=0}^k\|(\grad^E)^i(\psi_\alpha u
)\|^q_{L^q(M,(T^*M)^{\otimes i}\otimes
E)}\\
&=\sum_{i=0}^k\|\sum_{j=0}^i{i\choose j }\grad^j
\psi_\alpha\otimes (\grad^E)^{i-j}u \|^q_{L^q(M,(T^*M)^{\otimes
i}\otimes E)}\\
& \stackrel{\textrm{Fact 5}}{\preceq} \sum_{i=0}^k
\sum_{j=0}^i\|(\grad^E)^{i-j}u\|^q_{L^q(M,(T^*M)^{\otimes
(i-j)}\otimes E)}\\
&\preceq \sum_{s=0}^k\|(\grad^E)^s u
\|^q_{L^q(M,(T^*M)^{\otimes
 s}\otimes E)}\simeq |u|^q_{W^{k,q}(M,E)}
\end{align*}
\end{mdframed}}}
${}$\\
\begin{itemizeX}
\item \textbf{Part II:} Now we show that $|u|_{W^{k,q}(M,E)}\preceq \|u\|_{W^{k,q}(M,E)}$.\\
${}$\\
\begin{enumerate}
\item \textbf{Case 1:} Suppose there exists $1\leq \beta \leq N$
such that $\textrm{supp}u\subseteq U_\beta$.
\begin{align*}
&|u|^q_{W^{k,q}(M,E)}\stackrel{\textrm{Fact
2}}{\simeq}\sum_{s=0}^k\sum_{a=1}^r\sum_{1\leq j_1,\cdots,j_s\leq
n}\|\big( (\grad^E)^s u\big)^a_{j_1\cdots j_s }\circ
\varphi_\beta^{-1}\|_{L^q(\varphi_\beta(U_\beta))}^q\\
& \stackrel{\textrm{Observation
1 in 5.5.4}}{=}\sum_{s=0}^k\sum_{a=1}^r\sum_{1\leq j_1,\cdots,j_s\leq
n}\|\sum_{|\eta|\leq s}\sum_{l=1}^r(C_{\eta l })^a_{j_1\cdots j_s
}\partial^\eta(\underbrace{u^l}_{\rho_\beta^l\circ u}\circ
\varphi_\beta^{-1})\|_{L^q(\varphi_\beta(U_\beta))}^q\\
&\preceq \sum_{l=1}^r\sum_{|\eta|\leq k}\|\partial^\eta(u^l\circ
\varphi_\beta^{-1}
)\|_{L^q(\varphi_\beta(U_\beta))}^q=\sum_{l=1}^r\|u^l\circ
\varphi_\beta^{-1} \|^q_{W^{k,q}(\varphi_\beta(U_\beta))}\\
&\simeq \|u\|^q_{W^{k,q}(M,E)}
\end{align*}
\item \textbf{Case 2:} Now let $u$ be an arbitrary element of
$C^\infty(M,E)$.
\begin{align*}
|u|_{W^{k,q}(M,E)}&=|\sum_{\alpha=1}^N\psi_\alpha
u|_{W^{k,q}(M,E)}\leq \sum_{\alpha=1}^N|\psi_\alpha
u|_{W^{k,q}(M,E)}\\
&\stackrel{\textrm{Case 1}}{\preceq}\sum_{\alpha=1}^N \|\psi_\alpha u
\|_{W^{k,q}(M,E)}\stackrel{\textrm{Fact
3}}{\simeq}\sum_{\alpha=1}^N\sum_{l=1}^r \|\rho^l_\alpha\circ
(\psi_\alpha u)\circ \varphi_\alpha^{-1}
\|_{W^{k,q}(\varphi_\alpha(U_\alpha))}\\
&\simeq \|u\|_{W^{k,q}(M,E)}
\end{align*}
\end{enumerate}
\end{itemizeX}
\section{Some Results on Differential Operators}
Let $M^n$ be a compact smooth manifold. Let $E$ and $\tilde{E}$ be two vector bundles over $M$ of
ranks $r$ and $\tilde{r}$, respectively. A linear operator $P:
C^{\infty} (M,E)\rightarrow \Gamma (M,\tilde{E})$ is called \textbf{local} if
\begin{equation*}
\forall\, u\in C^{\infty}(M,E)\qquad\textrm{supp}\, Pu\subseteq
\textrm{supp}\,u
\end{equation*}
If $P$ is a local operator, then it is possible to have a well-defined notion of restriction of $P$ to open sets $U\subseteq M$,
that is, if $P: C^{\infty} (M,E)\rightarrow \Gamma (M,\tilde{E})$ is
local and $U\subseteq M$ is open, then we can define a map
\begin{equation*}
P|_U: C^{\infty} (U,E_U)\rightarrow \Gamma (U,\tilde{E}_U)
\end{equation*}
with the property that
\begin{equation*}
\forall\, u\in C^{\infty}(M,E)\qquad (Pu)|_U=P|_U(u|_U)
\end{equation*}
Indeed suppose $u,\tilde{u}\in C^\infty(M,E)$ agree on $U$, then as a result of $P$ being local we have
\begin{equation*}
\textrm{supp}\,(Pu-P\tilde{u})\subseteq \textrm{supp}\,(u-\tilde{u})\subseteq M\setminus U
\end{equation*}
Therefore if $u|_U=\tilde{u}|_U$, then $(Pu)|_U=(P\tilde{u})|_U$. Thus, if $v\in C^{\infty} (U,E_U)$ and $x\in U$, we can define $(P|_U)(v)(x)$ as follows: choose any $u\in C^{\infty}(M,E)$ such that $u=v$ on a neighborhood of $x$ and then let $(P|_U)(v)(x)=(Pu)(x)$.\\\\
Recall that for any nonempty set $V$, $\textrm{Func} (V,\reals^t)$
 denotes the vector space of all functions from $V$ to $\reals^t$.
By the \textbf{local representation of $P$} with respect to the total
trivialization triples $(U,\varphi,\rho)$ of $E$ and
$(U,\varphi,\tilde{\rho})$ of $\tilde{E}$ we mean the linear
transformation $Q: C^{\infty}(\varphi(U),\reals^r)\rightarrow
\textrm{Func}(\varphi(U),\reals^{\tilde{r}})$ defined by
\begin{equation*}
Q(f)=\tilde{\rho}\circ P(\rho^{-1}\circ f\circ \varphi)\circ
\varphi^{-1}
\end{equation*}
Note that $\rho^{-1}\circ f\circ \varphi$ is a section of
$E_U\rightarrow U$. Also note that for all $u\in C^{\infty}(M,E)$
\begin{equation}\lab{eqnjan1}
\tilde{\rho}\circ (P(u|_U))\circ
\varphi^{-1}=Q(\rho\circ(u|_U)\circ \varphi^{-1} )
\end{equation}
Let's denote the components of $f\in
C^{\infty}(\varphi(U),\reals^r)$ by $(f^1,\cdots,f^r)$. Then
we can write $Q(f^1,\cdots,f^r)=(h^1,\cdots,h^{\tilde{r}})$ where for all $1\leq k\leq \tilde{r}$
\begin{equation*}
h^k=\pi_k\circ Q(f^1,\cdots,f^r)\stackrel{\textrm{$Q$ is linear}}{=}\pi_k\circ Q(f^1,0,\cdots,0)+\cdots+\pi_k\circ Q(0,\cdots,0,f^r)
\end{equation*}
So if for each $1\leq k\leq \tilde{r}$ and $1\leq i\leq r$ we define $Q_{ki}:
C^{\infty}(\varphi(U),\reals)\rightarrow
\textrm{Func}(\varphi(U),\reals)$ by
\begin{equation*}
Q_{ki}(g)=\pi_k\circ Q(0,\cdots,0,\underbrace{g}_{\textrm{$i^{th}$ position}},0,\cdots,0)
\end{equation*}
then we have
\begin{equation*}
Q(f^1,\cdots,f^r)=(\sum_{i=1}^r Q_{1i}(f^i),\cdots,\sum_{i=1}^r
Q_{\tilde{r}i}(f^i))
\end{equation*}
In particular, note that the $s^{th}$ component of $\tilde{\rho}\circ Pu \circ \varphi^{-1}$, that is $\tilde{\rho}^s\circ Pu \circ \varphi^{-1}$, is equal to the $s^{th}$ component of $Q(\rho^1\circ u\circ \varphi^{-1},\cdots,\rho^r\circ u\circ \varphi^{-1} )$ (see Equation \ref{eqnjan1}) which is equal to
\begin{equation*}
\sum_{i=1}^r Q_{si}(\rho^i\circ u\circ \varphi^{-1})
\end{equation*}
\begin{theorem}\lab{thmjan2}
Let $M^n$ be a compact smooth manifold. Let $P:C^{\infty}(M,E)\rightarrow \Gamma(M,\tilde{E})$ be a local
operator. Let $\Lambda=\{(U_\alpha,\varphi_\alpha,\rho_\alpha,\psi_\alpha)\}_{1\leq \alpha \leq N}$ and $\tilde{\Lambda}=\{(U_\alpha,\varphi_\alpha,\tilde{\rho}_\alpha,\psi_\alpha)\}_{1\leq \alpha \leq N}$ be two augmented total trivialization atlases for $E$ and $\tilde{E}$, respectively. Suppose the atlas $\{(U_\alpha,\varphi_\alpha)\}_{1\leq \alpha\leq N}$ is GL compatible with itself. For each $1\leq \alpha\leq N$, let $Q^{\alpha}$ denote the local representation of $P$
with respect to the total trivialization triples
$(U_\alpha,\varphi_\alpha, \rho_\alpha)$ and
$(U_\alpha,\varphi_\alpha, \tilde{\rho}_\alpha)$ of $E$ and
$\tilde{E}$, respectively. Suppose for each $1\leq\alpha\leq N$, $1\leq i\leq \tilde{r}$, and $1\leq j\leq r$,
\begin{enumerateX}
\item
$Q^\alpha_{ij}:
(C_c^{\infty}(\varphi_\alpha(U_\alpha)),\|.\|_{e,q})\rightarrow
W^{\tilde{e},\tilde{q}}(\varphi_\alpha(U_\alpha))$ is well-defined and continuous and does not increase support, and
\item if $\Omega=\varphi_\alpha(U_\alpha)$ or $\Omega$ is an open bounded subset of $\varphi_\alpha(U_\alpha)$ with Lipschitz continuous boundary, then for all $h\in C^{\infty}(\Omega)$ and $\eta,\psi\in C_c^{\infty}(\varphi_\alpha(U_\alpha))$ with $\eta h\in C_c^{\infty}(\Omega)$, we have
    \begin{equation}\lab{eqnspring2017}
    \|\eta[Q^\alpha_{ij},\psi]h\|_{W^{\tilde{e},\tilde{q}}(\Omega)}\preceq \|\eta h\|_{W^{e,q}(\Omega)}
    \end{equation}
\end{enumerateX}
where $[Q^\alpha_{ij},\psi]h:=Q^\alpha_{ij}(\psi h)-\psi Q^\alpha_{ij}(h)$ (the implicit constant may depend on $\eta$ and $\psi$ but it does not depend on $h$).\\ Then
\begin{itemize}
\item $P(C^{\infty}(M,E))\subseteq
W^{\tilde{e},\tilde{q}}(M,\tilde{E};\tilde{\Lambda})$
\item $P: (C^{\infty}(M,E),\| .\|_{e,q})\rightarrow
W^{\tilde{e},\tilde{q}}(M,\tilde{E};\tilde{\Lambda})$ is continuous and so it can be
extended to a continuous linear map $P: W^{e,q}(M,E;\Lambda)\rightarrow
W^{\tilde{e},\tilde{q}}(M,\tilde{E};\tilde{\Lambda})$.
\end{itemize}
\end{theorem}
\begin{proof}
First note that
\begin{align*}
& \parallel P
u\parallel_{W^{\tilde{e},\tilde{q}}(M,\tilde{E};\tilde{\Lambda})}=\sum_{\alpha=1}^N
\sum_{i=1}^{\tilde{r}}
\parallel \tilde{\rho}^i_{\alpha}\circ (\psi_{\alpha}(P u))\circ \varphi_\alpha^{-1}
\parallel_{W^{\tilde{e},\tilde{q}}(\varphi_\alpha(U_\alpha))}\\
&\parallel u\parallel_{W^{e,q}(M,E;\Lambda)}=\sum_{\alpha=1}^N
\sum_{j=1}^r\parallel \rho_\alpha^j\circ (\psi_{\alpha}u)\circ
\varphi_\alpha^{-1}
\parallel_{W^{e,q}(\varphi_\alpha(U_\alpha))}
\end{align*}
It is enough to show that for all $1\leq \alpha\leq N$, $1\leq
i\leq \tilde{r}$
\begin{equation*}
\parallel
\tilde{\rho}^i_{\alpha}\circ(\psi_{\alpha}(P u))\circ
\varphi_\alpha^{-1}
\parallel_{W^{\tilde{e},\tilde{q}}(\varphi_\alpha(U_\alpha))}\preceq \sum_{\beta=1}^N \sum_{j=1}^r\parallel
\rho_\beta^j\circ(\psi_{\beta}u)\circ \varphi_\beta^{-1}
\parallel_{W^{e,q}(\varphi_\beta(U_\beta))}
\end{equation*}
We have
\begin{align*}
\parallel
\tilde{\rho}_\alpha^i&\circ(\psi_{\alpha}(P u))\circ
\varphi_\alpha^{-1}
\parallel_{W^{\tilde{e},\tilde{q}}(\varphi_\alpha(U_\alpha))}=\parallel (\psi_\alpha\circ\varphi_\alpha^{-1})\cdot(\tilde{\rho}_\alpha^i\circ(P u)\circ
\varphi_\alpha^{-1})\parallel_{W^{\tilde{e},\tilde{q}}(\varphi_\alpha(U_\alpha))}\\
&\leq\sum_{j=1}^r \parallel (\psi_\alpha\circ\varphi_\alpha^{-1})\cdot
Q_{ij}^{\alpha}(\rho_{\alpha}^j\circ(\sum_{\beta=1}^N \psi_\beta
u)\circ \varphi_\alpha^{-1})
\parallel_{W^{\tilde{e},\tilde{q}}(\varphi_\alpha(U_\alpha))}\\
& {\fontsize{10}{10}{\textrm{(see the paragraph above Theorem \ref{thmjan2})}}}\\
&\leq \sum_{\beta=1}^N \sum_{j=1}^r \parallel (\psi_\alpha\circ\varphi_\alpha^{-1})\cdot
Q_{ij}^{\alpha}(\rho_{\alpha}^j\circ(\psi_\beta u)\circ
\varphi_\alpha^{-1})
\parallel_{W^{\tilde{e},\tilde{q}}(\varphi_\alpha(U_\alpha))}
\end{align*}
\begin{align*}
&\preceq  \sum_{\beta=1}^N \sum_{j=1}^r\bigg(\parallel
Q_{ij}^{\alpha}(\rho_{\alpha}^j\circ(\psi_\alpha\psi_\beta u)\circ
\varphi_\alpha^{-1})
\parallel_{W^{\tilde{e},\tilde{q}}(\varphi_\alpha(U_\alpha))}\\
&\hspace{3cm}+\parallel
[Q_{ij}^{\alpha},\psi_\alpha\circ\varphi_\alpha^{-1}](\rho_{\alpha}^j\circ(\psi_\beta u)\circ
\varphi_\alpha^{-1})
\parallel_{W^{\tilde{e},\tilde{q}}(\varphi_\alpha(U_\alpha))}\bigg)\\
&\preceq  \sum_{\beta=1}^N \sum_{j=1}^r\parallel
\rho_{\alpha}^j\circ(\psi_\alpha\psi_\beta u)\circ \varphi_\alpha^{-1}
\parallel_{W^{e,q}(\varphi_\alpha(U_\alpha))}\\
&\hspace{3cm}+\sum_{\beta=1}^N \sum_{j=1}^r  \parallel
[Q_{ij}^{\alpha},\psi_\alpha\circ\varphi_\alpha^{-1}](\rho_{\alpha}^j\circ(\psi_\beta u)\circ
\varphi_\alpha^{-1})
\parallel_{W^{\tilde{e},\tilde{q}}(\varphi_\alpha(U_\alpha))}
\end{align*}
Note that $\rho_{\alpha}^j\circ(\psi_\alpha\psi_\beta u)\circ \varphi_\alpha^{-1}=(\psi_\alpha\psi_\beta\circ\varphi_\alpha^{-1})(\rho_{\alpha}^j\circ u\circ \varphi_\alpha^{-1})$ and $[Q_{ij}^{\alpha},\psi_\alpha\circ\varphi_\alpha^{-1}](\rho_{\alpha}^j\circ(\psi_\beta u)\circ
\varphi_\alpha^{-1})$ both have compact support in $\varphi_\alpha(U_\alpha\cap U_\beta)$. So it follows from Corollary \ref{corofallusef1} that
\begin{align*}
&\parallel
\rho_{\alpha}^j\circ(\psi_\alpha\psi_\beta u)\circ \varphi_\alpha^{-1}
\parallel_{W^{e,q}(\varphi_\alpha(U_\alpha))}\simeq\parallel
\rho_{\alpha}^j\circ(\psi_\alpha\psi_\beta u)\circ \varphi_\alpha^{-1}
\parallel_{W^{e,q}(\varphi_\alpha(U_\alpha\cap U_\beta))}\\
& \parallel
[Q_{ij}^{\alpha},\psi_\alpha\circ\varphi_\alpha^{-1}](\rho_{\alpha}^j\circ(\psi_\beta u)\circ
\varphi_\alpha^{-1})
\parallel_{W^{\tilde{e},\tilde{q}}(\varphi_\alpha(U_\alpha))}\\
&\hspace{3.5cm}\simeq \parallel
[Q_{ij}^{\alpha},\psi_\alpha\circ\varphi_\alpha^{-1}](\rho_{\alpha}^j\circ(\psi_\beta u)\circ
\varphi_\alpha^{-1})
\parallel_{W^{\tilde{e},\tilde{q}}(\varphi_\alpha(U_\alpha\cap U_\beta))}
\end{align*}
Let $\xi\in C_c^\infty(U_\alpha)$ be such that $\xi=1$ on $\textrm{supp}\,\psi_\alpha$. Clearly we have
\begin{align*}
&\parallel
[Q_{ij}^{\alpha},\psi_\alpha\circ\varphi_\alpha^{-1}](\rho_{\alpha}^j\circ(\psi_\beta u)\circ
\varphi_\alpha^{-1})
\parallel_{W^{\tilde{e},\tilde{q}}(\varphi_\alpha(U_\alpha\cap U_\beta))}\\
&\hspace{3cm}=
\parallel
(\xi\circ \varphi_\alpha^{-1})[Q_{ij}^{\alpha},\psi_\alpha\circ\varphi_\alpha^{-1}](\rho_{\alpha}^j\circ(\psi_\beta u)\circ
\varphi_\alpha^{-1})
\parallel_{W^{\tilde{e},\tilde{q}}(\varphi_\alpha(U_\alpha\cap U_\beta))}\\
&\hspace{2.5cm}\stackrel{\textrm{Equation \ref{eqnspring2017}}}{\preceq} \parallel
\rho_{\alpha}^j\circ(\xi\psi_\beta u)\circ
\varphi_\alpha^{-1}
\parallel_{W^{e,q}(\varphi_\alpha(U_\alpha\cap U_\beta))}
\end{align*}
Therefore
\begin{align*}
\parallel
\tilde{\rho}^i_{\alpha}&\circ(\psi_{\alpha}(P u))\circ
\varphi_\alpha^{-1}
\parallel_{W^{\tilde{e},\tilde{q}}(\varphi_\alpha(U_\alpha))}\\
&\preceq
\sum_{\beta=1}^N\sum_{j=1}^r\parallel
\rho_{\alpha}^j\circ(\psi_\alpha\psi_\beta u)\circ \varphi_\alpha^{-1}
\parallel_{W^{e,q}(\varphi_\alpha(U_\alpha\cap U_\beta))}\\
&\hspace{2cm}+\sum_{\beta=1}^N\sum_{j=1}^r\parallel
\rho_{\alpha}^j\circ(\xi\psi_\beta u)\circ \varphi_\alpha^{-1}
\parallel_{W^{e,q}(\varphi_\alpha(U_\alpha\cap U_\beta))}\\
&=\sum_{\beta=1}^N\sum_{j=1}^r\parallel
\rho_\alpha^j\circ(\psi_\alpha\psi_\beta u)\circ
\varphi_\beta^{-1}\circ\varphi_\beta\circ\varphi_\alpha^{-1}
\parallel_{W^{e,q}(\varphi_\alpha(U_\alpha\cap U_\beta))}\\
&\hspace{2cm} +\sum_{\beta=1}^N\sum_{j=1}^r\parallel
\rho_\alpha^j\circ(\xi\psi_\beta u)\circ
\varphi_\beta^{-1}\circ\varphi_\beta\circ\varphi_\alpha^{-1}
\parallel_{W^{e,q}(\varphi_\alpha(U_\alpha\cap U_\beta))}\\
&\stackrel{\textrm{Theorem \ref{winter105}}}{\preceq} \sum_{\beta=1}^N\sum_{j=1}^r\parallel
\rho_\alpha^j\circ(\psi_\alpha\psi_\beta u)\circ \varphi_\beta^{-1}
\parallel_{W^{e,q}(\varphi_\beta(U_\alpha\cap U_\beta))}\\
&\hspace{3cm} +\sum_{\beta=1}^N\sum_{j=1}^r\parallel
\rho_\alpha^j\circ(\xi\psi_\beta u)\circ \varphi_\beta^{-1}
\parallel_{W^{e,q}(\varphi_\beta(U_\alpha\cap U_\beta))}
\end{align*}
So it is enough to prove that $\parallel
\rho_\alpha^j\circ(\psi_\alpha\psi_\beta u)\circ \varphi_\beta^{-1}
\parallel_{W^{e,q}(\varphi_\beta(U_\alpha\cap U_\beta))}$ and  $\parallel
\rho_\alpha^j\circ(\xi\psi_\beta u)\circ \varphi_\beta^{-1}
\parallel_{W^{e,q}(\varphi_\beta(U_\alpha\cap U_\beta))}$ can be bounded by $\sum_{\beta=1}^N \sum_{j=1}^r\parallel
\rho_\beta^j\circ(\psi_{\beta}u)\circ \varphi_\beta^{-1}
\parallel_{W^{e,q}(\varphi_\beta(U_\beta))}$. Since this can be done in the exact same way as the proof of Theorem \ref{thmwinter20181}, we do not repeat the argument here.
\end{proof}

Here we will discuss one simple application of the above theorem. Let $(M^n,g)$ be a compact Riemannian manifold with $g\in W^{s,p}(M,T^2M)$,
$sp>n$, and $s\geq 1$. Consider $d: C^{\infty}(M)\rightarrow
C^{\infty}(T^{*}M)$. The local representations are all assumed to be with respect to charts in a super nice total trivialization atlas that is GL compatible with itself. The local representation of $d$ is
$Q:C^{\infty}(\varphi(U))\rightarrow
C^{\infty}(\varphi(U),\reals^n)$ which is defined by
\begin{align*}
Q(f)(a)&=\tilde{\rho}\circ d(\rho^{-1}\circ f \circ \varphi)\circ
\varphi^{-1}(a)\\
&=\tilde{\rho}\circ (\frac{\partial f}{\partial
x^i}|_{\varphi(\varphi^{-1}(a))}dx^i|_{\varphi^{-1}(a)})\\
&=(\frac{\partial f}{\partial x^1}|_a,\cdots,\frac{\partial
f}{\partial x^n}|_a)
\end{align*}
Here we used $\rho=Id$ and the fact that if $g:M\rightarrow
\reals$ is smooth, then
\begin{equation*}
(dg)(p)=\frac{\partial (g\circ \varphi^{-1})}{\partial
x^i}|_{\varphi(p)}dx^i|_p
\end{equation*}
Clearly each component of $Q$ is a continuous operator from
$(C_c^{\infty}(\varphi(U)),\|.\|_{e,q})$ to
$W^{e-1,q}(\varphi(U))$ (see Theorem \ref{winter88}; note that $\varphi(U)=\reals^n$). Also considering that
\begin{equation*}
\forall\,1\leq i\leq n\qquad Q_{i1}(h)=\frac{\partial h}{\partial x^i},\quad [Q_{i1},\psi]h=\frac{\partial \psi}{\partial x^i}h
\end{equation*}
the required property for $[Q_{i1},\psi]$ holds true. Hence $d$ can be viewed as a continuous
operator from $W^{e,q}(M)$ to $W^{e-1,q}(T^{*}M)$.\\

Several other interesting applications of Theorem \ref{thmjan2} can be found in \cite{holstbehzadan2017c}.

\section*{Acknowledgments}
   \label{sec:ack}

MH was supported in part by NSF Awards~1262982, 1318480, and 1620366.
AB was supported by NSF Award~1262982.

\bibliographystyle{abbrv}
\bibliography{refssob}

\begin{thebibliography}{10}

\bibitem{Adams75}
R.~A. Adams.
\newblock {\em Sobolev Spaces}.
\newblock Academic Press, New York, 1975.

\bibitem{32}
R.~A. Adams and J.~J.~F. Fournier.
\newblock {\em Sobolev Spaces}.
\newblock New York: Academic Press, second edition, 2003.

\bibitem{Apostol74}
T.~Apostol.
\newblock {\em Mathematical Analysis}.
\newblock Pearson, 1974.
\newblock Second edition.

\bibitem{Aubin1998}
T.~Aubin.
\newblock {\em Some Nonlinear Problems in {R}iemannian Geometry}.
\newblock Springer, 1998.

\bibitem{Bastos2014}
M.~A. Bastos, A.~Lebre, S.~Samko, and I.~M. Spitkovsky.
\newblock {\em Operator Theory, Operator Algebras and Applications}.
\newblock Birkh{\"a}user, 2014.

\bibitem{holstbehzadan2015b}
A.~Behzadan and M.~Holst.
\newblock Multiplication in {S}obolev spaces, revisited.
\newblock Available as arXiv:1512.07379 [math.AP], 2015.

\bibitem{holstbehzadan2017c}
A.~Behzadan and M.~Holst.
\newblock On certain geometric operators between {S}obolev spaces of sections
  of tensor bundles on compact manifolds equipped with rough metrics.
\newblock 2018.

\bibitem{holstbehzadan2018c}
A.~Behzadan and M.~Holst.
\newblock Some remarks on the space of locally {S}obolev-{S}lobodeckij
  functions.
\newblock Available as arXiv:1806.02188 [math.AP], 2018.

\bibitem{holstbehzadan2018d}
A.~Behzadan and M.~Holst.
\newblock Some remarks on {$W^{s,p}$} interior elliptic regularity estimates.
\newblock 2018.

\bibitem{33}
P.~K. Bhattacharyya.
\newblock {\em Distributions: {G}eneralized {f}unctions with {A}pplications in
  {S}obolev {S}paces}.
\newblock de Gruyter, 2012.

\bibitem{38}
H.~Brezis and P.~Mironescu.
\newblock {G}agliardo-{N}irenberg, composition and products in fractional
  {S}obolev spaces.
\newblock {\em Journal of Evolution Equations}, 1(4):387--404, 2001.

\bibitem{debnath2005}
L.~Debnath and P.~Mikusinski.
\newblock {\em Introduction to Hilbert Spaces with Applications}.
\newblock Academic Press, 2005.
\newblock Third edition.

\bibitem{Demengel2012}
F.~Demengel and G.~Demengel.
\newblock {\em Function {S}paces for the {T}heory of {E}lliptic {P}artial
  {D}ifferential {E}quations}.
\newblock Springer, 2012.

\bibitem{13}
B.~Driver.
\newblock Analysis tools with applications.
\newblock 2003.

\bibitem{Eichhorn2007}
J.~Eichhorn.
\newblock {\em Global Analysis on Open Manifolds}.
\newblock Nova Science Publishers, 2007.

\bibitem{Evans2010}
L.~Evans.
\newblock {\em Partial Differential Equations}.
\newblock American Mathematical Society, 2010.
\newblock Second edition.

\bibitem{Folland07}
G.~Folland.
\newblock {\em Real Analysis: Modern Techniques and Their Applications}.
\newblock Wiley, 2007.
\newblock Second edition.

\bibitem{Gavrilov2008}
A.~V. Gavrilov.
\newblock Higher covariant derivatives.
\newblock {\em Siberian Mathematical Journal}, 49(6):997--1007, 2008.

\bibitem{Gris85}
P.~Grisvard.
\newblock {\em Elliptic Problems in Nonsmooth Domains}.
\newblock Pitman Publishing, Marshfield, MA, 1985.

\bibitem{grosse1}
N.~Grosse and C.~Schneider.
\newblock {S}obolev spaces on {R}iemannian manifolds with bounded geometry:
  general coordinates and traces.
\newblock {\em Mathematische Nachrichten}, 286:1586--1613, 2013.

\bibitem{Grosser2001}
M.~Grosser, M.~Kunzinger, M.~Oberguggenberger, and R.~Steinbauer.
\newblock {\em Geometric Theory of Generalized Functions with Applications to
  General Relativity}.
\newblock Springer, 2001.

\bibitem{9}
G.~Grubb.
\newblock {\em Distributions and Operators}.
\newblock Springer-Verlag, New York, NY, 2009.

\bibitem{Hebey96}
E.~Hebey.
\newblock {\em Sobolev spaces on {R}iemannian manifolds}, volume 1635 of {\em
  Lecture notes in mathematics}.
\newblock Springer, Berlin, New York, 1996.

\bibitem{Hebey2000}
E.~Hebey.
\newblock {\em Nonlinear Analysis on Manifolds: {S}obolev Spaces and
  Inequalities}.
\newblock American Mathematical Society, 2000.

\bibitem{Inc2013}
H.~Inci, T.~Kappeler, and P.~Topalov.
\newblock {\em On the Regularity of the Composition of Diffeomorphisms}, volume
  226.
\newblock Memoirs of the American Mathematical Society, 2013.

\bibitem{Moore2009}
{J}ohn~{D}ouglas Moore.
\newblock Lectures on differential geometry.
\newblock 2009.

\bibitem{Lee2}
J.~M. Lee.
\newblock {\em {R}iemannian {M}anifolds: {A}n {I}ntroduction to {C}urvature}.
\newblock Springer, 1997.

\bibitem{Lee1}
J.~M. Lee.
\newblock {\em Introduction to smooth manifolds}.
\newblock Springer, 2002.

\bibitem{Lee3}
J.~M. Lee.
\newblock {\em Introduction to {S}mooth {M}anifolds}.
\newblock Springer, 2012.
\newblock Second edition.

\bibitem{12}
E.~D. Nezza, G.~Palatucci, and E.~Valdinoci.
\newblock Hitchhiker's guide to the fractional {S}obolev spaces.
\newblock {\em Bulletin des Sciences Mathématiques}, 136(5):521--573, 2012.

\bibitem{Palais65}
R.~Palais.
\newblock {\em Seminar on the {A}tiyah-{S}inger {I}ndex {T}heorem}.
\newblock Princeton University Press, Princeton, 1965.

\bibitem{18}
M.~Renardy and R.~Rogers.
\newblock {\em An Introduction to Partial Differential Equations}.
\newblock Springer-Verlag, Berlin Heildelberg New York, second edition, 2004.

\bibitem{Reus1}
M.~D. Reus.
\newblock An introduction to functional spaces.
\newblock Master's thesis, Utrecht University, 2011.

\bibitem{Rudi73}
W.~Rudin.
\newblock {\em Functional Analysis}.
\newblock McGraw-Hill, New York, NY, 1973.

\bibitem{37}
T.~Runst and W.~Sickel.
\newblock {\em {S}obolev Spaces of Fractional Order, {N}emytskij Operators, and
  Nonlinear Partial Differential Equations}.
\newblock Walter de Gruyter, 1996.

\bibitem{Taubes2011}
C.~Taubes.
\newblock {\em Differential Geometry: Bundles, Connections, Metrics and
  Curvature}.
\newblock Oxford University Press, 2011.

\bibitem{Treves1}
F.~Treves.
\newblock {\em Topological {V}ector {S}paces, {D}istributions and {K}ernels}.
\newblock ACADEMIC PRESS, New York, London, 1967.

\bibitem{36}
H.~Triebel.
\newblock {\em Interpolation {T}heory, {F}unction {S}paces, {D}ifferential
  {O}perators}.
\newblock North-Holland Publishing Company, 1977.

\bibitem{Trie83}
H.~Triebel.
\newblock {\em Theory of {F}unction {S}paces}, volume~78 of {\em Monographs in
  Mathematics}.
\newblock Birkh\"auser Verlag, Basel, 1983.

\bibitem{Trie92}
H.~Triebel.
\newblock {\em Theory of {F}unction {S}paces. {II}}, volume~84 of {\em
  Monographs in Mathematics}.
\newblock Birkh\"auser Verlag, Basel, 1992.

\bibitem{Trie2002}
H.~Triebel.
\newblock Function spaces in {L}ipschitz domains and on {L}ipschitz manifolds.
  {C}haracteristic functions as pointwise multipliers.
\newblock {\em Revista Matematica Complutense}, 15(2):475--524, 2002.

\bibitem{loring2011}
L.~Tu.
\newblock {\em An {I}ntroduction to {M}anifolds}.
\newblock Springer, 2011.
\newblock Second edition.

\bibitem{Wal2004}
G.~Walschap.
\newblock {\em Metric {S}tructures in {D}ifferential {G}eometry}.
\newblock Springer, 2004.

\end{thebibliography}
\end{document}